\documentclass[12pt]{amsart}
\pdfoutput=1
\usepackage[left=1in,top=1in,right=1in,bottom=1in,letterpaper]{geometry}
\usepackage{epsfig}
\usepackage{amsmath}

\usepackage{amssymb}
\usepackage{amscd}
\usepackage{fixltx2e}

\usepackage{url}
\usepackage{verbatim} 
\usepackage{color}
\usepackage{epsfig}
\usepackage{stmaryrd}
\usepackage{amsthm}
\usepackage{mathrsfs}

\usepackage{bm}
\usepackage{longtable}
\usepackage{booktabs}

\usepackage{hyperref}

\makeindex

\title{The invariant measures of some infinite interval exchange maps}
\author{W. Patrick Hooper}
\thanks{Supported by N.S.F. Postdoctoral Fellowship DMS-0803013, N.S.F. grant DMS-1101233 and
a PSC-CUNY Award (funded by The Professional Staﬀ Congress and The City University of New York).}
\address{Department of Mathematics, City College of New York}
\email{whooper@ccny.cuny.edu}
\date{\today}

\newtheorem{theorem}{Theorem}
\newtheorem{proposition}[theorem]{Proposition}
\newtheorem{lemma}[theorem]{Lemma}
\newtheorem{corollary}[theorem]{Corollary}

\newtheorem{claim}[theorem]{Claim}
\theoremstyle{definition}
\newtheorem{definition}[theorem]{Definition}
\theoremstyle{definition}
\newtheorem{remark}[theorem]{Remark}
\theoremstyle{definition}
\newtheorem{example}[theorem]{Example}

\def\N{\mathbb{N}}%
\def\P{\mathbb{P}}%
\def\Q{\mathbb{Q}}%
\def\R{\mathbb{R}}%
\def\Z{\mathbb{Z}}%
\def\RP{\mathbb{RP}}%

\def\GL{\textit{GL}}
\def\SL{\textit{SL}}
\def\SO{\textit{SO}}

\renewcommand{\hom}[1]{\ensuremath{{\llbracket #1 \rrbracket}}}%
\def\ker{\textit{ker}}%
\newcommand{\nullset}{\emptyset}

\def\H{\mathbb H}%

\def\isomS1{\textit{Isom}_+(S^1)}
\def\rt3{\sqrt{3}}

\def\sgn{\mathrm{sgn}} 

\def\v{{\bf v}}
\def\w{{\bf w}}

\numberwithin{equation}{section}
\numberwithin{theorem}{section}

\makeatletter
\def\@strippedMR{}
\def\@scanforMR#1#2#3\endscan{%
  \ifx#1M\ifx#2R\def\@strippedMR{#3}%
  \else\def\@strippedMR{#1#2#3}%
  \fi\fi}
\renewcommand\MR[1]{\relax\ifhmode\unskip\spacefactor3000 \space\fi
  \@scanforMR#1\endscan
  MR\MRhref{\@strippedMR}{\@strippedMR}}
\makeatother 

\newcommand{\rhoa}[1]{\rho_\lambda^{#1}}

\def\Alpha{{\mathcal A}}
\def\Beta{{\mathcal B}} 
\def\E{{\mathcal E}} 
\def\F{{\mathcal F}} 
\def\G{{\mathcal G}} 
\def\H{{\mathbb H}} 
\def\V{{\mathcal V}} 
\def\I{{\mathbf I}} 
\def\sC{{\mathcal C}} 
\def\sD{{\mathcal D}} 

\def\North{{\mathscr N}}
\def\East{{\mathscr E}}

\def\Aff{{\textit A\hspace{-1pt}f\hspace{-3pt}f}}

\def\sm{{\smallsetminus}}

\def\flow{F_{\bm \theta}}
\def\ret{T_{\bm \theta}}%
\def\M{{\mathcal M}}

\def\Vert{{V}}
\def\Edge{{E}}

\def\f{{\mathbf f}}
\def\g{{\mathbf g}}
\def\k{{\mathbf k}}
\def\u{{\mathbf u}}
\def\w{{\mathbf w}}
\def\x{{\mathbf x}}
\def\y{{\mathbf y}}
\def\Chi{{\mathcal X}}
\def\Surv{{\mathcal S}} 
\def\Rn{{\mathcal R}} 
\def\Act{{\Upsilon}}
\def\A{{\mathbf A}}
\def\Ho{{\mathbf H}}
\def\K{{\mathbf K}}
\def\Vo{{\mathbf V}}

\def\Q{{\mathcal Q}} 

\def\hol{{\mathit hol}}
\def\cyl{{\mathit cyl}}
\def\Shrink{{\textit Shr}} 
\def\Exp{{\textit Exp}} 
\def\cl{{\textit cl}} 
\def\Zem{{\mathcal Z}}
\def\Fmap{{\mathbf F}}

\def\vv{\mathrm{v}}
\def\vx{\mathrm{x}}
\def\vy{\mathrm{y}}
\def\vw{\mathrm{w}}
\def\va{\mathrm{a}}
\def\vb{\mathrm{b}}
\def\vc{\mathrm{c}}
\def\vd{\mathrm{d}}

\def\val{\mathit{val}} 
\def\e{{\mathbf e}} 
\def\a{{\mathbf a}} 
\def\b{{\mathbf b}} 
\def\z{{\mathbf z}} 
\def\h{{\mathbf h}} 
\def\0{{\mathbf 0}} 
\def\sa{{\Sigma}} 

\def\SP{{\mathit {S\hspace{-2pt}P}}} 
\def\supp{{\mathit {supp}}}
\def\suppbar{{\overline{\mathit {supp}}}}
\def\Circ{{{\mathbb S}^1}} 
\newcommand{\transpose}[1]{{\vphantom{#1}}^{t}{#1}}
\def\GH{{\mathcal H}}
\def\1{{\mathbf 1}}

\def\Grp{{\mathbb G}}
\def\DM{{\mathbf D}}
\def\RM{{\mathbf R}}
\def\PM{{\mathbf P}}
\def\vo{{\mathrm o}}
\def\K{{\mathbf K}}
\def\z{{\lambda}}
\def\zz{{\mathbf z}}
\def\one{{\bm 1}}
\def\Nu{{\mathcal N}}
\newcommand{\red}[1]{\textcolor{red}{#1}}

\def\Coh{{H^1}}
\def\btheta{{\bm \theta}}

\makeatletter
\def\imod#1{\allowbreak\mkern10mu({\operator@font mod}\,\,#1)}
\makeatother

\newif\ifdraft\drafttrue
\draftfalse

\newcommand{\name}[1]{\label{#1}{\ifdraft{\textcolor{blue}{\{\textrm{#1}\}}}\else\ignorespaces\fi}}

\newcommand{\pagelink}[1]{\hyperlink{page.\getpagerefnumber{#1}}{\getpagerefnumber{#1}}}

\begin{document}

\begin{abstract}
We classify the locally finite ergodic invariant measures of certain infinite interval exchange transformations (IETs). These transformations naturally arise from return maps of the straight-line flow on certain translation surfaces, and the study of the invariant measures for these IETs is equivalent to the study of invariant measures for the straight-line flow in some direction on these translation surfaces. For the surfaces and directions for which our methods apply, we can characterize the locally finite ergodic invariant measures of the straight-line flow in a set of directions of Hausdorff dimension larger than 1/2. We promote this characterization to a classification in some cases.
For instance, when the surfaces admit a cocompact action by a nilpotent group, we prove
each ergodic invariant measure for the straight-line flow is a Maharam measure, and we describe precisely which Maharam measures arise. When the surfaces under consideration are finite area, the straight-line flows in the directions we understand are uniquely ergodic. Our methods apply to translation surfaces admitting multi-twists in a pair of cylinder decompositions in non-parallel directions.
\end{abstract}

\maketitle


\section{Introduction}
An {\em interval exchange transformation (IET)} is a bijective piecewise isometry from the interval $[0,1]$ to itself which is orientation preserving and has only finitely many discontinuities. 
These maps are natural generalizations of rotations, and interesting because they are simple systems of low symbolic complexity but nonetheless many phenomena that appear in these systems are not yet fully understood. Perhaps the most well-studied problem in the subject is the classification of the ergodic invariant measures. For instance, there are many results which guarantee unique ergodicity. See for instance \cite{V87}, \cite[Theorem 8.2]{V}, \cite{M92}, \cite[Theorem 1.1]{CE07}, and \cite[Theorem 4]{Trev14}.
Relevant surveys of the subject include \cite{MT}, \cite{Zorich06} and \cite{Y10}.

For our purposes, an {\em infinite interval exchange transformation} is an orientation preserving bijective piecewise isometry of an
interval in $\R$ of possibly infinite length with a countably infinite collection of discontinuities. 
Here many related questions become interesting. Is Lebesgue measure ergodic? Is there an invariant probability measure? If so, is it unique? 
An {\em interval of continuity} is an interval on which the
map is continuous. We call a measure {\em locally finite} if it is finite on the intervals of continuity of the map. 
What are the locally finite ergodic invariant measures? Questions of this type were first answered in \cite[Theorem 1.4]{ANSS02}, where the locally finite ergodic invariant measures of certain skew products were classified. 

This article contributes to the subject of infinite interval exchange transformations in several ways. Foremost, we introduce methods to characterize (and in many cases classify) the locally finite ergodic invariant measures of some infinite IETs. 
Our approach draws inspiration from the intimate connection between finite interval exchange maps and Teichm\"uller theory, and the inherent connection to translation surfaces. (Section \ref{sect:trans} defines translation surface.) 
This point of view reveals that many infinite interval exchange maps are unexpectedly interesting. Finally, we develop an interpretation of certain infinite interval exchange maps as ``deterministic random walks,'' and our results draw connections between these systems and  corresponding random walks. Our strongest measure classification results are corollaries to our measure characterization theorem, which utilizes the theory of random walks (on graphs and discrete groups) to promote a characterization to a classification.

\section{Organization and overview of this article}
\name{sect:organization}
The main result of this paper is a characterization of the invariant measures of a class of infinite IETs. This class arises from return maps of the straight-line flow on a class of infinite translation surfaces
produced using a construction of Thurston. Section \ref{sect:background} introduces translation surfaces,
the straight-line flow, affine automorphism groups, and Thurston's construction.
The main idea of this paper is to use the affine automorphism group of such surfaces to renormalize the straight-line flow.

In section \ref{sect:restatement}, we state the main results of this paper. Our results hold for straight-line flows
in ``renormalizable directions'' and the section begins by defining these directions. Subsection \ref{ss:orbit} contains the first of our two main results: an orbit equivalence result for straight-line flows on surfaces produced using Thurston's construction. In subsection \ref{ss:emc}, we describe the second of our main results. We describe a characterization of locally finite ergodic invariant measures for these straight-line flows. In the applicable cases, all ergodic measures arise by pulling back Lebesgue measure under orbit equivalences described by the first main result.

In section \ref{sect:hist}, we provide an abbreviated history of the subject of infinite IETs and infinite translation surfaces.

The remaining sections of the paper are devoted to the proof of our two main results. The beginning of section \ref{sect:outline} provides a birds eye view of the proof of these results. We fill in the details in the subsections of section \ref{sect:outline}, where we describe a sequence of results which build up to the proofs of these main results. The logic of the proofs of two main results is entirely contained in this section, but we state many results in this section which require substantial further work to prove. 
The proofs of these auxiliary results
are contained in the remaining sections, \S \ref{sect:symmetry}-\ref{sect:no valance one}. In addition, three of the appendices
survey relevant aspects of mathematical theories related to the proof:
\begin{itemize}
\item Appendix \ref{sect:coding} surveys the use of coding to understand the invariant measures of IETs. This well developed topic is described in the context of this work, and in particular discusses IETs coming from infinite translation surfaces.
\item Appendix \ref{sect:farkas} discusses a known generalization of the duality theorem due to Farkas.
\item Appendix \ref{sect:boundary} discusses the Martin boundary of a graph as it relates to the positive eigenfunctions of the adjacency operator.
\end{itemize}
The topics in these appendices play a role in our proofs, but we only contextualize known results.
In order to fully understand the proof of our main results, the reader should read 
Section \ref{sect:outline} and should refer to sections \ref{sect:symmetry}-\ref{sect:no valance one}
and appendices \ref{sect:coding}-\ref{sect:boundary} as necessary. 

To further aid the reader,
\begin{itemize}
\item At the end of the paper (pages \pagelink{sect:notations}-\pagelink{sect:notations end}), we provide a list of notations introduced.
\end{itemize}

This article includes five additional appendices which are used to provide context for our main results
and to describe some relevant special cases. In these special cases, we are often able to promote our measure characterization results to measure classification results
which give a complete description of the ergodic measures.
\begin{itemize}
\item Appendix \ref{sect:cylinder decompositions} provides a more geometric view of from Thurston's construction. Briefly, any surface admitting distinct decompositions into cylinders which support an affine multi-twist can be viewed as arising from Thurston's construction.
\item Appendix \ref{sect:skew rotations} describes some infinite IETs which appear. We stress skew product transformations and define Maharam measures.
\item Appendix \ref{sect:hyperbolic} describes some results which hold for surfaces built from Thurston's construction using a hyperbolic graph. In some cases, we can obtain ergodic measure classifications which seem unlike others that have appeared in the study of infinite IETs.
\item Appendix \ref{app:nilpotent} describes results which pertain to surfaces with nilpotent symmetry groups. We show that when our main results apply in this setting, the ergodic measures are always Maharam measures. This gives a generalization of a measure classification result of Aaronson-Nakada-Sarig-Solomyak \cite[Theorem 1.4]{ANSS02}.
\item Appendix \ref{sect:unique ergodicity} gives unique ergodicity results for both infinite IETs and straight-line flows that are studied in this paper.
\end{itemize} 
\section{Background on translation surfaces}
\name{sect:background}

\subsection{Translation surfaces}
\name{sect:trans}

Let $\Sigma$ be a connected topological surface. A {\em translation atlas on $\Sigma$} is an atlas of local homeomorphisms from open sets of $\Sigma$ to the plane so that the transition functions are translations. Such an atlas specifies local
coordinates to the plane which are canonical up to translation. Formally, a {\em translation surface} is a connected topological surface together with a maximal translation atlas. 

We will now give a more utilitarian viewpoint. 
A translation surface $S$ \label{not:S} can be formed from a disjoint collection of convex polygons $\{P_i \subset \R^2\}_{i \in \Lambda}$ with edges glued in pairs by translations. Let $\Vert \subset S$ denote the collection of (equivalence classes of) vertices of polygons in $S$. We consider the points of $\Vert$ to be singularities. Formally,
we only have a translation structure on $S \smallsetminus \Vert$, but we will abuse our definition by calling $S$ a translation surface and we will work with the points in $\Vert$. For example, a closed translation surface can be formed by identifying edges of a finite collection of polygons. In this case, the points of $\Vert$ are cone singularities, whose cone angle is an integral multiple of $2 \pi$. In this paper, our surfaces will be built from countably many polygons, so we will see more exotic singularities but their structure will not concern us. (See \cite{BV13} for an analysis some of the singularities that can appear.)

\subsection{The straight-line flow}
\name{sect:straight-line flow}
Let $\Circ$ \label{not:Circ} denote the collection of unit vectors in $\R^2$. Given a translation surface $S$, it is commonplace to study the family $\{ F_{\bm \theta}^t:S \to S\}_{\bm \theta \in \Circ}$ \label{not:F}
of {\em straight-line flows}, which are parametrized by a unit vector (a {\em direction}) $\bm \theta \in \Circ \subset \R^2$. \label{not:btheta} In local coordinates, these flows are given by 
\begin{equation}
\label{eq:straight_line_flow}
F_{\bm \theta}^t(x,y)=(x,y)+t \bm \theta.
\end{equation}
A primary goal of this article will be to understand the invariant measures of $F^t_{\bm \theta}$. This flow also gives rise to infinite IETs. The union of the boundaries of the polygons making up $S$ form a section for the flow, and the return map to this section (equipped with the Lebesgue transverse measure to the foliation in $\bm \theta$ direction) is conjugate to an infinite IET.

\subsection{The affine automorphism group}
\name{sect:veech}
An {\em affine automorphism} of a translation surface $S$
is a homeomorphism $\phi:S \to S$ so that in local coordinates near every non-singular point, there are constants $a$, $b$, $c$, $d$, $t_1$ and $t_2$ so that
$$\phi(x,y)=(ax+by+t_1,cx+dy+t_2).$$
Because $S$ is connected and a translation surface, the 
values of $a$, $b$, $c$ and $d$ are independent of the coordinate chart. We say that the matrix
$$D(\phi)=\left[\begin{array}{rr}
a & b \\ c & d\end{array}\right] \in \GL(2,\R)$$ 
\label{not:D}
is the {\em derivative} of $\phi$.
The collection $\Aff(S) \label{not:Aff}$ of all affine automorphisms form a group,
called the {\em affine automorphism group} of $S$. The collection of derivatives of affine automorphisms forms a group called the {\em Veech group} of $S$.

\subsection{Thurston's construction}
\name{sect:graphs}
We will now describe a variant of a construction due to Thurston
which produced the first examples of pseudo-Anosov homeomorphisms. See \cite[\S 6]{T88},which was long preceded by an earlier preprint. Following Veech, we note that the closed translation surfaces admitting two non-commuting parabolic affine automorphisms can be characterized in terms of eigenvectors of graphs \cite[\S 9]{V}. We will make use of some simplifying ideas introduced by McMullen \cite[\S 4]{McM06}
in the closed surface case. (The article \cite[\S 3]{HooperGrid}
carefully describes this construction in the closed surface case.) Here, we extend these ideas to infinite graphs.

We begin by describing some graph theoretic terminology.
Throughout this paper, $\G$ \label{not:G} is an infinite, connected, bipartite, ribbon graph with bounded valance. These terms are defined below.
\begin{enumerate}
\item ({\em Infinite}) The vertex set $\V$ \label{not:V} is countably infinite.
\item ({\em Connected}) For every $\vv, \vw \in \V$, \label{not:v} there is a sequence of vertices $\vv=\vv_0, \vv_1, \ldots, \vv_k=\vw$
so that every $\overline{\vv_i \vv_{i+1}}$ lies in the edge set $\E$ \label{not:E}.
\item ({\em Bipartite}) \name{item4} The vertex set $\V$ decomposes into a disjoint union of two sets, $\V=\Alpha \cup \Beta$, and the edge set $\E$ consists only of edges of the form $\overline{\va \vb}$ with $\va \label{not:a} \in \Alpha$ \label{not:Alpha} and $\vb \in \Beta$. Thus we have natural maps
$\alpha:\E \to \Alpha$ and $\beta:\E \to \Beta$ given by the maps
$$\label{not:alpha} \alpha:\overline{\va \vb} \mapsto \va \quad \textrm{and} \quad \beta:\overline{\va \vb} \mapsto \vb.$$ 
\item ({\em Bounded valance}) The sets $\alpha^{-1}(\va)$ and $\beta^{-1}(\vb)$ are finite sets whose sizes are bounded from above.
\item ({\em Ribbon structure}) For every $\vv \in \V$, the ribbon graph structure specifies a cyclic permutation
$p_\vv$ \label{not:Pv} of the edges that contain $\vv$ as an endpoint.
\end{enumerate}

We use $\R^\V \label{not:RV}$ to denote the collection of all functions from $\V$ to $\R$. 

\begin{definition}
\name{def:adj}
The {\em adjacency operator} is the operator
$\A:\R^\V \to \R^\V$ defined by 
\begin{equation}
\label{not:f}
\A(\f)(\vv)=\sum_{\vw \sim \vv} \f(\vw),
\end{equation}
where the sum is taken over edges $\overline{\vw \vv}$ with $\vv$ as one endpoint. An {\em eigenfunction of $\A$} is a function $\w \label{not:w} \in \R^\V$ which satisfies $\A(\w)=\lambda \w$ for some \label{not:lambda} $\lambda \in \R$.
\end{definition}

\begin{definition}[East and North edge permutations]
Given the above structure on $\G$, we define bijections \label{not:East}
$\East, \North :\E \to \E$. These are given by
$$\East(\overline{\va \vb})=p_\va(\overline{\va \vb})
\quad \textrm{and} \quad
\North(\overline{\va \vb})=p_\vb(\overline{\va \vb}).$$
Since the permutations $p_\vv$ are cyclic, these maps satisfy 
\begin{equation}
\name{eq:ribbon condition}
\{\East^k(e)\}_{k \in \N}=\alpha^{-1}\big(\alpha(e)\big) \quad \textrm{and} \quad \{\North^k(e)\}_{k \in \N}=\beta^{-1}\big(\beta(e)\big)
\quad 
\text{for each $e \in \E$.}
\end{equation}
\end{definition}

\begin{definition}[The surface $S(\G,\w)$]
\name{def:S}
\label{not:S2}
Let $\G$ be a graph as above. Let $\w \in \R^\V$ be a positive eigenfunction of $\A$. Using the associated data, we will construct a translation surface
$S(\G,\w)$. This surface will be a union of rectangles $R_e$ with $e \in \E$, with each $R_e$ given by
$$R_e \label{not:R}=[0,\w \circ \beta(e)] \times [0, \w \circ \alpha(e)].$$
We glue the rectangles so that the right (east) side of $R_e$ is glued isometrically to the left side of $R_{\East(e)}$, and the top (north) side of $R_e$ is glued isometrically to the
bottom of $R_{\North(e)}$. (This explains the notation for the bijections $\East, \North :\E \to \E$.)
\end{definition}

Note that the surface $S(\G,\w)$ admits horizontal and vertical cylinder decompositions which intersect in the given rectangles. For each $\va \in \Alpha$ and each $\vb \in \Beta$, we have the cylinders:
\begin{equation}
\name{eq:cyl}
\cyl_\va=\bigcup_{e \in \alpha^{-1}(\va)} R_e
\quad \textrm{and} \quad
\cyl_\vb=\bigcup_{e \in \beta^{-1}(\vb)} R_e.
\end{equation}
Each cylinder $\cyl_\va$ is horizontal and each cylinder $\cyl_\vb$ is vertical. 
The {\em modulus \label{def:modulus}} of a cylinder is the ratio $\frac{\textit{width}}{\textit{circumference}}$. 
Therefore, the condition that all horizontal and vertical cylinders have equal moduli is equivalent to saying that our function $\w \in \R^\V$ is an eigenfunction of the adjacency operator. (The modulus is given by $1/\lambda$, where $\lambda$ is the eigenvalue.)
In particular, by remarks of Veech \cite[\S 9]{V} this guarantees the existence of two non-commuting parabolic automorphisms of our surface. The eigenvalue of a positive eigenfunction of an infinite connected graph satisfies $\lambda \geq 2$. 
For these values of $\lambda$, the two parabolics generate a free subgroup of $\SL(2, \R)$. Throughout this paper, we will denote the free group with two generators by
$G=\langle h, v\rangle$, \label{not:Gfree} with the choice of generators names representing {\em horizontal} and {\em vertical}.

\begin{definition}[Group representations to $\SL(2,\R)$]
\name{def:rho}
For each $\lambda >0$, let $\rho_\lambda \label{not:rho} :G \to \SL(2, \Z)$ denote the representation generated by
\begin{equation}
\rhoa{h}=\left[\begin{array}{rr} 1 & \lambda \\ 0 & 1 \end{array}\right] 
\quad \textrm{and} \quad
\rhoa{v}=\left[\begin{array}{rr} 1 & 0 \\ \lambda & 1 \end{array}\right].
\end{equation}
This representation is faithful so long as $\lambda \geq 2$.
\end{definition}

We have the following from work of Thurston and Veech \cite[\S 6]{T88} cite[\S 9]{V}.

\begin{proposition}[Automorphisms of $S(\G,\w)$]
\name{prop:eigenfunction}
Suppose $\w$ is a positive eigenfunction for the adjacency operator with eigenvalue $\lambda>0$. Then, there is an endomorphism from the free group on
two generators into the affine automorphism group, 
$\Phi \label{not:Phi} :G \to \Aff\big(S(\G,\w)\big)$, so that $D(\Phi^g)=\rhoa{g}$ for all $g \in G \label{not:g}$. Moreover, we can take
$\Phi^h$ (and respectively, $\Phi^v$) to preserve all horizontal (resp. vertical)
cylinders in the horizontal (resp. vertical) cylinder decomposition
and to act as a single Dehn twists on each preserved cylinder.
\end{proposition}

The main idea of this paper is to use the subgroup $\Phi^G \subset \Aff\big(S(\G,\w)\big)$
to renormalize straight-line flows on these surfaces.

\begin{remark}[Dihedral group action]
\name{rem:dihedral}
The dihedral group of order eight generated by 
$$\left[\begin{array}{rr} 0 & -1 \\ 1 & 0 \end{array}\right] 
\quad \textrm{and} \quad
\left[\begin{array}{rr} -1 & 0 \\ 0 & 1 \end{array}\right]$$
acts on the collection of surfaces obtainable from Thurston's construction. If the matrix $A$ lies in this group, the surface $A \big(S(\G, \w)\big)=S(\G', \w)$ with $\G'$ arising from $\G$ by changing the bipartite and ribbon graph structures on $\G$ in a way that depends on $A$. Note that this matrix group also acts on directions, and on the group $G$ through its action on the representation $\rhoa{G}$ by conjugation.
We will use this dihedral group action to reduce the number of cases we need to consider in several proofs in this paper.
\end{remark}

\section{Main results}
\name{sect:restatement}
In this section, we describe our main results: Theorems \ref{thm:topological conjugacy} and \ref{thm:ergodic2}.
We will also introduce several ideas needed to state these results.

\subsection{Renormalizable directions}
\label{ss:renormalizable directions}
Fix a real constant $\lambda \geq 2$. We will define what it means for a direction $\bm \theta \in \Circ$ to be {\em $\lambda$-renormalizable}. 

The constant $\lambda$ determines a representation $\rho_\lambda:G \to \SL(2,\R)$, where $G=\langle h,v \rangle$ is the free group on two generators. See Definition \ref{def:rho}. The group $\SL(2,\R)$ acts on the real projective line, $\R \P^1 \label{not:RP1}=\big(\R^2 \smallsetminus \{\0\}\big) /\R$ in the standard (linear) way. Recall that the {\em limit set} for the action
of a subgroup of $\SL(2,\R)$ on $\R \P^1$ is the set of accumulation points of an orbit. (Because $\rho_\lambda^G$ is a non-elementary Fuchsian group, the limit set is independent of the choice of the orbit.)

We say a direction $\bm \theta$ is {\em $\lambda$-renormalizable} if the following two statements are satisfied:
\begin{enumerate}
\item The projectivization of $\bm \theta$ lies
in the limit set of $\rho_\lambda^G$.
\item $\bm \theta$ is not an eigendirection of any matrix $\rho_\lambda^g$
where $g$ is conjugate in $G$ to an element of the set $\{h,v,v^{-1}h\}.$
\end{enumerate}
We use $\Rn_\lambda \label{not:Rn} \subset \Circ$ to denote the set of all $\lambda$-renormalizable directions.

\begin{remark}[The size of the set of $\lambda$-renormalizable directions]
\name{rem:renormalizable directions}
We note that since $\rho_\lambda^G$ is a non-elementary
subgroup of $\SL(2,\R)$, the limit set of $\rho_\lambda^G$
is always uncountable, and statement (2) of our definition
removes only countably many direction from this set.
When $\lambda=2$, $\Rn_\lambda$ is $\Circ$ with the vectors of rational slope removed. When $\lambda>2$, then the limit set of $\rho_\lambda^G$ is a Cantor set. The Hausdorff dimension of the limit set varies continuously in $\lambda \geq 2$ and is strictly monotone decreasing. At $\lambda=2$, the dimension is $1$, the dimension has a limiting value of $\frac{1}{2}$ as $\lambda \to +\infty$.
These results on Hausdorff dimension are due to Sato \cite[\S 2]{Sato96}.
\end{remark}

We will now explain how to relate $\lambda$-renormalizable directions as we vary $\lambda$. Consider the {\em Cayley graph} of the free group $G=\langle h, v\rangle$. This is the graph where elements of $G$ are the vertices, and two elements $g_1,g_2 \in G$ are joined by an edge if $g_2 g_1^{-1}$ lies in the symmetric generating set $\{h,v,h^{-1},v^{-1}\}.$ In particular, we have the notion of a geodesic ray in $G$, which we characterize now:

\begin{proposition}
\name{prop:ray}
A sequence  $\langle g_0, g_1, g_2, \ldots \rangle \name{not:ray}$ is a geodesic ray in $G$ if and only if it satisfies the following two statements:
\begin{enumerate}
\item[1.] $g_{n+1} g_{n}^{-1} \in \{h,v,h^{-1}, v^{-1}\}$ for all $n \geq 0$.
\item[2.] $g_{n+2} \neq g_{n}$ for all $n \geq 0$.
\end{enumerate}
\end{proposition}
This follows from the fact that the Cayley graph of $G$ is  homeomorphic to the 4-valent tree.
We use geodesic rays to relate the $\lambda$-renormalizable directions as we vary $\lambda$.
\begin{lemma}
\name{lem:compatible directions}
Let $\lambda_1 \geq 2$ and let ${\bm \theta}_1 \in \Rn_{\lambda_1}$. Then: 
\begin{enumerate}
\item There is a unique geodesic ray
$\langle g_0, g_1, \ldots \rangle$ with $g_0=e$ so
that the length $\|\rho_{\lambda_1}^{g_i}(\bm \theta_1)\|$ decreases strictly monotonically as $i \to \infty$.
We call $\langle g_0, g_1, \ldots \rangle$ the {\em $\lambda_1$-shrinking sequence} of ${\bm \theta}_1$.
\item For any $\lambda_2 \geq 2$, there is a unique
pair of antipodal vectors $\pm {\bm \theta}_2 \in \Rn_{\lambda_2}$ 
so that the $\lambda_1$-shrinking sequence of ${\bm \theta}_1$ coincides with the $\lambda_2$-shrinking sequence of either of the vectors $\pm {\bm \theta}_2$.
\end{enumerate}
\end{lemma}
Renormalizable directions are the topic of \S \ref{sect:symmetry}, and we prove this lemma at the end of this section. Statement (2) has the consequence that the
$\lambda$-shrinking sequences that arise from $\lambda$-renormalizable directions are independent of the choice of $\lambda \geq 2$. So, we call a geodesic ray $\langle g_0, g_1, \ldots \rangle$ with $g_0=e$ a {\em renormalizing sequence} if it is the $\lambda$-shrinking sequence for some $\lambda$-renormalizable direction. We write $\pm \bm \theta(\langle g_n \rangle, \lambda) \label{not:theta2}$ when
we determine an antipodal pair in this way.

\subsection{Orbit equivalence}
\name{ss:orbit}
Definition \ref{def:rho} produced a surface based on an infinite graph $\G$
and a positive eigenfunction $\w$. It is worth observing that fixing a graph satisfying our conditions, there are uncountably many positive eigenfunctions.
(See \cite[Theorem 6.3]{MW89} or our treatment in Appendix \ref{sect:boundary}.)
We will describe a result which shows how surfaces
determined by the same graph but differing eigenfunctions have similar dynamical
properties.

If $S$ is a translation surface and $\bm \theta \in \Circ$, then we can consider the foliation $\F_{\bm \theta}$ of the surface $S$ by orbits of the straight-line flow $\flow^t$ of equation \ref{eq:straight_line_flow}.

\begin{theorem}[Orbit Equivalence]
\name{thm:topological conjugacy}
Let $\G$ be an infinite, connected, bipartite, ribbon graph with bounded valance as in Section \ref{sect:graphs}.
Suppose $\w_1, \w_2 \in \R^\V$ are positive functions satisfying $\A(\w_i)=\lambda_i \w_i$ for $i=1,2$. Let $\langle g_n \rangle$ be a renormalizing sequence, and let $\pm \bm \theta_i=\pm \bm \theta(\langle g_n \rangle, \lambda_i)$ be associated pairs of antipodal $\lambda_i$-renormalizable directions
(as in statement (2) of Lemma \ref{lem:compatible directions}).
Then, there is a homeomorphism $\phi:S(\G, \w_1) \to S(\G, \w_2)$ such that
$\phi(\F_{\bm \theta_1})=\F_{\bm \theta_2}$. Moreover, $\phi$ can be taken to preserve the decomposition of the surfaces
into labeled rectangles as in Definition \ref{def:S},
$$S(\G, \w_1)=\bigcup_{e \in \E} R_e^1 \quad \textrm{and} \quad S(\G, \w_2)=\bigcup_{e \in \E} R_e^2,$$
and so that the restricted maps $\phi|_{R_e^1}:R_e^1 \to R_e^2$ sends the bottom (resp. top, left, right)
edge of $R_e^1$ to the bottom (resp. top, left, right) edge of $R_e^2$ for all $e \in \E$. In this case, the restriction of $\phi$ to a map
from $\bigcup_{e \in \E} \partial R_e^1$ to $\bigcup_{e \in \E} \partial R_e^2$ is uniquely determined.
\end{theorem}

This theorem is proved in \S \ref{sect:outline}, with the final step in the proof appearing in subsection \ref{sect:top_conj}.

\subsection{Extremal positive eigenfunctions}
\name{ss:eigenfunction}
Before discussing the ergodic measures, we need to further analyze the positive eigenfunctions of the adjacency operator of an infinite connected graph $\G$ with bounded valance.
For $\lambda>0$, consider the set
$$E_\lambda=\{\textrm{non-negative $\f \in \R^\V$ satisfying $\A(\f)=\lambda \f$}\} \label{not:E lambda 1}.$$
The set $E_\lambda$ is a closed convex cone in the
topology of pointwise convergence. Furthermore, every $\f \in E_\lambda$ is positive except for the zero function $\0 \in \R^\V$.
We call a positive eigenfunction $\f \in E_\lambda$ {\em extremal} if $\f=\f_1+\f_2$ for $\f_1,\f_2 \in E_\lambda$ implies $\f_1=c\f$ for some real number $c$ with $0 \leq c \leq 1$.  

\subsection{Ergodic measure characterization}
\name{ss:emc}
Let $S$ be a translation surface and $\bm \theta \in \Circ$. 
Such a choice of direction determines a foliation $\F_{\bm \theta} \label{not:foliation}$ of $S$ by orbits of the straight line flow in direction $\bm \theta$. This foliation is {\em singular} in the sense that some leaves
hit singularities. There is a standard method of constructing a non-singular leaf space from such a foliation; we split all singular leaves into two leaves. These leaves are then joined up to make continuous leaves: one split leaf moves 
leftward around each singularity it hit and the other leaf moves rightward around each the singularity. We use $\hat \F_\btheta \label{not:split foliation}$
to denote this non-singular leaf space, which we call the {\em split leaf space}. In our setting, $\hat \F_\btheta$ is a lamination. See Appendix \ref{sect:coding}
for a more rigorous description of this construction.

A {\em locally finite $\hat \F_{\bm \theta}$-transverse measure} is one which assigns finite measure to every compact transversal to the leaf space. An example of such a measure is the {\em Lebesgue transverse measure} on $S$. If $\gamma:[0,1] \to S$ is a differentiable transversal path, the Lebesgue $\hat \F_{\bm \theta}$-transverse measure satisfies
$$\gamma \mapsto \int_t |\gamma'(t) \wedge \bm \theta|~dt,$$
where $\wedge$ denotes the usual wedge product between vectors in the plane;
\begin{equation}
\name{eq:wedge}
\wedge:\R^2 \times \R^2 \to \R 
\quad \textrm{defined by} \quad
(a,b) \wedge (c, d)=ad-bc.
\end{equation}
We allow our transverse measures to be atomic. An atomic measure could be supported on a single leaf for instance. We will be only be considering locally finite transverse measures, so a leaf supporting an atomic measure can not accumulate in the surface.

\begin{theorem}[Ergodic measure characterization]
\name{thm:ergodic2}
Assume $\G$, $\w_1$, and $\bm \theta_1$ are as in Theorem \ref{thm:topological conjugacy}.
Additionally assume $\G$ has no vertices of valance one. Then, the locally finite ergodic $\hat \F_{\bm \theta_1}$-transverse measures on $S=S(\G,\w_1)$ are 
precisely those measures which arise from 
pulling back the Lebesgue $\hat \F_{\bm \theta_2}$-transverse measure on $S(\G, \w_2)$ under the homeomorphisms $\phi$ given in Theorem \ref{thm:topological conjugacy}, where $\w_2$ is an extremal positive eigenfunctions of $\A$.
\end{theorem}

Section \ref{sect:outline} culminates in a proof of this theorem.

\begin{remark}[Valance one] The author conjectures that the condition that $\G$ has no vertices of valance one is unnecessary here. It is  needed to verify some combinatorial conditions
(Definitions \ref{def:subseqence_decay}, \ref{def:critical_decay} and \ref{def:adjacency sign}), which we believe are still true in the valance one case, but which we could not prove.
Indeed, the theorem holds with the valance one condition removed if these conditions can be proved to hold
in the valance one case. See Theorem \ref{thm:general_ergodic_measure_classification}.
\end{remark}

\section{Historical remarks}
\label{sect:hist}

Infinite translation surfaces were first studied in the context of billiards, because when a polygon has angles which are irrational multiples of $\pi$, the Zemljakov-Katok unfolding construction produces an infinite translation surface \cite{ZK}. See for instance \cite{VGS92}.
Recently however, there has been interest in infinite translation surfaces which arise from other constructions. In \cite{Higl1} and \cite{Bowman13}, infinite translation surfaces have been studied 
which arise from geometric limits. Additionally, there are now many papers concerned with geometric questions about infinite covers of translation surfaces.
See \cite{HS09}, \cite{HW10}, and \cite{Schmoll11} in addition to papers on the popular Ehrenfest wind-tree model (see \S \ref{sect:Ehrenfest} of the appendix). Other recent work on infinite translation surfaces has included \cite{CGL} and \cite{PSV11}. 

Dynamicists have long been interested in skew products, and skew rotations have been studied since the papers \cite{S78} and \cite{CK76}. The greatest influence on this paper was Theorem 1.4 of Aaronson, Nakada, Sarig and Solomyak \cite{ANSS02} (which we restate as Theorem \ref{thm:ANSS}).
Other papers in this area include \cite{C09}, \cite{CK11}, \cite{CG12}.
While interval exchange transformations have been studied in close connection with translation surfaces since at least the 1970s,
only recently have skew rotations been studied using ideas from the theory of translation surfaces. (Perhaps \cite{HHW10} was the first such example.) 
Now it appears quite natural to use translation surfaces to study dynamical questions about skew products created using $\Z$-valued cocycles over an interval exchange.

Recently the ergodic theory of straight line flows on infinite covers of translation surfaces has been rapidly developing. In addition to this paper, work of Hubert and Weiss shows that if a $\Z$-cover of a translation surface has a Veech group which is a lattice and the surface contains a strip, then the straight-line flow is ergodic in almost every direction \cite{HW10}. 
Since the first draft of this paper appeared, it has become apparent from work of Fr{\c{a}}czek and Ulcigrai that typically we should not expect many ergodic directions for the straight-line flow on a $\Z$-cover of a translation surface. For instance, it is shown in \cite{FU14} that for any (unbranched) $\Z$-cover of a genus two translation surface the straight-line flow is not ergodic in almost every direction, even though many of these surfaces have recurrent straight-line flows in almost every direction. 

The philosophy of the proofs in this paper come from Teichm\"uller theory. In particular,
we follow the spirit of a criterion of Masur \cite{M92} which guarantees unique ergodicity of an interval exchange transformation involving finitely many intervals. Masur's criterion uses the Teichm\"uller flow on moduli space to demonstrate unique ergodicity. 
Rather than using the Teichm\"uller flow directly, we make use of the inherent symmetries (affine automorphisms) of surfaces produced using Thurston's construction to
renormalize the space of invariant measures. 
This idea goes back to work of Veech \cite{V}.
This iterative process was inspired by Smillie and Ulcigrai's work on the regular octagon \cite{SU10}, which produced
detailed information about the trajectories of the straight-line flow. Our techniques give weaker information about the behavior of trajectories than explicit coding, but the information we extract is sufficient for classifying invariant measures.

\section{Proofs of main results}
\name{sect:outline}

The Orbit Equivalence Theorem and Ergodic Measure Characterization Theorem (Theorems \ref{thm:topological conjugacy} and \ref{thm:ergodic2}) are the main results of this paper. We prove these results in this section, though we rely on work in later sections to flesh out many of the details.

We will now give an overview of the proofs of the main results. As in the statements of these results, $\G$ will be an infinite, connected, bipartite, ribbon graph with bounded valance. We let $\w$ be a positive eigenfunction
of the adjacency operator $\A:\R^\V \to \R^\V$ with eigenvalue $\lambda$. Our surface $S=S(\G,\w)$ is built as in Definition \ref{def:S}. We let $\btheta=\bm \theta(\langle g_n \rangle, \lambda)$ be a $\lambda$-renormalizable direction.
\begin{itemize}
\item In \S \ref{sect:reinterpret}, we explain that there
is a linear embedding of the space of locally finite $\hat \F_{\bm \theta}$-transverse measures into a cohomological space, $H^1$. We derive a necessary and sufficient criterion for an $m \in H^1$ to arise from a transverse measure in terms of pairings of $m$ with homology classes of saddle connections.
\item In \S \ref{sect:affine action}, we consider the affine action of the sequence of group elements $\langle g_n \in G \rangle$. This paper uses these group elements as renormalization operators. These group elements act on $H^1$. We give a necessary and sufficient criterion for $m \in H^1$ to arise from a transverse measure in terms of the images of $m$ under the action of 
$\langle g_n \rangle$.
\item In \S \ref{sect:operators}, we observe that there is a natural linear embedding of $\R^\V$ into $H^1$. The action of $G$ on $H^1$ leaves invariant
this image of $\R^\V$, and the induced action of the generators on $\R^\V$ is quite simple (see equations \ref{eq:Ho} and \ref{eq:Vo}). We describe
necessary and sufficient conditions for the image of $\f \in \R^\V$ in $H^1$
to come from a transverse measure.
\item In \S \ref{sect:top_conj}, we prove 
the Orbit Equivalence Theorem (Theorem \ref{thm:topological conjugacy}). 
Suppose $\w_2$ is another eigenfunction of $\A$ with 
eigenvalue $\lambda_2$ and consider the direction $\btheta_2=\bm \theta(\langle g_n \rangle, \lambda_2)$ on the surface $S_2=S(\G,\w_2)$. Since $S$ and $S_2$ are built in the same combinatorial way, there is a natural isotopy class of homeomorphisms between them. We use this to pull the cohomology class associated to the $\hat \F_{\btheta_2}$-Lebesgue transverse measure back to $S$. We use our understanding of transverse measures to observe that this pullback cohomology class came from such a measure. By integrating this measure, we improve our identification between these surfaces to a canonical homeomorphism $S \to S_2$ which carries $\F_{\bm \theta}$ to $\F_{\btheta_2}$. This is the desired orbit equivalence.
\item In \S \ref{sect:surviving}, we introduce a hypothesis called the  subsequence decay property. We show that under this hypothesis all cohomology classes arising from transverse measures lie in the image of $\R^\V$ inside $H^1$.
Going forward, we will consider the cone of transverse measures as linearly embedded in $\R^\V$.
\item In  \S \ref{sect:outline_adjacency}, we introduce the action of the adjacency operator, $\A$. We make some more hypotheses: the critical decay property and the adjacency sign property. We note these properties as well as the subsequence decay property hold for surfaces built from graphs with no vertices of valance one. Under these hypotheses, we show that $\A^2$ restricts to a bijection preserving the cone of transverse invariant measures. Furthermore, we show that if $\f \in \R^\V$ arises from an ergodic transverse measure, then there is a $\lambda_2>0$ so that $\A^2(\f)=\lambda_2^2 \f$. 
\item In \S \ref{sect:martin boundary intro}, we use Martin boundary theory to finish the proof of the Ergodic Measure Characterization Theorem (Theorem \ref{thm:ergodic2}). The prior section left us to consider precisely which functions $\f$ satisfying $\A^2(\f)=\lambda_2^2 \f$ arise from ergodic transverse measures. It turns out that for each $\lambda_2$, the set of $\f$ arising from transverse measures and satisfying $\A^2(\f)=\lambda_2^2 \f$ is the image under a linear map of the set of positive solutions to $\A(\g)=\lambda \g$. Via the Poisson-Martin representation theorem, the extreme points of this latter space are understood in terms of the minimal Martin boundary.
\end{itemize}

\subsection{Reinterpretation of invariant measures}
\label{sect:reinterpret}

Let $S$ be an infinite translation surface constructed as a union of polygons as in Section \ref{sect:trans}, and let $V \label{not:singularities}$ be the collection of all points in $S$ which are (identified) vertices of those polygons. Let
$\bm \theta \in \Circ$ be a direction and $\hat \F_{\bm \theta}$ be the leaf space of orbits of the straight line flow in direction $\bm \theta$, 
with singular leaves split. Let $\M_{\bm \theta} \label{not:measures}$ denote the space of all locally finite transverse measures for $\hat \F_{\bm \theta}$.

Let $H_1(S, V, \R) \label{not:homology}$ denote the real homology classes of closed curves in $S/V$ (i.e., $S$ with the points in $V$ collapsed to a single point). Here, we only allow such curves to visit finitely many polygons in the decomposition of $S$ into polygons.
We will always use $\Coh \label{not:cohomology}$ to indicate the the dual space to $H_1(S, V, \R)$. That is,
$\Coh$ is the collection of all linear maps from $H_1(S, V, \R) \to \R$. We choose this notation because $\Coh$ will be the only cohomological space we consider.

A transverse measure $\mu \in \M_{\bm \theta}$ determines a linear map $\Psi_{\bm \theta}(\mu):H_1 (S, V, \Z) \to \R \label{not:Psi}$.
This map is defined to be $\hom{x} \mapsto \mu(x)$ if $x$ is a curve in $S$ joining a point in $V$ to a point in $V$ and everywhere crossing 
the leaf space $\hat \F_{\bm \theta}$ with positive algebraic sign. (To be precise, if the leaves of $\hat \F_{\bm \theta}$ are upward pointing,
then the transversals moving rightward across the leaf space cross with positive algebraic sign.)
This map can be uniquely extended to a linear map $\Psi_{\bm \theta}(\mu):H_1(S, V, \Z) \to \R$,
since the map is determined on a basis. Note that $\Psi_{\bm \theta}(\mu)$ is a cohomology class in $H^1$. So, 
$\Psi_\btheta$ defines a linear map $\M_{\bm \theta} \to \Coh.$

Recall $\flow^t$ denoted the flow on a translation surface in direction $\bm \theta$. 
We call the flow $\flow^t:S \to S$ {\em conservative} if given any subset $A \subset S$ of positive measure and any $T>0$, for Lebesgue a.e. $x \in S$ there is a $t>T$ for which $\flow^t(x) \in A$. We have the following.

\begin{lemma}
\label{lem:injective}
Let $S$ be a translation surface with non-empty singular set $V$.
The map $\Psi_{\bm \theta}:\M_{\bm \theta} \to \Coh$ is injective if the straight-line flow 
$\flow^t$ has no periodic trajectories and is conservative.
\end{lemma}

This statement is proved in Appendix \ref{sect:homology}. 

We can apply this result because we have the following three results
for the translation surface $S=S(\G,\w)$ constructed as in Definition \ref{def:S},
with $\w \in \R^\V$ a positive eigenfunction for the adjacency operator with eigenvalue $\lambda$.

A {\em saddle connection } is an oriented geodesic segment $\sigma \subset S \label{not:sigma}$ which visits only finitely many polygons making up $S$ and intersects $V$ precisely at its endpoints.
The {\em holonomy} of a saddle connection, $\hol(\sigma) \in \R^2 \label{not:holonomy}$, is the vector difference between the end and starting points of a developed image of $\sigma$ into the plane. 

\begin{theorem}[No saddle connections]
\label{thm:no saddles}
No saddle connection of $S(\G,\w)$ has a
holonomy vector which points in a $\lambda$-renormalizable direction.
\end{theorem}

This theorem is proved in section \ref{sect:geometry surfaces}.
Note that in any translation surface, a periodic trajectory  of the straight-line flow lies in a maximal Euclidean cylinder whose core curves are parallel to the trajectory. The obstruction to further enlarging such a cylinder is the presence of vertices in the boundary, so such a cylinder is bounded by a finite number of saddle connections. So as a consequence of the above theorem, we see:

\begin{corollary}[Aperiodicity]
\label{cor:aperiodicity}
The surface $S(\G,\w)$ admits no periodic straight-line trajectories in $\lambda$-renormalizable directions.
\end{corollary}

Finally, in section \ref{sect:recurrence}, we prove the following:

\begin{theorem}[Conservativity]
\label{thm:pr}
The straight-line flow on $S(\G,\w)$ in a $\lambda$-renormalizable direction
is conservative.
\end{theorem}

So, by the lemma above, we have the following:
\begin{corollary}
If $\btheta$ is a $\lambda$ renormalizable direction on
the surface $S(\G,\w)$, then the map $\Psi_{\bm \theta}:\M_{\bm \theta} \to \Coh$ is injective.
\end{corollary}

The first idea in the proof of our main results is to work with the convex cone
$\Psi_{\bm \theta}(\M_{\bm \theta})$ rather than working directly with measures. In order to do this, we give a criterion for a cohomology class to come from a measure. 

\begin{lemma}
\label{lem:sign}
Given $m \in \Coh$, we have $m \in \Psi_{\bm \theta}(\M_{\bm \theta})$ if and only if for all saddle connections $\sigma$ with 
$m(\hom{\sigma}) \neq 0$ we have
$$\sgn \big( m(\hom{\sigma})\big)=\sgn \big(\hol(\sigma) \wedge \bm \theta\big).$$
Here, $\sgn:\R \to \{-1,0,1\} \label{not:signum}$ is the {\em signum} function, which assigns to a real number the element of $\{-1,0,1\}$ with the same sign.
\end{lemma}
The proof of this lemma is in Appendix \ref{sect:homology}.

\begin{remark}[Awkward inequalities]
\label{rem:ineq}
The construct ``$a \neq 0$ implies $\sgn(a)= \sgn(b)$'' is equivalent to the longer statement
``if $b>0$ then $a \geq 0$ and if $b<0$ then $a \leq 0$.'' We will repeatedly use this more compact construct in this paper.
\end{remark}

\subsection{Action of the affine automorphism group}
\label{sect:affine action}

For the remainder of Section \ref{sect:outline}, we will let $S=S(\G, \w)$ be a translation surface as constructed in section \ref{sect:graphs}.
Let $E$ denote the set of horizontal and vertical edges of the rectangles
making up $S$, and $V$ denotes the set of vertices of rectangles. By definition, the edges in $E \label{not:Edges}$ are all saddle connections, since they join a point in $V$ to a point in $V$. 
We orient the horizontal saddle connections in $E$ rightward, and
the vertical saddle connections upward. We have the following.

\begin{proposition}
\label{prop:generate}
The homology classes of the saddle connections in $E$ span $H_1(S,V,\Z)$. 
\end{proposition}

The proof is to simply note that by cutting along the horizontal and vertical saddle connections, we decompose $S$ into rectangles.
In particular, an $m \in \Coh$ is determined by what it does to $E$. 

Let $\phi:S \to S$ be an orientation preserving affine automorphism of $S$. Then,
$\phi$ acts on the space of transverse measures by pushing forward the measure. 
For all $\bm \theta \in \Circ$, $\phi$ induces a bijection 
$$\phi_\ast \label{not:phi ast}:\M_{\bm \theta} \to \M_{D (\phi)(\bm \theta)}:\mu \mapsto \mu \circ \phi^{-1}.$$
Now suppose that $\phi$ preserves the singular set $V$. (It may not, since we allow removable singularities.)
Then, this action on measures is compatible with the pushforward action on $\Coh$.
We abuse notation by also denoting this pushforward by $\phi_\ast:\Coh \to \Coh$.
For $m \in \Coh$, we define $\phi_\ast(m)$ by
\begin{equation}
\label{eq:pushforward}
\big(\phi_\ast(m)\big)(\hom{x})=m \circ \phi^{-1}(\hom{x})
\end{equation}
for all $\hom{x}  \label{not:hom} \in H_1(S, V, \Z)$.
For all $\mu \in \M_{\bm \theta}$ we have
$$\phi_\ast \circ \Psi_{\bm \theta}(\mu)=\Psi_{D (\phi)(\bm \theta)} \circ \phi_\ast(\mu).$$

We now recall Proposition \ref{prop:eigenfunction}. If $\w$ is a positive eigenfunction for $\G$ with eigenvalue $\lambda$, we have an action $\Phi$ of the free group $G=\langle h,v \rangle$ by affine automorphisms of $S(\G,\w)$. 
This group action preserves the vertex set $V$.
The action induced on homology is independent of the choice of such a positive eigenfunction $\w$. The following proposition will explain the action 
on the generating set $E$ of $H_1(S, V, \Z)$. 

For each saddle connection $\sigma \in \Edge$, we use $\hom{\sigma} \in H_1(S,V,\Z)$ to denote its homology class. Recall that vertices $\vv \in \V$ correspond to cylinders $\cyl_\vv$
as in Equation \ref{eq:cyl}. If $\vv \in \Alpha$, this cylinder is horizontal. Otherwise it is vertical. We use $\hom{\cyl_\vv}$ to denote the homology class of a core curve. As with $\hom{\sigma}$ we orient $\hom{\cyl_\vv}$ either rightward or upward. Note that any $\sigma \in E$ traverses exactly one cylinder $\cyl_\vv$. 

\begin{proposition}[Action on homology]
\name{prop:action}
Given any $\sigma \in \Edge$, we have
$$
\Phi^{h^k}(\hom{\sigma})=\begin{cases}
\hom{\sigma} & \textrm{if $\sigma$ is horizontal,} \\
\hom{\sigma}+k \hom{\cyl_\va} & \textrm{if $\sigma$ is vertical and traverses $\cyl_\va$ with $\va \in \Alpha$,}
\end{cases}
$$
$$
\Phi^{v^k}(\hom{\sigma})=\begin{cases}
\hom{\sigma}+k \hom{\cyl_\vb} & \textrm{if $\sigma$ is horizontal and traverses $\cyl_\vb$ with $\vb \in \Beta$,} \\
\hom{\sigma} & \textrm{if $\sigma$ is vertical.} 
\end{cases}
$$
\end{proposition}

Recall, a $\lambda$-renormalizable direction $\bm \theta$ has a $\lambda$-shrinking sequence of elements of $G$, $\langle g_n \rangle$. We also call 
$\langle g_n \rangle$ a renormalizing sequence. See \S \ref{ss:renormalizable directions} and, in particular, Lemma \ref{lem:compatible directions}.

\begin{definition}
\label{def:survivor}
We say that an $m \in \Coh$ is a {\em $(\bm \theta,n)$-survivor} if for all saddle connections $\sigma \in \Phi^{g_n^{-1}}(E)$ we have
$m(\hom{\sigma}) \neq 0$ implies
$$\sgn \big( m(\hom{\sigma})\big)=\sgn \big(\hol(\sigma) \wedge \bm \theta\big).$$
\end{definition}
The collection of $(\bm \theta, n)$-survivors $m\in \Coh$ is a closed convex cone by Remark \ref{rem:ineq}.
Note that this definition verifies the conditions of Lemma \ref{lem:sign} on a subset of saddle connections depending on $n$. But, it turns out that checking each of these subsets is sufficient for concluding that $m$ arises as $\Psi(\mu)$ for some $\mu \in \M_{\bm \theta}$. 

\begin{theorem}[Survivors and measures]
\label{thm:operator_theorem}
Let $\bm \theta$ be a $\lambda$-renormalizable direction.
For any $m \in \Coh$, 
$m \in \Psi_{\bm \theta}(\M_{\bm \theta})$ if and only if $m$ is a $(\bm \theta,n)$-survivor for all $n \geq 0$. 
\end{theorem}

The proof follows from analyzing the action of the sequence $\rhoa{g_n}$ on $\Circ$. We show that if
$m$ is a $(\bm \theta,n)$-survivor for all $n$, then $m$ satisfies Lemma \ref{lem:sign}. Note that Theorem \ref{thm:no saddles} implies we do not need to worry about saddle connections in the direction of $\btheta$.
See the end of Section \ref{sect:quadrants} for the proof.

We can derive an equivalent definition of being a $(\bm \theta,n)$-survivor by acting by $\Phi^{g_n}$
on $H^1$. This action was described in equation \ref{eq:pushforward}, and we denote this $G$ action by $\Phi^{G}_\ast \label{not:Phi ast}:H^1 \to H^1$.  

\begin{proposition}
\label{prop:equiv_def}
The condition that $m$ be a $(\bm \theta,n)$-survivor is equivalent to the statement that for all $\sigma \in \Edge$, we have
$\big(\Phi^{g_n}_\ast(m)\big)(\hom{\sigma}) \neq 0$ implies
$$\sgn \big(\Phi^{g_n}_\ast(m)\big)(\hom{\sigma})=\sgn \big(\hol(\sigma) \wedge \rhoa{g_n}(\bm \theta) \big).$$
\end{proposition}

\begin{remark}[Renormalization argument]
The interpretation of being a $(\bm \theta,n)$-survivor provided in by Proposition \ref{prop:equiv_def} allows for a renormalization argument in the following sense.
The conditions on being a $(\bm \theta,0)$-survivor are relatively weak, and depend only on the quadrant containing $\bm \theta$. Observe that
$m$ is a $(\bm \theta,n)$-survivor if and only if $\Phi^{g_n}_\ast(m)$ is a $\big(\rhoa{g_n}(\bm \theta),0\big)$-survivor. By Theorem \ref{thm:operator_theorem},
the cohomology classes associated to invariant measures are precisely those for which  $\Phi^{g_n}_\ast(m)$ is a $\big(\rhoa{g_n}(\bm \theta),0\big)$-survivor for all $n$.
In that sense, we are utilizing the action of the shrinking sequence  $\langle g_n \in G \rangle$ on the space of straight-line flows of $S(\G, \w)$ to understand the invariant measures for these flows.
\end{remark}


\subsection{Operators on graphs}
\label{sect:operators}
In addition to the adjacency operator, $\A$, defined in Definition \ref{def:adj},
we will be studying the following linear operators on $\R^\V$:
\begin{equation}
\label{eq:Ho}
\Ho^k(\f)(\vx)= \begin{cases}
\f(\vx) + k \sum_{\vy \sim \vx} \f(\vy) &  \textrm{if $\vx \in \Alpha$,}\\
\f(\vx) & \textrm{if $\vx \in \Beta$.}
\end{cases}
\end{equation}
\begin{equation}
\label{eq:Vo}
\Vo^k(\f)(\vx)= \begin{cases}
\f(\vx) &  \textrm{if $\vx \in \Alpha$,}\\
\f(\vx)+ k \sum_{\vy \sim \vx} \f(\vy) & \textrm{if $\vx \in \Beta$.}
\end{cases}
\end{equation}
We define the group action $\Upsilon \label{not:Upsilon}:G \times \R^\V \to \R^\V$ by extending the definition
$\Act^h=\Ho$ and $\Act^v=\Vo$. 
As long as $\G$ is an infinite connected graph, the group action $\Upsilon$ is a faithful action 
of the free group with two generators. The operators relate to the adjacency operator, $\A$, by the equations
\begin{equation}
\label{eq:interactions}
\A \Ho=\Vo \A \quad \textrm{and} \quad \A \Vo=\Ho \A.
\end{equation}

There is a natural linear embedding $\Xi \label{not:Xi}:\R^\V \to \Coh$. Given $\f \in \R^\V$, and $\hom{x} \in H_1(S,V,\Z)$
we define 
\begin{equation}
\label{eq:Xi}
\Xi(\f)(\hom{x})=\sum_{\vv \in \V} i(\hom{x}, \hom{\cyl_{\vv}}) \f(\vv).
\end{equation}
Here, $i:H_1(S,V,\Z) \times H_1(S \smallsetminus V,\Z) \to \Z \label{not:intersection}$ denotes the usual algebraic intersection number,
and $\hom{\cyl_{\vv}}$ is as defined above Proposition \ref{prop:action}.
This sum is well defined, because $i(\hom{x}, \hom{\cyl_{\vv}}) = 0$ for all but finitely many $\vv \in \V$.

It is not difficult to see that the image $\Xi(\R^\V)$ is invariant under the action of the affine automorphism group. In fact, we have the following.
\begin{proposition}
\label{prop:pullback_action}
For all $g \in G$ we have $\Phi^g_\ast \circ \Xi=\Xi \circ \Act^{g}$.
\end{proposition}
This proposition is proved in Section \ref{sect:operators2}.

Suppose that $\f \in \R^\V$ is a positive function. Then we can think of $\Xi(\f)$ as the element of
$\Coh$ arising from the Lebesgue transverse measure to the foliation in direction of angle $\frac{\pi}{4}$
on the surface $S(\G, \f)$. Similar interpretations arise from considering an $\f \in \R^\V$ where the signs of $\f$ on the subsets of vertices $\Alpha$ and $\Beta$ 
are fixed. This idea is what gives rise to the notions of quadrants, which we are about to introduce.
We will show that $\Psi_{\bm \theta}(\M_{\bm \theta}) \subset \Xi(\R^\V)$. Moreover, 
we will show that the set of transverse measures
to the split leaf space $\hat \F_{\bm \theta}$ (see \S \ref{ss:emc} and Appendix \ref{sect:coding}) all can be interpreted as such Lebesgue transverse measures on $S(\G, \f)$ with $\f$ satisfying {\em survivor} conditions,
as defined below.

\begin{definition}[Survivors in $\R^\V$]
\label{def:function_survivors}
Assume $\bm \theta \in \Circ$ is a $\lambda$-renormalizable direction with $\lambda$-shrinking sequence $\langle g_n \rangle$. 
We say that $\f \in \R^\V$ is a {\em $(\bm \theta, n)$-survivor} if $\Xi(\f)$ is a $(\bm \theta, n)$-survivor.
We say $\f$ is a {\em $\bm \theta$-survivor} if it is a $(\bm \theta, n)$-survivor for all $n \geq 0$. 
We use $\Surv_{\bm \theta} \label{not:survivors} \subset \R^\V$ to denote the set of all $\bm \theta$-survivors.
\end{definition}

We will introduce an equivalent definition which is intrinsic to $\R^\V$. For this, we need a few more definitions.

\begin{definition}[Sign pairs]
\label{def:sign_pairs}
The set of sign pairs is the set of four elements $\SP \label{not:sign pairs}=\{(\pm 1, \pm 1)\}$. We abbreviate these elements by writing
$$++=(1,1) \quad +-=(1, -1) \quad -+=(-1,1) \quad --=(-1,-1). \label{not:sign pairs 2}$$
We will use $\pi_1$ and $\pi_2$ to denote the projection functions. For instance $\pi_1(-+)=-1$ and $\pi_2(-+)=1$. 
\end{definition}

\begin{definition}[Quadrants in $\R^2$]
\label{def:quadrants1}
The four open quadrants in $\R^2$ are naturally in bijection to the elements $s \in \SP$. We define
$$\Q_s \label{not:quadrants 1} =\{(x,y) \in \R^2 ~:~ \textrm{$\sgn~x=\pi_1(s)$ and $\sgn~y=\pi_2(s)$}\}.$$
We use $\cl(\Q_{s}) \label{not:closure}$ to denote the closure $\Q_s$.
\end{definition}

\begin{definition}[Quadrants in $\R^\V$]
\label{def:quadrants2}
The four quadrants in $\R^V$ are 
$$\hat \Q_{s} \label{not:quadrants 2}=\left\{\f \in \R^\V ~:~ \begin{array}{l}
\textrm{$\f(\va)=0$ or $\sgn~\f(\va)=\pi_1(s)$ for all $\va \in \Alpha$.} \\
\textrm{And, $\f(\vb)=0$ or $\sgn~\f(\vb)=\pi_2(s)$ for all $\vb \in \Beta$.}
\end{array}\right\}.$$
\end{definition}

Recall that the orbit under $\rhoa{G}$ of a $\lambda$-renormalizable direction does not include the horizontal or vertical directions.

\begin{definition}[Sign sequence of $\bm \theta$]
\label{def:sign_sequence}
Suppose $\bm \theta$ has $\lambda$-shrinking sequence $\langle g_n\rangle$. We define the corresponding {\em ($\lambda$-) sign sequence} of $\bm \theta$ to be the sequence $\langle s_0, s_1, s_2, \ldots \rangle$ with $s_n \in \SP$ so that
$\rhoa{g_n}(\bm \theta) \in \Q_{s_n}$
for all $i$. 
\end{definition}

The $\lambda$-sign and $\lambda$-shrinking sequences turn out to uniquely determine $\bm \theta$. 

\begin{proposition}[Quadrant sequences]
\label{prop:quadrant_sequence}
Let $\bm \theta \in \Rn_\lambda$ with $\lambda$-shrinking sequence $\langle g_n \rangle$ and sign sequence $\langle s_n \rangle$. The only $\v \in \Circ$ for which
$\rhoa{g_n}(\v) \in \overline{\Q}_{s_n}$ for all $n$ is $\v=\bm \theta$.
\end{proposition}
This proposition is proved in Section \ref{sect:proof_operator_theorem}. 

Recall from \S \ref{ss:renormalizable directions}
that a renormalizing sequence $\langle g_n \rangle$
and a $\lambda \geq 2$ determine an antipodal pair of unit vectors $\pm \bm \theta(\langle g_n \rangle, \lambda) \in \Circ$
for which $\langle g_n \rangle$ is the $\lambda$-shrinking sequence. The quadrants containing these vectors do not change as we vary $\lambda$.

\begin{proposition}
\name{prop:quadrants_match}
If $\langle g_n \rangle$ is a renormalizing sequence
and $\lambda, \lambda' \geq 2$, then the antipodal 
pairs $\pm \btheta(\langle g_n \rangle, \lambda)$
and $\pm \btheta(\langle g_n \rangle, \lambda')$
lie in the same pair of quadrants.
\end{proposition}
This proposition is proved at the end of \S \ref{ss:shrinking}.
Because of this, given a choice of $\btheta \in \Rn_\lambda$ with $\lambda$-shrinking sequence $\langle g_n \rangle$,
we can let $\btheta(\langle g_n \rangle,\lambda')$
denote the choice from the antipodal pair
$\pm \btheta(\langle g_n \rangle,\lambda')$ which lies in
the same quadrant as $\btheta$.

We will need the following result in the proof of the orbit equivalence (Theorem \ref{thm:topological conjugacy}). 

\begin{proposition}
\label{prop:same_sign}
The $\lambda$-sign sequence of $\bm \theta$ is the same as the $\lambda'$-sign sequence of $\bm \theta(\langle g_n \rangle, \lambda')$.
\end{proposition}

This proposition will be proved at the end of section \ref{sect:quadrants and expansion}.
The sign sequence gives us a more natural definition of being a $(\bm \theta,n)$-survivor.

\begin{proposition}[Equivalent notion of survivors]
\label{prop:equiv_survivors}
Let $\f \in \R^\V$. Then $\f$ is a $(\bm \theta, n)$-survivor if and only if
$\Act^{g_n}(\f) \in \hat \Q_{s_n}$. 
\end{proposition}
\begin{proof}
By Proposition \ref{prop:pullback_action},
$\Phi^{g_n}_\ast \circ \Xi(\f)=\Xi \circ \Act^{g_n}(\f)$.
Showing that $\f$ is a $(\bm \theta, n)$-survivor is equivalent to showing that 
$\Xi \circ \Upsilon^{g_n}(\f)$ is a $(\rhoa{g_n}(\bm \theta), 0)$-survivor. 
Therefore, it is enough to consider the case when $n=0$.

Write $\bm \theta=(x,y)$ and set $s_x=\sgn~x$ and $s_y=\sgn~y$. As $\bm \theta$ is $\lambda$-renormalizable,
$s_x, s_y \in \{-1,1\}$. If $\sigma_\vv \in \Edge$ is a vertical saddle connection, then
$\sgn(\hol(\sigma_v) \wedge \bm \theta)= -s_x$. If $\sigma_v \in \Edge$ is horizontal, then
$\sgn(\hol(\sigma_h) \wedge \bm \theta)= s_y$. Assuming $\sigma_v$ crosses the horizontal cylinder $a \in \Alpha$
and $\sigma_h$ crosses the vertical cylinder $b \in \Beta$, we have
$$\Xi(\f)(\sigma_v)=-\f(\va) \quad \textrm{and} \quad \Xi(\f)(\sigma_h)=\f(\vb).$$
Therefore, $\Xi(f)$ is a $(\bm \theta,0)$-survivor if and only if 
$\sgn~\f(\va) \in \{0, s_x\}$ and $\sgn~\f(\vb) \in \{0, s_y\}$ for all
$a \in \Alpha$ and $b \in \Beta$. The conclusion of the proposition follows.
\end{proof}

\subsection{The topological conjugacy}
\label{sect:top_conj}

In this subsection, we will prove Theorem \ref{thm:topological conjugacy}. Throughout this subsection $\w_1, \w_2 \in \R^\V$ will be positive functions satisfying $\A(\w_i)=\lambda_i \w_i$. The sequence $\langle g_n \rangle$ will be a renormalizing sequence, and $\bm \theta_i= \bm \theta(\langle g_n \rangle, \lambda_i) \in \Circ$ will be chosen from the pairs of antipodal $\lambda_i$-renormalizable directions with $\lambda_i$-shrinking sequence $\langle g_n \rangle$ so that $\bm \theta_1$ and $\bm \theta_2$ lie in the same quadrant. We write $\bm \theta_i=(x_i, y_i)$. 

For any positive $\f \in \R^\V$ with $\A(\f)=\lambda \f$, there is a naturally related parameterized plane in $\R^\V$.
\begin{equation}
\label{eq:P}
P_\f:\R^2 \to \R^\V; \quad P_\f(x,y)(\vv)=\begin{cases}
x \f(\vv) & \textrm{if $\vv \in \Alpha$} \\
y \f(\vv) & \textrm{if $\vv \in \Beta$.}
\end{cases}
\end{equation}
This plane is invariant under the action of $\Ho$ and $\Vo$:
\begin{proposition}[Invariant planes]
\label{prop:invt_planes}
Let $\v \in \R^2$. Then, for all $g \in G$, we have $\Act^g \big(P_\f(\v)\big)=P_\f \big(\rho_{\lambda}^g(\v)\big)$. 
\end{proposition}
The proof is a simple calculation.

We will be considering the parametrized planes $P_{\w_1}$ and $P_{\w_2}$.
For $i=1,2$, let $\mu_i$ be the Lebesgue $\hat \F_{\bm \theta_i}$-transverse measures on $S(\G, \w_i)$. 
These measures are closely connected to the planes constructed above.

\begin{proposition}
For $i=1,2$, we have $\Psi_{\bm \theta_i}(\mu_i)=\Xi\big(P_{\w_i}(\bm \theta_i)\big)$.
\end{proposition}
\begin{proof}
Let $\va \in \Alpha$ and $\vb \in \Beta$. Let $\sigma_v$ be a vertical saddle connection crossing the horizontal cylinder $\cyl_\va$ and oriented upward. Let
$\sigma_h$ be a rightward oriented horizontal saddle connection crossing the vertical cylinder $\cyl_\vb$.
By definition of $\Psi_{\bm \theta_i}(\mu_i)$ and the Lebesgue transverse measure, we have the following.
$$\Psi_{\bm \theta_i}(\mu_i)(\hom{\sigma_v})=\big(0,\w_i(\va)\big) \wedge \bm \theta_i=-x_i \w_i(\va).$$
$$\Psi_{\bm \theta_i}(\mu_i)(\hom{\sigma_h})=\big(\w_i(\vb),0\big) \wedge \bm \theta_i=y_i \w_i(\vb).$$
Similarly, by definition of $\Xi$ we have the following.
$$\Xi\big(P_{\w_i}(\bm \theta_i)\big)(\hom{\sigma_v})=i(\hom{\sigma_v}, \hom{\cyl_\va}) P_{\w_i}(\bm \theta_i)(\va)=-x_i \w_i(\va).$$
$$\Xi\big(P_{\w_i}(\bm \theta_i)\big)(\hom{\sigma_h})=i(\hom{\sigma_h}, \hom{\cyl_\vb}) P_{\w_i}(\bm \theta_i)(\vb)=y_i \w_i(\vb).$$
\end{proof}

\begin{proposition}
\label{prop:survivor_in_plane}
$P_{\w_2}(\bm \theta_2)$ is a $\bm \theta_1$-survivor. Moreover $\v=\btheta_2$ is the unique $\v \in \Circ$ for which 
$P_{\w_2}(\v)$ is a $\bm \theta_1$-survivor.
\end{proposition}
\begin{proof}
We show that $P_{\w_2}(\bm \theta_2)$ is a $(\bm \theta_1, n)$-survivor for all $n$. Let $s_n$ be the $n$-th entry in the sign sequences of $\bm \theta_1$ and $\bm \theta_2$. (These are the same by Proposition \ref{prop:same_sign}.)
By Proposition \ref{prop:invt_planes}, we have that $\Upsilon^{g_n}\big(P_{\w_2}(\bm \theta_2)\big)=P_{\w_2}\big(\rho_{\lambda_2}^{g_n}(\bm \theta_2)\big)$. We know $\rho_{\lambda_2}^{g_n}(\bm \theta_2) \in \Q_{s_n}$. So by definition of $P_{\w_2}$, we have $P_{\w_2}\big(\rho_{\lambda_2}^{g_n}(\bm \theta_2)\big) \in \hat \Q_{s_n}$ and we may apply Proposition \ref{prop:equiv_survivors}.

Uniqueness of the choice $\bm \theta_1$-survivor follows by combining the above argument with Proposition \ref{prop:quadrant_sequence}.
\end{proof}

Let $\M_{i}$ denote the locally finite  $\hat \F_{\bm \theta_i}$-transverse measures on $S(\G, \w_i)$. From Theorem \ref{thm:operator_theorem} and Definition \ref{def:function_survivors}, we can observe:

\begin{corollary}
\label{cor:pullback}
There is an unique measure $\mu_2' \in  \M_{1}$ with $\Psi_{\bm \theta_1}(\mu_2')=\Xi\big(P_{\w_2}(\bm \theta_2)\big)$.
\end{corollary}
This measure $\mu_2'$ is our candidate pull-back measure. It remains to build our homomorphism. We begin by building a continuous map from $\phi_1:S(\G, \w_1) \to S(\G, \w_2)$. For $e \in \E$, let $R_e^i$ be the rectangles of $S(\G, \w_i)$ associated to the edge of the graph $e=\overline{\va \vb}$. We define the map on these rectangles. Recall we may view each $R_e^i$ as the rectangular subset of the plane, $[0,\w_i (\vb)] \times [0, \w_i(\va)]$. (See Definition \ref{def:S}.) 
We define the restriction map 
$\phi_1|_{ \partial R_e^1}: \partial R_e^1 \to \partial R_e^2$ along the edges by integrating the measure $\mu_2'$ and rescaling appropriately. 
For instance along the bottom edge $B \subset R_e^1$, which is identified with the interval $[0, \w_1(\vb)]\times \{0\}$, we use the formula
$$\phi_1|_B(t,0)=\Big(\frac{1}{|y_2|} \mu_2'([0,t] \times \{0\}),0\Big) \in R_e^2.$$
(Recall $y_2$ is the $y$-coordinate of $\bm \theta_2$.)
We must check that this map sends the bottom edge of $R_e^1$ to the bottom edge of $R_e^2$. 
This follows from the fact that $\Psi_{\bm \theta_1}(\mu_2')=\Xi\big(P_{\w_2}(\bm \theta_2)\big)$, because the bottom edge
of $R_e^1$ is a saddle connection, $\sigma_h=[0,\w_1(\vb)] \times \{0\}$. We evaluate the $x$-coordinate of the lower right endpoint as
\begin{equation}
\label{eq:bottom}
\begin{array}{rcl}
\displaystyle \phi_1|_{B}\big(\w_i(\vb),0\big) & = & \displaystyle \frac{1}{|y_2|} \mu_2'(\sigma_h) =
\displaystyle \frac{1}{|y_2|} \big|\Xi\big(P_{\w_2}(\bm \theta_2)\big)(\hom{\sigma_h})\big| \\
& = & \displaystyle \frac{1}{|y_2|} \big|P_{\w_2}(\bm \theta_2)(\vb)\big| = \frac{1}{|y_2|} \big|y_2 \w_2(\vb)\big|=\w_2(\vb).
\end{array}
\end{equation}
This shows that $\phi_1|_{L}$ maps to the lower edge, but it may not be surjective or continuous if $\mu'$ contains atoms. 
For the left, top and right edges ($L$, $T$ and $R$ respectively), we use the following formulas:
\begin{equation}
\label{eq:not_bottom}
\begin{array}{c}
\displaystyle 
\phi_1|_{L}(0,t)=\Big(0,\frac{1}{|x_2|} \mu_2'(\{0\} \times [0,t]) \Big). \\
\quad \phi_1|_{T}\big(t,\w_1(\va)\big)=\Big(\frac{1}{|y_2|} \mu_2'([0,t] \times \{\w_1(\va)\}),\w_2(\va)\Big). \\
\displaystyle 
\quad \phi_1|_{R}\big(\w_1(\vb),t\big)=\Big(\w_2(\vb),\frac{1}{|x_2|} \mu_2'(\{w_1(\vb)\} \times [0,t])\Big).
\end{array}
\end{equation}
Together these define the map $\phi_1|_{\partial R_e^1}: \partial R_e^1 \to \partial R_e^2$.
A similar check for each edge shows that $\phi_1|_{\partial R_e^1}$ sends the respective vertices of $R_e^1$ to the respective vertices of $R_e^2$.

\begin{proposition}
The map $\phi_1|_{\partial R_e^1}$ sends points on the same connected component of a leaf of $\hat \F_{\bm \theta_1}$
inside $R_e^1$ to points on the same connected component of a leaf of $\hat \F_{\bm \theta_2}$
inside $R_e^2$.
\end{proposition}
\begin{proof}
We can assume without loss of generality that $\bm \theta_1$ and $\bm \theta_2$ are in the first quadrant. We must check that if the points $P, Q \in \partial R_e^1$ satisfy $\overrightarrow{PQ} \parallel {\bm \theta}_1$, then
$\overrightarrow{P'Q'} \parallel {\bm \theta}_2$ where 
$P'=\phi_1|_{\partial R_e^1}(P)$ and $Q'=\phi_1|_{\partial R_e^1}(Q)$. We will check this in case $P$ lies on the left edge $L$ and $Q$ lies on the top edge $T$. We leave the remaining cases to the reader. In this case via the identification
of $R_e^1$ with the rectangle $[0,\w_i (\vb)] \times [0, \w_i(\va)]$, there is a $r>0$ so that
$P=\big(0,\w_1(\va)-r y_1\big)$ and $Q=\big(r x_1, \w_1(\va)\big)$. Let $O=\big( 0, \w_1(\va)\big)$. 
We note that the transversals $\overline{OP}$ and $\overline{OQ}$ cross the same collection of leaves.
Let $s=\mu_2'(\overline{OP})$. Then by this observation $s=\mu_2'(\overline{OQ})$.
We compute
$$\begin{array}{rcl}
P' & = & \Big(0,\frac{1}{x_2} \mu_2'(\{0\} \times [0,\w_1(\va)-r y_1]) \Big)=
\Big(0,\frac{1}{x_2} \big(\mu_2'(\{0\} \times [0,\w_1(\va)])-\mu_2'(\overline{OP}\big) \Big) \\
& = & 
\Big(0, \frac{1}{x_2}\big(x_2 \w_2(\va)-s\big)\Big)=\Big(0, \w_2(\va)-\frac{s}{x_2}\Big), \textrm{ and} \\
Q' & = & \Big(\frac{1}{y_2} \mu_2'(\overline{OQ}),\w_2(a)\Big)=\Big(\frac{s}{y_2},\w_2(a)\Big).\end{array}$$
Indeed, $\overrightarrow{P' Q'}=\frac{s}{x_2 y_2} {\bm \theta}_2$, with $\bm \theta_2=(x_2, y_2)$. 
\end{proof}
By the proposition, we can extend definition of $\phi_1|_{\partial R_{e}^1}$ to a map 
$\phi_1|_{R_{e}^1}: R_{e}^1 \to R_{e}^2$. If $\{P, Q\}$ is the boundary of a connected component of a leaf of $\hat \F_{\bm \theta_1}$ inside $R_e^1$, then by the proposition we know $\overline{P' Q'}$ is a connected component of a leaf of $\hat \F_{\bm \theta_2}$ inside $R_e^2$, 
where $P'=\phi_1|_{\partial R_e^1}(P)$ and $Q'=\phi_1|_{\partial R_e^1}(Q)$. Then we can define $\phi_1$ onto $\overline{PQ}$ by defining $\overline{PQ} \to \overline{P'Q'}$ affinely (so it scales distance linearly). Thus we have defined
$\phi_1|_{R_{e}^1}: R_{e}^1 \to R_{e}^2$.

We similarly define $\phi_1|_{R_{e}^1}: R_e^1 \to R_e^2$ for all $e \in \E$. 
These maps agree on the boundaries of rectangles, so we have defined the map
$\phi_1:S(\G, \w_1) \to S(\G, \w_2)$. We must show that $\phi_1$ is continuous. Discontinuities of $\phi_1$ can only arise from atoms of $\mu_2'$. For each atom $A$ of $\mu_2'$ there is a strip in $S(\G, \w_2) \smallsetminus \phi_1\big(S(\G, \w_1)\big)$ whose width is the measure of the atom. Note that each vertex of each rectangles $R_e^2$ is in $\phi_1\big(S(\G, \w_1)\big)$. Therefore for each atom $A$ we have an isometrically embedded strip $\epsilon:\big(0, \mu_2'(\{A\})\big) \times \R \to S(\G, \w_2)$ which sends the vertical foliation of the strip to the foliation $\hat \F_{{\bm \theta}_2}$. But, the vertical flow in the strip is not conservative, and so the existence of this strip contradicts the fact that the straight-line flow in direction $\bm \theta_2$ on $S(\G, \w_2)$ is conservative.
See Theorem \ref{thm:pr}.
Thus $\phi_1:S(\G, \w_1) \to S(\G, \w_2)$ is indeed continuous and surjective.

It remains to show that $\phi_1$ is invertible. We can construct a map $\phi_2:S(\G, \w_2) \to S(\G, \w_1)$ by switching the rolls of $1$ and $2$. The composition $\phi_2 \circ \phi_1:S(\G, \w_1) \to S(\G, \w_1)$ preserves the Lebesgue measures of transversals to $\hat \F_{\bm \theta_1}$, and it preserves all vertices of rectangles $R_e^1$. 
Therefore, the composition must act trivially on the boundaries of rectangles $R_e^1$. Moreover,
the composition acts affinely on connected components of leaves intersected with $R_e^1$, and so must act trivially on each
such connected component of a leaf. Therefore, $\phi_2 \circ \phi_1$ is the identity map.

Finally, we consider the uniqueness statement. Now suppose there were $\mu_2', \mu_2'' \in \M_1$ so that there are homeomorphisms $\phi_1', \phi_1'':S(\G, \w_1) \to S(\G, \w_2)$ 
which send $R_e^1$ to $R_e^2$ and respect the names of the boundary edges of the rectangles (bottom, top, left and right)
and which push the measures $\mu_2', \mu_2''$ forward to the Lebesgue $\hat \F_{\bm \theta_2}$-transverse measure on  $S(\G, \w_2)$.  Then we notice that the cohomology classes $\Psi_{\bm \theta_1}(\mu_2')$ and $\Psi_{\bm \theta_1}(\mu_2'')$
must be the same. Therefore, $\mu_2' = \mu_2''$ by Lemma \ref{lem:injective} and Theorem \ref{thm:pr}. Finally, we see 
that the restrictions of $\phi_1'$ and $\phi_1''$ to the boundaries of rectangles are determined by $\mu_2'$ and $\mu_2''$. (In fact they must be determined as in equations
\ref{eq:bottom} and \ref{eq:not_bottom}.) Thus, the restrictions of $\phi_1'$ and $\phi_1''$  to the boundaries of rectangles must be the same. This concludes the proof of
Theorem \ref{thm:topological conjugacy}.

We have the following corollary to the topological conjugacy theorem, which will be useful later.
It is stated in the context of Theorem \ref{thm:topological conjugacy}.

\begin{corollary}
\label{cor:pullback_measure}
The pullback of the Lebesgue $\hat \F_{\bm \theta_2}$-transverse measure on $S(\G, \w_2)$ 
under the homeomorphism $\phi:S(\G, \w_1) \to S(\G, \w_2)$ (which was denoted $\phi_1$ and constructed above)
is $\Psi_{\bm \theta_1}^{-1} \circ \Xi\big(P_{\w_2}(\bm \theta_2)\big)$.
\end{corollary}

\subsection{Surviving functions and measures}
\label{sect:surviving}
We will now continue the discussion from Section \ref{sect:operators}. 
Recall that if $\f$ was a $\btheta$-survivor, then we obtain an invariant measure 
$\Psi_{\bm \theta}^{-1} \circ \Xi(\f)$ by Theorem \ref{thm:operator_theorem} and Definition \ref{def:function_survivors}.
It turns out that under a ``reasonable'' assumption all measures arise in this way.
Recall that $\Surv_{\bm \theta}$ denotes the set of all $\bm \theta$-survivors in $\R^\V$. 

\begin{definition}[Subsequence decay property]
\label{def:subseqence_decay}
We say that $S(\G,\w)$ has the {\em subsequence decay property}, if for any 
$\bm \theta \in \Rn_\lambda$ and any $\f \in \Surv_{\bm \theta}$ there is a subsequence $\langle g_{n_i} \rangle$ of the
$\lambda$-shrinking sequence of $\bm \theta$ for which
$$\lim_{i \to \infty} \Act^{g_{n_i}}(\f)(\vv) = 0 \quad \textrm{for all $\vv \in \V$.}$$
\end{definition}

Recall that by Theorem \ref{thm:operator_theorem}, $m \in \Coh$ is the cohomology class of a 
$\hat \F_{\bm \theta}$-transverse measure if and only if $m$ is a $(\btheta,n)$ survivor for all $n \geq 0$. 

\begin{theorem}
\label{thm:surviving_functions}
Suppose $S(\G,\w)$ has the subsequence decay property, and 
let $\bm \theta \in \Rn_\lambda$. If $m \in \Coh$ is a $(\bm \theta,n)$-survivor
for all integers $n \geq 0$, then $m=\Xi(\f)$ where $\f \in \Surv_{\bm \theta}$.
\end{theorem}

We prove Theorem \ref{thm:surviving_functions} in Section \ref{sect:operators2}. In Section \ref{sect:no valance one}, we will show that
if $\G$ has no nodes of valance one, then $S(\G, \w)$ has the subsequence decay property.
(In fact, we prove that $S(\G, \w)$ has a stronger property. See Definition \ref{def:critical_decay} and Theorem \ref{thm:properties}, below.)

By combining Theorems \ref{thm:surviving_functions} and \ref{thm:operator_theorem}, we have the following.

\begin{corollary}
\label{cor:symbolic}
If $S(\G,\w)$ has the subsequence decay property, then
$\displaystyle \Psi_{\bm \theta}(\M_{\bm \theta})=\Xi(\Surv_{\bm \theta})$.
Moreover, both $\Psi_{\bm \theta}$ and $\Xi$ are injective, so this yields a linear isomorphism between $\M_{\bm \theta}$ and $\Surv_{\bm \theta}$.
\end{corollary}

In other words, we have reduced the problem of classifying locally finite transverse measures to the study
of $\btheta$-survivors in $\R^\V$.

\subsection{The action of the adjacency operator}
\label{sect:outline_adjacency}

In this section, we introduce the adjacency operator $\A$ to our arguments.

Let $\bar \cdot:\R^2 \to \R^2 \label{not:bar1}$ denote the involution $\overline{(x,y)}=(y,x)$. The action of $\bar \cdot$ permutes the quadrants. We define the action of $\bar \cdot$
on pairs of signs so that
$\overline{\Q_s}=\Q_{\bar s}$. 
The action of $\bar \cdot$ conjugates $\rhoa{h}$ to $\rhoa{v}$ and vice versa. 
The induced action on $G$ is given by the group homomorphism $\bar \cdot : G \to G \label{not:bar2}$ defined so that  $\bar h=v$ and $\bar v=h$. It follows that for all $g \in G$ and all $\v \in \R^2$, we have
$$\rhoa{\bar g}(\v)=\overline{\rhoa{g}(\bar \v)}.$$ 
Recall from equation 
\ref{eq:interactions}
that $\A \Ho=\Vo \A$ and $\A \Vo=\Ho \A$. Thus, for $g \in G$, we have
\begin{equation}
\label{eq:conj_by_A}
\A \Act^g=\Act^{\bar g} \A.
\end{equation}

Corollary \ref{cor:symbolic} indicated that in the presence of the subsequence decay property, $\Xi ^{-1} \circ \Psi_{\bm \theta}$ is a bijective correspondence between
$\M_{\bm \theta}$, the space of transverse measures for the foliation in direction $\bm \theta$, and the space $\Surv_{\bm \theta}$ of $\bm \theta$-survivors in $\R^\V$. We have the following.

\begin{proposition}
If $\f \in \R^\V$ is a $\bm \theta$-survivor then $\A(\f)$ is a $\bar {\bm \theta}$-survivor.
\end{proposition}
\begin{proof}
Let $\langle g_n \rangle$ and $\langle s_n \rangle$ denote the shrinking and sign sequences of $\bm \theta$, respectively. Then
the shrinking sequence of $\bar {\bm \theta}$ is $\langle \bar g_n \rangle$ and the sign sequence is $\langle \bar s_n \rangle$. 
By the notion of survivors given in Proposition \ref{prop:equiv_survivors},
as $\f$ is a $\bm \theta$-survivor, we have $\Act^{g_n}(\f) \in \hat \Q_{s_n}$ for all $n$. 
By equation \ref{eq:conj_by_A}, we know
$$\Act^{\bar g_n} \circ \A (\f)=\A \circ \Act^{g_n}(\f) \in \A(\hat \Q_{s_n}) \subset \hat \Q_{\bar s_n}$$
for all $n$. Hence, $\A (\f)$ is a $\bar {\bm \theta}$-survivor by Proposition \ref{prop:equiv_survivors}.
\end{proof}

We will show that $\A$ is a bijection from the cone of $\bm \theta$-survivors to the cone of $\bar {\bm \theta}$-survivors.
First we will discuss injectivity. Note that $\A$ is not injective when considered on $\R^\V$. We have the following description of the the group action $\Act$ on 
$\ker~\A=\{\k \in \R^\V~:~\A \k={\mathbf 0}\}$. For any subset $U \subset \V$ we define $\pi_U:\R^\V \to \R^\V$ by
\begin{equation}
\label{eq:pi_U}
\pi_{U}(\f)(\vv)=\begin{cases}
\f(\vv) & \textrm{if $\vv \in U$}\\
0 & \textrm{otherwise.}
\end{cases}
\end{equation}

\begin{proposition}[Kernel of $\A$]
\label{prop:kernel}
For all $\f \in \R^\V$ and all $k \in \Z$, we have 
$$\Ho^k(\f)=\f+k \pi_{\Beta} \circ \A(\f)
\quad \textrm{and} \quad \Vo^k(\f)=\f+k \pi_{\Alpha} \circ \A(\f).$$
In particular, 
if $\k \in \ker~\A$ then $\Act^g(\k)=\k$ for all $g \in G$. 
\end{proposition}
The proof is trivial, and follows from comparing the definitions of $\A$, $\Ho$ and $\Vo$. (See Definition \ref{def:adj} and equations \ref{eq:Ho} and \ref{eq:Vo}, respectively.)

\begin{proposition}[Injectivity of $\A$]
\label{prop:injectivity}
Suppose $S(\G,\w)$ has the subsequence decay property. Then, the restriction of $\A$ to the cone of all $\bm \theta$-survivors is injective.
\end{proposition}
\begin{proof}
Suppose that $\f_1, \f_2$ are both ${\bm \theta}$-survivors and that $\A(\f_1)=\A(\f_2)$. Then $\f_2=\f_1+\k$ with $\k \in \ker ~ \A$. We will show $\k(\vx)=0$ for all $\vx \in \V$. 
By Proposition \ref{prop:kernel}, for all $g \in G$ we have
$$\k=\Act^g(\k)=\Act^g(\f_2)-\Act^g(\f_1)$$
By the subsequence decay property, we know that there is a subsequence $g_{n_i}$ of the $\lambda$-shrinking sequence for $\bm \theta$ so that
$\lim_{i \to \infty} \Act^{g_i} (\f_j)={\mathbf 0}$ coordinatewise for $j=1,2$. For all $\vx \in \V$,
$$\k(\vx)=\lim_{i \to \infty} \Act^{g_{n_i}}(\k)(\vx)=\lim_{i \to \infty} \big( \Act^{g_{n_i}}(\f_2)- \Act^{g_{n_i}}(\f_1) \big)(\vx)=0.$$
\end{proof}

We will now discuss why $\A$ surjectively maps the $\bm \theta$-survivors onto the space of $\bar {\bm \theta}$-survivors. 
To do this we need to be able to find inverses under the adjacency operator.

We use $\R^\V_c \label{not:RVc}$ to denote the finitely supported functions $\V \to \R$. There is a natural bilinear pairing $\langle, \rangle:\R^\V \times \R^\V_c \to \R$ given by
\begin{equation}
\label{eq:pairing}
\langle \f, \x \rangle=\sum_{\vv \in \V} \x(\vv) \f(\vv).
\end{equation}
This sum is well defined because $\x$ is only non-zero at finitely many $\vv \in \V$. The operators $\A$, $\Ho$ and $\Vo$ restrict to actions on $\R^\V_c$. Furthermore,
\begin{equation}
\label{eq:operator_relations}
\langle \A \f, \x \rangle=\langle \f, \A \x \rangle, \quad
\langle \Ho \f, \x \rangle=\langle \f, \Vo \x \rangle, \quad \textrm{and} \quad
\langle \Vo \f, \x \rangle=\langle \f, \Ho \x \rangle
\end{equation}
for all $\f \in \R^\V$ and all $\x \in \R^\V_c$.
We define the group automorphism $\gamma:G \to G$ by extending the definition on generators
\begin{equation}
\label{eq:gamma}
\gamma(h)=v^{-1} \quad\textrm{and}\quad \gamma(v)=h^{-1}.
\end{equation} Note $\gamma^2$ is trivial. The natural extension of equation \ref{eq:operator_relations} to all $g \in G$ is
\begin{equation}
\label{eq:operator_relations2}
\langle \Act^g \f, \Act^{\gamma(g)} \x \rangle= \langle \f, \x \rangle
\end{equation}

We have the following variation of a theorem of Farkas \cite{Farkas1902}. (See \cite[Corollary 3.46]{AT07}, for a more modern treatment.)

\begin{lemma}[Farkas' theorem for the adjacency operator]
\label{lem:farkas}
Let $\f \in \hat \Q_{++}$. The following two statements are equivalent.
\begin{enumerate}
\item There is a $\g \in \hat \Q_{++}$ with $\A \g=\f$.
\item For all $\x \in \R^\V_c$, if  $\A \x \in \hat \Q_{++}$ then $\langle \f, \x \rangle \geq 0$.
\end{enumerate}
\end{lemma}

This lemma is proved in Appendix \ref{sect:farkas}; it follows from a generalized form of Farkas' theorem.
We will use criterion (2) to check for surjectivity. We will simplify this criterion.
First, we introduce a more strict version of the subsequence decay property. In order to define this property,
we need the following useful fact about renormalizable directions.

\begin{lemma}[Critical times]
\label{lem:critical_times}
Suppose $\bm \theta$ is a $\lambda$-renormalizable direction. Then there are infinitely many $n \in \N$ such that 
if $\f \in \R^\V$ is a $(\bm \theta, n)$-survivor, then $\f$ is a $(\bm \theta, m)$-survivor for all $m \leq n$.
\end{lemma}
This lemma is proved in section \ref{sect:quadrant tracking}. We prove this lemma, by showing that there a sequence of values of $n$ which are guaranteed to have the property above.
These values of $n$ satisfy a simple combinatorial criterion related to the shrinking sequence, and we call them {\em critical times}. See Definition \ref{def:critical_time}.

\begin{definition}[Critical decay property]
\label{def:critical_decay}
$S(\G,\w)$ has the {\em critical decay property}, if for any 
$\bm \theta \in \Rn_\lambda$ and any $\f \in \Surv_{\bm \theta}$ the sequence of critical times, $\langle g_{n_i} \rangle$, of the
$\lambda$-shrinking sequence of $\bm \theta$ satisfies
$$\lim_{i \to \infty} \Act^{g_{n_i}}(\f)(\vv) = 0 \quad \textrm{for all $\vv \in \V$.}$$
\end{definition}

The critical decay property implies the subsequence decay property of Definition \ref{def:subseqence_decay}.

\begin{definition}[Adjacency sign property]
\label{def:adjacency sign}
The graph $\G$ has the {\em adjacency sign property}, 
if for any $\bm \theta \in \Q_{++} \cap \Rn_\lambda$, any $\f \in \Surv_{\bm \theta}$, any 
$\x \in \R^\V_c$ such that $\A \x \in \hat \Q_{++}$ and any
critical time $t$ for the $\lambda$-shrinking sequence $\langle g_{n} \rangle$ we have
$$\langle \Upsilon^{g_t}(\f),  \Upsilon^{\gamma(g_t)}(\x)-\x \rangle \geq 0.$$
\end{definition}

In section \ref{sect:no valance one}, we prove the following.

\begin{theorem}[Graphs without vertices of valance one]
\label{thm:properties}
If $\G$ is an infinite connected bipartite graph with no vertices of valance one and $\w$ is a positive eigenfunction for the adjacency operator, then $S(\G,\w)$ has the critical decay property and $\G$ has the adjacency sign property.
\end{theorem}

\begin{remark}
The author believes that these properties should hold even when $\G$ has vertices of valance one.
\end{remark}

\begin{proposition}
\label{prop:weak_surj}
Assume $S(\G,\w)$ has the critical decay property and $\G$ has
the adjacency sign property. Let $\bm \theta \in \Rn_\lambda$ and $\f \in \Surv_{\bar {\bm \theta}}$.
Then there is a $(\bm \theta, 0)$-survivor, $\g$, with $\A\g=\f$.
\end{proposition}
\begin{proof}
Without loss of generality,
we may assume $\f \in \hat \Q_{++}$. The statement that $\g$ is a $(\bm \theta, 0)$-survivor is equivalent to saying $\g \in \hat \Q_{++}$. We show
the existence of such a $\g$ using Farkas' theorem for the adjacency operator. Let $\x \in \R^\V_c$ with $\A \x \in \Q_{++}$. By equation \ref{eq:operator_relations2}, for any $g \in G$, 
$$\langle \f, \x \rangle=\langle \Upsilon^g (\f), \Upsilon^{\gamma(g)}(\x) - \x\rangle +  \langle \Upsilon^g (\f), \x\rangle.$$
Let $\langle g_{n_i} \rangle$ be the critical subsequence of the $\lambda$-shrinking sequence of $\bm \theta$.
By the critical decay property, we see $\langle \Upsilon^{g_{n_i}} (\f), \x\rangle \to 0$ as $i \to \infty$. By the adjacency sign property,
we have $\langle \Upsilon^{g_{n_i}} (\f), \Upsilon^{\gamma(g_{n_i})}(\x) - \x\rangle$ is always non-negative. Thus, 
$$\langle \f, \x \rangle= \lim_{i \to \infty} \langle \Upsilon^{g_{n_i}} (\f), \Upsilon^{\gamma(g_{n_i})}(\x) - \x\rangle+\langle \Upsilon^{g_{n_i}} (\f), \x\rangle \geq 0$$
as needed to apply Farkas' Theorem (Lemma \ref{lem:farkas}).
\end{proof}

Given Proposition \ref{prop:weak_surj}, the proof that there is a $\bm \theta$-survivor, $\g$, with $\A \g=\f$ is not difficult. Roughly, the proof
carefully combines this Proposition with the Critical Times Lemma. Let $n$ be a critical time. Then,
we can find a $\big(\rho_\lambda^{g_n}(\bm \theta), 0\big)$-survivor, $\g_n$, such that $\A(\g_n)=\Act^{g_n}(\f)$.
Let $\g=\Act^{\bar g_n^{-1}}(\g_n)$. By Lemma \ref{lem:critical_times}, $\g$ is a $(\bar {\bm \theta}, m)$-survivor for 
all $m=0, \ldots, n$. Thus, the main technical difficulty is showing that the sets of all such candidates of the form 
$\Act^{\bar g_n^{-1}}(\g_n)$ for each critical time $n$ have a non-trivial intersection. This argument is fully explained in
section \ref{sect:surjectivity} in the proof of the following lemma.

\begin{lemma}[Surjectivity]
\label{lem:surjectivity}
Assume $S(\G,\w)$ has the critical decay property and $\G$ has
the adjacency sign property. Let $\bm \theta \in \Rn_\lambda$ and $\f \in \Surv_{\bar {\bm \theta}}$.
Then, there is a $\g \in \Surv_{\bm \theta}$ with $\A(\g)=\f$.
\end{lemma}

Combining Lemma \ref{lem:surjectivity} with Proposition \ref{prop:injectivity}, we observe the following.

\begin{theorem}
\label{thm:bijection}
Assume $S(\G,\w)$ has the critical decay property and $\G$ has
the adjacency sign property. If $\bm \theta$ is a $\lambda$-renormalizable direction, then the adjacency operator restricts to bijection between the 
set of $\bm \theta$-survivors and the set of $\bar {\bm \theta}$-survivors.
\end{theorem}

We have the following trivial proposition.

\begin{proposition}
Suppose $\f \in \Surv_{\bm \theta}$ then $\A^2(\f)-\f \in \Surv_{\bm \theta}$.
\end{proposition}
\begin{proof}
We consider the operator $\A^2-\I$. First observe that for all $s \in \SP$, we have $(\A^2-\I)(\hat \Q_{s}) \subset \hat \Q_{s}$. To see this, we may take $s=++$ without loss of generality. Then, if $\vv \in \V$ we can find an $\vw \sim \vv$. Then by definition of $\A$ we have
$$\A^2(\f)(\vv) \geq \A(\f)(\vw) \geq \f(\vv),$$
and therefore $\A^2(\f)(\vv) -\f(\vv) \geq 0$ as desired.

Second, we may observe that $\A^2-\I$ commutes with $\Act^g$ for all $g \in G$. This follows from linearity and equation \ref{eq:conj_by_A}. 
Since we assumed $\f$ was a $(\theta,n)$ survivor, we know $\Act^{g_n}(\f)\in \hat \Q_{s_n}$. Therefore, by our first observation, we see 
$$\Act^{g_n} \circ (\A^2-\I)(\f)= (\A^2-\I) \circ \Act^{g_n}(\f) \in \hat \Q_{s_n}$$ 
as well. 
Thus, $(\A^2-\I)(\f)$ is a $(\theta, n)$-survivor for all $n$.
\end{proof}

Our main theorem is really a corollary of Theorem \ref{thm:bijection}. The Ergodic measures theorem (Theorem \ref{thm:ergodic2}) follows by combining
the theorem below with Theorem \ref{thm:properties} and some Martin boundary theory as described in the next subsection.

\begin{theorem}
\label{main_thm:properties}
Assume $S(\G,\w)$ has the critical decay property and $\G$ has
the adjacency sign property.
If $\f \in \Surv_{\bm \theta}$ is extremal, then $\A^2(\f)=\lambda_2^2(\f)$ for some positive $\lambda_2 \in \R$. 
\end{theorem}
\begin{proof}
If $\f$ is extremal, then $\A^2(\f)$ must be extremal, since $\A^2: \Surv_{\bm \theta} \to \Surv_{\bm \theta}$ is a linear bijection. 
But, $\A^2(\f)=\f+\big(\A^2(\f)-\f\big)$ and both $\f, \A^2(\f)-\f \in \Surv_{\bm \theta}$. These two functions are linearly independent (contradicting extremality of $\A^2(\f)$)
unless $\A^2(\f)=\lambda_2^2(\f)$ for some positive $\lambda_2 \in \R$.
\end{proof}

\subsection{Martin boundary theory}
\label{sect:martin boundary intro}
In this subsection, we use Martin boundary theory to finish the proof of the ergodic measures theorem. 
We formally define the Martin boundary in Appendix \ref{sect:boundary}. Here, we will only introduce the facts we need to use to prove our main theorem.

Choose any vertex $o \in \V$ called the root. Given $\lambda$ such that positive eigenfunctions of $\A$ exist with eigenvalue $\lambda$, the Martin compactification ${\mathcal V}_\lambda \label{not:V lambda 1}$ is a compactification of the vertex set $\V$. The Martin boundary is the set ${\mathcal M}_\lambda=
{\mathcal V}_\lambda \smallsetminus \V \label{not:M lambda 1}$. Points $\zeta \in {\mathcal M}_\lambda \label{not:zeta}$ correspond to positive functions $\k_\zeta \in \R^\V \label{not:k zeta}$ so that $\A\big(\k_\zeta\big)=\lambda \k_\zeta$ and $\k_\zeta(o)=1$. The function $\zeta \mapsto \k_\zeta$ is continuous when $\R^\V$ is given the topology of pointwise convergence. The subset 
$${\mathcal M}^{\textit{min}}_\lambda=\{\zeta \in {\mathcal M}_\lambda~:~\textrm{$\k_\zeta$ is an extremal positive eigenfunction}\}.
\label{not:M lambda min}$$
is called the minimal Martin boundary, and is a Borel subset of ${\mathcal M}_\lambda$. 

For the following result see \cite[Part IV, Theorems 24.7-24.9]{W00}.

\begin{theorem}[Poisson-Martin representation theorem]
\label{thm:poisson-martin}
For each non-negative $\f \in \R^\V$ with $\A(\f)=\lambda \f$, there is a unique Borel measure $\nu_\f \label{not:nu f}$ on ${\mathcal M}_\lambda$ with $\nu_\f({\mathcal M}_\lambda \smallsetminus {\mathcal M}^{\textit{min}}_\lambda)=0$ and 
$$\f(\vv)=\int_{{\mathcal M}_\lambda} \k_\zeta(\vv) ~d\nu_\f(\zeta),$$
for all $\vv \in \V$. Furthermore, if $\f$ is an extremal positive eigenfunction, then $\f=c \k_\zeta$ for some $c>0$ and $\zeta \in {\mathcal M}^{\textit{min}}_\lambda$, and $\nu_\f$ is the Dirac measure with mass $c$ at $\zeta$.
\end{theorem}

Now we will study those $\f \in \hat \Q_s$ with $s$ a sign pair and for which $\A^2(\f)=\lambda^2(\f)$. 
We let $\G^2$ denote the graph with vertex set $\V$ and edge set $\E^2$, where each path of length two in $\E$
corresponds to an edge of $\G^2$. We allow loops in the graph $\G^2$, which arise from moving forward along an edge in $\E$ and then back to the starting point along
the same edge or another edge with the same endpoints. Since $\G$ is connected and bipartite, the graph $\G^2$ consists of two connected components with vertex sets $\Alpha$ and $\Beta$. We denote these two components by $\G_\Alpha$ and $\G_\Beta$. We will let $\A_\Alpha$ and $\A_\Beta$ be the adjacency operators on these two graphs and they satisfy the equations
$$\A^2(\f)(\va)=\A_\Alpha(\f|_\Alpha)(\va) \quad \textrm{and} \quad \A^2(\f)(\vb)=\A_\Beta(\f|_\Beta)(\vb)$$
for all $\f \in \R^\V$ and all $\va \in \Alpha$ and $\vb \in \Beta$. 

As in \S \ref{ss:eigenfunction}, we let $E_\lambda \label{not:E lambda}$ be the collection of non-negative $\f \in \R^\V$ such that $\A(\f)=\lambda \f$. For $s \in \SP$, define
$\widehat E_s=\{\f \in \hat \Q_s~:~\A^2(\f)=\lambda^2 \f\}. \label{not:hat E s}$
Note that $E_\lambda \subsetneq \hat E_{++}$.
For $\f \in \R^\V$ we define the functions $\f_\Alpha , \f_\Beta \in \R^\V$ according to the rule
\begin{equation}
\label{eq:fsub}
\f_\Alpha(\vv)=\begin{cases}
\f(\vv) & \textrm{if $\vv \in \Alpha$} \\
\frac{1}{\lambda} \A(\f)(\vv) & \textrm{if $\vv \in \Beta$}
\end{cases}
\quad \textrm{and} \quad
\f_\Beta(\vv)=\begin{cases}
\frac{1}{\lambda} \A(\f)(\vv) & \textrm{if $\vv \in \Alpha$} \\
\f(\vv) & \textrm{if $\vv \in \Beta$.}
\end{cases}
\end{equation}
Observe that if $\f \in \R^\V$ is a non-negative function satisfying $\A^2(\f)=\lambda^2 \f$, then we have 
$\A(\f_\Alpha)=\lambda \f_\Alpha$ and $\A(\f_\Beta)=\lambda \f_\Beta$. In fact,

\begin{proposition}
For all $s=(s_1, s_2) \in \SP$, the function $\f \mapsto (\f_\Alpha, \f_\Beta)$ restricts to a linear bijection 
$\widehat E_s \to s_1 E_\lambda \times s_2 E_\lambda$. Here, $s_i E_\lambda$ denotes non-negative eigenfunctions
with eigenvalue $\lambda$ when $s_i=1$ and the non-positive eigenfunctions when $s_i=-1$. The inverse map is given by
$(\f_\Alpha, \f_\Beta) \mapsto \pi_\Alpha(\f_\Alpha)+\pi_\Beta(\f_\Beta)$, with 
$\pi_\ast$ defined as in equation \ref{eq:pi_U}.
\end{proposition}

This means that we can use the Poison-Martin representation theorem to express those $\f \in \widehat E_s$ for $s \in \SP$. For this it is natural to consider signed measures on the Martin boundary. For $\nu$ a Borel measure on 
${\mathcal M}_\lambda$, we let $\nu^+$ and $\nu^-$ be the mutually singular unsigned Borel measures satisfying $\nu=\nu^+ - \nu^-$
obtained via the Hahn decomposition theorem. We let $M \label{not: M}$ be the space of signed Borel measures $\nu$ on ${\mathcal M}_\lambda$ such that 
\begin{enumerate}
\item $\nu(A)=0$ for all measurable $A \subset {\mathcal M}_\lambda \smallsetminus {\mathcal M}^{\textit{min}}_\lambda$, and
\item $\int_{{\mathcal M}_\lambda} \k_\zeta(\vv) ~d\nu(\zeta)$ is defined and finite for all $\vv \in \V$.
\end{enumerate}
Note the space $M$ is a real vector space.
We also define the subsets
$$\label{not: M +} M^+=\{\nu \in M~:~\nu^-({\mathcal M}_\lambda)=0\} \quad \textrm{and} \quad
M^-=\{\nu \in M~:~\nu^+({\mathcal M}_\lambda)=0\}.$$
We have the following corollary to the Poisson-Martin representation theorem.

\begin{corollary}
\label{cor:measure_pair}
Let $s=(s_1, s_2) \in \SP$ and $\f \in \widehat E_s$. Then there is a unique $(\nu_\Alpha, \nu_\Beta) \in M^2=M \times M$ for which 
$$\f(\va)=\int_{{\mathcal M}_\lambda} \k_\zeta(\va) ~d\nu_\Alpha(\zeta) \quad \textrm{and} \quad
\f(\vb)=\int_{{\mathcal M}_\lambda} \k_\zeta(\vb) ~d\nu_\Beta(\zeta)$$
for all $\va \in \Alpha$ and all $\vb \in \Beta$.
Moreover, $\nu_\Alpha \in M^{s_1}$ and $\nu_\Beta \in M^{s_2}$.
\end{corollary}
\begin{proof}
We first prove existence. Given $\f$, we may consider $\f_\Alpha, \f_\Beta$ defined as in equation \ref{eq:fsub}. We have that $s_1 \f_\Alpha, s_2 \f_\Beta \in E_\lambda$. Therefore using the notation from the Poisson-Martin representation theorem, we may set 
$\nu_\Alpha=s_1 \nu_{s_1 \f_\Alpha}$ and $\nu_\Beta=s_2 \nu_{s_2 \f_\Beta}$. The pair $(\nu_\Alpha, \nu_\Beta)$ satisfy the equations of the corollary and we have $\nu_\Alpha \in M^{s_1}$ and $\nu_\Beta \in M^{s_2}$.

Now assume that $(\mu_\Alpha, \mu_\Beta) \in M^2$ is a second pair of measures satisfying the statement of the corollary. We will prove that $\mu_\Alpha=\nu_\Alpha$. Define the unsigned measures $\sigma_1=\nu^+_\Alpha+\mu_\Alpha^-$ and  $\sigma_2=\nu^-_\Alpha+\mu_\Alpha^+$. Define $\g_1, \g_2 \in \R^\V$ by 
$$\g_i(\vv)=\int_{{\mathcal M}_\lambda} \k_\zeta(\vv) ~d\sigma_i(\zeta) \quad \textrm{for all $\vv \in \V$.} $$
Note that for all $\va \in \Alpha$ we have $\g_1(\va)=\g_2(\va)$ since
$$\f(\va)=\int_{{\mathcal M}_\lambda} \k_\zeta(\va) ~d\nu_\Alpha(\zeta)=\int_{{\mathcal M}_\lambda} \k_\zeta(\va) ~d\mu_\Alpha(\zeta).$$
So that $\g_1(\va)=\g_2(\va)=\f(\va)+\int_{{\mathcal M}_\lambda} \k_\zeta(\va) ~d\nu^-_\Alpha(\zeta)+ \int_{{\mathcal M}_\lambda} \k_\zeta(\va) ~d\mu^-_\Alpha(\zeta)$. Now observe that both $\g_1$ and $\g_2$ satisfy
$\A(\g_i)=\lambda \g_i$, since each $\k_\zeta$ is such an eigenfunction. It follows that for all $\vb \in \Beta$,
$$\g_i(\vb)=\frac{1}{\lambda} \A(\g_i)(\vb)=\frac{1}{\lambda} \sum_{\va \sim \vb} \g_i(\va).$$
Since the expression on the right only depends on vertices $\va \in \Alpha$, we see $\g_1=\g_2$. Therefore, the uniqueness
part of the Poisson-Martin representation theorem implies $\sigma_1=\sigma_2$. Finally, observe that if $s_1=+1$ then $\nu_\Alpha^- \equiv 0$. Therefore, $\nu^+_\Alpha+\mu_\Alpha^-=\mu_\Alpha^+$. But, the fact that $\mu_\Alpha^- \perp \mu_\Alpha^+$ implies that $\mu_\Alpha^- \equiv 0$ and $\nu^+_\Alpha=\mu_\Alpha^+$. Similarly, if $s_1=-1$ then $\nu_\Alpha^+ \equiv 0$, and so $\mu_\Alpha^+ \equiv 0$ and $\nu^-_\Alpha=\mu_\Alpha^-$.
In either case, $\mu_\Alpha=\nu_\Alpha$. An identical argument shows $\mu_\Beta=\nu_\Beta$, concluding the uniqueness part of the proof. 
\end{proof}

Corollary \ref{cor:measure_pair} defines a linear map ${\mathcal N}:\bigcup_{s \in \SP} \widehat E_s \to M^2 \label{not:Nu}$
according to the rule ${\mathcal N}(\f)=(\nu_\Alpha, \nu_\Beta)$ with $(\nu_\Alpha, \nu_\Beta)$ the unique pair of signed
measures guaranteed to exist by the corollary.

The group of matrices $\GL(2, \R)$ acts on $M^2$ by matrix multiplication. Namely,
$$\left[\begin{array}{cc} a & b \\
c & d \end{array}\right]\left[\begin{array}{c} \nu_\Alpha \\
\nu_\Beta \end{array}\right]=\left[\begin{array}{c} a \nu_\Alpha+b \nu_\Beta \\
c \nu_\Alpha+d \nu_\Beta\end{array}\right].$$
Therefore, our representation $\rho_\lambda:G \to \SL(2, \R)$ induces an action on $M^2$. 

\begin{proposition}[Adjacency operation on measures]
\label{prop:adj_op_measures}
Let $\f \in \R^\V$ and suppose there exists $(\nu_\Alpha, \nu_\Beta) \in M^2$ so that
$$\f(\va)=\int_{{\mathcal M}_\lambda} \k_\zeta(\va) ~d\nu_\Alpha(\zeta) \quad \textrm{and} \quad
\f(\vb)=\int_{{\mathcal M}_\lambda} \k_\zeta(\vb) ~d\nu_\Beta(\zeta)$$
for all $\va \in \Alpha$ and $\vb \in \Beta$. Then
$$\A(\f)(\va)=\lambda \int_{{\mathcal M}_\lambda} \k_\zeta(\va) ~d\nu_\Beta(\zeta) \quad
\textrm{and} \quad
\A(\f)(\vb)=\lambda \int_{{\mathcal M}_\lambda} \k_\zeta(\vb) ~d\nu_\Alpha(\zeta).$$
\end{proposition}
\begin{proof}
Define $\f_\Alpha, \f_\Beta \in \R^\V$ according to the rule that
$$\f_\ast(\vv)=\int_{{\mathcal M}_\lambda} \k_\zeta(\vv) ~d\nu_\ast(\zeta) \quad \textrm{for all $\vv \in \V$ and $\ast \in \{\Alpha, \Beta\}$.}$$
Then $\A(\f_\ast)=\lambda \f_\ast$ since each $\k_\zeta$ satisfies $\A(\k_\zeta)=\lambda \k_\zeta$. Then we have $\f=\pi_\Alpha(\f_\Alpha)+\pi_\Beta(\f_\Beta)$. It follows that
$$\A(\f)=\A\big(\pi_\Alpha(\f_\Alpha)+\pi_\Beta(\f_\Beta)\big)=
\pi_\Beta \circ \A(\f_\Alpha)+\pi_\Alpha \circ \A(\f_\Beta)=\pi_\Beta (\lambda \f_\Alpha)+\pi_\Alpha (\lambda \f_\Beta).$$
Therefore, the conclusion follows. 
\end{proof}

\begin{proposition}[Group operation on measures]
\label{prop:group_op_measures}
Suppose $\f \in \R^\V$ satisfies $\A^2(\f)=\lambda^2 \f$ and there exist $(\nu_\Alpha, \nu_\Beta) \in M^2$ so that
$$\f(\va)=\int_{{\mathcal M}_\lambda} \k_\zeta(\va) ~d\nu_\Alpha(\zeta) \quad \textrm{and} \quad
\f(\vb)=\int_{{\mathcal M}_\lambda} \k_\zeta(\vb) ~d\nu_\Beta(\zeta)$$
for all $\va \in \Alpha$ and $\vb \in \Beta$. Let $g \in G$, and set $(\mu_\Alpha, \mu_\Beta)=\rhoa{g} \cdot (\nu_\Alpha, \nu_\Beta)$
and ${\hat \f}= \Upsilon^g(\f)$. Then $\A^2({\hat \f})=\lambda^2 {\hat \f}$ and
$${\hat \f}(\va)=\int_{{\mathcal M}_\lambda} \k_\zeta(\va) ~d\mu_\Alpha(\zeta) \quad \textrm{and} \quad
{\hat \f}(\vb)=\int_{{\mathcal M}_\lambda} \k_\zeta(\vb) ~d\mu_\Beta(\zeta)$$
for all $\va \in \Alpha$ and $\vb \in \Beta$. 
\end{proposition}
\begin{proof}
First, the statement that $\A^2({\hat \f})=\lambda^2 {\hat \f}$ follows from equation \ref{eq:conj_by_A}. Namely,
$$\A^2({\hat \f})=\Upsilon^{g}\circ \A^2(\f)=\lambda^2 \Upsilon^{g}(\f)=\lambda^2 {\hat \f}.$$

We prove the remainder of the statement by induction. In fact, it is sufficient to prove
it for $g$ the identity and a collection of generators. The case of the identity is trivial. 
We consider the case of $g=h^k$ and $k \in \Z$. Then 
$(\mu_\Alpha, \mu_\Beta)=(\nu_\Alpha+k \nu_\Beta, \nu_\Beta)$. We know that
$${\hat \f}=\Ho^k(\f)=\f+k \pi_\Alpha \circ \A(\f)=\f+k \pi_\Alpha \circ \A(\f).$$
In this case, the conclusion follows from Proposition \ref{prop:adj_op_measures}. The case of $g=v^K$ is similar.
\end{proof}

\begin{corollary}
\label{cor:Nu_conj}
Let $\f \in \widehat E_s$ for some $s \in \SP$. Suppose for some $g \in G$ we have $\Upsilon^g(\f) \in \widehat \Q_{s'}$ for some $s' \in \SP$. Then, $\Upsilon^g(\f) \in \widehat E_{s'}$ and 
$$\Nu \big( \Upsilon^g(\f) \big)=\rhoa{g} \cdot \Nu(\f).$$
\end{corollary}
\begin{proof}
Let ${\hat \f}=\Upsilon^g(\f)$. Proposition \ref{prop:group_op_measures} implies that $\A^2({\hat \f})=\lambda^2 {\hat \f}$. Therefore,
${\hat \f} \in \widehat E_{s'}$. Set $(\mu_\Alpha, \mu_\Beta)=\rhoa{g} \cdot \Nu(\f) \in M^2$. Proposition \ref{prop:group_op_measures} also indicates that 
$${\hat \f}(\va)=\int_{{\mathcal M}_\lambda} \k_\zeta(\va) ~d\mu_\Alpha(\zeta) \quad \textrm{and} \quad
{\hat \f}(\vb)=\int_{{\mathcal M}_\lambda} \k_\zeta(\vb) ~d\mu_\Beta(\zeta)$$
for all $\va \in \Alpha$ and $\vb \in \Beta$. By Corollary \ref{cor:measure_pair}, the pair of measures that satisfy this
statement are uniquely determined by ${\hat \f}$. Therefore, it must be that $\Nu({\hat \f})=(\mu_\Alpha, \mu_\Beta)$ as desired.
\end{proof}

This corollary allows us to prove the following important lemma.

\begin{lemma}
\label{lem:square}
Suppose $\f \in \R^\V$ satisfies $\A^2(\f)=\lambda^2 \f$ and is a $\bm \theta$-survivor with $\bm \theta=(x,y)\in \Rn_\lambda$. 
Let $(\nu_\Alpha, \nu_\Beta)=\Nu(\f)$. Then, 
$\frac{1}{x} \nu_\Alpha=\frac{1}{y} \nu_\Beta$. 
\end{lemma}
\begin{proof}
Let $\langle g_n \rangle$ be the $\lambda$-shrinking sequence of $\bm \theta$, and let $\langle s_n \rangle$ be the shrinking sequence. Since $\f$ is a $\bm \theta$-survivor, $\Upsilon^{g_n}(\f) \in \widehat \Q_{s_n}$.
Moreover, by equation \ref{eq:conj_by_A}, $\A^2 \big(\Upsilon^{g_n}(\f)\big) = \lambda^2 \Upsilon^{g_n}(\f)$.
Therefore, $\Upsilon^{g_n}(\f) \in \widehat E_{s_n}$ for all $n$. Let $(\nu_{\Alpha,n}, \nu_{\Beta,n}) =\Nu \circ \Upsilon^{g_n}(\f)$. Write $s_n=(s_{n,1}, s_{n,2})$. By Corollary \ref{cor:measure_pair}, we have $\nu_{\Alpha,n} \in M^{s_{n,1}}$
and $\nu_{\Beta,n} \in M^{s_{n,2}}$. Then by Corollary \ref{cor:Nu_conj}, we have
$$(\nu_{\Alpha,n}, \nu_{\Beta,n})=\rhoa{g_n} \cdot (\nu_{\Alpha,0}, \nu_{\Beta,0}).$$

Now let $A \subset {\mathcal M}_\lambda$ be any measurable subset. Let $\v=\big(\nu_{\Alpha,0}(A), \nu_{\Beta,0}(A)\big)$.
Then we know $\big(\nu_{\Alpha,n}(A), \nu_{\Beta,n}(A)\big)=\rhoa{g_n}(\v)$ for all $n$. Moreover, since $\nu_{\Alpha,n} \in M^{s_{n,1}}$
and $\nu_{\Beta,n} \in M^{s_{n,2}}$, we have $\rhoa{g_n}(\v) \in \cl({Q}_{s_n})$ for all $n$. 
Since $\langle g_n \rangle$  and $\langle s_n \rangle$  are $\lambda$-shrinking and sign sequences for a $\lambda$-renormalizable $\bm \theta$, by Proposition \ref{prop:quadrant_sequence}, we see that $\v$ must be a non-negative scalar multiple of $\bm \theta$. Equivalently, $\frac{1}{x} \nu_\Alpha(A)=\frac{1}{y} \nu_\Beta(A)$. Since this is true for all measurable sets $A$, it must be that $\frac{1}{x} \nu_\Alpha=\frac{1}{y} \nu_\Beta$.
\end{proof}

A converse to Lemma \ref{lem:square} follows from Proposition \ref{prop:survivor_in_plane}.
If $\nu$ is any Borel measure on ${\mathcal M}_\lambda$ with 
$\int_{{\mathcal M}_\lambda} \k_\zeta(\v) ~d\nu(\zeta)$ defined and finite for all $\vv \in \V$, then the pair $(x \nu, y \nu)=\Nu(\f)$ for a $\bm \theta$-survivor $\f \in \R^\V$. Indeed, if we define $\g \in \R^\V$ according to the rule that
$\g(\vv)=\int_{{\mathcal M}_\lambda} \k_\zeta(\vv) ~d\nu(\zeta)$ for all $\vv \in \V$, then we see that
$\f=P_\g(\bm \theta)$ where $P_\g$ denotes the parametrized plane defined in section  
\ref{sect:top_conj}. Therefore, we have the following two results.

\begin{corollary}[Survivors and the Martin boundary]
\label{cor:bijection}
The linear map $\Nu$ restricts to a bijection from the space of $\bm \theta$-survivors
$\f \in \R^\V$ satisfying $\A^2(\f)=\lambda^2 \f$ to pairs of measures of the form $(x \nu, y \nu)$ with $\nu$ an unsigned Borel measure on ${\mathcal M}_\lambda$ satisfying $\nu({\mathcal M}_\lambda \smallsetminus {\mathcal M}^{\textit{min}}_\lambda)=0$ and so that $\int_{{\mathcal M}_\lambda} \k_\zeta(\vv) ~d\nu(\zeta)$ is defined and finite for all $\vv \in \V$.
\end{corollary}

We have been working in the specific case when $\bm \theta$ is a $\lambda$-renormalizable direction and 
$\A(\f)=\lambda \f$. But, the statements of our main results in section \ref{sect:restatement} concern two eigenfunctions with possibly different eigenvalues. In the context of the statements of these theorems we have the following.  

\begin{corollary}[Extremal survivors]
\label{cor:bijection2}
Suppose $\w_1 \in \R^\V$ satisfies $\A(\w_1)=\lambda_1 \w_1$ and 
assume $S(\G,\w_1)$ has the critical decay property and $\G$ has the adjacency sign property.
Let $\bm \theta_1$ be a $\lambda_1$-renormalizable direction with
$\lambda_1$-shrinking sequence $\langle g_n \rangle$. Suppose $\f \in \Surv_{\bm \theta_1}$ is extremal. By Theorem \ref{main_thm:properties}, there is a 
$\lambda_2$ so that $\A(\f)=\lambda_2^2 \f$. Set $\bm \theta_2=(x_2, y_2)=\bm \theta(\langle g_n \rangle, \lambda_2)$ as in \S \ref{ss:renormalizable directions} and define $\w_2 \in \R^\V$ by
$$\w_2(\va)=\frac{1}{x_2} \f(\va) \quad \textrm{and} \quad
\w_2(\vb)=\frac{1}{y_2} \f(\vb) \quad \textrm{for $\va \in \Alpha$ and $\vb \in \Beta$.}$$
Then, $\w_2$ is an extremal positive eigenfunction of $\A$ with eigenvalue $\lambda_2$.
\end{corollary}

Note that $\f=P_{\w_2}(\bm \theta_2)$ by definition of $P_{\w_2}$ in Section \ref{sect:top_conj}. This is the connection with the topological conjugacy construction.

\begin{proof}
Let $\f \in \Surv_{\bm \theta_1}$ be extremal.
The statement of the corollary defines $\w_2$. We only need to prove that $\w_2$ is an extremal positive eigenfunction. 
From the previous corollary with $\lambda=\lambda_2$ and $\bm \theta=\bm \theta_2$, we know that $\Nu(\f)=(x_2 \nu, y_2 \nu)$ for some
Borel measure $\nu$ on ${\mathcal M}_\lambda$ satisfying $\nu({\mathcal M}_\lambda \smallsetminus {\mathcal M}^{\textit{min}}_\lambda)=0$. Moreover, since $\Nu$ is a linear bijection, we know that since $\f$ is extremal, $\nu$ is a Dirac measure supported on a single point $\zeta \in {\mathcal M}^{\textit{min}}_\lambda$. Let $c$ be the total mass of this atomic measure. Then, since $\Nu(\f)=(x_2 \nu, y_2 \nu)$,
$$\f(\va)=c x_2 \k_\zeta(\va) \quad \textrm{and} \quad \f(\vb)=c y_2 \k_\zeta(\vb) \quad \textrm{for $\va \in \Alpha$ and $\vb \in \Beta$.}$$
Therefore by definition of $\w_2$, we have $\w_2=c \k_\zeta$. Since $\zeta \in {\mathcal M}^{\textit{min}}_\lambda$, we know $\w_2$ is an extremal positive eigenfunction.
\end{proof}

We need to show that each measure associated to an extremal positive eigenfunction $\w_2$ is ergodic. This is essentially a converse of the above
Corollary. 

\begin{lemma}
\label{lem:ergodicity_implication}
Assume $S(\G,\w_1)$ has the critical decay property and $\G$ has the adjacency sign property. Let $\bm \theta_1$ be a $\lambda_1$-renormalizable direction with
$\lambda_1$ shrinking sequence $\langle g_n \rangle$. 
Suppose $\w_2$ is an extremal positive eigenfunction with eigenvalue $\lambda_2$, and let $\bm \theta_2=\bm \theta(\langle g_n \rangle, \lambda_2)$. Then the transverse measure 
$\mu_2=\Psi_{\bm \theta_1}^{-1} \circ \Xi\big(P_{\w_2}(\bm \theta_2)\big)$
is ergodic. 
\end{lemma}
\begin{proof}
By Proposition \ref{prop:survivor_in_plane}, the function $\f=P_{\w_2}(\bm \theta_2)$ is a $\bm \theta_1$-survivor. Therefore, by Lemma \ref{lem:injective} and Theorem \ref{thm:pr} there is a unique transverse measure $\mu_2=\Psi_{\bm \theta_1}^{-1} \circ \Xi(\f)$.
By Corollary \ref{cor:pullback_measure}, we know that $\mu_2$ arises from the pullback of a homomorphism
$\phi:S(\G, \w_1) \to S(\G, \w_2)$ as described in the Topological conjugacy theorem (Theorem \ref{thm:topological conjugacy}). In particular, Theorem \ref{thm:pr} implies that the system consisting of the surface $S(\G, \w_1)$ and measure $\mu_2$ is conservative. Hence the measure $\mu_2$ has an ergodic decomposition. See \cite[Theorem 2.2.9]{A97}. Recall that ${\mathcal M}_{\bm \theta_1}$ denotes the collection of all locally finite transverse measures to the straight-line foliation of $S(\G, \w_1)$ in direction $\bm \theta_1$. The ergodic decomposition yields a space $\Omega$ equipped with a probability measure $m$ and a collection of ergodic transverse measures 
$$\{\mu_\omega \in {\mathcal M}_{\bm \theta_1}~:~ \omega \in \Omega\}$$
(all depending on $\mu_2$) so that for any measurable transversal $\tau \subset S(\G, \w_1)$ the map $\omega \mapsto \mu_\omega(\tau)$ is measurable and 
$$\mu_2(\tau)=\int_{\Omega} \mu_\omega(\tau)~dm(\omega).$$
For each $\va \in \Alpha$, let $\sigma_\va$ be a vertical saddle connection crossing $\cyl_\va$. Similarly for each $\vb \in \Beta$, let $\sigma_\vb$ be a horizontal saddle connection crossing the vertical cylinder $\cyl_\vb$. Note that since $\Psi_{\bm \theta_1}(\mu_2)= \Xi(\f)$ we have 
$|\f(\va)|=\mu_2(\sigma_\va)$ and 
$|\f(\vb)|=\mu_2(\sigma_\vb)$, with a sign depending only on the quadrant containing $\bm \theta$. In particular,
$$|\f(\va)|=\int_{\Omega} \mu_\omega(\sigma_\va)~dm(\omega) \quad
\textrm{and} \quad 
|\f(\vb)|=\int_{\Omega} \mu_\omega(\sigma_\vb)~dm(\omega).$$
For each $\omega \in \Omega$, set $\g_\omega=\Xi^{-1} \circ \Psi_{\bm \theta_1}(\mu_\omega)$. Each $\g_\omega$ is well defined and is a $\bm \theta_1$-survivor by 
Corollary \ref{cor:symbolic}. Then $|\g_\omega(\va)|=\mu_\omega(\sigma_\va)$ and
$|\g_\omega(\vb)|=\mu_\omega(\sigma_\vb)$, with the same sign considerations as before. Therefore, we have
$$\f(\vv)=\int_{\Omega} \g_\omega(\vv)~dm(\omega) \quad \textrm{for all $\vv \in \V$}.$$
The map $\Xi^{-1} \circ \Psi_{\bm \theta_1}$ is a linear bijection ${\mathcal M}_{\bm \theta_1} \to \Surv_{\bm \theta_1}$. Therefore, each $\g_\omega$ is extremal in 
$\Surv_{\bm \theta_1}$, because $\mu_\omega$ is ergodic. From Theorem \ref{main_thm:properties}, it follows that there is an $\lambda_\omega>0$ depending on $\omega$ for which $\A^2(\g_\omega)=\lambda_\omega^2 \g_\omega$.
Moreover the map $\omega \mapsto \lambda_\omega$ is measurable, since it can be computed as a ratio of sums of measures of saddle connections. Then for any integer $k$ we have
\begin{equation}
\label{eq:adjacency_powers}
\lambda_2^{2k} \f(\v)=\A^{2k}(\f)(\vv)=\int_{\Omega} \lambda_\omega^{2k} 
\g_\omega(\vv)~dm(\omega) \quad \textrm{for all $\vv \in \V$},
\end{equation}
where we make sense of negative of $\A$ powers by noting that $\A^2$ restricts to a bijection 
$\Surv_{\bm \theta_1} \to \Surv_{\bm \theta_1}$. See Theorem \ref{thm:bijection}.
Now suppose that 
$$m\big(\{\omega~:~\lambda_\omega>\lambda_2\}\big)>0.$$
Then we observe from the right side of equation \ref{eq:adjacency_powers} that
$$\lim_{k \to +\infty} \big(\A^{2k}(\f)(\vv)\big)^\frac{1}{2k}>\lambda_2,$$
but this contradicts the left side of the equation which says that this limit must be $\lambda_2$. Similarly, if
$$m\big(\{\omega~:~\lambda_\omega<\lambda_2\}\big)<0 \quad \textrm{then} 
\lim_{k \to -\infty} \big(\A^{2k}(\f)(\vv)\big)^\frac{1}{2k}<\lambda_2,$$
which again is a contradiction. Thus $m$-almost everywhere we have $\lambda_\omega=\lambda_2$. Then as in Corollary \ref{cor:bijection2}, we may set 
$$\w_\omega(\va)=\frac{1}{x_2} \f(\va) \quad \textrm{and} \quad \w_\omega(\vb)=\frac{1}{y_2} \f(\vb) \quad \textrm{for all $\va \in \Alpha$ and $\vb \in \Beta$.}$$
Then by this corollary, $\w_\omega$ is an extremal positive eigenfunction with eigenvalue $\lambda_2$ $m$-a.e.. Since $\f=P_{\w_2}(\bm \theta_2)$, we know that
$$\w_2(\vv)=\int_{\Omega} \w_\omega(\vv)~dm(\omega) \quad \textrm{for all $\vv \in \V$}.$$
But since $\w_2$ is an extremal positive eigenfunction and almost every $\w_\omega$
is an eigenfunction with the same eigenvalue, we have that $\w_\omega$ is a scalar multiple of $\w_2$ $m$-almost everywhere. 
Therefore $m$-a.e., $\g_\omega$ is a scalar multiple of $\f$. Since $\mu_\omega=\Psi_{\bm \theta_1}^{-1} \circ \Xi(\g_\omega)$ and $\mu_2=\Psi_{\bm \theta_1}^{-1} \circ \Xi(\f)$, we know $m$-a.e. that $\mu_\omega$ is a scalar multiple of $\mu_2$.
In particular, each $\mu_\omega$ was known to be ergodic, so the measure $\mu_2$ must be ergodic as well. 
\end{proof}

By combining the corollary and lemma above, we have the following.

\begin{theorem}
\label{thm:general_ergodic_measure_classification}
Assume $S(\G,\w_1)$ has the critical decay property and $\G$ has the adjacency sign property.
Let $\bm \theta_1$ be a $\lambda_1$-renormalizable direction with
$\lambda_1$ shrinking sequence $\langle g_n \rangle$. Then, the collection of locally finite ergodic invariant measures is given by 
$$\big\{\Psi_{\bm \theta_1}^{-1} \circ \Xi\big(P_{\w_2}(\bm \theta_2)\big)\big\},$$
where $\w_2$ varies over the extremal positive eigenfunctions of $\G$,
$\bm \theta_2=\bm \theta(\langle g_n \rangle, \lambda_2)$ and $\lambda_2$
is the eigenvalue of $\w_2$. 
\end{theorem}
\begin{proof}
All locally finite ergodic invariant measures arise in this way by Corollary \ref{cor:bijection2} and Corollary \ref{cor:symbolic}. Conversely, every such measure is ergodic by Lemma \ref{lem:ergodicity_implication}.
\end{proof}

Finally, we can prove the ergodic measure classification theorem.

\begin{proof}[Proof of the Ergodic Measures Theorem II (Theorem \ref{thm:ergodic2})]
Since $\G$ has no vertices of valance one, by Theorem \ref{thm:properties}
we know that $S(\G,\w_1)$ has the critical decay property and $\G$ has
the adjacency sign property. Therefore by Theorem \ref{thm:general_ergodic_measure_classification} each ergodic measure is of the form 
$\mu_2=\Psi_{\bm \theta_1}^{-1} \circ \Xi\big(P_{\w_2}(\bm \theta_2)\big)$, where $\w_2$ is an extremal positive eigenfunction. Then it follows from Corollary \ref{cor:pullback_measure} that $\mu_2$ arises from a pullback construction as described in the statement of the theorem.
\end{proof}

\section{Renormalizable directions}
\name{sect:symmetry}

In this section, we revisit the idea of renormalizable directions which was introduced in section \ref{ss:renormalizable directions}. We work out a number of their basic properties. 

\subsection{Shrinking sequences and directions}
\label{sect:unique}
We begin by recalling some of the ideas from section \ref{ss:renormalizable directions}. We gave $G$ a metric structure coming from viewing $G$ as a subset of its Cayley graph constructed from the symmetric generating set $\{h,v,h^{-1},v^{-1}\}$. 
In particular, an infinite (resp. finite) sequence $\langle g_0, g_1, \ldots \rangle$ of elements of $G$ is a {\em geodesic ray (resp. segment) in $G$} if and only if it satisfies the following
statements.
\begin{enumerate}
\item[1.] $g_{n+1} g_{n}^{-1} \in \{h,v,h^{-1}, v^{-1}\}$ for all $n \geq 0$.
\item[2.] $g_{n+2} g_{n}^{-1} \neq e$ for all $n \geq 0$.
\end{enumerate}
Here, we have restated Proposition \ref{prop:ray}.

Now consider the representation $\rho_\lambda:G \to \SL(2,\R)$ for $\lambda \geq 2$, as defined in Definition \ref{def:rho}.
Because $\lambda \geq 2$, the representations $\rho_\lambda$ are always discrete and faithful. 

\begin{definition}[Shrinking]
\label{def:shrinking sequences}
Let $\btheta \in \Circ$. A geodesic ray $\langle g_i\rangle_{i \geq 0}$ with $g_0=e$ is a {\em $\lambda$-shrinking sequence of $\bm \theta$} if
$\|\rhoa{g_i}(\btheta)\|$ is strictly monotone decreasing.
 We say $\langle g_i \rangle$ is {\em a $\lambda$-shrinking sequence} if it is the $\lambda$-shrinking sequence of some $\bm \theta \in \Circ$,
 and we say the associated $\bm \theta \in \Circ$ is {\em $\lambda$-shrinkable}.
\end{definition}

Our main result of this subsection is the following:

\begin{theorem}[The Correspondence Theorem]
\label{thm:unique shrinking}~
\begin{enumerate}
\item If $\bm \theta \in \Circ$ is $\lambda$-shrinkable, then there is a unique geodesic ray $\langle g_i \rangle$ with $g_0=e$
which is a $\lambda$-shrinking sequence of $\bm \theta$.
\item Conversely,
if $\langle g_i \rangle$ is a geodesic ray with $g_0=e$, then there is at most one pair of antipodal $\lambda$-shinkable directions $\btheta$ for which $\langle g_i \rangle$ is its $\lambda$-shrinking sequence.
\end{enumerate} 
\end{theorem}

In addition to proving this result, in this subsection, we will also state a number of facts which will be useful later about the action of a shrinking sequence.

Since the Cayley graph of $G$ is a tree, 
for any $g \in G$, there is a unique geodesic segment (in the word metric) $\langle g_0, \ldots, g_n \rangle$ in $G$ which satisfies $g_0=e$ and $g_n=g$. We begin by studying which 
unit vectors are shrunk or expanded along such a sequence.

\begin{definition}[Shrinking and expanding sets]
\label{def:shrink_exp}
Let $g \in G$ and let $\langle g_0, \ldots, g_n \rangle$ be the unique geodesic segment with $g_0=e$ and $g_n=g$.
Define $\Shrink_\lambda(g)$ and $\Exp_\lambda(g)$ to be the sets
$$\Shrink_\lambda(g) =
\{ \bm \theta \in \Circ ~:~  \|\rho_\lambda^{g_n}(\bm \theta)\| < \|\rho_\lambda^{g_{n-1}}(\bm \theta)\| < \ldots < \|\rho_\lambda^{g_1}(\bm \theta)\| < \|\bm \theta\|=1 \}. \label{not:shrink}$$
$$\Exp_\lambda(g)=
\{ \bm \theta \in \Circ ~:~  \|\rho_\lambda^{g_n}(\bm \theta)\| > \|\rho_\lambda^{g_{n-1}}(\bm \theta)\| > \ldots > \|\rho_\lambda^{g_1}(\bm \theta)\| > \|\bm \theta\|=1 \}. \label{not:expand}$$
\end{definition}

The next proposition follows from the observation that these two definitions are related by switching the directions of the inequalities. We define the projection map $\pi_\Circ:\R^2 \smallsetminus \{\0\} \to \Circ \label{not:pi circ}$ to be the map
$\pi_\Circ(\v)=\frac{\v}{\|\v\|}$.

\begin{proposition}
\label{prop:shrink_exp_relation}
For all $g \in G$, $\Shrink_\lambda(g)=\pi_\Circ \circ \rhoa{g^{-1}} \big(\Exp_\lambda(g^{-1})\big)$.
\end{proposition}
\begin{proof}
The geodesic segment joining $e$ to $g^{-1}$ is the sequence $\langle g'_i=g_{n-i} g^{-1} \rangle$. Thus,
$$
\begin{array}{rcl}
\Exp_\lambda(g^{-1}) & = & \{ \bm \theta \in \Circ ~:~  \|\rho_\lambda^{g^{-1}}(\bm \theta)\| > \|\rho_\lambda^{g_{1} g^{-1}}(\bm \theta)\| > \ldots > \|\rho_\lambda^{g_{n-1} g^{-1}}(\bm \theta)\| > \|\bm \theta\| \} \\
& = & \pi_\Circ \circ \rhoa{g} \big(\{ \bm \theta' \in \Circ ~:~  \|\bm \theta'\| > \|\rho_\lambda^{g_{1}}(\bm \theta')\| > \ldots > \|\rho_\lambda^{g_{n-1}}(\bm \theta')\| > \|\rho_\lambda^{g_{n}}(\bm \theta')\| \}\big) \\
& = & \pi_\Circ \circ \rhoa{g} \big(\Shrink_\lambda(g)\big)
\end{array}$$
This is equivalent to the version stated in the proposition.
\end{proof}

\begin{proposition}
\label{prop:shrink_exp}
We have the following descriptions of the shrinking sets of generators.
$$
\begin{array}{lcl}
\Shrink_\lambda(h)=\{(x,y)\in \Circ~:~\frac{-2}{\lambda}<\frac{y}{x}<0\}.
& \quad &
\Shrink_\lambda(h^{-1})=\{(x,y)\in \Circ~:~0<\frac{y}{x}<\frac{2}{\lambda}\}. \\
\Shrink_\lambda(v)=\{(x,y)\in \Circ~:~-\infty<\frac{y}{x}<\frac{-\lambda}{2}\}.
& \quad &
\Shrink_\lambda(v^{-1})=\{(x,y)\in \Circ~:~\frac{\lambda}{2}<\frac{y}{x}<\infty\}.
\end{array}$$
When $g \in \{h,v,h^{-1}, v^{-1}\}$ we have $\Exp_\lambda(g)=\Circ \smallsetminus \cl\big(\Shrink_\lambda(g)\big)$, where $\cl$ denotes the closure.
\end{proposition}

We will not prove this proposition as it is a simple computation. Disjointness of these sets gives the following consequence:

\begin{corollary}[Unique shrinking generator]
\label{cor:shrink}
Given any $\lambda \geq 2$ and $\bm \theta \in \R^2 \smallsetminus \{{\mathbf 0}\}$, there is at most one element 
$g \in \{h,h^{-1},v,v^{-1}\}$ so that 
$\|\rho_\lambda^g(\bm \theta)\|<\|\bm \theta\|$. Moreover, if $g_1 \in \{h,h^{-1},v,v^{-1}\}$ satisfies
$|\rhoa{g_1}(\bm \theta)|<|\bm \theta|$ (resp. $|\rhoa{g_1}(\bm \theta)| \leq |\bm \theta|$), then every $g_2 \in \{h,h^{-1},v,v^{-1}\}$ with $g_1 \neq g_2$ satisfies
$|\rhoa{g_2}(\bm \theta)|>|\bm \theta|$ (resp. $|\rhoa{g_2}(\bm \theta)| \geq |\bm \theta|$).
\end{corollary}
\begin{proof}
The proof follows from the fact that if $g_1$ and $g_2$ are distinct elements of $\{h,h^{-1},v,v^{-1}\}$ then 
$\Shrink_\lambda(g_1) \subset \Exp_\lambda(g_2)$ (resp. $\cl\big(\Shrink_\lambda(g_1)\big) \subset \cl\big(\Exp_\lambda(g_2)\big)$). This can be derived directly from Proposition \ref{prop:shrink_exp}.
\end{proof}

\begin{proof}[Proof of Theorem \ref{thm:unique shrinking}]
First consider statement (1). Suppose $\bm \theta \in \Circ$ is $\lambda$-shrinkable, i.e., there is a geodesic ray
$\langle g_i \rangle$ with $g_0=e$ so that $\|\rho_\lambda^{g_i}(\btheta)\|$ decreases monotonically.
Suppose $\langle g_i' \rangle$ is another such geodesic ray. We claim they are equal. Otherwise,
there is a smallest $i$ so that $g_i \neq g_i'$. Since $g_0=g_0'=e$, we know $i \geq 1$ and $g_{i-1}=g_{i-1}'$. Thus,
$$\|\rhoa{g_{i} g_{i-1}^{-1}}\big(\rhoa{g_{i-1}}(\btheta)\big)\|<\|\rhoa{g_{i-1}}(\btheta)\| \quad \textrm{and} \quad \|\rhoa{g_{i}' g_{i-1}^{-1}}\big(\rhoa{g_{i-1}}(\btheta)\big)\|<\|\rhoa{g_{i-1}}(\btheta)\|.$$
But, $g_{i} g_{i-1}^{-1}, g_{i}' g_{i-1}^{-1} \in \{h,v,h^{-1},v^{-1}\}$ are distinct,
which contradicts Corollary \ref{cor:shrink}.

Now consider statement (2). Let $\langle g_i \rangle$ be a geodesic ray, and suppose that it is the $\lambda$-shrinking sequence for two non-parallel directions $\btheta$ and $\btheta'$. Observe that by definition we have 
$$\|\rhoa{g_i}(\btheta)\| \leq 1 \quad \text{and} \quad \|\rhoa{g_i}(\btheta')\| \leq 1.$$
But the set of all $M \in \SL(2, \R)$ for which $\|M \btheta\| \leq 1$ and
$\|M \btheta'\| \leq 1$ is compact. So by discreteness of $\rhoa{G}$, the sequence $\langle \rhoa{g_i} \rangle$ can only take finitely may values. But this contradicts the definition of a $\lambda$-shrinking sequence. \end{proof}

\begin{proposition}
\label{prop:increasing}
Let $\bm \theta \in \Circ$.
Let $\langle g_i \rangle$ be a geodesic segment or ray in $G$ for which $\|\rhoa{g_{n+1}}(\bm \theta)\| > \|\rhoa{g_{n}}(\bm \theta)\|$
(resp. $\|\rhoa{g_{n+1}}(\bm \theta)\| \geq \|\rhoa{g_{n}}(\bm \theta)\|$)
for some $n \geq 0$. Then, the subsequence $\|\rhoa{g_{n+j}}(\bm \theta)\|$ for $j \geq 0$
is a strictly increasing (resp. non-strictly increasing) sequence.
\end{proposition}

\begin{proof}
We will prove the strictly increasing case. The non-strict case follows similarly.
The proof is by induction. Suppose $\|\rhoa{g_{i+1}}(\bm \theta)\| > \|\rhoa{g_{i}}(\bm \theta)\|$. Then
$\pi_\Circ\big(\rhoa{g_{i+1}}(\bm \theta)\big) \in \Shrink_\lambda(g_{i}g_{i+1}^{-1})$. 
Since $\langle g_i \rangle$ is a geodesic,
 $g_{i+2}g_{i+1}^{-1} \neq g_{i}g_{i+1}^{-1}$. So by Corollary \ref{cor:shrink},  
$$\pi_\Circ\big(\rhoa{g_{i+1}}(\bm \theta)\big) \in \Exp_\lambda(g_{i+2}g_{i+1}^{-1})\big).$$ 
In other words, $\|\rhoa{g_{i+2}}(\bm \theta)\| > \|\rhoa{g_{i+1}}(\bm \theta)\|$.
\end{proof}

We also have the following consequence:

\begin{corollary}
\label{cor:length_minimizer}
If $\bm \theta \in \Circ \smallsetminus \big(\Shrink_\lambda(h) \cup \Shrink_\lambda(h^{-1}) \cup \Shrink_\lambda(v) \cup \Shrink_\lambda(v^{-1}) \big)$, 
then for all $g \in G$ we have $\| \bm \theta \| \leq \|\rhoa{g}(\bm \theta)\|$.
\end{corollary}
\begin{proof}
Apply Proposition \ref{prop:increasing} to the geodesic segment joining $e$ to $g$.
\end{proof}

Using Proposition \ref{prop:shrink_exp_relation}, we can obtain a concrete description of the shrinking and expanding sets. (To make it extremely concrete, you can combine it with Proposition \ref{prop:shrink_exp}.)

\begin{corollary}
\label{cor:shrunk}
Let $g \in G$ and let $\langle g_0=e, \ldots, g_n=g\rangle$ be the corresponding geodesic segment.
For $\lambda \geq 2$, we have $\Exp_\lambda(g)= \Exp_\lambda(g_1)$ and
$$\Shrink_\lambda(g)= \pi_\Circ \circ \rhoa{g^{-1}} \big(\Exp_\lambda(g_{n-1} g^{-1})\big)=\pi_\Circ \circ \rhoa{g_{n-1}^{-1}}\big(\Shrink_\lambda(g g_{n-1}^{-1})\big).$$
\end{corollary}
\begin{proof}
The statement $\Exp_\lambda(g_1)=\Exp_\lambda(g)$ follows directly from Proposition \ref{prop:increasing}. To prove the identity for
$\Shrink_\lambda(g)$, we apply the equation for $\Exp_\lambda(g)$ and Proposition \ref{prop:shrink_exp_relation}:
$$\Shrink_\lambda(g)=\pi_\Circ \circ \rhoa{g^{-1}} \big(\Exp_\lambda(g^{-1})\big)=\pi_\Circ \circ \rhoa{g^{-1}} \big(\Exp_\lambda(g_{n-1} g^{-1})\big).
$$
Finally by Proposition \ref{prop:shrink_exp_relation}, we see
$\Exp_\lambda(g_{n-1} g^{-1})=\pi_\Circ \circ \rhoa{g g_{n-1}^{-1}}\big(\Shrink_\lambda(g g_{n-1}^{-1})\big)$. 
\end{proof}

We can now prove our main result for this subsection.

\subsection{The limit set}
\label{sect:the limit set}
Our original definition of $\lambda$-renormalizable directions involved the limit set. See \S \ref{ss:renormalizable directions}. Here, we will introduce 
 hyperbolic geometry and the limit set. For the necessary background, we refer the reader to the book of Matsuzaki and Taniguchi \cite{MT98}.

The {\em hyperbolic plane} is $\H^2=\SO(2) \setminus \SL(2, \R) \label{not:hyperolic plane}$.
The boundary of the hyperbolic plane is naturally defined to be $\partial \H^2=\R\P^1=(\R^2 \smallsetminus\{0\})/\R$. For $M \in \SL(2, \R)$ and
$\v \in \R^2\smallsetminus\{0\}$, we will use $[M] \in \H^2$ and $[\v] \in \R\P^1$ to denote the corresponding equivalence classes.
To make  $\R\P^1$ the boundary of the hyperbolic plane, we say a sequence $\langle [M_n] \in \H^2\rangle$ converges to $[\v] \in \R\P^1$
if for any $[\w] \neq [\v] \in \R\P^1$, 
\begin{equation}
\label{eq:limit set def}
\frac{\|M_n \v\|}{\|M_n \w\|} \to 0 \textrm{ as } n \to \infty.
\end{equation}

The {\em limit set}, $\Lambda(\Gamma) \subset \R\P^1$ of a discrete group $\Gamma \subset \SL(2, \R)$ is the set of all limit points of sequences 
in $\SO(2) \setminus \Gamma \subset \H^2$. Equivalently, the limit set $\Lambda(\Gamma) \subset \R\P^1$ is the smallest non-empty closed $\Gamma$-invariant subset of 
$\R\P^1$.
An {\em open horodisk} in $\H^2$ at $[\v ] \in \R\P^1$ is a set of the form 
$$\{[M] \in \H^2 ~:~ \|M \v\| < \epsilon \}$$
for some $\epsilon>0$. The only accumulation point in $\R \P^1$ of the horodisk defined above is $[\v]$. 
The {\em horospherical limit set} $\Lambda_h(\Gamma) \subset \R\P^1$ is the set of all $[\v] \in \Lambda(\Gamma)$
for which any horodisk at $[\v]$ contains points in the orbit $[\Gamma]=\{[M]~:~M \in \Gamma\}$. 

From work in the last section, we can conclude the following.

\begin{lemma}
\label{lem:horocyclic implies shrinkable}
Suppose $[\btheta] \in \Lambda_h(\Gamma)$. Then, $\btheta$ is $\lambda$-shrinkable. Moreover, if $\langle g_i \rangle$ is the $\lambda$-shrinking sequence of $\btheta$, then $\lim_{i \to \infty} \|\rhoa{g_i}(\btheta)\|=0$.
\end{lemma}
\begin{proof}
Suppose that $[\bm \theta] \in \Lambda_h(\rhoa{G})$. From the definitions above, we know that
for every $\epsilon>0$, there is a $g \in G$ so that $\|\rhoa{g}(\btheta)\|<\epsilon$. Now we attempt to build a $\lambda$-shrinking sequence for $\btheta$. We define the geodesic ray $\langle g_i \rangle$ inductively so that 
$g_0=e$ and $g_{i} g_{i-1}^{-1} \in \{h,v,h^{-1},v^{-1}\}$ is the unique choice for which 
$$\|\rhoa{g_{i} g_{i-1}^{-1}} \big( \rhoa{g_{i-1}}(\btheta)\big)\| < \|\rhoa{g_{i-1}}(\btheta)\|.$$
Here, uniqueness is provided by Corollary \ref{cor:shrink}. Existence of this choice is a consequence
of Corollary \ref{cor:length_minimizer}: if there is no generator which shrinks $\rhoa{g_{i-1}}(\btheta)$, then
$$\|\rhoa{g_{i-1}}(\btheta)\|= \min_{g \in G} \|\rhoa{g}(\btheta)\|.$$
But, this contradicts the hypothesis that $[\bm \theta] \in \Lambda_h(\rhoa{G})$.

It remains to show that
$\lim_{i \to \infty} \|\rhoa{g_i}(\btheta)\|=0$. Suppose not, then there is an $\epsilon>0$ so that
$\|\rhoa{g_i}(\btheta)\|> \epsilon$ for all $i \in \N$. Since this is true for $i=0$, we know $\epsilon<1$.
On the other hand, since $[\bm \theta] \in \Lambda_h(\rhoa{G})$,
we know that there is a $\gamma \in G$ so that $\|\rhoa{\gamma}(\btheta) \|< \epsilon$.
Since $\epsilon<1$, we know $\gamma \neq e$. Consider the geodesic segment
$\langle e=\gamma_0, \gamma_1, \ldots, \gamma_n=\gamma \rangle$. By hypothesis, we know the segment
can not coincide with the initial segment of the ray $\langle g_i \rangle$. Thus, there is a minimal $i>0$ so that
$\gamma_i \neq g_i$. Since $\langle g_i \rangle$ is a shrinking sequence, we know
$\|\rhoa{g_i}(\btheta)\|<\|\rhoa{g_{i-1}}(\btheta)\|$. Then by uniqueness of the shrinking generator (i.e., Corollary \ref{cor:shrink} applied to
$\rhoa{g_{i-1}}(\btheta)$), we know that $\|\rhoa{\gamma_i}(\btheta)\| > \|\rhoa{g_{i-1}}(\btheta)\|$. 
But then Proposition \ref{prop:increasing} implies that 
$$\|\rhoa{\gamma_n}(\btheta)\| > \|\rhoa{\gamma_{n-1}}(\btheta)\| > \ldots >\|\rhoa{\gamma_i}(\btheta)\| > \|\rhoa{g_{i-1}}(\btheta)\|.$$
Since $\|\rhoa{\gamma_n}(\btheta)\|< \epsilon$, we conclude that $\|\rhoa{g_{i-1}}(\btheta)\|<\epsilon$,
which is a contradiction.
\end{proof}

We now recall a theorem of Beardon and Maskit \cite[Theorem 2]{BM74}. If $\Gamma \subset \SL(2, \R)$ is geometrically finite, then 
$\Lambda(\Gamma) \smallsetminus \Lambda_h(\Gamma)$ is the collection of fixed points of parabolics. Because $\rhoa{G}$ is geometrically finite, we have:

\begin{corollary}
\label{cor:dichotomy}
If $[\bm \theta] \in \Lambda(\rhoa{G})$ then either $\bm \theta$ is fixed by a parabolic in $\rhoa{G}$ or $[\bm \theta] \in \Lambda_h(\rhoa{G})$.
\end{corollary}

Recall that the collection $\Rn_\lambda \subset \Circ$ of all $\lambda$-renormalizable directions was defined in \S \ref{ss:renormalizable directions} 
to be the limit set with orbits of some eigenvectors removed. In particular,
because we removed the conjugacy classes of the parabolics, $\Rn_\lambda \subset \Lambda_h(\rhoa{G})$. As a consequence, we can apply Lemma \ref{lem:horocyclic implies shrinkable} to obtain:

\begin{theorem}[Shrinking renormalizable directions]
\label{thm:shrinking}
If $\btheta \in \Rn_\lambda$, then it is $\lambda$-shrinkable. Furthermore,
its shrinking sequence $\langle g_i \rangle$ satisfies $\lim_{i \to \infty} \|\rho_\lambda^{g_i}(\btheta)\|=0$.
\end{theorem}

We will now use the limit set to improve Proposition \ref{prop:shrink_exp},
which described the vectors shrunk by the generators of $G$. 
In this paper, we only care about the behavior of $\lambda$-renormalizable directions, which we now understand to lie in $\Lambda_h(\rhoa{G})$. But,
note that we also removed the eigendirections of conjugates of $\rho_\lambda^{vh^{-1}}$ from the limit set to obtain $\Rn_\lambda$. This viewpoint gives the following:

\begin{proposition}
\label{prop:limit_set2}
For all $\lambda \geq 2$, we have the following statements.
$$\Shrink_\lambda({h^{-1}}) \cap \Rn_\lambda \subset \big\{\bm \theta=(x,y) \in \Circ~:~0 < \frac{y}{x} < \frac{\lambda-\sqrt{\lambda^2-4}}{2}\big\}.$$
$$\Shrink_\lambda({v^{-1}}) \cap \Rn_\lambda \subset \big\{\bm \theta=(x,y) \in \Circ~:~ \frac{\lambda+\sqrt{\lambda^2-4}}{2} < \frac{y}{x} < \infty \big\}.$$
$$\Shrink_\lambda({h}) \cap \Rn_\lambda \subset \big\{\bm \theta=(x,y) \in \Circ~:~\frac{-\lambda+\sqrt{\lambda^2-4}}{2} < \frac{y}{x} < 0 \big\}.$$
$$\Shrink_\lambda({v}) \cap \Rn_\lambda \subset \big\{\bm \theta=(x,y) \in \Circ~:~ -\infty < \frac{y}{x} < \frac{-\lambda-\sqrt{\lambda^2-4}}{2} \big\}.$$
\end{proposition}
\begin{proof}
Note that when $\lambda=2$, this statement is directly implied by Proposition \ref{prop:shrink_exp}. So, we will assume $\lambda>2$. 
The quotient $\H^2/\rhoa{G}$ is homeomorphic to a thrice punctured sphere. Recall that the homotopy classes of loops on $\H^2/\rhoa{G}$ are in bijective correspondence
with conjugacy classes in $G$. The conjugacy classes of $h$, $v$, and $v h^{-1}$ correspond to simple loops traveling around each of the three punctures. 
As $\lambda>2$, $\rhoa{v h^{-1}}$ is a hyperbolic isometry of $\H^2$ whose axis is the hyperbolic geodesic with whose endpoints are the projectivizations of the eigendirections of $\rho^{v h^{-1}}$. 
The endpoints of this geodesic are the points of $\R\P^1$,
$$e_1=[(2,\lambda-\sqrt{\lambda^2-4}] 
\quad \textrm{and} \quad 
e_2=[(2,\lambda+\sqrt{\lambda^2-4})].$$
Note that these points were explicitly removed from the set of $\lambda$-renomalizable directions. The geodesic joining these points in $\H^2$ descends
to a closed loop $\H^2/\rhoa{G}$ around the flaring end. Thus, the interval 
$$I_1=\big\{[(2,y)]~:~\lambda-\sqrt{\lambda^2-4}<y<\lambda+\sqrt{\lambda^2-4}\big\} \subset \R \P^1$$
is a maximal interval in the compliment of $\Lambda(\rhoa{G})$. The set of such complimentary intervals is $\rhoa{G}$ invariant, thus
$$I_2=\rhoa{h^{-1}}(I_2)=\big\{[(2,y)]~:~\lambda-\sqrt{\lambda^2-4}<-y<\lambda+\sqrt{\lambda^2-4}\big\}$$
is also a complimentary interval. Our formulas follow from the fact that
$$\Rn_\lambda \subset \Circ \smallsetminus \big(\bar I_1 \cup \bar I_2 \cup \{(1,0),(-1,0), (0,1), (0,-1)\}\big),$$
where we are using $\bar I_1$ and $\bar I_2$ to denote the closure of the lifts of $I_1$ and $I_2$ from $\RP^1$ to $\Circ$. The conclusion follows from intersecting the set on the right with our formulas for $\Shrink_\lambda(g)$ in Proposition \ref{prop:shrink_exp} for each $g \in \{h,v,h^{-1},v^{-1}\}$.
\end{proof}

\subsection{Combinatorics of renormalizable sequences}
\name{ss:shrinking}
We would like to understand which geodesic rays are $\lambda$-shrinking sequences for $\lambda$-renormalizable directions. In order to describe the answer, we define the following shift action on geodesic rays 
$\langle g_n \rangle_{n \geq 0}$ with $g_0=e$:
$${\mathcal S}(\langle g_0, g_1, g_2, \ldots \rangle)=\langle g_1 g_1^{-1}, g_2 g_1^{-1}, g_3 g_1^{-1}, \ldots \rangle.$$
We call two sequences of elements in $G$, $\langle g_i \rangle$ and $\langle g_i' \rangle$, {\em tail equivalent} if there are non-negative integers $m$ and $n$ for which
${\mathcal S}^m(\langle g_i\rangle)={\mathcal S}^n(\langle g_i' \rangle)$.

\begin{theorem}[Shrinking sequences of renormalizable directions]
\label{thm:characterization}
Let $\lambda \geq 2$. Then, the geodesic ray $\langle g_i \rangle$ with $g_0=e$ is a $\lambda$-shrinking sequence
for a $\lambda$-renormalizable direction if and only if $\langle g_i \rangle$ is not tail equivalent to any of the four geodesic rays fixed by ${\mathcal S}$,
$$
\langle e, h, h^2, \ldots \rangle, \quad \langle e, h^{-1}, h^{-2}, \ldots \rangle, \quad
\langle e, v, v^2, \ldots \rangle \quad \textrm{and} \quad \langle e, v^{-1}, v^{-2}, \ldots \rangle,
$$
and is not tail equivalent to either of the following two geodesic rays of period two,
$$
\langle e, h, v^{-1}h, h v^{-1}h, v^{-1} h v^{-1}h, \ldots \rangle \quad \textrm{and} \quad
\langle e, h^{-1}, v h^{-1}, h^{-1} v h^{-1}, v h^{-1} v h^{-1}, \ldots \rangle.
$$ 
\end{theorem}
\begin{proof}
First suppose that $\langle g_i \rangle$ is tail equivalent to one of the first four listed sequences.
Then there is a parabolic $P \in \rhoa{G}$ (namely, $\rhoa{g}$ for some $g \in \{h,v,h^{-1},v^{-1}\}$), and an $n$ so that
$\rhoa{g_{n+k}}=P^k \rhoa{g_{n}}$. But, you can not make any vector shrink infinitely often by successively applying the same parabolic. So,
$\langle g_i \rangle$ is not a $\lambda$-shrinking sequence. 
Now suppose it is tail equivalent to either of the last two listed sequences.
We break into cases. If $\lambda=2$, then there is a parabolic $P \in \rhoa{G}$ (namely, $\rhoa{g}$ for some $g \in \{v^{-1}h,h^{-1}v\}$),
and an $n$ so that $\rhoa{g_{n+2k}}=P^k \rhoa{g_{n}}$ for all $k \geq 0$. The same reasoning works as before. Now suppose that $\lambda>2$. This time there is a hyperbolic matrix $H \in \rhoa{G}$ (namely $H=\rhoa{g}$ for some $g \in \{v^{-1}h,h^{-1}v\}$) and an $n$ so that $\rhoa{g_{n+2k}}=H^k \rhoa{g_{n}}$ for all $k \geq 0$. The only vectors which can be repeatedly shrunk by a hyperbolic matrix are its contracting eigenvectors. So, $\rhoa{g_n}(\btheta)$ is a contracting eigenvector of $H$. Equivalently, $\btheta$ is
a contracting eigenvector of $\rhoa{g_n} H \rhoa{g_n^{-1}}$. But these
eigenvectors were explicitly thrown out by the definition of $\lambda$-renormalizable directions.

Now let $\langle g_i \rangle$ be any geodesic ray with $g_0=e$. We will
show that this sequences is a $\lambda$-shrinking sequence unless
it is tail equivalent to one of the six sequences listed in the theorem.

Since $\rho_\lambda^G$ is discrete and $\langle g_i \rangle$ is infinite, there must be an accumulation point $[\btheta] \in \RP^1$. Let $\btheta \in \Circ$ be a lift of $[\btheta]$. We claim that $\btheta$ is $\lambda$-shinkable and $\langle g_i \rangle$ is its shrinking sequence. 

First we claim that $\btheta \in {\overline \Shrink}_\lambda(g_i)$ for all $i$,
where ${\overline \Shrink}_\lambda(g_i)$ denotes the closure of $\Shrink_\lambda(g_i)$. By definition, the sets $\Shrink_\lambda(g_i)$ are nested. (See Definition \ref{def:shrink_exp}.) Observe that Corollary \ref{cor:shrunk} implies that each set is non-empty. We conclude that there is a direction ${\bm \eta} \in \bigcap_i {\overline \Shrink}_\lambda(g_i).$  It follows that 
$\|\rho_\lambda^{g_i}({\bm \eta})\|$ is non-strictly monotone decreasing. 
By equation \ref{eq:limit set def}, which defined convergence to the boundary in $\H^2$, we know that infinitely often the inequality $\|\rho_\lambda^{g_i}({\bm \theta})\| \leq \|\rho_\lambda^{g_i}({\bm \eta})\|$ is satisfied. So, it follows
that infinitely often $\rho_\lambda^{g_{i-1}}({\bm \theta})$ must be non-strictly shrunk
by $\rho_\lambda^{g_{i} g_{i-1}^{-1}}$. So, infinitely often
$$\pi_\Circ \circ \rhoa{g_i}(\btheta) \in \overline{\Exp}_\lambda(g_{i-1} g_i^{-1}).$$
So by Corollary \ref{cor:shrunk}, infinitely often $\btheta \in {\overline \Shrink}_\lambda(g_i)$.
Since these sets are nested, we conclude that $\btheta \in {\overline \Shrink}_\lambda(g_i)$ for all $i$, as desired.

Now we claim that $\btheta \in \Shrink_\lambda(g_i)$.
Observe that the projectivized sets $[{\overline \Shrink}_\lambda(g_i)]$ form
a nested intersection of closed intervals. Thus, the conclusion follows
unless there is an $n$ so that $\btheta \in \partial {\overline \Shrink}_\lambda(g_i)$ for all $i \geq n$. 
Suppose this is the case.
It follows then from Corollary \ref{cor:shrunk} that
$$\pi_\Circ \circ \rhoa{g_{i-1}} (\btheta) \in \partial \Shrink_\lambda(g_i g_{i-1}^{-1})$$
for all $i \geq n$. Observe that $g_i g_{i-1}^{-1} \in \{h,v,h^{-1},v^{-1}\}$. 
Then using the explicit description for shrinking sets provided by Proposition \ref{prop:shrink_exp}, we see that one of the following must hold for each ``time'' $i \geq n$:
\begin{enumerate}
\item[(a)] $[\pi_\Circ \circ \rhoa{g_{i-1}}(\btheta)]=[(1,0)]$ and $g_i g_{i-1}^{-1} \in \{h,h^{-1}\}$. 
\item[(b)] $[\pi_\Circ \circ \rhoa{g_{i-1}}(\btheta)]=[(0,1)]$ and $g_i g_{i-1}^{-1} \in \{v,v^{-1}\}$. 
\item[(c)] $[\pi_\Circ \circ \rhoa{g_{i-1}}(\btheta)]=[(-\lambda,2)]$ and $g_i g_{i-1}^{-1} = h$.
\item[(d)] $[\pi_\Circ \circ \rhoa{g_{i-1}}(\btheta)]=[(\lambda,2)]$ and $g_i g_{i-1}^{-1} = h^{-1}$.
\item[(e)] $[\pi_\Circ \circ \rhoa{g_{i-1}}(\btheta)]=[(2,-\lambda)]$ and $g_i g_{i-1}^{-1} = v$.
\item[(f)] $[\pi_\Circ \circ \rhoa{g_{i-1}}(\btheta)]=[(2,\lambda)]$ and $g_i g_{i-1}^{-1} = v^{-1}$.
\end{enumerate}
We will use these statements to show that the sequence $\langle g_i \rangle$ 
must in fact be tail equivalent to one of the sequences from the theorem.
For example, suppose we are in case (a) for some time $i$.
Then, we have $\rhoa{g_{i-1}^{-1}}(\btheta)=(x,0)$ for some $x \neq 0$, and $g_i g_{i-1}^{-1} \in \{h,h^{-1}\}$. Then,
$$\rhoa{g_i}(\btheta)=\rhoa{g_i g_{i-1}^{-1}}(x,0)=(x,0).$$
So, we are again in case (a) but at time $i+1$. We conclude that if we are in case (a) at time $n$, then $g_i g_{i-1}^{-1} \in \{h,h^{-1}\}$ for all $i \geq n$. 
But this means that $\langle g_i \rangle$ is tail equivalent to either
$\langle e, h, h^2,\ldots \rangle$ or $\langle e, h^{-1}, h^{-2},\ldots \rangle$. Case (b) works similar. Now suppose at time $n$ we are in case (c). Then, $\btheta=c(-\lambda,2)$ for some $c \neq 0$ and $g_n g_{n-1}^{-1}=h$. We see that
$$\rhoa{g_n}(\btheta)=c \rhoa{g_n g_{n-1}^{-1}}(-\lambda,2)=(\lambda,2).$$
Now we must be in one of the six cases at time $n+1$. If $\lambda \neq 2$, then
we must be in case (d), but this is a contradiction because
then $g_n g_{n-1}^{-1}=h$ and $g_{n+1} g_{n}^{-1}=h^{-1}$ which contradicts the definition of geodesic ray. If $\lambda=2$, for the same reason, we must be in case (f) at time $n+1$. Thus, $g_{n+1} g_{n}^{-1}=v^{-1}$. Continuing inductively, we see that 
$\langle g_i\rangle$ is tail equivalent to $\langle e, h, v^{-1} h, h v^{-1} h, \ldots\rangle$. The remaining cases work in the same way.
\end{proof}

Lemma \ref{lem:compatible directions} really follows as a corollary to the above result. 

\begin{proof}[Proof of Lemma \ref{lem:compatible directions}]
Let $\lambda_1 \geq 2$ and let $\btheta_1 \in \Rn_{\lambda_1}$. 
By Theorem \ref{thm:shrinking}, we know that $\btheta_1$ is $\lambda_1$-shrinkable. So by definition it has a $\lambda_1$-shrinking sequence $\langle g_i \rangle$. Moreover, this sequence is unique by 
The Correspondence Theorem (Theorem \ref{thm:unique shrinking}). This proves statement (1) of the Lemma.

Now let $\lambda_2 \geq 2$. Statement (2) says there us a unique pair of antipodal $\lambda_2$-renormalizable vectors $\pm \btheta_2$ so that the same $\langle g_i \rangle$ is the  $\lambda_2$-shrinking sequence for each.  The characterization of shrinking sequences of $\lambda$-renormalizable directions given in Theorem \ref{thm:characterization} is independent of $\lambda$. So, because $\langle g_i\rangle$ is $\lambda_1$-shrinking sequence for a $\lambda_1$-renormalizable direction, it is also the $\lambda_2$-shrinking sequence for a $\lambda_2$-renormalizable direction. We call this new direction $\btheta_2$. The Correspondence Theorem tells us that 
$\btheta_2$ is unique up to the antipodal map.
\end{proof}

We conclude this section by giving a proof of Proposition \ref{prop:quadrants_match}, which stated that the pair of quadrants containing
a $\lambda$-renormalizable direction depends only on its shrinking sequence
and not on $\lambda$. 

\begin{proof}[Proof of Proposition \ref{prop:quadrants_match}]
Suppose $\langle g_n \rangle$ is a renormalizing sequence
and $\lambda, \lambda' \geq 2$. Then we have 
$$\pm \btheta(\langle g_n \rangle,\lambda) \in \Shrink_\lambda(g_1) \quad \text{and} \quad \pm \btheta(\langle g_n \rangle,\lambda') \in \Shrink_{\lambda'}(g_1).$$
Note that because $g_0=e$ and $\langle g_n \rangle$ is a geodesic ray, $g_1 \in \{h,v,h^{-1},v^{-1}\}$. These shrinking sets are contained in the same pair of opposite quadrants by Proposition \ref{prop:shrink_exp}.
\end{proof}

\section{Geometry of graph surfaces}
\label{sect:geometry surfaces}

In this section, we discuss the geometry of surfaces of the form $S(\G,\w)$, where $\w$ is a positive eigenfunction of the adjacency operator. In the subsection \ref{sect:facts}, we discuss features
that distinguish eigenfunctions on graphs with a vertex of valance one from
graphs with no vertices of valance one. In subsection \ref{sect:no saddles}, we use these observations to prove that our surfaces have no saddle connections in renormalizable directions.

\subsection{Facts about eigenfunctions}
\label{sect:facts}
In this section, we discuss some facts about eigenfunctions of graphs which distinguish graphs with no vertices of valance one. Note that $\G_\Z$ is the only infinite 
connected graph with no vertices of valance one and no vertices of valance larger than two. We will pay particular attention to the case of vertices
of valance larger than two.

To distinguish graphs with vertices of valance one, we make the following definition.

\begin{definition}
\label{def:spoke}
For an integer $k \geq 2$, a {\em $k$-spoke} is an ordered $k$-tuple
$(\vv_1, \ldots, \vv_k)\in \V^k$
such that
\begin{enumerate}
\item $\val(\vv_1)=1$.
\item For each $i \in \N$ with $2 \leq i \leq k-1$, $\val(\vv_i)=2$, $\vv_i \sim \vv_{i-1}$ and $\vv_i \sim \vv_{i+1}$.
\end{enumerate}
\end{definition}

Note that there is no condition on the valance of the vertex $\vv_k$ of a $k$-spoke. See Figure \ref{fig:spoke} for an example.

\begin{figure}[ht]
\begin{center}
\includegraphics[height=1in]{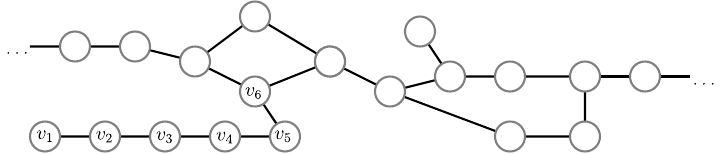}
\caption{$(\vv_1, \ldots, \vv_6)$ is a $6$-spoke.}
\label{fig:spoke}
\end{center}
\end{figure}

\begin{proposition}[Eigenfunctions and spokes]
\label{prop:eig_spokes}
Let $(\vv_1, \ldots, \vv_k)$ be a spoke. Assume $\w \in \R^\V$ is a positive function satisfying $\A \w=\lambda \w$. 
If $\lambda=2$, then for all $j$ with $2 \leq j \leq k$, we have $\w(\vv_j)=j \w(\vv_1)$. If $\lambda>2$ then
$$\w(\vv_j)=\frac{\sinh (jz)}{\sinh(z)} \w(\vv_1) \quad \textrm{where $z=\ln \frac{\lambda+\sqrt{\lambda^2-4}}{2} =\cosh^{-1}(\frac{\lambda}{2})$.}$$
\end{proposition}
\begin{proof}[Sketch of proof]
Note that $\w(\vv_2)=\lambda \w(\vv_1)$ and $\w(\vv_{j})=\lambda \w(\vv_{j-1})-\w(\vv_{j-2})$.
Thus, the values of $\w(\vv_j)$ are determined by the previous values, and thus the value of $\w(\vv_j)$ is uniquely
determined by the value of $\w(\vv_1)$. Finally by inspection and trigonometry, it can be checked that the values stated in the proposition do give rise to a solution to these equations.
\end{proof}

Note that every vertex of valance one belongs to a spoke. The following handles vertices with greater valance.

\begin{proposition}[Detecting spokes]
\label{prop:detecting_spokes}
Let $\w$ be a positive function satisfying $\A \w=\lambda \w$. Let $x$ be a vertex with $\val(x) \geq 2$. 
Then, $\vx$ belongs to a spoke if and only if
there is a $\vy \sim \vx$ such that 
$$\frac{\w(\vy)}{\w(\vx)} < \frac{\lambda-\sqrt{\lambda^2-4}}{2}.$$
Furthermore, if this inequality is satisfied then $x$ and $y$ belong to the same spoke.
\end{proposition}
\begin{proof}
We will assume $\lambda>2$. (There are only three infinite connected graphs with a positive eigenfunction with eigenvalue $2$, and these satisfy the statement.)

First suppose that $\vx=\vv_j$ belongs to a spoke. Let $\vy=\vv_{j-1}$. Then by Proposition \ref{prop:eig_spokes} and angle addition formulas,
we know that 
$$\frac{\w(\vy)}{\w(\vx)}=\frac{\sinh \big((j-1) z\big)}{\sinh(jz)}=\frac{\sinh(jz) \cosh z- \cosh(jz)\sinh(z)}{\sinh(jz)}.$$
We have $\cosh(z)=\lambda/2$ and $\sinh(z)=\sqrt{\lambda^2-4}/2$. Thus,
$$\frac{\w(\vy)}{\w(\vx)}=\frac{\lambda}{2}-\frac{\cosh(jz)\sqrt{\lambda^2-4}}{2 \sinh(jz)}<\frac{\lambda-\sqrt{\lambda^2-4}}{2}.$$

Now we will approach the converse. We claim that if $\va$ and $\vb$ are two adjacent vertices, with 
\begin{equation}
\label{eq:ratio_inequality}
\frac{\w(\vb)}{\w(\va)} < \frac{\lambda-\sqrt{\lambda^2-4}}{2},
\end{equation}
then they are elements of the same spoke. First, we will prove that
this inequality implies that $\val (\vb) \leq 2$. Suppose $\val(\vb) = k \geq 3$. Let $\va$ and $\vc_1, \ldots, \vc_{k-1}$ be the vertices adjacent to $\vb$.
We have that 
$$\lambda=\frac{\w(\va)+\sum_{i=1}^{k-1} \w(\vc_i)}{\w(\vb)}>\frac{\lambda+\sqrt{\lambda^2-4}}{2}+\frac{\sum_{i=1}^{k-1} \w(\vc_i)}{\w(\vb)}.$$
Thus, there is a $j$ such that 
$$\frac{\w(\vc_j)}{\w(\vb)}<\frac{\lambda-\sqrt{\lambda^2-4}}{2(k-1)}=\frac{2}{(k-1)(\lambda+\sqrt{\lambda^{2}-4})}
\leq \frac{1}{\lambda+\sqrt{\lambda^{2}-4}} \leq \frac{1}{\lambda}.$$
In summary $\w(\vc_j) < \w(\vb)/\lambda$. Therefore,
$$\big(\A \w\big)(\vc_j)=\sum_{\vd \sim \vc_j} \w(\vd) \geq \w(\vb)> \lambda \w(\vc_j),$$
which contradicts our assumption that $\A \w=\lambda \w$.

Now suppose $\va \sim \vb$, $\val (\vb)=2$ and equation \ref{eq:ratio_inequality} holds. 
We will show $\va$ belongs to a spoke, completing the proof.
Let $\vc$ denote the other vertex adjacent to $\vb$.
We have
$\lambda \w(\vb)=\w(\va)+\w(\vc)$. Thus,
$$\frac{\w(\vc)}{\w(\vb)}=\lambda-\frac{\w(\va)}{\w(\vb)}>\lambda-\big(\frac{\lambda-\sqrt{\lambda^2-4}}{2}\big)^{-1}=\frac{\lambda-\sqrt{\lambda^2-4}}{2}.$$
Thus, equation \ref{eq:ratio_inequality} is satisfied with $\va$ replaced by $\vb$ and $\vb$ replaced by $\vc$. By the claim above, we know $\val(\vc) \leq 2$. 
By induction, we see that either $\{\va,\vb,\vc\}$ is a subset of a spoke, or there is an infinite sequence of vertices
$\{\vx_0=\va, \vx_1=\vb, \vx_2=\vc, \vx_3, \vx_4, \ldots \}$ with each $\vx_j$ for $j \geq 1$ satisfying $\val(\vx_j)=2$, $\vx_j \sim \vx_{j-1}$, and $\vx_j \sim \vx_{j+1}$. 
We will show that $\{\va,\vb,\vc\}$ must be a subset of a spoke, by proving this other possibility is false. Note that the values of $\w(\vx_j)$ is uniquely determined by $\w(\vx_{j-1})$ and $\w(\vx_{j-2})$ according to the rule $\w(\vx_j)+\w(\vx_{j-2})=\lambda \w(\vx_{j-1})$. In particular, the value of each $\w(\vx_j)$ may be determined inductively from $\w(\vx_0)$ and $\w(\vx_1)$.
Any such solution can be written as 
$$\w(\vx_j)=r \big(\frac{\lambda+\sqrt{\lambda^2-4}}{2}\big)^j+s \big(\frac{\lambda-\sqrt{\lambda^2-4}}{2}\big)^j$$
for all $j$ and some $r,s \in \R$. We will now solve for $r$ and $s$. We have $\w(\vx_0)=r+s$. Then,
$$
\begin{array}{rcl}
\displaystyle \frac{\w(\vb)}{\w(\va)} & \displaystyle = & \displaystyle \frac{\w(\vx_1)}{\w(\vx_0)}=
\frac{r \big(\frac{\lambda+\sqrt{\lambda^2-4}}{2}\big)+s \big(\frac{\lambda-\sqrt{\lambda^2-4}}{2}\big)}{\w(\vx_0)}= \\
& \displaystyle = & 
\displaystyle \frac{r \big(\frac{\lambda+\sqrt{\lambda^2-4}}{2}\big)+\big(\w(\vx_0)-r\big) \big(\frac{\lambda-\sqrt{\lambda^2-4}}{2}\big)}{\w(\vx_0)}= 
\big(\frac{\lambda-\sqrt{\lambda^2-4}}{2}\big) + r\sqrt{\lambda^2-4}.
\end{array}
$$
Thus, by equation \ref{eq:ratio_inequality}, we have $r<0$. It follows that there is a $j \in \N$ such that
$\w(\vx_j)<0$. This contradicts our initial assumption that $\w$ is a positive eigenfunction. Thus, 
$\va$ is an element of a spoke.
\end{proof}

\subsection{Absence of saddle connections}
\label{sect:no saddles}

Let $S=S(\G,\w)$ be a surface built as in \S \ref{sect:graphs}. Here, $\w$ is a positive eigenfunction of the adjacency operator with eigenvalue $\lambda$. Recall that $\V=\Alpha \cup \Beta$ is the vertex set of the graph, and each $\vv \in \V$ represents a cylinder,
$\cyl_\vv$, which is horizontal when $\vv \in \Alpha$ and vertical when $\vv \in \Beta$. 

\begin{definition}
The {\em support} of a saddle connection $\sigma$ in $S$ is
the collection $\supp(\sigma) \subset \V$ defined so that
$\vv \in \supp(\sigma)$ if $\sigma$ intersects a core curve of $\cyl_\vv$. 
\end{definition}

Recall Definition \ref{def:spoke} of a spoke.

\begin{definition}
The {\em extended support} of a saddle connection $\sigma$,
denoted $\suppbar(\sigma)$ is the union of the support $\supp(\sigma)$ and all spokes which intersect the support. 
\end{definition}

\begin{lemma}
\label{lemma:support lemma}
Let $\btheta$ be a $\lambda$-renormalizable direction, and let $\langle g_i \rangle$ be its $\lambda$-shrinking sequence. If $\sigma$ is a saddle connection whose holonomy is parallel to $\btheta$, then
$\suppbar\big(\Phi^{g_1}(\sigma)\big) \subset \suppbar(\sigma).$
\end{lemma}
\begin{proof}
By Remark \ref{rem:dihedral}, it suffices to consider the case when
$g_1=h^{-1}$. Let $\btheta=(x,y)$. Since $\btheta \in \Shrink_\lambda(h^{-1}) \cap \Rn_\lambda$, by Proposition \ref{prop:limit_set2}, we know that 
\begin{equation}
\label{eq:limit set inequality}
0<\frac{y}{x}<\frac{\lambda-\sqrt{\lambda^2-4}}{2}.
\end{equation}
Now suppose that the statement is not true for some $\sigma$.
Let $\sigma'=\Phi^{g_1}(\sigma)$. If our statement is false,
there is a $\vv \in \suppbar(\sigma') \smallsetminus \suppbar(\sigma)$. 

We first claim that we can assume $\vv \in \supp(\sigma') \smallsetminus \suppbar(\sigma)$. Otherwise
$\vv$ would lie in a spoke which intersects $\supp(\sigma')$
but not $\supp(\sigma)$. But, if this is the case, we can replace our choice of $\vv$ with a vertex in this intersection. 

Now we claim that $\vv \in \Beta$. Indeed, the effect of
applying $\Phi^{g_1}=\Phi^{h^{-1}}$ is to simultaneously left Dehn twist all horizontal cylinders. It follows that $\sigma'$ intersects
the same horizontal cylinders that $\sigma$ does. So,
if $\vv \in \Alpha$, then $\vv \in \supp(\sigma')$ implies
$\vv \in \supp(\sigma)$. 

We now know that $\vv \in \supp(\sigma')$ and $\vv \in \Beta$. 
Note that because $\sigma'$ is not horizontal or vertical, 
there must be a vertex $\va \in \supp(\sigma') \cap \Alpha$ so that $\va$ is adjacent to $\vv$. Again, because $\supp(\sigma) \cap \Alpha=\supp(\sigma') \cap \Alpha$, we know that $\va \in \supp(\sigma)$. 
Consider the cylinder $\cyl_\va$, which intersects the vertical cylinder $\cyl_\vv$. Note that $\va$ and $\vv$ can not be elements
of the same spoke (or else $\va \in \supp(\sigma)$ implies $\vv \in \suppbar(\sigma)$). Therefore Proposition \ref{prop:detecting_spokes} implies that
$$\frac{\w(\vv)}{\w(\va)} \geq \frac{\lambda-\sqrt{\lambda^2-4}}{2}.$$
So, there is at least one rectangle in the intersection
$\cyl_\va \cap \cyl_\vv$ and its width is non-strictly greater than
$\frac{\lambda-\sqrt{\lambda^2-4}}{2} \w(\va)$. Since the circumference of $\cyl_\va$ is $\lambda \w(\va)$, it follows that any geodesic segment of positive slope crossing from the bottom of $\cyl_{\va}$ to the top without passing through the interior of $\cyl_\va \cap \cyl_\vv$
has slope greater than or equal to
$$\frac{\w(\va)}{\lambda \w(\va)-\frac{\lambda-\sqrt{\lambda^2-4}}{2} \w(\va)}=\frac{2}{\lambda+\sqrt{\lambda^2-4}}=\frac{\lambda-\sqrt{\lambda^2-4}}{2}.$$
But, $\sigma$ is supposed to be such a segment. Moreover, 
$\sigma$ points in the direction of $\btheta$, which satisfies equation \ref{eq:limit set inequality}. This is a contradiction.
\end{proof}

Now we will show that there are no saddle connections
on $S(\G,\w)$ which point in $\lambda$-renormalizable directions.

\begin{proof}[Theorem \ref{thm:no saddles}]
Suppose to the contrary that there is a surface $S(\G,\w)$
with $\w$ an eigenfunction with eigenvalue $\lambda$, a $\lambda$-renormalizable direction $\btheta$, and a saddle connection $\sigma$ whose holonomy is parallel to $\btheta$. 
Let $\langle g_i \rangle$ be the $\lambda$-shrinking sequence of $\btheta$. 
Observe that the extended support $\suppbar(\sigma)$ is a finite set. Furthermore, by inductively applying Lemma \ref{lemma:support lemma}, we see that
$$\suppbar\big(\Phi^{g_i}(\sigma)\big) \subset \suppbar(\sigma).$$
Thus, each $\Phi^{g_i}(\sigma)$ is a saddle connection contained in a finite union of cylinders indexed by $\suppbar(\sigma)$. We conclude that 
\begin{equation}
\label{eq:bar inequality}
\|\hol\big(\Phi^{g_i}(\sigma)\big)\| \geq \min~\{\w(\vv)~:~\vv \in \suppbar(\sigma)\}>0\quad \text{for all $i$.}
\end{equation}
On the other hand, we know that $\hol(\sigma)=c \btheta$ for some $c \neq 0$. Therefore,
$$\hol \big(\Phi^{g_i}(\sigma)\big)=c \rhoa{g_i}(\btheta)\quad \text{for all $i$.}$$
But, Theorem \ref{thm:shrinking} stated that $\lim_{i \to \infty} \|\rhoa{g_i} (\btheta)\|=0$, which contradicts equation \ref{eq:bar inequality}.
\end{proof}

\section{Conservativity of the straight-line flow}
\label{sect:recurrence}
The goal of this section is to prove Theorem \ref{thm:pr}. By hypothesis,
$S=S(\G,\w)$ is a surface built as in \S \ref{sect:graphs},
with $\w$ and eigenfunction for the adjacency operator with eigenvalue $\lambda \geq 2$, and $\bm \theta$ is a $\lambda$-renormalizable direction. The theorem concludes that the straight line flow in direction $\bm \theta$ is conservative.

The following lemma describes the technique we use to prove conservativity. For this lemma, let $S$ be a translation surface described as a countable union of polygons with edge identifications. We refer to a subset of $S$ as {\em bounded} if it is contained in a finite union of the polygons
making up $S$. We let $\bm \theta \in S^1$ be a direction and let $\mu$ be the Lebesgue transverse measure to the foliation in direction $\bm \theta$.

\begin{lemma}[Criterion for conservativity]
\label{lem:recurrence_criterion}
Suppose that for all bounded subsets $K \subset S$ and all $\epsilon>0$, there is a bounded subset $U \subset S$ such that $K \subset U$ and $\mu(\partial U) < \epsilon$.
Then the straight line flow in direction $\bm \theta$ is conservative.
\end{lemma}

The proof uses ideas from \cite[Proof of Theorem 1]{Tro04}. See \cite[Proof of Lemma 15]{Higl} for another variant of a proof.

\begin{proof}
It suffices to consider a bounded transversal $K$ to the foliation in direction $\btheta$ and demonstrate that there is no wandering $X \subset K$ with $\mu(X)>0$. ($X$ is {\em wandering} if no backward orbit of a point in $X$ returns to $X$.) Suppose not. By hypothesis we can find a bounded subset $U$ containing $K$ so that $\mu(\partial U)<\mu(X)$. Consider the backward straight line flow $F_{-\btheta}^t$ applied to $x \in X$. Let $t(x) \in \R \cup \{\infty\}$ be the first positive time the trajectory hits $X \cup \partial U$,
or $\infty$ if it never hits.
Observe that the portion of the trajectories $F_{-\btheta}^t(x)$ with $0 \leq t < t(x)$ are disjoint and contained entirely within $U$. The measure of the union of such trajectories is $\int_X t(x)~d\mu(x)$, which is bounded by the area of $U$. We conclude that $t(x)$ is almost everywhere finite. Now consider the $\mu$-a.e. defined map
$f:X \to X \cup \partial U$ defined by $f(x)=F_{-\btheta}^{t(x)}(x)$. This map is measure preserving in the sense that $\mu \big( f(A)\big)=\mu(A)$ for all measurable $A \subset X$. So because of our choice of $U$, we know 
$$\mu\big(f(X) \cap \partial U\big) \leq \mu (\partial U) < \mu(X).$$
We conclude that $\mu\big(f(X) \cap X\big)>0$, but this contradicts our original assumption that $X$ was wandering. 
\end{proof}

Let us restate the lemma in the context of our work. We let $S=S(\G,\w)$ and assume $\btheta$ is $\lambda$-renormalizable. 
For any bounded set $K$, we will find a nested sequence of bounded sets $U_1 \subset U_2 \subset \ldots \subset S$ so that $K \subset U_1$. The lemma tells us that if $\liminf_{n \to \infty} \mu(\partial U_n) =0$, then the straight-line flow in direction $\bm \theta$ is conservative.

By Lemma \ref{lem:compatible directions}, a $\lambda$-renormalizable direction $\bm \theta$ has a unique $\lambda$-shrinking sequence $\langle g_0, g_1, \ldots \rangle$. By an action of the dihedral group, we may assume that $g_1=h$ (see Remark \ref{rem:dihedral}). Our sequence of sets $U_i$ will be defined
using a natural subsequence $\gamma_n=g_{i(n)}$ of the shrinking sequence. We will inductively define this subsequence.
We define $i(0)=0$ so that $\gamma_0=e$, the identity. Recall that $g_i \circ g_{i-1}^{-1} \in \{h, h^{-1}, v, v^{-1}\}$ for all $i \geq 1$. 
For $n>0$, inductively define $$i(n)=\min \{ j>i(n-1) ~:~g_j \circ g_{j-1}^{-1} \neq g_{j+1} \circ g_{j}^{-1}\}.$$
Such an $i(n)$ exists because $\btheta$ is $\lambda$-renormalizable by Theorem \ref{thm:characterization}.
For example if 
$$\langle g_i \rangle=\langle e, h, h^2, vh^2, v^2 h^2, v^3 h^2, h^{-1}v^3 h^2, \ldots \rangle
\quad \textrm{then} \quad
\langle \gamma_n \rangle=\langle e, h^2, v^3 h^2, \ldots \rangle.$$
In particular, $\gamma_n \circ \gamma_{n-1}^{-1}$ is a non-zero power of $h$ when $n$ is odd and a non-zero power of $v$ when $n$ is even.

Recall that for a vertex $v \in \V$, $\cyl_v \subset S$ denotes the cylinder associated to $v$. 
Recall $\cyl_a$ is horizontal for $a \in \Alpha$ and $\cyl_b$ is vertical for $b \in \Beta$. 
We will now define subsets $V_n \subset \V$. Recall $K \subset S$ is bounded. 
Thus, the following set is finite:
$$V_0=\{a \in \Alpha~:~\cyl_a \cap K \neq \nullset\}.$$
For $i>0$ inductively define 
$$V_n=\{v \in \V~:~v \sim w \textrm{ for some $w \in V_{n-1}$}\}.$$
Note that $V_n \subset \Alpha$ for $n$ even, and $V_n \subset \Beta$ for $n$ odd. 
It can be observed inductively that each $V_n$ is finite because $\G$ has bounded valance. 

Using the $V_n$ we define subsets of $U_n \subset S$. Let $X_n=\bigcup_{v \in V_n} \cyl_v$. 
Note that $X_{n-1} \subset X_n$ for all $n \geq 1$. 
We define
$$U_n=\Phi^{\gamma_n^{-1}}(X_n).$$
We can see that these sets are nested by noting that 
$$U_n=\Phi^{\gamma_{n-1}^{-1}} \circ \Phi^{(\gamma_n \gamma_{n-1}^{-1})^{-1}}\big(X_n\big)
\supset \Phi^{\gamma_{n-1}^{-1}} \circ \Phi^{(\gamma_n \gamma_{n-1}^{-1})^{-1}}\big(X_{n-1}\big)=U_{n-1}.$$
The last equality follows from two statements. First, $X_{n-1}$ is a union of horizontal cylinders when $n$ is odd, and a union of
vertical cylinders when $n$ is even. And second, $\gamma_n \gamma_{n-1}$ is a power of $h$ when $n$ is odd, and a power of $v$ when $n$ is even.
Thus,  $\Phi^{(\gamma_n \gamma_{n-1})^{-1}}\big(X_{n-1}\big)=X_{n-1}$ for all $n$.

The $U_n$ are likely getting larger with much longer boundary measured using Euclidean length. But, $\partial U_n$ is getting closer to pointing in the direction $\bm \theta$. We will show that this convergence in direction is happening so fast that the Lebesgue transverse measure in direction $\bm \theta$ of $\partial U_n$ decays to zero.

\begin{proposition}[Boundary growth]
\label{prop:boundary_growth}
Let $\ell_n$ denote the Euclidean length of $\partial X_n$.
If $\lambda>2$ then there is a constant $c$ such that
$$\ell_n \leq c \Big(\frac{\lambda+\sqrt{\lambda^2-4}}{2}\Big)^n$$
for all $n$. If $\lambda=2$ then there is a $c$ for which $\ell_n<cn$ for all $n$.
\end{proposition}

In order to prove this proposition, we first verify the following:

\begin{claim}
\label{claim:growth}
Let $\lambda \geq 2$ and assume that $\langle a_0, a_1, \ldots \rangle$ is a sequence of non-negative real numbers satisfying
$a_{n+1} \leq \lambda a_n - a_{n-1}$.
If $\lambda>2$ then there is a constant $c$ such that
$$a_n \leq c \Big(\frac{\lambda+\sqrt{\lambda^2-4}}{2}\Big)^n$$
for all $n$. If $\lambda=2$ then there is a $c$ for which $a_n<cn$ for all $n$.
\end{claim}
\begin{proof}
We consider the linear map $\phi(x,y)=(y,\lambda y-x)$. We note that $(a_n,a_{n+1})$ has
the same $x$ coordinate and non-strictly smaller $y$ coordinate than $\phi(a_{n-1},a_{n})$. 
Let $\omega=\frac{\lambda+\sqrt{\lambda^2-4}}{2}$. 
The eigenvectors of $\phi$ are ${\mathbf v_1}=(1, \omega)$ and ${\mathbf v_2}=(\omega,1)$. They satisfy
$\phi(\mathbf v_1)=\omega \mathbf v_1$ and $\phi(\mathbf v_2)=\omega^{-1} \mathbf v_2$.
We first note that if $\frac{a_n}{a_{n-1}}<\omega^{-1}$ for some $n$, then there is no way to infinitely continue
the sequence so that $a_{n} \geq 0$ for all $n$. This because for every point $(x,y)$ with $\frac{y}{x}<\omega^{-1}$
is eventually ejected from the positive quadrant by a power of $\phi$. This is illustrated in the left half of Figure \ref{fig:growth}. 
Moreover, lowering the $y$-coordinate along the way, will only result in the vector being ejected from the quadrant faster.

Now assume $\lambda>2$. By the above argument, we know that $\frac{a_n}{a_{n-1}} \geq \omega^{-1}$ for all $n$. Let $P$ denote the closed parallelogram constructed from the convex
hull of the points $(0,0)$, $\omega^{-1} \mathbf v_2=(1,\omega^{-1})$, $\mathbf v_1$ and $(0,\omega-1)$. We note that $\phi(P) \subset \omega P$. 
See the right half of Figure \ref{fig:growth}. 
Moreover
$$\{(x,y)~:~\exists (x_0, y_0) \in P \textrm{ such that } x=y_0, y<\lambda y_0-x_0 \textrm{ and } \frac{y}{x} \geq \omega^{-1} \}\subset \omega(P).$$
We may assume that $(a_0,a_1) \in c P$ for some $c>0$. 
By induction, we conclude that $(a_{n}, a_{n+1}) \in c \omega^n P$ for all $n$. The conclusion for $\lambda>2$ follows.

When $\lambda=2$, we set $P_k$ to be the convex hull of the points $(0,0)$, $(k,k)$, $(k,k+1)$, and $(0,1)$.  We may check that
$\phi(P_k) \subset P_{k+1}$ for all $k>0$. Moreover, 
$$\{(x,y)~:~\exists (x_0, y_0) \in P_k \textrm{ such that } x=y_0, y<2 y_0-x_0 \textrm{ and } \frac{y}{x} \geq 1\}\subset P_{k+1}.$$
Assuming $(a_0,a_1) \in c P_0$, we have $(a_{n}, a_{n+1}) \in c P_n$ for all $n$, and the
conclusion follows.
\end{proof}

\begin{figure}[ht]
\begin{center}
\includegraphics[width=6in]{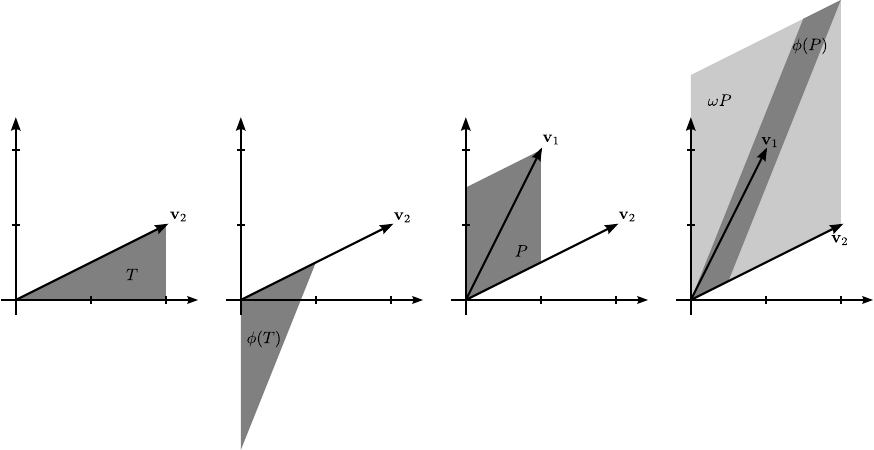}
\caption{This figure illustrates the action of $\phi$ from the proof of Claim \ref{claim:growth} on $\R^2$ in the case of $\lambda=\frac{5}{2}$. 
The left two pictures illustrate a triangle $T$ of points satisfying $\frac{y}{x}<\omega^{-1}$ being (eventually) ejected from the positive quadrant.
The right two pictures indicate how $\phi(P) \subset \omega P$.}
\label{fig:growth}
\end{center}
\end{figure}

\begin{proof}[Proof of Proposition \ref{prop:boundary_growth}]
We will show that the numbers $\ell_n$ are related by the inequality
\begin{equation}
\label{eq:recurrence_eqn}
\ell_{n+1} \leq \lambda \ell_n - \ell_{n-1}.
\end{equation}
By Claim \ref{claim:growth} above, this is sufficient to guarantee the result.
Consider $X_n$ with boundary of length $\ell_n$. Let 
$W_n=V_{n+1} \smallsetminus V_{n-1}$. 
The set of all $\cyl_v$ for $v \in W_n$ is the set of all horizontal or vertical cylinders which cross $\partial X_n$. (These cylinders are all horizontal if $n$ is odd and all vertical if $n$ is even.) As every such cylinder has inverse modulus
$\lambda$, we know that
$$\sum_{v \in W_n} \textit{length}(\partial \cyl_v) \leq \lambda \ell_n.$$
(Each such cylinder $\cyl_v$ for $v \in W_n$ crosses $\partial X_n$ at least twice. A cylinder which crosses segments of length $l$, must have circumference $\lambda l$.)
Moreover, 
$$X_{n+1}=X_{n-1} \cup \bigcup_{v \in W_n} \cyl_v.$$ 
Explicitly, $X_{n+1}$
is formed from $X_{n-1}$ by attaching horizontal (or vertical) cylinders to the horizontal  (or vertical) boundaries of $X_{n-1}$.
Thus, we have that 
$$\partial X_{n-1} \cup \partial X_{n+1} \subset \bigcup_{b \in W_n} \partial \cyl_v.$$
The sets $\partial X_{n-1}$ and $\partial X_{n+1}$ are disjoint other than they may contain common singular points. Hence equation \ref{eq:recurrence_eqn} holds.
\end{proof}

Now as $U_n=\Phi^{\gamma_n^{-1}}(X_n)$, we can compute the Lebesgue transverse measure of the boundary of $U_n$ in terms of the unit vector $\bm \theta$ and $\ell_n$.
Recall that the derivative $D(\Phi^g)$ is $\rhoa{g}$. From this observation, we have
\begin{equation}
\mu(\partial U_n)=\begin{cases}
|\rhoa{\gamma_n^{-1}}(\ell_n,0) \wedge \bm \theta| & \textrm{if $n$ is even} \\
|\rhoa{\gamma_n^{-1}}(0,\ell_n) \wedge \bm \theta| & \textrm{if $n$ is odd} 
\end{cases}=\begin{cases}
|(\ell_n,0) \wedge \rhoa{\gamma_n}(\bm \theta)| & \textrm{if $n$ is even} \\
|(0,\ell_n) \wedge \rhoa{\gamma_n}(\bm \theta)| & \textrm{if $n$ is odd}, 
\end{cases}
\end{equation}
where $(a,b)\wedge (c,d)=ad-bc$ is the usual wedge product in the plane. In any case, we have
\begin{equation}
\mu(\partial U_n) \leq \ell_n \| \rhoa{\gamma_n}(\bm \theta) \|.
\end{equation}
The proof of conservativity proceeds by observing that along a subsequence $\| \rhoa{\gamma_n}(\bm \theta) \|$ decays faster than $\ell_n$ grows,
and therefore Lemma \ref{lem:recurrence_criterion} applies. 

We define a collection of troublesome elements of the group $G$. Let 
$$\begin{array}{rcl}
{\mathcal T} & = & 
\{(hv^{-1})^k\} \cup \{v^{-1} (hv^{-1})^k\} \cup 
\{(h^{-1} v)^k\} \cup \{v (h^{-1} v)^k  \} \cup \\
& & \{(vh^{-1})^k\} \cup \{h^{-1} (vh^{-1})^k\} \cup 
\{(v^{-1} h)^k\} \cup \{h (v^{-1} h)^k \},\end{array}$$
where $k$ is allowed to range over the set $\{0,1,2,3, \ldots\}$. We define a subsequence $\langle \gamma_{n_j} \rangle$ by
$n_0=0$ and
$$n_j=\min \{n_j>n_{j-1}~:~\gamma_{n_j} \gamma_{n_{j-1}}^{-1} \not \in {\mathcal T}\}.$$
Note that $n_j$ is well defined for all $j$ so long as $\bm \theta$ is a $\lambda$-renormalizable direction by Theorem \ref{thm:characterization}.
(If $n_j$ is not well defined for some $j$, then $\langle g_n \rangle$ is tail equivalent to one of last two pair of sequences in the theorem.)

Given the above arguments, the following two lemmas imply the conservativity of the straight-line flow in a $\lambda$-renormalizable direction $\bm \theta$. (That is,
they imply Theorem \ref{thm:pr}.) The first lemma handles the case of $\lambda=2$, and the second handles the case of $\lambda>2$ which is made more technical because we have to work with the varying values of $\lambda$. The proofs of both these lemmas use the same ideas. 
\begin{lemma}
\label{lem:theta_decay2}
Assume $\lambda=2$. Then, there is a positive $\epsilon<1$ depending such that for all $j$
$$\eta(n_1) \eta(n_2-n_1) \ldots \eta(n_j-n_{j-1}) \| \rho_2^{\gamma_{n_j}}(\bm \theta) \| \leq \epsilon^j \quad \textrm{for all $j$},$$
where $\eta(1)=\frac{3}{2}$ and $\eta(n)=n$ for integers $n \geq 2$. 
\end{lemma}

The statement of the lemma was chosen because it admits a recursive proof. Given the lemma, we get the type of decay we really want:
\begin{corollary}
Assume $\lambda=2$. Then for all $j$, there is a positive $\epsilon<1$ such that
$$n_j \| \rho_2^{\gamma_{n_j}}(\bm \theta) \| \leq \epsilon^j \quad \textrm{for all $j$}.$$
\end{corollary}
\begin{proof}
This follows from the fact that for any finite collection of $k \geq 1$ positive integers,
$$m_1+m_2+\ldots+m_k \leq \eta(m_1) \eta(m_2) \ldots \eta(m_k).$$
(Here, $m_k=n_k-n_{k-1}$ with $n_0$ taken to be zero.) It is clearly true for for $k=1$. We now prove it for $k=2$. We break into cases:
\begin{itemize}
\item If $m_1=m_2=1$ then $m_1+m_2=2$ and $\eta(m_1)\eta(m_2)=\frac{9}{4}>2$. 
\item If $m_1=1$ and $m_2 \geq 2$ then $\frac{\eta(m_1)\eta(m_2)}{m_1+m_2}=\frac{\frac{3}{2} m_2}{m_2+1}=\frac{3}{2}-\frac{3}{2(m_2+1)} \geq 1$, with equality when $m_2=2$. 
\item If $m_1 \geq 2$ and $m_2 \geq 2$ then $\eta(m_1)\eta(m_2) \geq 2 \max(m_1, m_2) \geq m_1+m_2$.
\end{itemize}
Now suppose the statement is true for $k\geq 2$. Then,
$$\begin{array}{rcl}
m_1+m_2+\ldots+m_k+m_{k+1} & \leq & \eta(m_1+\ldots+m_k)\eta(m_k) = (m_1+\ldots+m_k)\eta(m_{k+1}) \\
& \leq  &
\eta(m_1) \eta(m_2) \ldots \eta(m_k) \eta(m_{k+1}).\end{array}$$
\end{proof}

\begin{proof}[Proof of Lemma \ref{lem:theta_decay2}]
To simplify notation, write $\delta_j=\gamma_{n_j}$. It is sufficient to prove the statement for the case of $j=1$. To see this, suppose the statement is true for $j=1$ and all 
$\lambda$-renormalizable directions $\bm \theta$. Then 
$$
\begin{array}{rcl}
\eta(n_1) \ldots \eta(n_j-n_{j-1}) \| \rho_2^{\delta_{j}}(\bm \theta) \|  & = & 
\eta(n_1) \ldots \eta(n_j-n_{j-1}) \| \rho_2^{\delta_{j} \delta_1^{-1}} \circ \rho_2^{\delta_1}(\bm \theta)\| \\
& = & \Big(\eta(n_1) \|\rho_2^{\delta_1}(\bm \theta)\|\Big)\Big(\eta(n_1')\ldots \eta(n'_{j-1}-n'_{j-2}) \| \rho_2^{\delta'_{j-1}}(\bm \theta') \| \\
& \leq & \epsilon \eta(n_1')\ldots \eta(n'_{j-1}-n'_{j-2}) \| \rho_2^{\delta'_{j-1}}(\bm \theta')\|.\end{array}$$ 
where $\bm \theta'=\rho_2^{\delta_1}(\bm \theta)/\|\rho_2^{\delta_1}(\bm \theta)\|$, $n_j'=n_{j+1}-n_1$, and $\delta'_j=\delta_{j+1} \delta_1^{-1}$. Note that $\delta'_j$ arises from the shrinking sequence for $\bm \theta'$ in the same way in which $\delta_j$ arises from the shrinking sequence of $\btheta$. The inequality arises from the statement of the lemma in the case $j=1$.
We repeat this argument $j-1$ more times to obtain the statement of the lemma.

We will now concentrate on proving the lemma in the case of $j=1$. We note that the possible words $\delta_1$ are of one of the following eight forms:
\begin{equation}
\label{eq:good_forms}
\begin{array}{cccc}
v^{-a} (hv^{-1})^k &
v^a (h^{-1} v)^k &
h^{-a} (vh^{-1})^k &
h^a (v^{-1} h)^k 
\\
h^a v^{-1} (hv^{-1})^k & 
h^{-a} v (h^{-1} v)^k &
v^a h^{-1} (vh^{-1})^k & 
v^{-a} h (v^{-1} h)^k 
\end{array}
\end{equation}
Here $k\geq 0$ and $a \in \Z \smallsetminus \{0,1\}$. In addition,  $(a, k) \neq (-1,0)$ for words in the first row. 
By the previous paragraph, it is enough to prove that if $\delta_1$ is one of these words then 
\begin{equation}
\label{eq:enough1B}
\eta(n_1) \|\rho_2^{\delta_1}(\bm \theta)\| \leq \epsilon < 1,
\end{equation}
where $n_1=2k+1$ if $\delta_1$ is chosen from the first line of equation \ref{eq:good_forms}, or
$n_1=2k+2$ if $\delta_1$ is chosen from the second. We simplify our job more by noting that many of these words are equivalent under the dihedral group.
See Remark \ref{rem:dihedral}.
Thus we really only need to cover one case from the first line and one case from the second.

We will prove the statement for the two cases 
$\delta_1=h^{-a} (vh^{-1})^k$ and $\delta_1=h^{a} v^{-1} (hv^{-1})^k$.
In both cases either $v \delta_1$ or $v^{-1} \delta_1$ is an element of the shrinking sequence for $\bm \theta$. 
(Otherwise $\delta_1$ would be reducible or a longer word.) 
Then by Corollary \ref{cor:shrunk}, we know that
$$\bm \theta \in \Shrink_2(v \delta_1) \cup \Shrink_2(v^{-1} \delta_1)=\pi_\Circ \circ \rho_2^{\delta_1^{-1}}\big(\Shrink_2(v) \cup \Shrink_2(v^{-1})\big).$$
For $\lambda=2$, $\Shrink_2(v) \cup \Shrink_2(v^{-1}) \subset \{(x,y) \in \Circ~:~-1<\frac{x}{y}<1\}$ by Proposition \ref{prop:shrink_exp}. Therefore,
$\bm \theta \in \pi_\Circ \circ \rho_2^{\delta_1^{-1}}(x,\pm 1)$
for some choice of $-1 < x <1$ and some choice of $\pm 1$. In particular it follows that 
$$\|\rho_2^{\delta_1}(\bm \theta)\|=\frac{\|(x,\pm 1)\|}{\|\rho_2^{\delta_1^{-1}}(x,\pm 1)\|} \leq \frac{\sqrt{2}}{\|\rho_2^{\delta_1^{-1}}(x,\pm 1)\|}.$$
Therefore, it is sufficient to prove the statement that
\begin{equation}
\label{eq:sufficient_eq2}
\inf_{-1<x<1} \frac{\|\rho_2^{\delta_1^{-1}}(x,1)\|}{\eta(n_1)\sqrt{2}} \geq \frac{1}{\epsilon}>1.
\end{equation}
(We can drop the $\pm 1$ by symmetry.)

We will now work out the case of $\delta_1=h^{-a} (vh^{-1})^k$ where $k\geq 0$, $a \in \Z \smallsetminus \{0,1\}$, $(a, k) \neq (-1,0)$ and $n_1=2k+1$. We need to estimate the length of 
$$\rho_2^{\delta_1^{-1}}(x,1)=\rho_2^{(hv^{-1})^k} \circ \rho_2^{h^a}(x, 1).$$
Note that $\rho_2^{h^a}(x, 1)=(x+2a,1)$. We break into cases depending on $a$. Recall $-1<x<1$.
\begin{enumerate}
\item $x+2a>3$ when $a \geq 2$, or 
\item $x+2a<-1$ when $k \geq 1$ and $a \leq -1$.
\item $x+2a<-3$ when $k=0$ and $a  \leq -2$.
\end{enumerate}
In each of these cases, we have shown that $\rho_2^{h^a}(x, 1)$ lies in a ray contained in the line $L=\{(x',1)~:~x' \in \R\}$. Now consider the matrix
$$\rho_2^{hv^{-1}}=\left[\begin{array}{cc} -3 & 2 \\ -2 & 1\end{array}\right].$$
This matrix is a parabolic with eigenvector $\v=(1,1)$ satisfying $\rho_2^{hv^{-1}}(\v)=-\v$. 
Let $\u=(1,0)$. Then $\rho_2^{hv^{-1}}(\u)=-\u-2 \v$. We see that $L=\{\v+t \u~:~t \in \R\}$. We compute that 
$$\rho_2^{(hv^{-1})^k}(\v+t \u)=(-1)^k\big((2kt+1) \v+t \u\big).$$
It can be checked that the point of $\rho_2^{(hv^{-1})^k}(L)$ which lies closest to the origin is 
$\rho_2^{(hv^{-1})^k}(\v+t \u)$ for $-1 \leq t <0$. This is the image under $\rho_2^{(hv^{-1})^k}$ of $\v+t \u=(x',1)$ for some $0 \leq x' < 1$. In particular, this point is not in any of the rays. It follows that the infimum
we need to find to apply equation \ref{eq:sufficient_eq2} is obtained at the endpoint of the ray. We must check this equation in each of the three cases. In case (1), the endpoint is $(3,1)=\v+2\u$. Therefore, 
\begin{equation*}
\inf_{-1<x<1} \frac{\|\rho_2^{\delta_1^{-1}}(x,1)\|}{\eta(n_1)\sqrt{2}} >
\frac{\|\rho_2^{(hv^{-1})^k}(\v+2 \u)\|}{\eta(2k+1)\sqrt{2}}=
\frac{\|(1+4k)\v+2 \u\|}{\eta(2k+1)\sqrt{2}}.
\end{equation*}
When $k=0$, we have
\begin{equation*}
\inf_{-1<x<1} \frac{\|\rho_2^{\delta_1^{-1}}(x,1)\|}{\eta(n_1)\sqrt{2}} > \frac{\|(3, 1)\|}{\frac{3 \sqrt{2}}{2}}=
\frac{\sqrt{10}}{\frac{3\sqrt{2}}{2}}=\frac{\sqrt{20}}{3}>1.
\end{equation*}
When $k \geq 1$, we know $\eta(2k+1)=2k+1$. We apply the triangle inequality
\begin{equation*}
\begin{array}{rcl}
\displaystyle \inf_{-1<x<1} \frac{\|\rho_2^{\delta_1^{-1}}(x,1)\|}{\eta(n_1)\sqrt{2}} & > &
\displaystyle \frac{(1+4k)\|\v\|-2\|\u\|}{(2k+1)\sqrt{2}}=\frac{(1+4k)\sqrt{2}-2}{2k+1} \\ 
& = & \displaystyle  2-\frac{1+\sqrt{2}}{2k+1} \geq 2-\frac{1+\sqrt{2}}{3}>1.\end{array}
\end{equation*}
In case (2), the endpoint is $(-1,1)=\v-2 \u$. We have $k \geq 1$. We check equation \ref{eq:sufficient_eq2}.
\begin{equation*}
\inf_{-1<x<1} \frac{\|\rho_2^{\delta_1^{-1}}(x,1)\|}{\eta(n_1)\sqrt{2}} >
\frac{\|\rho_2^{(hv^{-1})^k}(\v-2 \u)\|}{(2k+1)\sqrt{2}}=
\frac{\|(1-4k)\v-2 \u\|}{(2k+1)\sqrt{2}}.\end{equation*}
When $k=1$, we have $\|(1-4k)\v-2 \u\|=\|(-5,-3)\|=\sqrt{34}$. Therefore,
\begin{equation*}
\inf_{-1<x<1} \frac{\|\rho_2^{\delta_1^{-1}}(x,1)\|}{\eta(n_1)\sqrt{2}} >
\frac{\sqrt{34}}{4\sqrt{2}}=\frac{\sqrt{17}}{4}>1.\end{equation*}
When $k\geq 2$, we can apply the triangle inequality.
\begin{equation*}
\begin{array}{rcl}
\displaystyle \inf_{-1<x<1} \frac{\|\rho_2^{\delta_1^{-1}}(x,1)\|}{\eta(n_1)\sqrt{2}} & > &
\displaystyle \frac{\|(1-4k)\v-2 \u\|}{(2k+1)\sqrt{2}}>\frac{(4k-1)\sqrt{2}-2}{(2k+1)\sqrt{2}}
=\frac{4k-1-\sqrt{2}}{2k+1} \\
& = & \displaystyle 
2-\frac{3+\sqrt{2}}{2k+1} \geq 2-\frac{3+\sqrt{2}}{5}>1.\end{array}\end{equation*}
Finally, we consider case (3). Here $k=0$ and the ray endpoint is $(-3,1)$ so, 
\begin{equation*}
\inf_{-1<x<1} \frac{\|\rho_2^{\delta_1^{-1}}(x,1)\|}{\eta(n_1)\sqrt{2}}  > 
\frac{\|(-3,1)\|}{\frac{3}{2}\sqrt{2}}=\frac{\sqrt{20}}{3}>1.
\end{equation*}

Now we consider the case of $\delta_1=h^{a} v^{-1} (hv^{-1})^k$ where $n_1=2k+2$. We must estimate the length of 
$$\rho_2^{\delta_1^{-1}}(x,1)=\rho_2^{(vh^{-1})^k} \circ \rho_2^{v} \circ \rho_2^{h^{-a}}(x, 1).$$
Noting that $\rhoa{h^{-a}}(x,1)=(x-2a,1)$, we break into two cases depending on
the choice of $k \geq 0$ and $a \not \in \{0, 1\}$.
\begin{enumerate}
\item[(1')] $x-2a < -3$ when $a \geq 2$, or 
\item[(2')] $x-a > 1$ when $a \leq -1$.
\end{enumerate}
Let $\v$ and $\u$ be as in the previous paragraph. We have $\rho_2^{vh^{-1}}(\v)=-\v$ and 
$\rho_2^{vh^{-1}}(\u)=-(\u-2\v)$. Let $L=\{(x',1)\}$. Set $L'= \rho_2^v(L)=\{(t,1+2t)~:~t \in \R\}.$
The closest point to the origin on $L'$ is $(\frac{-2}{5},\frac{1}{5})$. Now consider 
the closest point on $\rho_2^{(vh^{-1})^k}(L')$. Using the action of this parabolic, it can be checked that
this closest point is of the form $\rho_2^{(vh^{-1})^k}(t,1+2t)$ where $-1<t\leq \frac{-2}{5}$. It follows that the closest point
on $\rho_2^{(vh^{-1})^k v}(L)$ is also of this form. The preimage of this point lies in 
$$\{\rho_2^{v^{-1}}(t,1+2t)~:~-1<t\leq \frac{-2}{5}\}=\{(t,1)~:~-1<t\leq \frac{-2}{5}\}.$$
Again we observe, that the rays in cases (1') and (2') do not contain such points. Therefore, the minimum length must occur at the endpoints of the rays. We now check equation \ref{eq:sufficient_eq2} in case (1'). 
Here the endpoint is $(-3,1)$ and $\rho_2^v(-3,1)=(-3,-5)=-5 \v+2 \u$. We simplify equation \ref{eq:sufficient_eq2} using the triangle inequality.
\begin{equation*}
\begin{array}{rcl}
\displaystyle \inf_{-1<x<1} \frac{\|\rho_2^{\delta_1^{-1}}(x,1)\|}{\eta(n_1)\sqrt{2}} & > &
\displaystyle \frac{\|\rho_2^{(vh^{-1})^k}(-5 \v+2 \u)\|}{(2k+2)\sqrt{2}}=
\frac{\|(-5-4k)\v+2 \u\|}{(2k+2)\sqrt{2}} \\
& \geq &
\displaystyle \frac{(5+4k)\sqrt{2}-2}{(2k+2)\sqrt{2}}=2-\frac{\sqrt{2}-1}{2k+2}\geq 2-\frac{\sqrt{2}-1}{2}>1.
\end{array}
\end{equation*}
In case (2'), the endpoint is $(1,1)$ and $\rho_2^v(1,1)=(1,3)=3\v-2 \u$. Therefore,
\begin{equation*}
\inf_{-1<x<1} \frac{\|\rho_2^{\delta_1^{-1}}(x,1)\|}{\eta(n_1)\sqrt{2}} > 
\frac{\|\rho_2^{(vh^{-1})^k}(3 \v-2 \u)\|}{(2k+2)\sqrt{2}}=
\frac{\|(3+4k)\v-2 \u\|}{(2k+2)\sqrt{2}}.
\end{equation*}
When $k=0$, we have $(3+4k)\v-2 \u=(1,3)$. Therefore, in this case,
\begin{equation*}
\inf_{-1<x<1} \frac{\|\rho_2^{\delta_1^{-1}}(x,1)\|}{\eta(n_1)\sqrt{2}} > 
\frac{\|(3,1)\|}{2\sqrt{2}}=\frac{\sqrt{5}}{2}>1.
\end{equation*}
When $k \geq 1$, we apply the triangle inequality
\begin{equation*}
\inf_{-1<x<1} \frac{\|\rho_2^{\delta_1^{-1}}(x,1)\|}{\eta(n_1)\sqrt{2}} > 
\frac{(3+4k)\sqrt{2}-2}{(2k+2)\sqrt{2}}=\frac{3+4k-\sqrt{2}}{2k+2}=2-\frac{1+\sqrt{2}}{2k+2} \geq
2-\frac{1+\sqrt{2}}{4}>1.
\end{equation*}
This concludes the argument. The constant $\epsilon$ can be taken so that $\frac{1}{\epsilon}$ is the minimum of the finite number of constants used the cases above.
\end{proof}

\begin{lemma}
\label{lem:theta_decay}
If $\lambda>2$, 
there is a positive $\epsilon<1$ depending only on $\lambda$ such that
$$\Big(\frac{\lambda+\sqrt{\lambda^2-4}}{2}\Big)^{n_j} \| \rhoa{\gamma_{n_j}}(\bm \theta) \| \leq \epsilon^j \quad \textrm{for all $j$}.$$
\end{lemma}
\begin{proof}
To simplify notation, write $\delta_j=\gamma_{n_j}$ and $\omega=(\lambda+\sqrt{\lambda^2-4})/2$. As in the previous case, it is in fact sufficient to prove the statement for $j=1$. To see this, suppose the statement is true for all $\lambda$-renormalizable directions $\bm \theta$
in the case of $j=1$. Then we have
$$
\omega^{n_{j+1}} \| \rhoa{\delta_{j+1}}(\bm \theta) \|  = 
\omega^{n_{j+1}} \| \rhoa{\delta_{j+1} \delta_1^{-1}} \circ \rhoa{\delta_1}(\bm \theta)\| =
\omega^{n_{j+1}} \| \rhoa{\delta_{j+1} \delta_1^{-1}}(\btheta')\| \|\rhoa{\delta_1}(\bm \theta)\|
< \epsilon \omega^{n_j'} \| \rhoa{\delta'_{j}} (\bm \theta')\|, 
$$
where $\bm \theta'=\rhoa{\delta_1}(\bm \theta)/\|\rhoa{\delta_1}(\bm \theta)\|$, $n_j'=n_{j+1}-n_1$, and $\delta'_j=\delta_{j+1} \delta_1^{-1}$. Again, we can apply induction to obtain the statement of the lemma.

We now prove the lemma in the case of $j=1$. As in the previous proof, it is sufficient to prove the lemma in the cases of 
$\delta_1=h^{-a} (vh^{-1})^k$ and $\delta_1=h^{a} v^{-1} (hv^{-1})^k$. Here, $k\geq 0$ and $a \in \Z \smallsetminus \{0,1\}$.
For $\delta_1=h^{-a} (vh^{-1})^k$ we can also assume $(a, k) \neq (-1,0)$.
It is enough to prove that if $\delta_1$ is one of these words then 
\begin{equation}
\label{eq:enough1}
\omega^{n_1} \|\rhoa{\delta_1}(\bm \theta)\| \leq \epsilon < 1.
\end{equation}
Here $n_1=2k+1$ if $\delta_1=h^{-a} (vh^{-1})^k$ and $n_1=2k+2$ if $\delta_1=h^{a} v^{-1} (hv^{-1})^k$.

In both cases either $v \delta_1$ or $v^{-1} \delta_1$ is an element of the shrinking sequence for $\bm \theta$. 
Then by Corollary \ref{cor:shrunk}, we know that in either case
$$\bm \theta \in \Shrink_\lambda(v \delta_1) \cup \Shrink_\lambda(v^{-1} \delta_1)=\pi_\Circ \circ \rhoa{\delta_1^{-1}}\big(\Shrink_\lambda(v) \cup \Shrink_\lambda(v^{-1})\big),$$
where $\pi_\Circ$ denotes projection onto the unit circle.
From Proposition \ref{prop:limit_set2}, we know 
\begin{equation}
\label{eq:shrink_interval}
\Rn_\lambda \cap \big(\Shrink_\lambda(v) \cup \Shrink_\lambda(v^{-1})\big) \subset \pi_\Circ( \{(x, \pm 1)~:~-\omega^{-1} < x < \omega^{-1} \}).
\end{equation}
In particular, $\pi_\Circ \circ \rhoa{\delta_1^{-1}}(\bm \theta)$ lies in this set. Without loss of generality, we may assume
$\bm \theta=\pi_\Circ \circ \rhoa{\delta_1^{-1}}(x,1)$ for some $x$ with $-\omega^{-1} < x < \omega^{-1}$. Thus,
\begin{equation}
\label{eq:accessible_form}
\begin{array}{rcl}
\|\rhoa{\delta_1}(\bm \theta)\|
 & \leq & \displaystyle \sup_{-\omega^{-1} < x < \omega^{-1}} ~ \frac{\|(x,1)\|}{\|\rhoa{\delta_1^{-1}}(x, 1)\|}  \\
& \leq & 
\displaystyle \sup_{-\omega^{-1} < x < \omega^{-1}} ~ \frac{\sqrt{\lambda \omega^{-1}}}{\|\rhoa{\delta_1^{-1}}(x,1)\|} =
\frac{\sqrt{\lambda \omega^{-1}}}{\inf_{-\omega^{-1} < x < \omega^{-1}} \|\rhoa{\delta_1^{-1}}(x,1)\|}.
\end{array}
\end{equation}
The last inequality follows from the fact that $\|(x,1)\| \leq \|(\omega^{-1},1)\| = \sqrt{\lambda \omega^{-1}}$. 
In particular from equation \ref{eq:enough1}, it is sufficient to prove that 
\begin{equation}
\label{eq:enough2}
\inf_{-\omega^{-1} < x < \omega^{-1}} \frac{\omega^{-n_1}}{\sqrt{\lambda \omega^{-1}}}  \|\rhoa{\delta_1^{-1}}(x,1)\| \geq
\frac{1}{\epsilon} >1.
\end{equation}
for some $\epsilon<1$.

Now we concentrate on the case
$\delta_1=h^{-a} (vh^{-1})^k$ and $n_1=2k+1$. We need to compute the infimum of $\rhoa{\delta_1^{-1}}(x,1)=\rhoa{(hv^{-1})^k h^{a}}(x,1)$ over those $x$ satisfying $-\omega^{-1} < x < \omega^{-1}$. We have that $\rhoa{h^a}(x,1)=(x+a \lambda,1)$. Now, we break into cases depending on $a$ and $k$. 
As $a \not \in \{0, 1\}$, $k \geq 0$ and $(a,k) \neq (-1,0)$, either
\begin{enumerate}
\item $x+a \lambda \geq 2\lambda-\omega^{-1}=\lambda+\omega$ when $a \geq 2$, or 
\item $x+a \lambda \leq \omega^{-1}-\lambda=-\omega$ when $k \geq 1$ and $a \leq -1$.
\item $x+a \lambda \leq \omega^{-1}-2\lambda=-\lambda-\omega$ when $k=0$ and $a  \leq -2$.
\end{enumerate}
(For the equalities above, we use the identity $\omega+\omega^{-1}=\lambda$.) 
We will compute the infimum over $x'$ of
$\|\rhoa{(hv^{-1})^k}(x',1)\|$, where $x'$ ranges over the possible values of $x+a \lambda$ allowed by the inequalities
in cases (1)-(3). First consider the infimum over the line
$\rhoa{(hv^{-1})^k} (\R \times \{1\})$. The two eigenvectors of $\rhoa{hv^{-1}}$ are 

\begin{equation}
\label{eq:eigenvectors1}
\mathbf v_1=(\omega,1)
\quad \textrm{and} \quad \mathbf v_2=(\omega^{-1},1)
,\end{equation} with
$\rhoa{hv^{-1}}(\mathbf v_1)=-\omega^2 \mathbf v_1$ and $\rhoa{hv^{-1}}(\mathbf v_2)=-\omega^{-2} \mathbf v_2$. Therefore,
$$\inf_{x' \in \R} \|\rhoa{(hv^{-1})^k}(x',1)\|=\|\rhoa{(hv^{-1})^k}(x^\circ,1)\|,$$
where $0 \leq x^\circ < \omega^{-1}$. As the inequalities above exclude the possibility of $x'=x^\circ$, we conclude
that the infimum in cases (1)-(3) must be realized at the endpoints of the rays. 

In case (1), we have
$$\inf_{x' \geq \lambda+\omega} \|\rhoa{(hv^{-1})^k}(x',1)\|=
\|\rhoa{(hv^{-1})^k}(\lambda+\omega,1)\|.$$
We break into two subcases depending if $k=0$ or $k \geq 1$. 
If $k=0$, then 
$$\inf_{x' \geq \lambda+\omega} \|\rhoa{(hv^{-1})^k}(x',1)\|=\|(\lambda+\omega,1)\|>
\lambda+\omega.$$
Applying this to equation \ref{eq:enough2} in this case yields
$$\inf_{-\omega^{-1} < x < \omega^{-1}} \frac{\omega^{-n_1}}{\sqrt{\lambda \omega^{-1}}}  \|\rhoa{\delta_1^{-1}}(x,1)\| \geq \frac{\omega^{-1}}{\sqrt{\lambda \omega^{-1}}} (\lambda+\omega)=\sqrt{\frac{(\lambda+\omega)^2}{\lambda \omega}} > \sqrt{2}.$$
Now consider $k \geq 1$.
We have that $(\lambda+\omega,1)=\frac{1}{\omega-\omega^{-1}} (2\omega \mathbf v_1-\lambda \mathbf v_2)$, with $\mathbf v_i$ as in equation \ref{eq:eigenvectors1}.
Thus,
$$\begin{array}{rcl}
\|\rhoa{(hv^{-1})^k}(\lambda+\omega,1)\| & = & 
\frac{\omega^{2k+1}}{\omega-\omega^{-1}}
\|2 \mathbf v_1-\lambda \omega^{-4k} \mathbf v_2\| \geq 
\frac{\omega^{2k+1}}{\omega-\omega^{-1}}
(2 \|\v_1\|- \lambda \omega^{-4} \|\v_2\|) \\
& = & \frac{\omega^{2k+1}}{\omega-\omega^{-1}}
(2 \sqrt{\lambda \omega}- \lambda \omega^{-4} \sqrt{\lambda \omega^{-1}}) =
\frac{\omega^{2k+1} \sqrt{\lambda \omega^{-1}}}{\omega-\omega^{-1}}(2 \omega -\lambda \omega^{-4}).\end{array}$$
Applying this to equation \ref{eq:enough2} yields
$$\begin{array}{rcl}
\displaystyle \inf_{-\omega^{-1} < x < \omega^{-1}} \frac{\omega^{-n_1}}{\sqrt{\lambda \omega^{-1}}}  \|\rhoa{\delta_1^{-1}}(x,1)\| 
 & \geq & \displaystyle \frac{\omega^{-2k-1}}{\sqrt{\lambda \omega^{-1}}} \Big(\frac{\omega^{2k+1} \sqrt{\lambda \omega^{-1}}}{\omega-\omega^{-1}}(2 \omega -\lambda \omega^{-4}) \Big) \\
& = & \displaystyle \frac{2 \omega -\lambda \omega^{-4}}{\omega-\omega^{-1}} =
\frac{\omega-\omega^{-1}+\lambda(1-\omega^{-4})}{\omega-\omega^{-1}}
>1. \end{array}
$$

In case (2), we have $k \geq 1$ and 
$$\inf_{x' \leq -\omega} \|\rhoa{(hv^{-1})^k)}(x',1)\|=
\|\rhoa{(hv^{-1})^k)}(-\omega,1)\|.$$
We may write $(-\omega,1)=\frac{1}{\omega-\omega^{-1}} (- \lambda \mathbf v_1+2 \omega \mathbf v_2)$. Therefore,
$$\begin{array}{rcl}
\displaystyle \|\rhoa{(hv^{-1})^k}(-\omega,1)\| & = & \displaystyle \frac{1}{\omega-\omega^{-1}} \| 2 \omega^{1-2k} \mathbf v_2-\omega^{2k} \lambda \mathbf v_1\|
\geq \frac{1}{\omega-\omega^{-1}} \big(\omega^{2k} \lambda \|\mathbf v_1\|-
2 \omega^{1-2k} \|\mathbf v_2\|\big) \\
& = & \displaystyle \frac{\omega^{2k+1} \sqrt{\lambda \omega^{-1}}}{\omega-\omega^{-1}} \big(\lambda
-2 \omega^{-4k}\big)>\frac{\omega^{2k+1} \sqrt{\lambda \omega^{-1}}}{\omega-\omega^{-1}} (\lambda-2 \omega^{-4}).
\end{array}$$
Applying this to equation \ref{eq:enough2} yields
$$\begin{array}{rcl}
\displaystyle \inf_{-\omega^{-1} < x < \omega^{-1}} \frac{\omega^{-n_1}}{\sqrt{\lambda \omega^{-1}}}  \|\rhoa{\delta_1^{-1}}(x,1)\| 
 & \geq & \displaystyle \frac{\omega^{-2k-1}}{\sqrt{\lambda \omega^{-1}}} \Big(\frac{\omega^{2k+1} \sqrt{\lambda \omega^{-1}}}{\omega-\omega^{-1}} (\lambda-2 \omega^{-4}) \Big) \\
& = & \displaystyle \frac{\lambda-2 \omega^{-4}}{\omega-\omega^{-1}}=
 \frac{\omega-\omega^{-1}+2(\omega^{-1}-\omega^{-4})}{\omega-\omega^{-1}}
>1. \end{array}
$$

Case (3) is simpler because $k=0$. For $\delta_1$ as in this case, we have that
$$\inf_{x' \leq -\lambda-\omega} \|\rhoa{(hv^{-1})^k}(x',1)\|=\|(-\lambda-\omega,1)\|>
\lambda+\omega,$$
and we can proceed as in the first subcase of case (1).

Now we consider the case of $\delta_1=h^{a} v^{-1} (hv^{-1})^k$ and $n_1=2k+2$. 
Noting that $\rhoa{h^{-a}}(x,1)=(x-a \lambda,1)$, we break into two cases depending on
the choice of $k \geq 0$ and $a \not \in \{0, 1\}$.
\begin{enumerate}
\item[(1')] $x-a \lambda \leq \omega^{-1}-2\lambda=-\lambda-\omega$ when $a \geq 2$, or 
\item[(2')] $x-a \lambda \geq \lambda-\omega^{-1}=\omega$ when $a \leq -1$.
\end{enumerate}
We will analyze
$\inf_{x'} \|\rhoa{(vh^{-1})^k v}(x',1)\|$, where $x'$ ranges over the possible values of $x-a \lambda$ allowed by the inequalities
in cases (1') and (2').
This time, we have that 
$$\inf_{x' \in \R} \|\rhoa{(vh^{-1})^k v}(x',1)\|=\|\rhoa{(vh^{-1})^k v}(x^\circ,1)\|,$$
where $-\frac{\lambda}{1+\lambda^2} \leq x^\circ < \omega^{-1}$. This interval is in the complement of the set of values allowed for $x-a\lambda$ by cases (1) and (2),
therefore the value of $\inf_{x'=x-a\lambda}  \|\rhoa{(vh^{-1})^k v}(x',1)\|$ is realized at the endpoints of the rays of these cases. 

In case (1') we have,
$$\inf_{x' \leq -\lambda-\omega} \|\rhoa{(hv^{-1})^k}(x',1)\|=\|\rhoa{(vh^{-1})^k v}(-\lambda-\omega,1)\|=
\|\rhoa{(vh^{-1})^k}(-\lambda-\omega,-\lambda^2-\omega^2)\|.$$
As before $\mathbf v_1=(\omega,1)$ and $\mathbf v_2=(\omega^{-1},1)$ are eigenvectors of $\rhoa{v h^{-1}}$, but here
$\rhoa{v h^{-1}}(\mathbf v_1)=\omega^{-2} \mathbf v_1$ and
$\rhoa{v h^{-1}}(\mathbf v_2)=\omega^{2} \mathbf v_2$.
We write our vector in terms of
these eigenvectors as
$$(-\lambda-\omega,-\lambda^2-\omega^2)=\frac{\lambda}{\omega^3-\omega} \mathbf v_1+\frac{-2 \omega^4}{\omega^2-1} \mathbf v_2.$$
Thus, we see 
$$\rhoa{(vh^{-1})^k v}(-\lambda-\omega,1)=\frac{\lambda}{\omega^3-\omega} \omega^{-2k} \mathbf v_1+\frac{-2 \omega^4}{\omega^2-1} \omega^{2k} \mathbf v_2.$$
By the triangle inequality, we have
$$\|\rhoa{(vh^{-1})^k v}(-\lambda-\omega,1)\|
\geq 
\frac{2 \omega^{2k+4}\|\mathbf v_2\|}{\omega^2-1}-\frac{\lambda \omega^{-2k-1}\|\mathbf v_1\|}{\omega^2-1}
=\frac{\sqrt{\lambda \omega^{-1}}}{\omega^2-1}
(2 \omega^{2k+4} - \lambda \omega^{-2k}).$$
In this case, we apply equation \ref{eq:enough2} and see
$$\begin{array}{rcl}
\displaystyle \inf_{-\omega^{-1} < x < \omega^{-1}} \frac{\omega^{-n_1}}{\sqrt{\lambda \omega^{-1}}}  \|\rhoa{\delta_1^{-1}}(x,1)\| 
 & \geq & \displaystyle \frac{\omega^{-2k-2}}{\sqrt{\lambda \omega^{-1}}} \Big(\frac{\sqrt{\lambda \omega^{-1}}}{\omega^2-1}
(2 \omega^{2k+4} - \lambda \omega^{-2k}) \Big) \\
& = & \displaystyle \frac{2 \omega^2-\lambda \omega^{-4k-2}}{\omega^2-1}=
\frac{2 \omega^2-\omega^{-4k-1}-\omega^{-4k-3}}{\omega^2-1}
 \\
& > & \displaystyle \frac{2 (\omega^2-1)}{\omega^2-1}=2.
\end{array}
$$

In case (2'), we have the following calculations.
$$\inf_{x' \geq \omega} \|\rhoa{(vh^{-1})^k v}(x',1)\|=|\rhoa{(vh^{-1})^k v}(\omega,1)\|=
|\rhoa{(vh^{-1})^k}(\omega,\lambda \omega+1)\|.$$
We may write $(\omega,\lambda \omega+1)=\frac{1}{\omega^2-1}(-2 \mathbf v_1+\omega^3 \lambda \v_2)$. Therefore,
$$\rhoa{(vh^{-1})^k}(\omega,\lambda \omega+1)=\frac{1}{\omega^2-1}(-2 \omega^{-2k} \mathbf v_1+\omega^{3+2k} \lambda \mathbf v_2).$$
By the triangle inequality, we have
$$|\rhoa{(vh^{-1})^k}(\omega,\lambda \omega+1)\| \leq \frac{\sqrt{\lambda\omega^{-1}}}{\omega^2-1}(\omega^{3+2k} \lambda -2 \omega^{1-2k}).$$
We apply equation \ref{eq:enough2} and see
$$\begin{array}{rcl}
\displaystyle \inf_{-\omega^{-1} < x < \omega^{-1}} \frac{\omega^{-n_1}}{\sqrt{\lambda \omega^{-1}}}  \|\rhoa{\delta_1^{-1}}(x,1)\| 
 & \geq & \displaystyle \frac{\omega^{-2k-2}}{\sqrt{\lambda \omega^{-1}}} \Big(\frac{\sqrt{\lambda\omega^{-1}}}{\omega^2-1}(\omega^{3+2k} \lambda -2 \omega^{1-2k}) \Big) \\
& = & \displaystyle \frac{\omega \lambda-2 \omega^{-4k-1}}{\omega^2-1} =
\frac{\omega^2+1-2 \omega^{-4k-1}}{\omega^2-1} \\
& \geq & \frac{\omega^2+1-2 \omega^{-1}}{\omega^2-1},
\end{array}
$$
and the fraction $\frac{\omega^2+1-2 \omega^{-1}}{\omega^2-1}$
is strictly greater than one because $1-2 \omega^{-1} > -1$. 
This completes the proof of equation \ref{eq:enough2} in all cases.
\end{proof}

\section{Quadrants and shrinking sequences}
\label{sect:quadrants}
Up to this point the $\lambda$-shrinking sequence $\langle g_i \rangle$ of a vector $\bm \theta$ has been characterized by the property that 
$$\|\bm \theta \|> \|\rhoa{g_1}(\bm \theta)\|>\|\rhoa{g_2}(\bm \theta)\|>\ldots.$$
In this section, we will see that this condition is essentially equivalent to saying that each $\rhoa{g_i}(\bm \theta)$ lies in some specific
quadrant depending on the shrinking sequence. This section culminates with a proof of Theorem \ref{thm:operator_theorem}.

\subsection{Quadrants and expansion}
\label{sect:quadrants and expansion}
Because $G$ is a free group, it is natural to identify each $g \in G$ with its unique expression as a reduced word in the generators. This is equivalent to identifying $g$ with a geodesic segment in the Cayley graph joining $e \in G$ to $g \in G$. In this context, it is natural to consider $\tilde G$ which we define to be the free monoid with generating set $\{h,v,h^{-1},v^{-1}\}$. That is, $\tilde G$ is the set of all words in these symbols. Our expression of $g \in G$ as its unique reduced word gives a set-theoretic embedding $G \hookrightarrow \tilde G$.

Recall Definitions \ref{def:sign_pairs} and \ref{def:quadrants1} of sign pairs and quadrants in $\R^2$, respectively.
We define the following action on sign pairs.
\begin{definition}[Expanding sign action]
\label{def:sign_action}
The monoid action $\sa:\tilde G \times \SP \to \SP \label{not:sa}$ on the set of sign pairs is the action determined by action of generators shown in the following diagram.
\begin{center}
\includegraphics{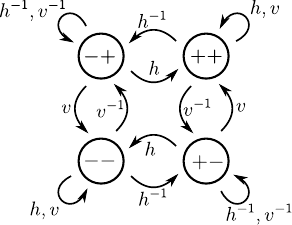}
\end{center}
For $g \in G$, we define $\Sigma^g$ to be the action of $g$ written as a reduced word in $\tilde G$.
\end{definition}

Now recall Definition \ref{def:shrink_exp} of $\Exp_\lambda(g)$. 
Motivation for the definition of the expanding sign action comes from the following.

\begin{proposition}[Quadrants and expansion]
\label{prop:quadrants_and_expansion}
Let $\langle g_i \rangle$ be a geodesic ray or segment with $g_0=e$ and let $s \in \SP$. Suppose that
$\bm \theta \in \cl\big(\Exp_\lambda(g_1) \cap \Q_s\big)$.
Define $s_i=\sa^{g_i}(s)$.  Then
$\rhoa{g_i}(\bm \theta) \in \cl(\Q_{s_i})$ for all $i$.
\end{proposition}

\begin{remark}
The technical condition $\bm \theta \in \cl(\Exp_\lambda(g_1) \cap \Q_s)$ 
allows the proposition to handle horizontal and vertical $\bm \theta$.
If $\bm \theta$ is not horizontal or vertical, the condition that 
$\bm \theta \in \cl(\Exp_\lambda(g_1) \cap \Q_s)$
is equivalent to saying that $\bm \theta \in \Q_s \cap \Circ$ and $\|\rho_\lambda^{g_1}(\bm \theta)\| \geq \| \bm \theta\|$. 
\end{remark}

\begin{proof}[Proof of Proposition \ref{prop:quadrants_and_expansion}]
First, we will prove that the statement is true for $i=1$. Any such statement must be invariant under the action of
the dihedral group on $\langle g_i \rangle$ and $\bm \theta$. See Remark \ref{rem:dihedral}. Thus we may assume that $g_1=h$ and $\bm \theta \in \cl(\Q_{++} \cup \Q_{-+})$.
We have $\bm \theta \in \Exp_\lambda(h) \cap (\cl(\Q_{++} \cup \Q_{-+})$. Thus,
$$\rhoa{h}(\bm \theta) \in \rhoa{h}\big( \Exp_\lambda(h) \cap (\cl(\Q_{++} \cup \Q_{-+}) \big) \subset \Q_{++}.$$
This agrees with the fact that $\sa^h(-+)=\sa^h(++)=++$. So, the proposition is true for $i=1$, by dihedral group invariance.

To see the statement is true for $i>1$, we apply induction. Proposition \ref{prop:increasing} implies that the sequence $\|\rhoa{g_i}(\bm \theta)\|$ is non-strictly increasing. Assume $\rhoa{g_i}(\bm \theta) \in \cl(\Q_{s_i})$. We have $\|\rhoa{g_{i+1}}(\bm \theta)\| \geq \|\rhoa{g_i}(\bm \theta)\|$.
Then by the first paragraph, $\rhoa{g_{i+1}}(\bm \theta) \in \cl(\Q_{s'})$, where $s'=\sa^{g_{i+1}g_{i}^{-1}}(s_i)=s_{i+1}$.
\end{proof}

\begin{proposition}
\label{prop:inside_self}
Let $g \in \{h,v,h^{-1},v^{-1}\}$ and $s \in \SP$. If $\rhoa{g}(\Q_s) \subset \Q_s$ then for any $\v \in \Q_s$ we have $\|\rhoa{g}(\v)\|>\|\v\|$.
\end{proposition}
\begin{proof}
By the dihedral group action, we may assume $g=h$. We have $\rhoa{h}(\Q_{++}) \subset \Q_{++}$ and $\rhoa{h}(\Q_{--}) \subset \Q_{--}$, but
$\rhoa{h}(\Q_{+-}) \not \subset \Q_{+-}$ and $\rhoa{h}(\Q_{-+}) \not \subset \Q_{-+}$. By Proposition \ref{prop:shrink_exp} we have
$\Q_{++} \cup \Q_{--} \subset \Exp_\lambda(h)$. 
\end{proof}

By combining these two propositions, we have the following.

\begin{corollary}
\label{cor:quadrants}
Suppose $\langle g_i \rangle$ is a geodesic ray and $\rhoa{g_1}(\Q_s) \subset \Q_s$, or equivalently
$$(g_1, s) \in \{(h,++), (h,--), (v,++), (v,--), (h^{-1},+-), (h^{-1},-+), (v^{-1},+-), (v^{-1},-+)\}.$$
Then for all $i$, $\rhoa{g_i}\big(\cl(\Q_s)\big) \subset \cl(\Q_{s_i})$ 
where $s_i = \sa^{g_i}(s)$.
\end{corollary}

Recall that given a $\lambda$-renormalizable direction $\bm \theta \in \Circ$ with shrinking sequence $\langle g_0, g_1, g_2, \ldots \rangle$, we can associate
a sign sequence $\langle s_0, s_1, s_2, \ldots \rangle$ with $s_i \in \SP$ so that $\rhoa{g_i}(\bm \theta) \in \Q_{s_i}$ for all $i$. 
See Definition \ref{def:sign_sequence}. These two sequences are related by the following.

\begin{proposition}[Shrinking and sign sequences]
\label{prop:shrinking_and_sign}
Let $\langle g_i\rangle$ be the shrinking sequence of a $\lambda$-renormalizable direction $\bm \theta \in \Circ \cap \Q_{s_0}$, with $s_0 \in \SP$. 
Then, there is a unique (infinite) path in the following diagram which begins at the node labeled $\Q_{s_0}$ and follows edges
labeled $g_1, g_2 g_1^{-1}, g_3 g_2^{-2}$, etc. Moreover,
the path visits the sequence of nodes $\langle \Q_{s_i}\rangle$ where $\langle s_i \rangle$ is the sign sequence of $\btheta$. 
\begin{center}
\includegraphics{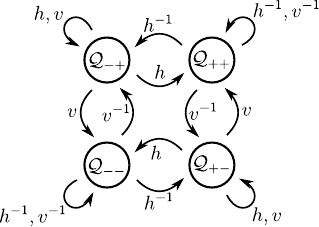}
\end{center}
\end{proposition}
\begin{proof}
We will prove that we can determine $s_{i+1}$ from $s_i$, $g_{i+1} g_{i}^{-1}$ and $g_{i+2} g_{i+1}^{-1}$.
(This is similar to the diagram; a node in a path is uniquely determined from the previous node and the labels on the adjacent arrows.)
Consider the following observations.
\begin{enumerate}
\item We have $\rhoa{g_{i+1}}(\bm \theta) \in \Shrink_\lambda(g_{i+2} g_{i+1}^{-1}) \cap \Q_{s_{i+1}}$. So the intersection
$\Shrink_\lambda(g_{i+2} g_{i+1}^{-1}) \cap \Q_{s_{i+1}}$ is non-empty. Therefore,
$g_{i+2} g_{i+1}^{-1} \in \{h,v\}$ implies $s_{i+1} \in \{+-,-+\}$ and
$g_{i+2} g_{i+1}^{-1} \in \{h^{-1},v^{-1}\}$ implies $s_{i+1} \in \{++,--\}$.
\item 
Similarly, we have $\rhoa{g_{i+1}}(\bm \theta) \in \rhoa{g_{i+1} g_{i}^{-1}}(\Q_{s_{i}}) \cap \Q_{s_{i+1}}$, so this intersection is non-empty. Therefore, we have the following.
$$\begin{array}{c}
s_i=+- \textrm{ and } g_{i+1} g_{i}^{-1}=h \textrm{ implies } s_{i+1} \in \{+-, --\}.\\
s_i=-+ \textrm{ and } g_{i+1} g_{i}^{-1}=h \textrm{ implies } s_{i+1} \in \{++, -+\}.\\
s_i=+- \textrm{ and } g_{i+1} g_{i}^{-1}=v \textrm{ implies } s_{i+1} \in \{++, +-\}.\\
s_i=-+ \textrm{ and } g_{i+1} g_{i}^{-1}=v \textrm{ implies } s_{i+1} \in \{-+, --\}.\\
s_i=++ \textrm{ and } g_{i+1} g_{i}^{-1}=h^{-1} \textrm{ implies } s_{i+1} \in \{++, -+\}.\\
s_i=-- \textrm{ and } g_{i+1} g_{i}^{-1}=h^{-1} \textrm{ implies } s_{i+1} \in \{+-, --\}.\\
s_i=++ \textrm{ and } g_{i+1} g_{i}^{-1}=v^{-1} \textrm{ implies } s_{i+1} \in \{++, +-\}.\\
s_i=-- \textrm{ and } g_{i+1} g_{i}^{-1}=v^{-1} \textrm{ implies } s_{i+1} \in \{-+, --\}.
\end{array}$$
\end{enumerate}
These two observations combine to uniquely determine $s_{i+1}$ as in the diagram. 
\end{proof}

\begin{proof}[Proof of Proposition \ref{prop:same_sign}]
Proposition \ref{prop:same_sign} claimed that the sign sequence of $\btheta=\btheta(\langle g_i \rangle,\lambda)$ only depended on the sequence $\langle g_i \rangle$ and on the quadrant containing $\btheta$ and not on $\lambda$. This is a consequence of \ref{prop:shrinking_and_sign}, because $\lambda$ plays no role.
\end{proof}

%

\subsection{Critical times for the shrinking sequence}

Let $\langle s_n \rangle$ be the sign sequence of $\bm \theta$. We have the following important definition.

\begin{definition}
\label{def:critical_time}
An integer $n \geq 1$ for which $s_{n-1}=s_{n}$ is a {\em critical time}. 
\end{definition}

There are various relationships between the sign sequence and the shrinking sequence $\langle g_n \rangle$ that appear at critical times.

\begin{proposition} 
\label{prop:crit_equiv}
Let $n>0$ be an integer. The following statements are equivalent.
\begin{enumerate}
\item $n$ is a critical time.
\label{crit1}
\item $\rhoa{g_{n-1} g_n^{-1}}(\Q_{s_n}) \subset \Q_{s_n}$.
\label{crit2}
\item \label{crit3} The pair $(g_{n-1} g_n^{-1}, s_{n})$ lies in the set
$$\{ (h,++), (h,--), (v,++), (v,--), (h^{-1},+-), (h^{-1},-+), (v^{-1},+-), (v^{-1},-+)\}.$$
\item \label{crit6}
$g_{n+1} g_{n-1}^{-1} \in \{h^2, hv, vh, v^2, h^{-2}, h^{-1}v^{-1}, v^{-1}h^{-1},v^{-2}\}.$
\end{enumerate}
\end{proposition}
\begin{proof}
It can be observed that each of these statements is invariant
under the action of the dihedral group. (See Remark \ref{rem:dihedral}.)
To simplify things, we use the dihedral group to arrange that 
$(x,y) = \rhoa{g_{n-1}}(\btheta)$ satisfy $x>0$, $y>0$ and $y<x$. 
This guarantees that $g_n=h^{-1} g_{n-1}$ 
(since $\rhoa{h^{-1}}$ is the only generator which shrinks $(x,y)$; see Proposition \ref{prop:shrink_exp}).
It also implies that $(x,y) \in \Q_{++}$ (i.e., $s_{n-1}=++$). So, saying that $n$ is a critical time is the same as saying that $\rhoa{h^{-1}}(x,y) \in \Q_{++}$, i.e., $s_n=++$. If $n$ is not a critical time, then it must be that
$\rhoa{h^{-1}}(x,y) \in \Q_{-+}$, i.e., $s_n=-+$. Using this information, the other statements can be observed to be equivalent by inspection.
For example, we will verify that (4) is equivalent to being a critical time. If it is a critical time, then because 
$\rhoa{g_n}(\btheta) \in \Q_{++}$, we know that in order to further shrink the vector we must have $g_{n+1} g_n^{-1} \in \{h^{-1},v^{-1}\}$. So, $g_{n+1} g_{n-1}^{-1} \in \{h^{-2},v^{-1}h^{-1}\}$, which is allowed by (4).
On other hand if $n$ is not a critical time, then $\rhoa{g_n}(\btheta) \in \Q_{-+}$. So in order to further shrink the vector, we have $g_{n+1} g_n^{-1} \in \{h,v\}$. But because $g_{n} g_{n-1}^{-1}=h^{-1}$, the first option is not allowed. We conclude that $g_{n+1} g_{n-1}^{-1}=v h^{-1}$, which is not in the list in item (4). \end{proof}

\begin{corollary}[Critical times occur]
\label{cor:critical_times_occur}
For $\bm \theta \in \Rn_\lambda$,
there are infinitely many critical times.
\end{corollary}
\begin{proof}
We apply statement (\ref{crit6}) of Proposition \ref{prop:crit_equiv}. Suppose the conclusion is false for $\bm \theta$.
Then for all but finitely many $n$, we would have 
$g_n g_{n-2}^{-1} \in \{hv^{-1}, h^{-1}v, vh^{-1}, v^{-1}h\}$.
But, then $\langle g_i \rangle$ is tail equivalent to one of the last two sequences of Theorem \ref{thm:characterization}, and so $\btheta$ is not $\lambda$-renormalizable.
\end{proof}

\subsection{Interaction with the dot product} 

Recall the group automorphism $\gamma:G \to G$ defined as in equation \ref{eq:gamma} on page \pagelink{eq:gamma}. This automorphism has special significance for the representation $\rhoa{G}$.
\begin{proposition}
For all $\lambda$ and all $g \in G$, we have $\rhoa{\gamma(g)}=\transpose{\rho}_\lambda^{g^{-1}}$, the inverse transpose of 
$\rhoa{g}$. 
\end{proposition}
The proof follows by checking that it is true on the generators. As a consequence, 
for all $\v, \w \in \R^2$ and all $g \in G$,
\begin{equation}
\label{eq:dot_product}
\v \cdot \w=\rhoa{g}(\v) \cdot \rhoa{\gamma(g)}(\w),
\end{equation}
where $\cdot$ denotes the usual dot product $\R^2 \times \R^2 \to \R$. 
We establish some simple corollaries of this observation.

\begin{corollary}[Sign pairs at critical times]
\label{cor:critical_time_formula}
Let $\langle g_i\rangle$ and $\langle s_i\rangle$ be shrinking sequences and sign pairs for a $\lambda$-shrinkable direction $\bm \theta \in \Circ$. Suppose $n>0$ is a critical time,
then 
$s_n=\Sigma^{\gamma(g_n)}(s_0).$
\end{corollary}
\begin{proof}
Consider that $1=\bm \theta \cdot \bm \theta=\rhoa{g_i}(\bm \theta) \cdot \rhoa{\gamma(g_i)}(\bm \theta)$ for all $i>0$. Since $\|\rhoa{g_1}(\bm \theta)\|<1$ we know
$\|\rhoa{\gamma(g_1)}(\bm \theta)\|>1$. Therefore, by Proposition \ref{prop:quadrants_and_expansion}, we know $\rhoa{\gamma(g_i)}(\bm \theta) \in \Q_{s'_i}$ where $s'_i=\Sigma^{\gamma(g_i)}(s_0)$ for all $i$. Since $n$ is a critical time, we know $s_n=s_{n-1}$. By the dihedral group action of Remark \ref{rem:dihedral}, without loss of generality, we can assume that
$s_{n-1}=s_n=++$ and $g_n \circ g_{n-1}^{-1}=h^{-1}$. Note that
$$s'_n \in \Sigma^{\gamma(h^{-1})}(\SP)=\Sigma^v(\SP)=\{++,--\}.$$
(See the diagram in Definition \ref{def:sign_action}.) Since $\rhoa{\gamma(g_n)}(\bm \theta) \in \Q_{s'_n}$, we know $s'_n=++$ because $\rhoa{g_n}(\bm \theta)\in \Q_{s_n}=\Q_{++}$ and $\rhoa{g_n}(\bm \theta) \cdot \rhoa{\gamma(g_n)}(\bm \theta)=1$.
\end{proof}

\begin{corollary}
\label{cor:dot_product_growth}
Suppose $\bm \theta$ is $\lambda$-renormalizable, and $\v \in \R^2$ satisfies $\bm \theta \cdot \v \neq 0$. Let $\langle g_n \rangle$ be the $\lambda$-shrinking sequence for $\bm \theta$. Then, 
$\|\rhoa{\gamma(g_n)}(\v)\| \to \infty$ as $n \to \infty$.
\end{corollary}
\begin{proof}
By equation \ref{eq:dot_product}, we have that
$\bm \theta \cdot \v=\rhoa{g_n}(\bm \theta) \cdot \rhoa{\gamma(g_n)}(\v)$. Therefore,
$$\|\rhoa{\gamma(g_n)}(\v)\| \geq \frac{|\bm \theta \cdot \v|}{\|\rhoa{g_n}(\bm \theta)\|}.$$
The denominator of this expression tends to $0$ as $n \to \infty$, so $\|\rhoa{\gamma(g_n)}(\v)\| \to \infty$. 
\end{proof}

Of particular importance, we can conclude that if $\bm \theta \cdot \v \neq 0$ then eventually there is an $i$ for which 
$\|\rhoa{\gamma(g_i)}(\v)\|>\|\rhoa{\gamma(g_{i-1})}(\v)\|$. Then, for $n>i$, the quadrant containing $\rhoa{\gamma(g_n)}(\v)$
is governed by the expanding sign action. See Proposition \ref{prop:quadrants_and_expansion}. 
Then, the following applies.

\begin{proposition}
\label{prop:detecting_postive_dot}
Let $\bm \theta \in \Rn_\lambda$. Let $\langle g_i \rangle$ be the $\lambda$-shrinking sequence of $\bm \theta$, and let 
$\langle s_i \rangle$ be the sign sequence. Suppose $\v \in \R^2$ satisfies 
$\bm \theta \cdot \v > 0$. By Corollary \ref{cor:dot_product_growth},
there is an $i$ for which $\|\rhoa{\gamma(g_i)}(\v)\|>\|\rhoa{\gamma(g_{i-1})}(\v)\|$.
For any critical time
$n \geq i$, we have $\rhoa{\gamma(g_{n})}(\v) \in \Q_{s_n}$. 
\end{proposition}
\begin{proof}
Note that if $g \in \{h,v, h^{-1}, v^{-1}\}$ then the image $\sa^g(\SP)$ consists of precisely two sign pairs. 
(For instance, $\sa^h(\SP)=\{+-,-+\}$. See Definition \ref{def:sign_action}.) In particular, for $n \geq i$, we have
$\rhoa{\gamma(g_{n})}(\v)$ lies in one of the two quadrants in $\sa^{g_n g_{n-1}^{-1}}(\SP)$.
Now suppose that $n$ is a critical time. Statement (\ref{crit3}) of Proposition \ref{prop:crit_equiv} explicitly describes the possible pairs
$(g_{n-1} g_n^{-1}, s_n)$. By inspection it can be observed that $\sa^{g_n g_{n-1}^{-1}}(\SP)=\{s_n, -s_n\}$. Now recalling equation \ref{eq:dot_product}, we have that 
$$\bm \theta \cdot \v=\rhoa{g_n}(\bm \theta) \cdot \rhoa{\gamma(g_n)}(\v).$$
Therefore $\bm \theta \cdot \v>0$ and $\rhoa{g_n}(\bm \theta) \in \Q_{s_n}$ implies that $\rhoa{\gamma(g_n)}(\v) \in \Q_{s_n}$.
\end{proof}

Now suppose that $\bm \theta \cdot \v = 0$. Let $R$ denote the linear map $R:\R^2 \to \R^2$ which rotates by $\frac{\pi}{2}$. 
As a matrix,
\begin{equation}
\label{eq:R}
R=\left[\begin{array}{rr}
0 & -1 \\ 1 & 0 \end{array}\right].
\end{equation}
Conjugation by this map induces an automorphism of $\rhoa{G}$. Inspection reveals that 
\begin{equation}
\label{eq:R_commute}
R \circ \rhoa{g}=\rhoa{\gamma(g)} \circ R
\end{equation}
for all $g \in G$. 
Note that this map $R$ permutes the quadrants of $\R^2$ and therefore induces permutation $r:\SP \to \SP$ so that
\begin{equation}
\label{eq:r}
R(\Q_s)=\Q_{r(s)}.
\end{equation}

The punchline of this section is that we can detect the sign of $\bm \theta \cdot \v$ using the shrinking and sign sequences of $\bm \theta$.

\begin{theorem}
\label{thm:dot_product}
Let $\bm \theta \in \Rn_\lambda$, and let $\langle g_n \rangle$ and $\langle s_n \rangle$ denote the $\lambda$-shrinking and sign sequences of $\bm \theta$, respectively. Choose any nonzero vector $\v \in \R^2$. Then,
\begin{enumerate}
\item[(a)] If $\bm \theta \cdot \v > 0$, then there is an $i \geq 0$ such that for any critical time $n>i$ we have $\rhoa{\gamma(g_{n})}(\v) \in \Q_{s_n}$. 
\item[(b)] If $\bm \theta \cdot \v < 0$, then there is an $i \geq 0$ such that for any critical time $n>i$ we have $-\rhoa{\gamma(g_{n})}(\v) \in \Q_{s_n}$. 
\item[(c)] If $\bm \theta \cdot \v = 0$, the for all $i \geq 0$ we have $\rhoa{\gamma(g_{n})}(\v) \in \Q_{r^{\pm 1} (s_n)}$ for some fixed choice of sign. 
\end{enumerate}
\end{theorem}
\begin{proof}
Statement (a) follows directly from Proposition \ref{prop:detecting_postive_dot}. Statement (b) follows by applying this proposition to $- \v$. Statement
(c) follows from equation \ref{eq:R_commute} by noting that
$$\rhoa{\gamma(g_n)} \circ R(\bm \theta)=R \circ \rhoa{g_n}(\bm \theta) \in \Q_{r(s_n)}.$$
\end{proof}

\subsection{Proofs of results from \S \ref{sect:outline}}
\label{sect:proof_operator_theorem}
We need to convert between the
dot and wedge product. We recall that
\begin{equation}
\label{eq:dot_to_wedge}
\v \cdot \w=\v \wedge R(\w)=R^{-1}(\v) \wedge \w
\end{equation}
with $R$ defined as in equation \ref{eq:R}. 

Recall that the Quadrant Sequence Proposition (Prop. \ref{prop:quadrant_sequence}) said that the only $\v \in \Circ$ for which $\rhoa{g_n}(\v) \in \overline{\Q}_{s_n}$ for all $n$ is $\v=\bm \theta$. We now give the proof:

\begin{proof}[Proof of Proposition \ref{prop:quadrant_sequence}]
Suppose such a $\v$ exists. Then $\v \wedge \bm \theta \neq 0$. Let $\w=R(\v)$. Then,
$$\v \wedge \bm \theta=\bm \theta \cdot \w=\rhoa{g_n}(\bm \theta) \cdot \rhoa{\gamma(g_{n})}(\w).$$
Since $\bm \theta \cdot \w$ is non-zero, Theorem \ref{thm:dot_product} guarantees that there is a $n$ for which 
$\pm \rhoa{\gamma(g_{n})}(\w) \in \Q_{s_n}$, with the sign equal to the sign of $\bm \theta \cdot \w$.
By equation \ref{eq:R_commute}, we have
$\pm \rhoa{\gamma(g_{n})}(\w)=\pm R \circ \rhoa{g_{n}}(\v)$, so we may conclude that 
$\rhoa{g_{n}}(\v) \in \Q_{r(s_n)} \cup \Q_{r^{-1}(s_n)}$. This is a contradiction, since we assumed that $\rhoa{g_{n}}(\v) \in {\overline \Q}_{s_n}$.
\end{proof}

Recall Definition \ref{def:survivor} of a $(\bm \theta, n)$-survivor $m\in \Coh$. Theorem \ref{thm:operator_theorem} stated that a cohomology class $m \in \Coh$
arises from applying $\Psi_\btheta$ to a locally finite transverse measure to the foliation in a $\lambda$-renormalizable direction $\bm \theta$ if and only if $m$ is a $(\bm \theta, n)$-survivor for all $n \geq 0$.

\begin{proof}[Proof of Theorem \ref{thm:operator_theorem}]
Note that the ``only if'' direction is trivial. Checking that $m$ is a $(\bm \theta, n)$-survivor for all $n \geq 0$
simply checks that $m$ pairs correctly with some saddle connections. See Definition \ref{def:survivor} of $(\theta, n)$-survivors. In particular, Lemma \ref{lem:sign} implies
this is a necessary condition for $m \in \Psi_\btheta(\M_\btheta)$. 

For the ``if'' direction, we will use the sufficiency criterion given by Lemma \ref{lem:sign} . Let $\sigma$ be any saddle oriented connection in $S$, and let $\v=\hol~ \sigma$. We will check that if
$m(\hom{\sigma}) \neq 0$ then $\sgn\big(m(\hom{\sigma}) \big)=\sgn\big(\hol(\sigma) \wedge \bm \theta\big)$. 
So, assume that $m(\hom{\sigma}) \neq 0$.
We know that $\hol \sigma$ is not parallel to $\bm \theta$
by Theorem \ref{thm:no saddles}. (There are no saddle connections in $\lambda$-renormalizable directions.)
Let $\langle g_n \rangle$ denote the shrinking sequence for $\bm \theta$, and $\langle s_n \rangle$ denote the sign sequence. By Proposition \ref{prop:quadrant_sequence},
there is a smallest $n$ for which $\rhoa{g_n}(\hol~\sigma) \not \in \Q_{s_n} \cup \Q_{-s_n}$. 
Then, there is a sign pair $s \not \in \{\pm s_n\}$ so that $\rhoa{g_n}(\hol~\sigma)$ lies in the closed quadrant
${\overline \Q}_{s}$. Let $\sigma'=\Phi^{g_n}(\sigma)$ whose holonomy is given by $\rhoa{g_n}(\hol~\sigma)$.
Because our surface decomposes into rectangles, we can write
$$\hom{\sigma'}=\sum_{i=1}^{k} \hom{\sigma_i'},$$
where the $\sigma_i'$ are horizontal or vertical saddle connections (boundary edges of the rectangles) oriented so that their holonomies lie in the boundary of the quadrant $\Q_s$. Now let $\sigma_i=\Phi^{g_n^{-1}}(\sigma_i)$
for all $i$ so that $\hom{\sigma}=\sum_{i=1}^k \hom{\sigma_i}$.
Since we know that $m$ is a $(\btheta,n)$-survivor, we know that for each $i$, either
$m(\hom{\sigma_i})=0$ or 
$$\sgn\big(m(\hom{\sigma_i})\big)=\sgn\big(\hol(\sigma_i)\wedge \bm \theta\big).$$
Given this, it suffices to prove that 
$\sgn\big(\hol(\sigma_i)\wedge \bm \theta\big)=\sgn(\v \wedge \bm \theta)$ for all $i$,
because then we have
$$\sgn \big(m(\hom{\sigma})\big)=\sgn \sum_{i=1}^k m(\hom{\sigma_i})= 
\sgn(\v \wedge \bm \theta)$$
since each term in the sum has the sign the same as $\hol(\sigma_i)\wedge \bm \theta$
whenever it is non-zero, and since the total sum is non-zero from the assumption that $m(\hom{\sigma}) \neq 0$. To verify this sufficiency condition, recall that for $\nu=\sigma$ or $\nu=\sigma_i$ for some $i$, we have that
$\rhoa{g_n}(\hol~\nu)$ lies in the closed quadrant ${\overline \Q}_{s}$ with $s \not \in \{\pm s_n\}$, while by definition $\rhoa{g_n}(\btheta)$ lies in $\Q_{s_n}$. By invariance of the wedge product under orientation preserving linear maps,
$$\hol(\nu)\wedge \bm \theta= \rhoa{g_n}(\hol~\nu) \wedge \rhoa{g_n}(\btheta).$$
The sign of the right hand side is the same for all non-zero vectors such as $\rhoa{g_n}(\hol~\nu)$
taken from ${\overline \Q}_{s}$ wedged with all vectors such as $\rhoa{g_n}(\btheta)$ taken from $\Q_{s_n}$.
In particular, the sign of this wedge product does not change if we set $\nu=\sigma$ or set $\nu=\sigma_i$ for some $i$.
\end{proof}

\section{Survivors and Operators}
\label{sect:operators2}

The purpose of this section is to prove Theorem \ref{thm:surviving_functions}, which says that all $\bm \theta$-survivors in $\Coh$ arise as $\Xi(\f)$ where $\f \in \R^\V$ is a $\bm \theta$-survivor. In order to prove this statement, we will transform several results proved in section \ref{sect:quadrants}
about the interplay between the $\rhoa{G}$ action on $\R^2$ and quadrants in $\R^2$ into statements about the
$\Phi^{G}_\ast$ action on $\Coh$ and the $\Upsilon^G$ action on $\R^\V$. 

\subsection{Homology, Cohomology and \texorpdfstring{$\R^\V$}{R{\textasciicircum}V}}

We begin by proving Proposition \ref{prop:pullback_action}, which says that $\Phi^{g}_\ast \circ \Xi=\Xi \circ \Upsilon^{g}$.

\begin{proof}[Proof of Proposition \ref{prop:pullback_action}]
Let $\hom{x}\in H_1(S,V,\Z)$. By definition of $\Phi^{g}_\ast$ and of $\Xi$, 
$$\big(\Phi^{g}_\ast \circ \Xi(\f)\big)(\hom{x})=\Xi(\f)\big(\Phi^{g^{-1}}(\hom{x})\big)=\sum_{\vv \in \V} i\big(\Phi^{g^{-1}}(\hom{x}), \hom{\cyl_\vv}\big) \f(\vv).$$
By acting by $\Phi^{g}$ on each side of each expression for intersection number, we have
$$\big(\Phi^{g}_\ast \circ \Xi(\f)\big)(\hom{x})=\sum_{\vv \in \V} i\big(\hom{x}, \Phi^{g}(\hom{\cyl_\vv})\big) \f(\vv).$$
Consider the case of $g=h^k$ with $k \in \Z$. By Proposition \ref{prop:action}, we have
$$\Phi^{h^k}(\hom{\cyl_\vv})=\begin{cases}
\hom{\cyl_\vv}+k\sum_{\vw \sim \vv} \hom{\cyl_\vw} & \textrm{if $\vv \in \Alpha$}\\
\hom{\cyl_\vv} & \textrm{if $\vv \in \Beta$.}
\end{cases}
$$
Thus, we may write
$$\big(\Phi^{h^k}_\ast \circ \Xi(\f)\big)(\hom{x})=\sum_{\va \in \Alpha} i(\hom{x}, \hom{\cyl_\va}) \f(\va)+\sum_{\vb \in \Beta} i \big(\hom{x}, \hom{\cyl_\vb}+k \sum_{\va \sim \vb} \hom{\cyl_\va}\big)\f(\vb).$$
By regrouping terms, we see 
$$\big(\Phi^{h^k}_\ast \circ \Xi(\f)\big)(\hom{x})=\sum_{\va \in \Alpha} i(\hom{x}, \hom{\cyl_\va}) \big(\f(\va)+k\sum_{\vb\sim \va} \f(\vb)\big) +\sum_{\vb \in \Beta} i(\hom{x}, \hom{\cyl_\vb}) \f(\vb).$$
This last expression is equal to $\sum_{\vv \in \V} i\big(\hom{x}, \hom{\cyl_\vv}) \big(\Ho^k(\f)(\vv)\big)=\big(\Xi \circ \Ho^k(\f)\big)(\hom{x}).$ Thus, $\Xi \circ \Ho^k=\Phi^{h^k}_\ast \circ \Xi$
as desired. By a similar argument or by the action of the dihedral group, the same holds for $g=v^k$.
\end{proof}

Recall that $\R^\V_c$ represents the set of finitely supported functions from $\V \to \R$. We introduce a canonical linear map 
$\Zem:H_1(S,V,\R) \to \R^\V_c \label{not:Zem}$. Define
\begin{equation}
\label{eq:Zem}
\Zem(\hom{x})(\vv)=i(\hom{x}, \hom{\cyl_\vv}),
\end{equation}
where $\hom{\cyl_\vv} \in H_1(S \smallsetminus V, \Z)$ represents the homology class of the core curve of cylinder $\cyl_\vv$ for $\vv\in \V$.
Recalling the bilinear pairing $\langle, \rangle:\R^\V \times \R^\V_c \to \R$ given in equation \ref{eq:pairing}. By definition of $\Xi$ (see equation \ref{eq:Xi}), we have
\begin{equation}
\label{eq:two_embeddings}
\langle \f, \Zem(\hom{x}) \rangle=\Xi(\f)(\hom{x}).
\end{equation}

Note that $\Zem$ is not injective. A useful consequence of the construction is that 
$$\Zem(\hom{x})=\Zem(\hom{y}) \quad \textrm{implies} \quad \Xi(\f)(\hom{x})=\Xi(\f)(\hom{y}).$$

We collect the following corollary to Proposition \ref{prop:pullback_action}.

\begin{corollary}
\label{cor:pullback_action}
For all $\hom{x} \in H_1(S,V,\R)$, we have $\Upsilon^g \circ \Zem(\hom{x})=\Zem \circ \Phi^{\gamma(g)}(\hom{x})$, where $\gamma:G\to G$ is
the involutive homomorphism defined above equation \ref{eq:operator_relations2}.
\end{corollary}
\begin{proof}
Note that two elements $\x, \y \in \R^\V_c$ are equal if and only if $\langle \f, \x \rangle=\langle \f, \y \rangle$
for all $\f \in \R^\V$. Let $\f \in \R^\V$ be arbitrary. By equation \ref{eq:operator_relations2},
$$\langle \f, \Upsilon^g \circ \Zem(\hom{x}) \rangle=\langle \Upsilon^{\gamma(g)^{-1}} (\f), \Zem(\hom{x}) \rangle.$$
Then by equation \ref{eq:two_embeddings} and Proposition \ref{prop:pullback_action}, 
$$\langle \f, \Upsilon^g \circ \Zem(\hom{x}) \rangle=\Xi\big(\Upsilon^{\gamma(g)^{-1}}(\f)\big)(\hom{x})=
\big(\Phi^{\gamma(g^{-1})}_\ast \circ \Xi(\f)\big)(\hom{x}).$$
By the definition of $\Phi^{\gamma(g^{-1})}_\ast$ given in equation \ref{eq:pushforward}, we continue
$$\langle \f, \Upsilon^g \circ \Zem(\hom{x}) \rangle=\Xi(\f)\big(\Phi^{\gamma(g)}(\hom{x})\big)=\langle \f, \Zem \circ \Phi^{\gamma(g)}(\hom{x})\rangle.$$
Thus, $\Upsilon^g \circ \Zem=\Zem \circ \Phi^{\gamma(g)}$ as desired.
\end{proof}

Recall that $R$ is the linear map $\R^2 \to \R^2$ which rotates by $\frac{\pi}{2}$. Also,
$r$ is the induced permutation on the signs of quadrants in $\R^2$. See equations \ref{eq:R} and \ref{eq:r}.

\begin{proposition}
\label{prop:r}
Let $\sigma$ be a saddle connection and $s \in \SP$. Then, $\hol~\sigma \in \cl(\Q_{s})$ if and only if $\Zem(\hom{\sigma}) \in \hat \Q_{r(s)}$.
\end{proposition}
\begin{proof}
Choose any $s$ for which $\hol~\sigma \in \cl(\Q_{s})$. 
Write $s=(s_x, s_y)$ with $s_x,s_y \in \{1,-1\}$. Then $r(s)=(-s_y, s_x)$. 
Because $\hol~\sigma \in \cl(\Q_{s})$, we have $\sgn\big(\pi_x (\hol~\sigma)\big) \in \{s_x, 0\}$ and $\sgn\big(\pi_y (\hol~\sigma)\big)\in \{s_y,0\}$.
Let $\va \in \Alpha$. Then 
$$\sgn~\Zem(\hom{\sigma})(\va)=\sgn~i(\hom{\sigma}, \hom{\cyl_\va})=\sgn~(\hol~\sigma) \wedge (1,0)=
\sgn\big(-\pi_y(\hol~\sigma)\big) \in \{-s_y, 0\}.$$
Similarly if $\vb \in \Beta$, then 
$$\sgn~\Zem(\hom{\sigma})(\vb)=\sgn~i(\hom{\sigma}, \hom{\cyl_\vb})=\sgn~(\hol~\sigma) \wedge (0,1)=
\sgn\big(\pi_x(\hol~\sigma)\big)\in\{s_x,0\}.$$
So, $\Zem(\hom{\sigma}) \in \hat \Q_{r(s)}$. The converse follows by reversing this argument.
\end{proof}

\subsection{Quadrant tracking via expansion}
\label{sect:quadrant tracking}
Recall that conjugation by $R$ preserves the subgroup $\rhoa{G} \subset \SL(2, \R)$.
Equation \ref{eq:R_commute} explains exactly how conjugation by $R$ acts on $G$.
This action on $G$ also relates to the expanding sign action $\sa^G$. Namely,
\begin{equation}
\label{eq:r_commute}
\sa^g \circ r=r \circ \sa^{\gamma(g)}.
\end{equation}
We view the following as a corollary of Proposition \ref{prop:quadrants_and_expansion}.

\begin{corollary}
\label{cor:expansion_sign_action}
Let $\langle g_i \rangle$ be a geodesic ray and $s \in \SP$. Assume that $\sigma$ is a saddle connection
for which $\hol~\sigma \in \cl\Big(\Exp_\lambda\big(\gamma(g_1)\big) \cap \Q_{s}\Big)$. Then, 
$\Upsilon^{g_i} \circ \Zem(\hom{\sigma}) \in \hat \Q_{r(s_i)}$ where $s_i=\sa^{g_i}(s)$.
\end{corollary}
\begin{proof}
By Proposition \ref{prop:quadrants_and_expansion}, we know 
$\rhoa{\gamma(g_i)}(\hol~\sigma) \in \cl(\Q_{s_i})$ for all $i$. 
We also know that 
$\hol \circ \Phi^{\gamma(g_i)}(\hom{\sigma})=\rhoa{\gamma(g_i)}(\hol~\sigma)$ is the holonomy of
a saddle connection. By Proposition \ref{prop:r}, 
$\Zem \circ \Phi^{\gamma(g_i)}(\hom{\sigma}) \in \hat \Q_{r(s_i)}$ .
By Corollary \ref{cor:pullback_action}, we see $\Upsilon^{g_i} \circ \Zem(\hom{\sigma})=\Zem \circ \Phi^{\gamma(g_i)}(\hom{\sigma}) \in \hat \Q_{r(s_i)}$.
\end{proof}

\begin{proposition}[Quadrant tracking]
\label{prop:quadrant_tracking}
Suppose $\langle g_i \rangle$ is a geodesic ray and $\rhoa{g_1}(\Q_s) \subset \Q_s$. Then
$\Upsilon^{g_i}(\hat \Q_s) \subset \hat \Q_{s_i}$ where $s_i=\sa^{g_i}(s)$. 
\end{proposition}
\begin{proof}
Let $\f \in \hat \Q_s$. Given $\vv \in \V$, let $\e_\vv \in \R^\V$ denote the function
\begin{equation}
\label{eq:e}
\e_\vv(\vw)=\begin{cases} 1 & \textrm{if $\vv=\vw$} \\ 0 & \textrm{if $\vv \neq \vw$.}\end{cases}
\end{equation}
Formally, we may write
$$\f=\sum_{\vv \in \V,~\f(\vv) \neq 0} \f(\vv) \e_\vv.$$
This sum makes sense, because for any $\vv \in \V$, there are only finitely many terms whose support includes $\vv$. 
For each $\vv \in \V$ with $\f(\vv) \neq 0$ we can choose a horizontal or vertical saddle connection $\sigma_\vv$ such that for any $\vw \in \V$ we have
$$i(\hom{\sigma_\vv}, \hom{\cyl_\vw})=\begin{cases}
\sgn~\f(\vv) & \textrm{if $\vv=\vw$} \\
0 & \textrm{otherwise.}
\end{cases}$$
Then by definition of $\Zem$ we have
\begin{equation}
\label{eq:f}
\f=\sum_{\vv \in \V,~\f(\vv) \neq 0} |\f(\vv)| \Zem(\hom{\sigma_\vv}).
\end{equation}
Note that each $\Zem(\hom{\sigma_\vv})\in \hat \Q_s$ and therefore by 
Proposition \ref{prop:r} we know $\hol~\sigma_\vv \in \cl(\Q_{r^{-1}(s)})$. 
We now apply $\Upsilon^{g_i}$ to equation \ref{eq:f}, yielding
\begin{equation}
\label{eq:f2}
\Upsilon^{g_i}(\f)=
\sum_{\vv \in \V,~\f(\vv) \neq 0} |\f(\vv)| \Upsilon^{g_i} \circ \Zem(\hom{\sigma_\vv})=
\sum_{\vv \in \V,~\f(\vv) \neq 0} |\f(\vv)|  \Zem \circ \Phi^{\gamma(g_i)} (\hom{\sigma_\vv}),
\end{equation}
with the last equality following from Corollary \ref{cor:pullback_action}. Now consider that if
$\rhoa{g_1}(\Q_s) \subset \Q_s$, then $\rhoa{\gamma(g_1)}\big(\cl(\Q_{r^{-1}(s)})\big) \subset \cl(\Q_{r^{-1}(s)})$. (This follows by inspecting the
pairs $(g_1, s)$ allowed by Corollary \ref{cor:quadrants}.) By Corollary \ref{cor:quadrants}, we may conclude that
$$\hol~\Phi^{\gamma(g_i)}(\hom{\sigma_\vv}) \in \cl(\Q_{s'_i}),$$
where $s'_i=\sa^{\gamma(g_i)} \circ r^{-1}(s)$. By equation \ref{eq:r_commute}, we have $s'_i=r^{-1} \circ \sa^{g_i}(s)$.
Then it follows from Proposition \ref{prop:r} that
$$\Zem \circ \Phi^{\gamma(g_i)} (\hom{\sigma_\vv}) \in \hat \Q_{r(s'_i)}=\hat \Q_{s_i},$$
where $s_i=r(s'_i)=\sa^{g_i}(s)$ as in the statement of the proposition. As each of the terms in the
sum for $\Upsilon^{g_i}(\f)$ in equation \ref{eq:f2} lie in $\Q_{s_i}$, we know $\Upsilon^{g_i}(\f) \in \Q_{s_i}$.
\end{proof}


Let $\langle g_i \rangle$ be the shrinking sequence for a $\lambda$-renormalizable direction $\bm \theta \in S^1$, and $\langle s_i \rangle$ denote the sign sequence.
Recall definition \ref{def:critical_time} that $n$ is a critical time if $s_{n-1}=s_{n}$. We have the following corollary to Proposition \ref{prop:quadrant_tracking}.

\begin{corollary}[Critical times]
\label{cor:survivor_implication}
Suppose $n$ is a critical time. Then if $\f \in \R^\V$ is a $(\bm \theta, n)$-survivor, then it is a
$(\bm \theta, k)$-survivor for all $k<n$.
\end{corollary}

\begin{proof}
Let $\f_k=\Upsilon^{g_k}(\f)$ and $\bm \theta_k=\rhoa{g_k}(\bm \theta)$.  We know that  
$\bm \theta_n \in \Q_{s_n}$ and
$\f_n \in \hat \Q_{s_n}$. By statement (\ref{crit2}) of Proposition \ref{prop:crit_equiv}, we know $\rhoa{g_{n-1} g_n^{-1}}(\Q_{s_n}) \subset \Q_{s_n}=\Q_{s_{n-1}}$.
Consider the geodesic segment $\langle h_i=g_{n-i} g_n^{-1}\rangle$.  By definition of the sign sequence $\langle s_i \rangle$, we have 
$\theta_{n-i} \in \Q_{s_{n-i}}$. By Proposition \ref{prop:quadrants_and_expansion}, we know 
$s_{n-i}=\sa^{h_i}(s_n)$. Finally, by Proposition \ref{prop:quadrant_tracking},
we know $\f_{n-i}=\Upsilon^{h_i}(\hat \Q_s) \subset \hat \Q_{s_{n-i}}$ too. Therefore, by definition, $\f_{n-i}$ is a $(\bm \theta, n-i)$-survivor.
\end{proof}

Recall that Lemma \ref{lem:critical_times} claimed that for $\btheta$ a $\lambda$-renormalizable direction,
there are infinitely many $n \in \N$ such that for $\f \in \R^\V$, being a $(\bm \theta, n)$-survivor implies 
being at  $(\bm \theta, k)$-survivor for all $k \leq n$. So, this was really a consequence of the above Corollary.

\begin{proof}[Proof of Lemma \ref{lem:critical_times}]
Let $n$ be a critical time for the shrinking sequence of $\btheta$. There are infinitely many of these times by Corollary \ref{cor:critical_times_occur}. Also if $\f$ is a $(\bm \theta, n)$-survivor then it is a $(\bm \theta, k)$-survivor for all $k < n$ by Corollary \ref{cor:survivor_implication}.
\end{proof}

\begin{corollary}
\label{cor:eventually}
Suppose $\bm \theta$ is a $\lambda$-renormalizable direction. If for some $N\geq 0$ we know that $\f \in \R^\V$ is a $(\bm \theta, k)$-survivor for all
$k \geq N$, then $\f$ is a $\bm \theta$-survivor.
\end{corollary}
\begin{proof}
Corollary \ref{cor:critical_times_occur} guarantees that we have infinitely many critical times $n$. Choose such an $n>N$. Corollary \ref{cor:survivor_implication} implies that
$\f$ is a $(\bm \theta, k)$-survivor for all $k<n$.
\end{proof}

\begin{corollary}[Group invariance]
\label{cor:group_invariance}
Assume that $\f$ is a $\bm \theta$-survivor, for $\bm \theta$ a $\lambda$-renormalizable direction. Then for all $g \in G$,
$\Upsilon^g(\f)$ is a $\rhoa{g}(\bm \theta)$-survivor. 
\end{corollary}
\begin{proof}
Let $\f'=\Upsilon^g(\f)$. Let $\langle g_i \rangle$ be the $\lambda$-shrinking sequence for $\bm \theta$, and
$\langle g_i' \rangle$ be the $\lambda$-shrinking sequence for $\bm \theta'=\pi_\Circ \circ \rhoa{g}(\bm \theta)$.
Then the $\lambda$-shrinking sequences of $\bm \theta$ and $\bm \theta'$ are tail equivalent, as defined in section \ref{sect:symmetry}.
Moreover, regardless of the choice of $g \in G$ there are positive integers $m$ and $n$ such that 
$$\Upsilon^{g_m}(\f)=\Upsilon^{g'_n}(\f') \quad \textrm{and} \quad \pi_\Circ \circ \rhoa{g_m}(\bm \theta)=\pi_\Circ \circ \rhoa{g'_n}(\bm \theta').$$
Then this is also true for $m$ replaced by $m+k$ and $n$ replaced by $n+k$ for all $k \geq 0$. Thus,
Corollary \ref{cor:eventually} implies that $\f'$ is a $\bm \theta'$-survivor.
\end{proof}

\subsection{Survivors in \texorpdfstring{$\Coh$}{H{\textasciicircum}1} and \texorpdfstring{$\R^\V$}{R{\textasciicircum}V}}
\label{sect:survivors}
In this section we will provide a proof of Theorem \ref{thm:surviving_functions}, i.e., in the presence of the subsequence decay property, we have that
$\Psi_\btheta(\M_\btheta) \subset \Xi(\R^\V)$.
The main idea of the proof is to understand the function $\Fmap:\Coh \to \R^\V$ given by
$$\Fmap(m)(\vv)=m(\hom{\cyl_\vv}),$$
where $\hom{\cyl_\vv}$ represents the homology class of the cylinder associated to $\vv \in \V$ oriented rightward or upward. 
Recall that $r:\SP \to \SP$ is the action on signs of quadrants induced by a rotation by $\frac{\pi}{2}$ radians.
See equation \ref{eq:r}.
We have the following.

\begin{proposition}
\label{prop:Fmap_conj}
$\Fmap \circ \Phi^{g}_{\ast}=\Upsilon^{\gamma(g)} \circ \Fmap$.
\end{proposition}
\begin{proof}
It is enough to handle the cases of $g \in \{h^{k}, v^{k}~:~k \in \Z\}$. Let $g=h^{k}$. Suppose $\va \in \Alpha$. 
Then by Proposition \ref{prop:action},
$$\big(\Fmap \circ \Phi^{h^{k}}_{\ast}(m)\big)(\va)=m \big(\Phi^{h^{-k}}(\hom{\cyl_\va})\big)=m(\hom{\cyl_\va}).$$
By definition of $\gamma$, we have $\gamma(h^{k})=v^{-k}$. By equation \ref{eq:Vo}, 
$$\big(\Upsilon^{v^{-k}} \circ \Fmap(m)\big)(\va)=\Fmap(m)(\va)=m(\hom{\cyl_\va}).$$
Now let $\vb \in \Beta$. We have
$$\big(\Fmap \circ \Phi^{h^{k}}_{\ast}(m)\big)(\vb)=m \big(\Phi^{h^{-k}}(\hom{\cyl_\vb})\big)=
m(\hom{\cyl_\vb}-k\sum_{\va \sim \vb} \hom{\cyl_\va}).$$
$$\big(\Upsilon^{v^{-k}} \circ \Fmap(m)\big)(\vb)=\Fmap(m)(\vb) -k \sum_{\va \sim \vb} \Fmap(m)(\va)=
m(\hom{\cyl_\vb}) -k \sum_{\va \sim \vb} m(\hom{\cyl_\va}).$$
These two expressions are equal by linearity of $m$. The proof is similar for the case of $g=v^{k}$.
\end{proof}

Recall that $R$ acts on $\R^2$ by rotation by $\frac{\pi}{2}$. See equation \ref{eq:R}.

\begin{proposition}
\label{prop:circumference_survivors}
If $m \in \Coh$ is a $(\bm \theta, n)$-survivor, then $\Fmap(m)\in \R^\V$ is a $\big(R^{-1}(\bm \theta),n\big)$-survivor.
\end{proposition}

\begin{proof}
To ease notation define $\bm \theta'=R^{-1}(\bm \theta)$. 
Let $\langle g_n \rangle$ denote the shrinking sequence for $\bm \theta$. Let
$g'_n=\gamma(g_n)$. 
By equation \ref{eq:R_commute}, $\langle g'_n \rangle$ is the shrinking sequence for $\bm \theta'$. 
Therefore, if $\langle s_n \rangle$ is the sign sequence of $\bm \theta$, then the sign sequence
for $\bm \theta'$ is $\langle s'_n=r^{-1}(s_n)\rangle$. By Proposition \ref{prop:equiv_survivors},
we must show that $\Upsilon^{g'_n}  \circ \Fmap(m) \in \hat \Q_{s'_n}$.  By Proposition \ref{prop:Fmap_conj},
$\Upsilon^{g'_n}  \circ \Fmap(m)=\Fmap \big(\Phi^{g_n}_\ast(m)\big)$. For $\vv \in \V$, we have
$$\Fmap \big(\Phi^{g_n}_\ast(m)\big)(\vv)=\big(\Phi^{g_n}_\ast(m)\big)(\hom{\cyl_\vv}).$$
Define $x,y \in \R$ so that $(x,y)=\bm \theta$.
We may apply Proposition \ref{prop:equiv_def} because because $m$ is a $(\bm \theta, n)$-survivor
and because $\hom{\cyl_\vv}$ can be written as a sum of homology classes of saddle connections parallel to $\hol(\hom{\cyl_\vv})$. 
Therefore, for all $\va \in \Alpha$, unless $\Fmap \big(\Phi^{g_n}_\ast(m)\big)(\va)=0$,
$$\sgn~\Fmap \big(\Phi^{g_n}_\ast(m)\big)(\va)=
\sgn~\hol(\hom{\cyl_\va})\wedge \rhoa{g_n}(\bm \theta)=
\sgn~(1,0) \wedge \rhoa{g_n}(\bm \theta)=\sgn~y.$$
Similarly, for $\vb \in \Beta$, unless $\Fmap \big(\Phi^{g_n}_\ast(m)\big)(\vb)=0$, we have
$$\sgn~\Fmap \big(\Phi^{g_n}_\ast(m)\big)(\vb)=
\sgn~\hol(\hom{\cyl_\vb})\wedge \rhoa{g_n}(\bm \theta)=
\sgn~(0,1) \wedge \rhoa{g_n}(\bm \theta)=-\sgn~x.$$
Noting that $(x,y) \in \Q_{s_n}$, we see $(y, -x) \in \Q_{r^{-1}(s_n)}$. 
Thus, $\Fmap \big(\Phi^{g_n}_\ast(m)\big)(\vv) \in \hat \Q_{s'_n}$ as desired.
\end{proof}

\begin{proof}[Proof of Theorem \ref{thm:surviving_functions}]
Suppose that $m \in \Coh$ is a $\bm \theta$-survivor, and
suppose that there is no $\f \in \R^\V$ for which $m=\Xi(\f)$. Then, 
there are horizontal or vertical saddle connections $\sigma, \sigma' \in E$ both of which cross $\cyl_\vv$ for some $\vv \in V$ and 
for which $m(\hom{\sigma}) \neq m(\hom{\sigma'})$. Note that for such a pair of saddle connections, we have
$$\Phi^{g}(\hom{\sigma}-\hom{\sigma'})=\hom{\sigma}-\hom{\sigma'}$$
for all $g \in G$, because the difference $\hom{\sigma}-\hom{\sigma'}$ has zero algebraic intersection number
with each horizontal and vertical cylinder. Consider the shrinking sequence $\langle g_n \rangle$ of $\bm \theta$. 
We have 
\begin{equation}
\label{eq:difference_constant}
\big(\Phi^{g_n}_\ast(m)\big)(\hom{\sigma}-\hom{\sigma'})=
m\big(\Phi^{g_n^{-1}}(\hom{\sigma}-\hom{\sigma'})\big)=
m(\hom{\sigma}-\hom{\sigma'}).
\end{equation}
So, this quantity remains constant.

On the other hand, let $\g=\Fmap(m)$. By 
Proposition \ref{prop:circumference_survivors}, $\g$ is a $\bm \theta'=R(\bm \theta)$-survivor. The $\lambda$-shrinking sequence of $\bm \theta'$ is $\langle g'_n=\gamma(g_n)\rangle$ by equation \ref{eq:R_commute}. Now we utilize the subsequence decay property.
By this property, there is a subsequence $\langle g_{n_i}' \rangle$ such that
for all $\vv \in \V$ we have
$$\lim_{i \to \infty} \Upsilon^{g'_{n_i}} (\g)(\vv) = 0.$$
The quantity $\Act(g'_{n_i})(\g)(\vv)$ has meaning to us because
$$\Upsilon^{g'_{n_i}}(\g)(\vv)=\big(\Upsilon^{g'_{n_i}} \circ \Fmap(m)\big)(\vv)=
\big(\Fmap \circ \Phi^{g_{n_i}}_\ast(m)\big)(\vv)=\Phi^{g_{n_i}}_\ast(m)(\hom{\cyl_\vv}).$$
Because $m$ is a $(\bm \theta, n)$-survivor for all $n$, we know that any saddle connection $\nu \in E$ satisfies either
$\Phi^{g_{n}}_\ast(m)(\hom{\nu})=0$ or
\begin{equation}
\label{eq:saddle_sign}
\sgn~\Phi^{g_{n}}_\ast(m)(\hom{\nu})=\sgn~\big(\hol(\hom{\nu}) \wedge \rhoa{g_n}(\bm \theta)\big).
\end{equation}
In particular, choose and orient a boundary component of $\cyl_\vv$ for some $\vv \in \V$. This component is made up of some number of oriented horizontal or vertical saddle connections
$\nu_1, \ldots, \nu_k$. Equation \ref{eq:saddle_sign} implies that all nonzero
$\Phi^{g_{n}}_\ast(m)(\hom{\nu_j})$ have the same sign regardless of the choice of $j=1, \ldots, k$. Since $\hom{\nu_1}+\ldots+\hom{\nu_k}=\hom{\cyl_\vv}$ in $H_1(S,V,\Z)$, it follows that 
$$\lim_{i \to \infty} \Phi^{g_{n_i}}_\ast(m)(\hom{\nu_j}) (\g)(\vv) = 0 \quad \textrm{for all $j=1, \ldots k$}.$$
Now note that every horizontal and every vertical saddle connection lies in the boundary of some cylinder $\cyl_\vv$ for $\vv \in \V$. In particular, this holds for $\sigma$ and $\sigma'$. We conclude that
$$\lim_{i \to \infty} \Phi^{g_{n_i}}_\ast(m)(\hom{\sigma}-\hom{\sigma'})=0.$$
This contradicts our assumption that $m(\hom{\sigma}-\hom{\sigma'}) \neq 0$  via equation \ref{eq:difference_constant}.
\end{proof}

\subsection{Surjectivity}
\label{sect:surjectivity}
In this section, we prove Lemma \ref{lem:surjectivity}, which states that, under our hypotheses, if $\btheta$ is a $\lambda$-renormalizable direction, then $\A$ surjectively sends $\Surv_\btheta$ to $\Surv_{\bar \btheta}$. 
The proof follows without much difficulty from Proposition \ref{prop:weak_surj} of the outline. However, the ideas
used in the proof are most similar to the ideas appearing in Appendix \ref{sect:farkas}.

\begin{proof}[Proof of Lemma \ref{lem:surjectivity}]
We assume $S(\G,\w)$ has the critical decay property and the $\G$ has the adjacency sign property.
We let $\btheta \in \R_\lambda$ and $\f \in \Surv_{\bar \btheta}$. We will show that there is a $\g \in \Surv_{\bm \theta}$ with $\A(\g)=\f$.

Let $\langle g_n \rangle$ and $\langle s_n \rangle$ denote the shrinking and sign sequences of $\bm \theta$. Then $\langle \bar g_n \rangle$ and $\langle \bar s_n \rangle$ are the corresponding sequences of $\bar {\bm \theta}$. See the beginning of section \ref{sect:outline_adjacency}. To simplify notation define $\f_n=\Upsilon^{\bar g_n}(\f)$. 

For each critical time $n_i$, define the set
${\mathbf H}_{n_i}=\{\h \in \hat \Q_{s_{n_i}}~:~\A(\h)=\f_{n_i}\}$.  Applying Proposition \ref{prop:weak_surj} with $\f$ replaced by $\f_{n_i}$ and $\bm \theta$ replaced by $\pi_\Circ \circ \rhoa{g_{n_i}}(\bm \theta)$ implies that ${\mathbf H}_{n_i}$ is non-empty. For $\h \in {\mathbf H}_{n_i}$, we have
$$\A \circ \Upsilon^{g_{n_i}^{-1}}(\h)=\Upsilon^{\bar g_{n_i}^{-1}} \circ \A (\h)= \Upsilon^{\bar g_{n_i}^{-1}}(\f_{n_i})=\f.$$
Therefore any $\g_0 \in \Upsilon^{g_{n_i}^{-1}}({\mathbf H}_{n_i})$ represents a candidate $\g$. Set 
${\mathbf G}_{n_i}=\Upsilon^{g_t^{-1}}({\mathbf H}_{n_i})$. By definition, we have
\begin{equation}
\label{eq:closed}
{\mathbf G}_{n_i}=\{\g_0 \in \R^\V ~:~ \textrm{$\A(\g_0)=\f$ and $\g_0 \in \Upsilon^{g_{n_i}}(\hat \Q_{s_{n_i}})$} \}.
\end{equation}
Note that Lemma \ref{lem:critical_times} implies that
\begin{equation}
\label{eq:Gni}
{\mathbf G}_{n_i}=\{\g_0 ~:~ \textrm{$\A(\g_0)=\f$ and $\g_0$ is a $(\bm \theta, n)$-survivor for $0 \leq n \leq n_i$}\}.
\end{equation}
Thus, we have ${\mathbf G}_{n_{i+1}} \subset {\mathbf G}_{n_i}$ for all $i$. We conclude that
$$\{\g \in \Surv_{\bm \theta}~:~ \A(\g)=\f\}=\bigcap_i {\mathbf G}_{n_i},$$
where the intersection is taken over the critical times. We will use the notation ${\mathbf G}_{\infty}=\bigcap_i {\mathbf G}_{n_i}$.
We must show that ${\mathbf G}_{\infty}$ is non-empty. 

Note that each ${\mathbf G}_{n_i}$ is non-empty, because each ${\mathbf H}_{n_i}$ is. Furthermore, we note that each ${\mathbf G}_{n_i}$ is weakly-closed. 
This can be most easily seen by looking at equation \ref{eq:closed}. It follows because both $\A$ and $\Upsilon^{g_{n_i}}$ are weakly-continuous and $\hat \Q_{s_{n_i}}$ is weakly-closed.

For each critical time $n_i$ choose a $\g_{n_i} \in {\mathbf G}_{n_i}$. We will find a weak limit point, $\g$, of the a sequence of $\langle \g_{n_i}~|~\textrm{$n_i$ a critical time} \rangle$. As each $\g_{n_i}$ is a $\bm (\theta,0)$-survivor, we know that $\g_{n_i} \in \hat \Q_{s_0}$. Without loss of generality, we may assume that $s_0=++$. Therefore, $\g_{n_i},\f \in \hat \Q_{++}$. For each $\vv \in \V$, there is a $\vw \sim \vv$ and we have
$$0 \leq \g_{n_i}(\vv) \leq \A(\g_{n_i})(\vw)=\f(\vw).$$
In particular the sequence of $\g_{n_i}(\vv)$ is bounded for each $\vv \in \V$. Therefore, we can apply a Cantor diagonal argument as in the proof of Proposition \ref{prop:weakly-closed}. This produces our desired limit $\g$. 
Since each ${\mathbf G}_{n_i}$ is weakly-closed, and this sequence of sets is nested, we see that $\g \in {\mathbf G}_{\infty}$. Such a $\g$ lies in ${\mathbf G}_{n_i}$ for every $i$, and in view of equation \ref{eq:Gni},
we see that $A(\g)=\f$ and $\g$ is a $(\bm \theta, n)$-survivor for every $n \geq 0$. That is, $\g \in \Surv_{\bm \theta}$.
\end{proof}

\subsection{Survivors and the pairing of \texorpdfstring{$\R^\V$}{R{\textasciicircum}V} with \texorpdfstring{$\R^\V_c$}{R{\textasciicircum}V{\textunderscore}c}}
Ideas developed in this section will be useful in sections \ref{sect:subgraphs_and_covers} and \ref{sect:no valance one}, which address the critical decay property and the adjacency sign condition.
We begin this section by summarizing results from sections \ref{sect:quadrant tracking} and \ref{sect:survivors}.
Then we will state some consequences of our previous work.

Let $\bm \theta$ be a $\lambda$-renormalizable direction. Let $\langle g_i \rangle$ and $\langle s_i \rangle$ be the shrinking and sign sequences
of $\bm \theta$. Recall $\Surv_{\bm \theta}$ denotes the set of $\bm \theta$-survivors in $\R^\V$. By definition \ref{def:function_survivors} this is 
$$\Surv_{\bm \theta}=\{\f \in \R^\V~:~\textrm{$\Upsilon^{g_n}(\f) \in \hat \Q_{s_n}$ for all $n \geq 0$}\}.$$
Alternately, by Theorems \ref{thm:operator_theorem} and  \ref{thm:surviving_functions}, $\Xi(\Surv_{\bm \theta})$ is the set of all cohomology classes arising from transverse invariant measures to the foliation in the direction $\bm \theta$. Then Lemma \ref{lem:sign} implies that 
\begin{equation}
\label{eq:surv_equivalent}
\Surv_{\bm \theta}=\Big\{\f \in \R^\V~:~
\begin{array}{l}
\textrm{$\sgn \langle \f, \Zem(\hom \sigma)\rangle=\sgn (\hol~\sigma \wedge \bm \theta)$ for all saddle} \\
\textrm{connections $\sigma$ with $\langle \f, \Zem(\hom \sigma)\rangle \neq 0$}\end{array}\Big\}.
\end{equation}
Here, the quantity $\langle \f, \Zem(\hom \sigma)\rangle$ represents the value assigned to $\hom{\sigma}$ by the associated cohomology class
(which is potentially induced by an invariant measure). See equation \ref{eq:two_embeddings}.

The later sections of this paper are primarily interested in questions of the following sort. Given $\x \in \R^\V_c$ with some properties
and any $\f \in \Surv_{\bm \theta}$, is it true that $\langle \f, \x \rangle \geq 0$? We pursue this sort of question from the following point of view.
One way to show that $\langle \f, \x\rangle \geq 0$ is to find saddle connections $\sigma_j$ so that 
$\x=\sum_{j} \Zem(\hom{\sigma_j})$ and $\hol~\sigma_j \wedge \bm \theta > 0$. If we can find this, then $\langle \f, \x\rangle \geq 0$
follows from equation \ref{eq:surv_equivalent}. 

There is a second, more operator theoretic point of view as well.
\begin{proposition}[Positivity checks]
\label{prop:positivity_checks}
Let $\bm \theta$ be a $\lambda$-renormalizable direction. Let $\langle g_i \rangle$ and $\langle s_i \rangle$ be the shrinking and sign sequences
of $\bm \theta$. Let $\x \in \R^\V_c$. Either of the following two statements imply that for all $\f \in \Surv_{\bm \theta}$, we have $\langle \f, \x \rangle \geq 0$. 
\begin{enumerate}
\item There exist a finite collection oriented saddle connections $\{\sigma_j\}$ so that $\hol~\sigma_j \wedge \bm \theta > 0$ for each $j$
and $\x=\sum_{j} \Zem(\hom{\sigma_j})$.
\item There is a critical time $t$ so that $\Upsilon^{\gamma(g_t)}(\x) \in \hat \Q_{s_t}$. 
\end{enumerate}
\end{proposition}

\begin{remark}
In fact, it can be observed that these two statements are equivalent. That (1) implies (2) follows essentially from Theorem \ref{thm:dot_product}. Conversely, if we have (2), then $\Upsilon^{\gamma(g_t)}(\x)$ can be expressed as $\Zem(\hom{\tau_1}+\ldots+\hom{\tau_N})$ where each $\tau_j$ is an appropriately oriented horizontal or vertical saddle connection, so that $\hol~\tau_j \in \cl(\Q_{r^{-1}(s_t)})$.
Then 
$$\x=\Upsilon^{\gamma(g_t)^{-1}} \circ \Zem(\hom{\tau_1}+\ldots+\hom{\tau_N})=
\Zem \circ \Phi^{g_t^{-1}}(\hom{\tau_1}+\ldots+\hom{\tau_N})$$
and we may take $\sigma_j=\Phi^{g_t^{-1}}(\tau_j)$.
\end{remark}

\begin{proof}[Proof of Proposition \ref{prop:positivity_checks}]
For statement (1), this follows from the paragraph preceding the proposition. In the case of statement (2), we have $\langle \f, \x \rangle=\langle \Upsilon^{g_t}(\f), \Upsilon^{\gamma(g_t)}(\x)\rangle$. Since
both $\Upsilon^{g_t}(\f)$ and $\Upsilon^{\gamma(g_t)}(\x)$ lie in $\hat \Q_{s_t}$, their product is non-negative.
\end{proof}

The following strengthens the utility of statement (2) of Proposition \ref{prop:positivity_checks}.

\begin{proposition}
\label{prop:gamma_critical}
Let $\bm \theta \in \Rn_\lambda$ and let $\f$ be a $\bm \theta$-survivor. Let $\langle g_i \rangle$ and $\langle s_i \rangle$ be the shrinking and sign sequences
of $\bm \theta$. If $\x \in \R^\V_c \cap \hat \Q_{s_0}$, then for any critical time $t$,
$\Upsilon^{\gamma(g_t)}(\x) \in \hat \Q_{s_t}$.
\end{proposition}
\begin{proof}
We can write $\x=\sum_{j=1}^{N} \Zem(\hom{\sigma_j})$ where each $\sigma_j$ is a horizontal or vertical saddle connection 
oriented so that $\hol~\sigma_j \in \cl(\Q_{r^{-1}(s_0)})$. See Proposition \ref{prop:r}. Then,
\begin{equation}
\label{eq:sum9}
\Upsilon^{\gamma(g_t)}(\x)=\sum_{j=1}^{N} \Zem \circ \Phi^{g_t} (\hom{\sigma_j}),
\end{equation}
by Corollary \ref{cor:pullback_action}. Looking at the action on holonomy, we have
$$\hol \circ \Phi^{g_t} (\hom{\sigma_j})=\rhoa{g_t} (\hol~\sigma_j).$$
By Proposition \ref{prop:shrink_exp}, since $\rhoa{g_1}$ shrinks the vector $\bm \theta \in \Q_{s_0}$, we know $\rhoa{g_1}$ expands every vector in $\Q_{r^{-1}(s_0)}$. Therefore,
the quadrant containing $\rhoa{g_t} (\hol~\sigma_j)$ is governed by the expanding sign action. Namely by Proposition \ref{prop:quadrants_and_expansion}, 
we must have
$\rhoa{g_i} (\hol~\sigma_j)\in \cl(\Q_{s'_i})$ where $s'_i=\Sigma^{g_i}\circ r^{-1}(s_0)$ for all $i$.
Therefore, by Proposition \ref{prop:r} we know each term from equation \ref{eq:sum9}
satisfies 
$$\Zem \circ \Phi^{g_t} (\hom{\sigma_j}) \in \hat \Q_{r(s'_t)}.$$
By equation \ref{eq:r_commute}, we know $s'_i=r^{-1} \circ \Sigma^{\gamma(g_i)}(s_0)$,
and so $r(s'_t)=\Sigma^{\gamma(g_t)}(s_0)$. Therefore $r(s'_t)=s_t$ by Corollary \ref{cor:critical_time_formula}. It follows from equation \ref{eq:sum9} that $\Upsilon^{\gamma(g_t)}(\x) \in \hat \Q_{s_t}$ as desired.
\end{proof}


\section{Subgraphs and covers}
\label{sect:subgraphs_and_covers}
In this section, we will consider how the group of operators $\Upsilon^G$ is affected by restricting to subgraphs and lifting to finite covers. This will be relevant to proofs 
in section \ref{sect:no valance one}, where we prove the adjacency sign property and the decay properties hold when $\G$ has no vertices of valance one. 

\subsection{Subgraphs}
Throughout this subsection, $\GH \subset \G$ will be a connected infinite subgraph. 
We will use subscripts to distinguish quantities related to the two graphs. For instance, we write $\V_\GH \subset \V_\G$ 
to indicate that the vertex set of $\GH$ is a subset of the vertex set of $\G$. The subgraph $\GH$ inherits a bipartite structure from $\G$. Namely, $\Alpha_\GH=\Alpha_\G \cap \V_\GH$ and  $\Beta_\GH=\Beta_\G \cap \V_\GH$.
We use $\Upsilon^G_\G:\R^{\V_\G} \to \R^{\V_\G}$
and $\Upsilon^G_\GH:\R^{\V_\GH} \to \R^{\V_\GH}$ to denote the appropriate groups of operators on the two spaces. We identify $\R^{\V_\GH}$ with 
the subset of $\R^{\V_\G}$ whose support is contained in $\V_\GH$. 

\begin{proposition}
\label{prop:difference_support}
Let $\f \in \R^{\V_\GH}$. Suppose that $\f \in \hat \Q_s$ for some $s \in \SP$. Write $s=(a, b)$ with $a, b \in \{\pm 1\}$. 
Then for all integers $k \neq 0$, the following statements hold.
\begin{enumerate}
\item $\Ho^k_\G(\f)-\Ho^k_\GH(\f)$ is supported on a subset of $\Alpha_\G$ and the sign of all non-zero values of this function is $b \cdot \sgn(k)$. 
\item $\Vo^k_\G(\f)-\Vo^k_\GH(\f)$ is supported on a subset of $\Beta_\G$ and the sign of all non-zero values of this function is $a \cdot \sgn(k)$. 
\end{enumerate}
\end{proposition}
\begin{proof}
We will prove statement (1). Statement (2) has a similar proof. By considering the definition of $\Ho$ given in equation \ref{eq:Ho}, for all $\vv \in \V_\G$ we have
$$\big(\Ho^k_\G(\f)-\Ho^k_\GH(\f)\big)(\vv)=\begin{cases}
\displaystyle \sum_{\vb \sim_\G \vv, \vb {\not \sim}_\GH \vv} k \f(\vb) & \textrm{if $\vv \in \Alpha_\G$} \\
0 & \textrm{if $\vv \in \Beta_\G$,}
\end{cases}$$
where the sum should be interpreted as over all edges leaving $\vv$ that appear in $\G$ but not in $\GH$. The conclusion follows.
\end{proof}

\begin{corollary}
\label{cor:diff_tracking}
Let $\langle g_i\rangle$ be a geodesic ray in $G$. Suppose $\x \in \R^{\V_\GH} \cap \hat \Q_{s_0}$ for some $s_0 \in \SP$, and define $s_i=\Sigma^{g_i}(s_0)$. Then for all $i \geq 1$,
$$\Upsilon_\G^{g_i g_1^{-1}}\big(\Upsilon_\G^{g_1}(\x)-\Upsilon_\GH^{g_1}(\x)\big) \in \hat \Q_{s_i}.$$
\end{corollary}
\begin{proof}
By Proposition \ref{prop:difference_support} and comparison to Definition \ref{def:sign_action} of the expanding sign action, 
we see that $\Upsilon_\G^{g_1}(\x)-\Upsilon_\GH^{g_1}(\x) \in \hat \Q_{s_1}$. Moreover, because of the description of the support as a subset of either $\Alpha$ or $\Beta$,
we have 
$$\Upsilon_\G^{g_1}\big(\Upsilon_\G^{g_1}(\x)-\Upsilon_\GH^{g_1}(\x)\big)=\Upsilon_\G^{g_1}(\x)-\Upsilon_\GH^{g_1}(\x).$$ 
Since $s_1=\Sigma^{g_1}(s_0)$, we know $s_1=\Sigma^{g_1}(s_1)$. Therefore we have $\rho_\lambda^{g_1}(\Q_{s_1}) \subset \Q_{s_1}$ and thus
$$\Upsilon_\G^{g_i g_1^{-1}}\big(\Upsilon_\G^{g_1}(\x)-\Upsilon_\GH^{g_1}(\x)\big)=
\Upsilon_\G^{g_i}\big(\Upsilon_\G^{g_1}(\x)-\Upsilon_\GH^{g_1}(\x)\big)\in \hat \Q_{s_i'}$$
where $s_i'=\Sigma^{g_i}(s_1)$ by Proposition \ref{prop:quadrant_tracking}. Finally, note that 
$$s_i=\Sigma^{g_i}(s_0)=\Sigma^{g_i g_1^{-1}}(s_1)=\Sigma^{g_i}(s_1)=s'_i$$
since $\Sigma^{g_1}(s_1)=s_1$. Thus, $\Upsilon_\G^{g_i g_1^{-1}}\big(\Upsilon_\G^{g_1}(\x)-\Upsilon_\GH^{g_1}(\x)\big) \in \hat \Q_{s_i}$ as desired.
\end{proof}

For $\vv \in \V_\GH$, we use $\e_\vv$ to denote the element of either $\R^{\V_\G}_c$
or $\R^{\V_\GH}_c$ (depending on context) which assigns one to $\vv$ and zero to everything else. 

\begin{theorem}
\label{thm:subgraph}
Suppose $\GH \subset \G$ is an infinite connected subgraph.  Let $\langle g_i \rangle$ be any geodesic ray in $G$. Then, for any vertex $\vv \in \V_\GH$,
$$|\Upsilon^{g_i}_{\GH}(\e_\vv)(\vw)| \leq |\Upsilon^{g_i}_{\G}(\e_\vv)(\vw)|$$
for all $\vw \in \V_\GH$ and all $i \geq 0$. 
\end{theorem}

\begin{proof}
Give $\GH$ an arbitrary ribbon graph structure, and find an arbitrary eigenvector of the adjacency operator. Then Corollary \ref{cor:expansion_sign_action} gives a sequence $s_i \in \SP$ such that
$\Upsilon^{g_i}_{\GH}(\e_\vv) \in \hat \Q_{s_i}$ and $\Upsilon^{g_i}_{\G}(\e_\vv) \in \hat \Q_{s_i}$. 
In particular, $s_i=\sa^{g_i}(s_0)$ for the appropriate choice of $s_0 \in \SP$. 

Now we will inductively define elements $\y_i \in \R^{\V_\G}_c$ for each integer $i \geq 1$. 
Assuming $\y_1, \ldots, \y_{i-1}$ are defined, we define $\y_i$ to be the unique element of 
 so that the following equation holds
$$\Upsilon^{g_i}_{\G}(\e_\vv)-\Upsilon^{g_i}_{\GH}(\e_\vv)=
\Upsilon^{g_i g_1^{-1}}_{\G}(\y_1)+\Upsilon^{g_i g_2^{-1}}_{\G}(\y_2)+\ldots+\y_i.$$
We will show that 
\begin{equation}
\label{eq:subgraph_show}
\Upsilon^{g_i g_j^{-1}}_{\G}(\y_j) \in \hat \Q_{s_i} \textrm{\quad for all $i \geq j$}.
\end{equation}
In particular, for all $\vw \in \V_\G$ we have
$$\Upsilon^{g_i}_{\G}(\e_\vv)(\vw)=\Upsilon^{g_i}_{\GH}(\e_\vv)(\vw)+
\Upsilon^{g_i g_1^{-1}}_{\G}(\y_1)(\vw)+\Upsilon^{g_i g_2^{-1}}_{\G}(\y_2)(\vw)+\ldots+\y_i(\vw),$$
and all terms in this sum are either zero or have the same sign. Therefore, equation \ref{eq:subgraph_show}
implies the theorem.

To shorten our formulas, let $\x_j=\Upsilon^{g_j}_{\GH}(\e_\vv)$. By induction, we observe that 
$$\y_j=\Upsilon^{g_j g_{j-1}^{-1}}_\G(\x_{j-1})-\Upsilon^{g_j g_{j-1}^{-1}}_\GH(\x_{j-1}).$$
Let $i \geq j$. Recall from the first paragraph that $\x_j \in \hat \Q_{s_j}$. Applying Corollary \ref{cor:diff_tracking} to the geodesic ray $\langle g_i g_j^{-1} \rangle_{i \geq j}$ yields that 
$\Upsilon^{g_i g_j^{-1}}(\y_j) \in \hat \Q_{s'_{i,j}}$ where $s'_{i,j}=\Sigma^{g_i g_j^{-1}}(s_j)$. The conclusion follows from the observation that 
$s_j=\Sigma^{g_j}(s_0)$ and therefore 
$$s'_{i,j}=\Sigma^{g_i g_j^{-1}}(s_j)=\Sigma^{g_i g_j^{-1}} \circ \Sigma^{g_j} (s_0)=s_i.$$
This proves equation \ref{eq:subgraph_show} as desired.
\end{proof}

\subsection{Finite covers}
Now let $\widetilde \G$ be a finite cover of $\G$. We will let $\pi:\V_{\widetilde \G} \to \V_\G$
denote the restriction of the covering map to the vertices. This map induces
a map $\pi_\ast:\R^{\V_{\widetilde \G}} \to \R^{\V_\G}$ given by
\begin{equation}
\label{eq:pi_ast}
\pi_\ast(\widetilde \f)(\vv)=\sum_{\widetilde \vv \in \pi^{-1}(\vv)} \widetilde \f(\widetilde \vv).
\end{equation}
Since $\pi$ arose from a covering map, we have the following.
\begin{proposition}
\label{prop:cover_commute}
For all $g \in G$, $\Upsilon^g_\G \circ \pi_\ast = \pi_\ast \circ \Upsilon^g_{\widetilde \G}$.
\end{proposition}
This may be proved by checking that the equation holds for powers of the generators of $G$. 

Finally, to use the previous propositions in our setting, we need to be able to find nice subgraphs of finite covers of $\G$. We use $\G_\Z$ to denote the graph with vertex set consisting of the integers which is formed by drawing edges between subsequent integers.

\begin{lemma}
\label{lem:covers}
Suppose $\G$ is an infinite connected graph so that every vertex has finite valance. Further suppose that $\G$ has no vertices of valance one. Let $\vv$ be any vertex of $\G$. Then either
\begin{enumerate}
\item there is an embedding $\phi:\G_\Z \to \G$ such that $\phi(0)=\vv$, or
\item there is a double cover $\widetilde \G$ of $\G$ (depending on $\vv$) and an embedding $\widetilde \phi:\G_\Z \to \widetilde \G$
such that $\widetilde \phi(0)$ is a lift of $\vv$. 
\end{enumerate}
\end{lemma}
\begin{proof}
For $S \subset \Z$ let $\G_S$ denote the graph with vertex set $S$ which is formed by adding edges between all pairs of integers whose difference is $1$. 
Because $\G$ is infinite and has bounded valance, there are vertices of
arbitrary large distance from $\vv$. Thus, we can define a metric embedding $\psi_0:\G_{\{0,-1,-2, \ldots\}} \to \V_\G$ such that
$\psi(0)=\vv$. (For each $n \geq 0$, choose $\psi_0(-n)$ so that its distance from $\vv$ is $n$.)
Now for $n>0$ inductively define $\psi_n:\G_{\{n, n-1, \ldots\}} \to \V_\G$ so that 
\begin{itemize}
\item $\psi_n$ restricted to $\G_{\{n-1, n-2, \ldots\}}$ is $\psi_{n-1}$, and
\item $\psi_n$ applied to the edge $(n-1,n)$ is distinct from the edge $(n-2,n-1)$. (This can be done because $\G$ has no vertices of valance one.)
\end{itemize}
Now assume $\psi_n$ is always injective. Then the limit $\lim_{n \to \infty} \psi_n$ is an embedding
of $\G_\Z \to \G$. So assume $\psi_n$ is not injective for some $n$.
Let $N$ be the smallest integer for which $\psi_n$ is not injective. Then, there is an $M<N$ for which
$\psi_N(N)=\psi_N(M)$. Let $\GH$ be the circular subgraph of $\G$ consisting of the image $\psi_N([M,N])$. Let $\hom{\GH} \in H_1(\G; \Z)$ denote the corresponding homology class with $\Z$ coefficients. 
Note that $\hom{\GH}$ must be primitive. Let $p: H_1(\G; \Z) \to \Z_2$ be any group homomorphism so that $p(\hom{\GH})=1$. Let $p':\pi_1(\G) \to \Z_2$ be the homomorphism constructed by taking homology class of an element of $\pi_1(\G)$ and then applying $p$. Let $\widetilde G$ denote the double cover of $\G$ which corresponds to the kernel of $p$. Then $\psi_N([M,N])$ lifts to an embedding 
$\widetilde \psi_N(M)$ and $\widetilde \psi_N(N)$ are both lifts of the point $\psi_N(N)=\psi_N(M)$.
Also consider the two disjoint lifts of $\psi_M$. These two lifts are rays
which end at $\widetilde \psi_N(M)$ and $\widetilde \psi_N(N)$. The two lifts of the ray $\psi_M$ and one of the lifts of the segment $\psi_N \big|_{[M,N]}$ can be stitched together to form our desired embedding $\widetilde \phi:\G_\Z \to \widetilde \G$.
\end{proof}

\section{Graphs without vertices of valance one}
\label{sect:no valance one}
In this section, we will primarily consider graphs $\G$ with no vertices of valance one. We prove the following results about these graphs.

\begin{theorem}[Decay]
\label{thm:decay}
Suppose $\G$ has no vertices of valance one. Let $\bm \theta \in \Rn_\lambda$ and let $\langle g_{n} \rangle$ be the $\lambda$-shrinking sequence of $\bm \theta$. Then for any $\f \in \Surv_{\bm \theta}$ and any $\vv \in \V$, the sequence
$$|\Act^{g_{n}}(\f)(\vv)|$$
tends monotonically to zero as $n \to \infty$. 
\end{theorem}

We break the proof of this theorem into two parts. We will prove that the sequence is monotonically decreasing in Section \ref{ss:montonic_decay}.
In Sections \ref{ss:effective_decay1} and \ref{ss:effective_decay2}, we will prove that this limit is zero. Clearly, this theorem implies that any subsequence also decays to zero. Thus, we have the following corollary.

\begin{corollary}[Decay properties]
If $\G$ has no vertices of valance one, then $S(\G, \w)$ has the subsequence decay property and the critical decay property.
\end{corollary}

We will prove that graphs with no vertices of valance one have the adjacency sign property in sections \ref{ss:adjacency_sign} and \ref{ss:asp_Z}. See Theorem \ref{thm:adj_sign_int}
and Corollary \ref{cor:adj_sign}

\subsection{Monotonic decay}
\label{ss:montonic_decay}

In this subsection we prove the monotonicity part of Theorem \ref{thm:decay}.

\begin{lemma}[Monotonic Decay]
\label{lem:monotonic_decay}
Suppose $\G$ has no vertices of valance one.
Let $\bm \theta \in \Circ$ be a $\lambda$-renormalizable direction with $\lambda$-shrinking sequence $\langle g_n \rangle$.  
For every $\bm \theta$-survivor $\f \in \R^\V$ and
every $\vv \in \V$, the sequence
$|\Upsilon^{g_n}(\f)(\vv)|$ decreases (non-strictly) monotonically in $n$. 
\end{lemma}

We will simplify the statement of this lemma. Applying Corollary \ref{cor:group_invariance},
we know that $\Upsilon^{g_1}(\f)$ is a $\rhoa{g_1}(\bm \theta)$-survivor. By induction, it is sufficient to prove that 
\begin{equation}
\label{eq:monotonic_decay_need}
|\Upsilon^{g_1}(\f)(\vv)| \leq |\f(\vv)|,
\end{equation}
for all $\lambda$-renormalizable $\bm \theta$, all $\bm \theta$-survivors $\f$, and all $\vv \in \V$. 
Up to the dihedral group action, 
we may assume that $\bm \theta \in \Q_{++}$ and that $g_1=h^{-1}$. See Remark \ref{rem:dihedral}. Since $\f$ is a $\bm \theta$-survivor, we know
$\f \in \hat \Q_{++}$. By definition, $\Ho^{-1}(\f)(\vb)=\f(\vb)$ for all $\vb \in \Beta$. So it suffices to consider $\vv \in \Alpha$. 
Suppose that $\rhoa{h^{-1}}(\bm \theta) \in \Q_{++}$. Recalling the formula in equation \ref{eq:Ho} for $\Ho(\f)$, we have that for $\va \in \Alpha$,
$$\Ho^{-1}(\f)(\va)=\f(\va)-\sum_{\vb \sim \va} \f(\vb) \geq 0.$$
But, each $\f(\vb) \geq 0$ as $\f \in \hat \Q_{++}$, so $\Ho^{-1}(\f)(\va) \leq \f(\va)$. 
Hence, equation \ref{eq:monotonic_decay_need} is trivially true when $\rhoa{h^{-1}}(\bm \theta) \in \Q_{++}$.
By Proposition \ref{prop:shrinking_and_sign}, the alternative is that $\rhoa{h^{-1}}(\bm \theta) \in \Q_{-+}$.
In this case $\Ho^{-1}(\f)(\va)<0$, so equation \ref{eq:monotonic_decay_need} is equivalent to showing
\begin{equation}
\label{eq:need2}
\f(\va) \geq - \Ho^{-1}(\f)(\va).
\end{equation}
In fact, we have that $\Ho^{-1}(\f)(\va)=\f(\va)-\A(\f)(\va)$ by Proposition \ref{prop:kernel}. Thus, equation \ref{eq:need2}
is equivalent to showing that $2 \f(\va) \geq \A(\f)(\va)$. Perhaps more usefully, we have
$$2 \f(\va)-\A(\f)(\va)=2 \langle \f, \e_a \rangle-\langle \A(\f), \e_a\rangle=\langle \f, 2 \e_a-\A(\e_a) \rangle.$$
In summary, the following lemma implies the Monotonic Decay lemma (\ref{lem:monotonic_decay}) above.

\begin{lemma}
\label{lem:2ndB}
Suppose $\G$ has no vertices of valance one. Let $\bm \theta \in \Circ \cap \Q_{++}$ be a $\lambda$-renormalizable direction with shrinking sequence $\langle g_n \rangle$.
Assume that $g_1=h^{-1}$. Then for every $\bm \theta$-survivor $\f$
and every $\va \in \Alpha$, we have $\langle \f, 2 \e_a-\A(\e_a) \rangle \geq 0$.
\end{lemma}

Recall, the {\em valance} of a vertex $\vv \in \V$ is the number $\textit{val}(\vv)=\# \{ \vx \in \V~:~ \vx \sim \vv\}$.

\begin{proposition}
\label{prop:case1}
The conclusion of Lemma \ref{lem:2ndB} holds when $\textit{val}(\va) \geq 3 \label{not:val}$ 
and $\va$ is not a member of a spoke.
\end{proposition}

For one step in the proof, we need the following result,
which is a consequence of Proposition \ref{prop:detecting_spokes}.

\begin{corollary}
\label{cor:bound}
Suppose $\G$ contains a vertex of valance $n$ which is not a member of a spoke. If $\G$ admits a positive eigenfunction with eigenvalue $\lambda$, then $\lambda \geq \frac{n}{\sqrt{n-1}}$. 
\end{corollary}
\begin{proof}
Let $\vv$ be the vertex of valance $n$ which is not a member of a spoke, and let $\f$ be a positive eigenfunction
with eigenvalue $\lambda$. 
Then by Proposition \ref{prop:detecting_spokes}, if $\vw \sim \vv$ we know $\frac{\f(\vw)}{\f(\vv)} \geq \frac{\lambda-\sqrt{\lambda^2-4}}{2}$. Since
$\sum_{\vw \sim \vv} \f(\vw)=\lambda \f(\vv)$, we know that
$$\lambda \geq n \big(\frac{\lambda-\sqrt{\lambda^2-4}}{2}\big).$$
This is equivalent to the inequality given in the corollary.
\end{proof}

\begin{proof}[Proof of Proposition \ref{prop:case1}]
By assumption $g_1=h^{-1}$. Thus, by Proposition \ref{prop:limit_set2}, we know $\bm \theta=(x,y) \in {\mathbb S}^1$ satisfies 
$$0<\frac{y}{x}<\frac{\lambda-\sqrt{\lambda^2-4}}{2}.$$
We will show that there exists two saddle connections $\sigma_1$ and $\sigma_2$ such that
\begin{enumerate}
\item $\Zem(\hom{\sigma_1}+\hom{\sigma_2})=2 \e_a-\A(\e_a)$, and
\item $\hol(\sigma_i) \wedge \bm \theta \geq 0$ for $i=1,2$.
\end{enumerate}
We will show these two statements imply the proposition. Statement (1) implies that
$$\langle \f, 2 \e_a-\A(\e_a) \rangle=\langle \f, \Zem(\hom{\sigma_1}+\hom{\sigma_2}) \rangle.$$
We wish to show that this quantity is non-negative. This follows from statement (2) of Proposition \ref{prop:positivity_checks} with 
$\x=\Zem(\hom{\sigma_1}+\hom{\sigma_2})$. Thus, the existence of such $\sigma_1$ and $\sigma_2$ imply the proposition.

Consider our surface $S=S(\G, \w)$, where $\A \w=\lambda \w$. We may assume that
$\w(\va)=1$ by scaling $\w$ if necessary.  The cylinder $\cyl_\va$ has a decomposition into rectangles of the form
$$\cyl_\va=\bigcup_{\vb \sim \va} (\cyl_\va \cap \cyl_\vb).$$
Let $k=\textit{val}(\va) \geq 3$.
We may number these rectangles $R_1, \ldots, R_k$ so that each $R_i$ is adjacent to $R_{i+1 \imod k}$.
Similarly, number the relevant vertices $\vb_1, \ldots, \vb_k \in \Beta$ so that $R_i=\cyl_\va \cap \cyl_{\vb_i}$ for all $i$.
Choose $j \in \{1, \ldots k\}$ so that
\begin{equation}
\label{eq:width_min}
\w(\vb_j)+\w(\vb_{j+1 \imod k})=\min~\{\w(\vb_i)+\w(\vb_{i+1 \imod k})~:~i=1, \ldots, k\}.
\end{equation}
Choose $\sigma_1$ to be the diagonal of the rectangle $R_j \cup R_{j+1 \imod k}$ which can be oriented downward and leftward. 
Choose $\sigma_2$ to be the diagonal of the complimentary rectangle, $\cyl_\va \smallsetminus (R_j \cup R_{j+1 \imod k})$, 
oriented downward and leftward. See Figure \ref{fig:rectangles}.

\begin{figure}[ht]
\begin{center}
\includegraphics[height=0.75in]{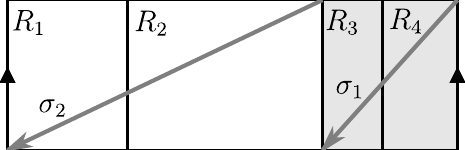}
\caption{An example decomposition of $\cyl_\va$ into rectangles. In this case, $\textit{val}(\va)=4$ and $j=3$ is the index satisfying equation \ref{eq:width_min}.}
\label{fig:rectangles}
\end{center}
\end{figure}

We now verify statements (1) and (2) hold for our choices of $\sigma_1$ and $\sigma_2$. Note that
$\sigma_1 \cup \sigma_2$ cross $\cyl_\va$ twice with positive algebraic sign, and cross
each of $\cyl_\vb$ for $\vb \sim \va$ once with negative algebraic sign. Moreover, $\sigma_1 \cup \sigma_2$ cross no other cylinders. 
This verifies statement (1). 

Now we verify statement (2). 
Write $\hol(\sigma_1)=(-w_1, -1)$ and $\hol(\sigma_2)=(-w_2, -1)$ with $w_1>0$ and $w_2>0$. 
Now recall that $\bm \theta=(x,y) \in S^1 \cap \Q_{++}$ is $\lambda$-renormalizable and that $g_1=h^{-1}$. Thus, Proposition \ref{prop:limit_set2}
implies that $\frac{y}{x}<\frac{\lambda-\sqrt{\lambda^2-4}}{2}$. 
For $i=1,2$, we have
\begin{equation}
\label{eq:hol_wedge}
\hol(\sigma_i) \wedge \bm \theta=(-w_i, -1) \wedge (x,y)=x-w_i y=x(1-w_i\frac{y}{x})>
\frac{x}{2}\big(2-w_i(\lambda-\sqrt{\lambda^2-4})\big).
\end{equation}
(Note $x>0$ since $\bm \theta \in \Q_{++}$.)
To proceed we must find upper bounds for $w_1$ and $w_2$. 
Because of equation \ref{eq:width_min}, we know that 
$$w_1 \leq \frac{2 \lambda}{\textit{val}(\va)} \leq \frac{2 \lambda}{3}.$$
($w_1$ must be less than or equal to the average length of pairs of adjacent rectangles.)
Therefore, by continuing equation \ref{eq:hol_wedge} for $i=1$ we have
$$\begin{array}{rcl}
\displaystyle \hol(\sigma_1) \wedge \bm \theta & \displaystyle > &  \displaystyle \frac{x}{2}\big(2-(\frac{2 \lambda}{3})(\lambda-\sqrt{\lambda^2-4})\big)=
\frac{x}{3}(3+\lambda\sqrt{\lambda^2-4}-\lambda^2) \\
& \displaystyle = & \displaystyle \frac{x}{3}\big(\frac{2 \lambda^2-9}{\lambda\sqrt{\lambda^2-4}+\lambda^2-3}\big) \geq 0.\end{array}$$
Here, the $\geq 0$ statement is somewhat subtle. The denominator $\lambda\sqrt{\lambda^2-4}+\lambda^2-3$ is positive because
$\lambda \geq 2$, and the numerator $2 \lambda^2-9 \geq 0$ because $\lambda \geq \frac{3 \sqrt{2}}{2}$ by Corollary \ref{cor:bound}.
To get an upper bound for $w_2$ we find a lower bound for $w_1$. We know 
$\va$ is not a member of a spoke and $\sigma_1$ crosses two rectangles, so by Proposition \ref{prop:detecting_spokes} we have
$$w_1 \geq \lambda-\sqrt{\lambda^2-4}.$$ Since the cylinder $\cyl_\va$ has inverse modulus $\lambda$ and width one, we know $w_1+w_2=\lambda$.
Therefore $w_2 \leq \sqrt{\lambda^2-4}$. By continuing equation \ref{eq:hol_wedge} for $i=2$ we have
$$\begin{array}{rcl}
\displaystyle \hol(\sigma_2) \wedge \bm \theta & \displaystyle > &  \displaystyle \frac{x}{2}\big(2-(\sqrt{\lambda^2-4})(\lambda-\sqrt{\lambda^2-4})\big)=
\frac{x}{2}(\lambda^2-2-\lambda\sqrt{\lambda^2-4}) \\
& \displaystyle = & \displaystyle \frac{x}{2}\Big(\frac{4}{\lambda^2-2+\lambda\sqrt{\lambda^2-4}}\Big)> 0.\end{array}$$
This proves statement (2) and concludes the proof.
\end{proof}

\begin{proposition}
Lemma \ref{lem:2ndB} holds when $\textit{val}(\va) = 2$ unless $\va$ belongs to a spoke.
\end{proposition}
\begin{proof}
Let $\vb_1$ and $\vb_2$ denote the two vertices adjacent to $a$. For $i=1,2$, let $\sigma_i$ denote the
diagonal of rectangle $\cyl_\va \cap \cyl_{\vb_i}$ which can be oriented downward and leftward.
We observe that $\Zem(\hom{\sigma_1}+\hom{\sigma_2})=2 \e_\va-\A(\e_\va)$. Write $\bm \theta=(x,y)\in \Q_{++}$. 
Because $g_1=h^{-1}$, Proposition \ref{prop:limit_set2}
implies that $y<x(\frac{\lambda-\sqrt{\lambda^2-4}}{2})$. For $i=1,2$, we have
$\hol(\sigma_i)=\big(-\w(\vb_i),-\w(\va)\big)$. By Proposition \ref{prop:detecting_spokes}, 
$\w(\vb_i) \geq \w(\va) (\frac{\lambda-\sqrt{\lambda^2-4}}{2})$.
We have $\w(\vb_1)+\w(\vb_2)=\lambda \w(\va)$. Thus,
$$\w(\vb_1)=\lambda \w(\va)-\w(\vb_2) \leq \w(\va) \big(\lambda-\frac{\lambda-\sqrt{\lambda^2-4}}{2}\big)=\w(\va) \frac{\lambda+\sqrt{\lambda^2-4}}{2}.$$
And similarly, $\w(\vb_2) \leq \w(\va) (\frac{\lambda+\sqrt{\lambda^2-4}}{2})$.
Thus, we compute
$$\hol(\sigma_i) \wedge \bm \theta = x\w(\va)-y\w(\vb_i)>x \w(\va)-\big(x(\frac{\lambda-\sqrt{\lambda^2-4}}{2})\big)\big(\w(a) (\frac{\lambda+\sqrt{\lambda^2-4}}{2})\big)=0.$$
Thus, if $\f$ is a $\bm \theta$-survivor, $\langle \f, \Zem(\hom{\sigma_i}) \rangle >0$.
And therefore $\langle \f, 2 \e_a-\A(\e_a) \rangle \geq 0$
as desired.
\end{proof}

\subsection{Effective decay for the integers}
\label{ss:effective_decay1}
In this subsection, we will only consider the case when $\G=\G_\Z$, the graph whose vertex set is $\Z$ and edges join two integers if and only if they differ by one.
Our decomposition $\V=\Alpha \cup \Beta$ is a decomposition into even and odd integers.

\begin{lemma}[Effective decay for the integers]
\label{lem:eff_int}
Let $\langle g_i \rangle$ be a shrinking sequence for a renormalizable direction $\bm \theta$.
Then, there is a critical time $t>0$ for which 
$$|\Act^{g_{t}}(\f)(\vv)| \leq \frac{1}{2} |\f(\vv)|$$
for any $\bm \theta$-survivor $\f$ and any $\vv \in \Z$. 
\end{lemma}
The idea of the proof is the following. We claim it is sufficient to find
a critical time $t$ for which 
\begin{equation}
\label{eq:half}
|\Upsilon^{\gamma(g_t)}(\e_\vv)(\vv)| \geq 2 \quad \textrm{
for all $\vv \in \Z$.}
\end{equation}
(The bulk of this section will be spent proving equation \ref{eq:half}.)
Let $\e_\vv \in \R^\Z_c$ be as in equation \ref{eq:e}.
Set $\e'_\vv=\pm \e_\vv \in \hat \Q_s$ where the sign is chosen depending on the quadrant $\Q_s$ containing $\bm \theta$. 
By equation \ref{eq:operator_relations2}, 
we have 
\begin{equation}
\label{eq:sum}
|\f(\vv)|=\langle \f, \e'_\vv \rangle=\langle \Upsilon^{g_t}(\f), \Upsilon^{\gamma(g_t)}(\e'_\vv)\rangle=
\sum_{\vw \in \V} \big(\Upsilon^{g_t}(\f)(\vw)\big)\big(\Upsilon^{\gamma(g_t)}(\e'_\vv)(\vw)\big).
\end{equation}
By Proposition \ref{prop:gamma_critical}, we have
$\Upsilon^{\gamma(g_t)}(\e'_\vv) \in \hat \Q_{s_t}$ for all critical times $t$.
Since both $\Upsilon^{g_t}(\f)$ and 
$\Upsilon^{\gamma(g_t)}(\e'_\vv)$ lie in $\hat \Q_{s_t}$, all terms in the sum above are non-negative. Therefore,
$$|\f(\vv)|=
\langle \Upsilon^{g_t}(\f), \Upsilon^{\gamma(g_t)}(\e'_\vv)\rangle
\geq \big(\Upsilon^{g_t}(\f)(\vv)\big)\big( \Upsilon^{\gamma(g_t)}(\e'_\vv)(\vv)\big)
\geq 2 |\Upsilon^{g_t}(\f)(\vv)|,$$
as claimed by the lemma.

\begin{proposition}
\label{prop:no_crit}
Consider the geodesic ray defined 
$$g_n=\begin{cases}
(v^{-1} h)^{\frac{n}{2}} & \textrm{if $n$ is even} \\
h(v^{-1} h)^{\frac{n-1}{2}} & \textrm{if $n$ is odd}.
\end{cases}$$
Then for all $n \geq 1$ and all $\va \in \Alpha$ and $\vb \in \Beta$, 
$$\Upsilon^{g_n}(\e_\va)=\begin{cases}
\sum_{\vc \in [\va-n+1,\va+n-1]} \e_\vc & \textrm{if $n \equiv 1 \pmod{4}$}, \\
\sum_{\vc \in \Alpha \cap [\va-n+1,\va+n-1]} \e_\vc-\sum_{\vc \in \Beta \cap [\va-n+1,\va+n-1]} \e_\vc & \textrm{if $n \equiv 2 \pmod{4}$}, \\
-\sum_{\vc \in [\va-n+1,\va+n-1]} \e_\vc & \textrm{if $n \equiv 3 \pmod{4}$}, \\
-\sum_{\vc \in \Alpha \cap [\va-n+1,\va+n-1]} \e_\vc+\sum_{\vc \in \Beta \cap [\va-n+1,\va+n-1]} \e_\vc & \textrm{if $n \equiv 0 \pmod{4}$.}
\end{cases}$$
$$\Upsilon^{g_n}(\e_\vb)=\begin{cases}
\sum_{\vc \in [\vb-n,\vb+n]} \e_\vc & \textrm{if $n \equiv 1 \pmod{4}$}, \\
\sum_{\vc \in \Alpha \cap [\vb-n,\vb+n]} \e_\vc-\sum_{\vc \in \Beta \cap [\vb-n,\vb+n]} \e_\vc & \textrm{if $n \equiv 2 \pmod{4}$}, \\
-\sum_{\vc \in [\vb-n,\vb+n]} \e_\vc & \textrm{if $n \equiv 3 \pmod{4}$}, \\
-\sum_{\vc \in \Alpha \cap [\vb-n,\vb+n]} \e_\vc+\sum_{\vc \in \Beta \cap [\vb-n,\vb+n]} \e_\vc & \textrm{if $n \equiv 0 \pmod{4}$.}
\end{cases}$$
\end{proposition}
It is straightforward to check that the formulas in the proposition above follow from the definition of the action $\Upsilon^G$. 

\begin{lemma}
\label{lem:growth}
Let $\langle g_n \rangle$ be a geodesic ray for which $\rhoa{g_1}(\Q_{s}) \subset \Q_s$.
Then for all $\x \in \R^\V_c \cap \hat \Q_s$ and all $\vv \in \R^\V_c$, the sequence
$|\Upsilon^{g_n}(\x)(\vv)|$
is non-decreasing.
\end{lemma}
\begin{proof}
Consider the orbit under the sign action, $s_n=\Sigma^{g_n}(s)$. 
Note that the condition that $\rhoa{g_1}(\Q_{s}) \subset \Q_s$ guarantees that 
$\Upsilon^{g_n}(\hat \Q_{s}) \subset \hat \Q_{s_n}$ for all $n$, by Proposition \ref{prop:quadrant_tracking}.

Suppose the lemma is false. Then there is a geodesic ray $\langle g_n \rangle$ such that 
$|\Upsilon^{g_n}(\x)(\vv)| < |\Upsilon^{g_{n-1}}(\x)(\vv)|$. We may assume that $n$ is minimal over all geodesic rays, all $\x$ and 
and all possible choices of $n$. Since the statement is invariant under the dihedral group, we may assume that $g_1=h$ and $s=++$. 
We will show that 
\begin{equation}
\label{eq:g_n}
g_n=\begin{cases}
(v^{-1} h)^{\frac{n}{2}} & \textrm{if $n$ is even} \\
h(v^{-1} h)^{\frac{n-1}{2}} & \textrm{if $n$ is odd}.
\end{cases}
\end{equation}
Otherwise, there is a $i \leq n-2$ for which $g_{i+2} g_i^{-1} \in \{vh,h^2, h^{-1} v^{-1}, v^{-2}\}$. 
(This is the first $i$ for which $g_{i+2}$ differs from the form above.) In this case, 
$s_{i+1}=s_{i+2}$, by the definition of the expanding sign action. Moreover, $\rhoa{g_{i+2} g_{i+1}^{-1}}(\Q_{s_{i+1}}) \subset \Q_{s_{i+1}}$. 
So letting $\y=\Upsilon^{g_{i+1}}(\x)$ and considering the geodesic ray $\langle g_{i+1+n} g_{i+1}^{-1} \rangle_n$ gives
a shorter counter example. 

Finally the case of $g_n$ as in equation \ref{eq:g_n} follows from Proposition \ref{prop:no_crit} above. Observe that
$$\Upsilon^{g_n}(\x)(\vv)=\sum_{\vw \in \V} \x(\vw) \Upsilon^{g_n}(\e_\vw)(\vv).$$
Note that the sign of each non-zero term is only dependent on $n$, and that by Proposition \ref{prop:no_crit} the sequence
$|\Upsilon^{g_n}(\e_\vw)(\vv)|$ is non-decreasing regardless of the choice of $\vw$.
\end{proof}

This proposition further lowers the bar for proving Lemma \ref{lem:eff_int}. It is sufficient to find any time $n$ for which 
\begin{equation}
\label{eq:half2}
|\Upsilon^{\gamma(g_n)}(\e_\vv)(\vv)| \geq 2 \quad \textrm{
for all $\vv \in \Z$.}
\end{equation}
We will then apply Lemma \ref{lem:growth}. We know that $\Q_{s_0}$ contains the vector $\bm \theta$, which is shrunk by 
$\rhoa{g_1}$. Then, $\rhoa{\gamma(g_1)}(\Q_{s_0})\subset \Q_{s_0}$ and  Lemma \ref{lem:growth} indicates that given any critical time $t \geq n$, we have
$|\Upsilon^{\gamma(g_t)}(\e_\vv)(\vv)|>|\Upsilon^{\gamma(g_n)}(\e_\vv)(\vv)| \geq 2$.
(We have such a critical time because of Corollary \ref{cor:critical_times_occur}.)
Thus, equation \ref{eq:half2} implies equation \ref{eq:half} which implies Lemma \ref{lem:eff_int}.

Recall from \S \ref{ss:renormalizable directions}
that a {\em renormalizing sequence} is a $\lambda$-shrinking sequences
for some $\lambda$-renormalizable direction.

\begin{proposition}
\label{prop:growth_list}
Let $\langle g_n \rangle$ be any renormalizing sequence. Then there is an $n$ for which $g_n$ is of one of the following forms. 
\begin{enumerate}
\item $g_n=h^c v^b h^a$ for some $a \in \Z$ and non-zero $b, c \in \Z$ such that 
$(b,c) \not \in \{(1,-1),(-1,1)\}$. 
\item $g_n=h^{-b} (h^{-b} v^b)^d h^a$ for some $a \in \Z$,  $b \in \{\pm 1\}$, and $d>0$. 
\item $g_n=h^f v^{e} (h^{-b} v^b)^d h^a$ for some $a \in \Z$,  $b \in \{\pm 1\}$, $d>0$, 
and nonzero $e, f \in \Z$ such that $(e,f) \neq (b, -b)$. 
\end{enumerate}
Moreover, for any such $g_n$, we have $|\Upsilon^{g_n}(\e_\vv)(\vv)| \geq 2$ for all $\vv \in \Alpha$. 
\end{proposition}
\begin{proof}
There is an $a \in \Z$ (possibly zero) of maximal absolute value such that $h^a =g_{|a|}$. 
(A renormalizing sequence cannot be $g_n=h^{\pm n}$ for all $n$, by definition.)
For all $\vv \in \Alpha$ and all $a \in \Z$, we have $\Ho^a(\e_\vv)=\e_\vv$
by definition of $\Ho$. See equation \ref{eq:Ho}.
Then to be a geodesic ray, one of $v h^a, v^{-1} h^a \in \{g_n\}$. 
Thus, there is a non-zero $b$ of maximal absolute value such that $v^b h^a=g_{|a|+|b|}$. 
By definition of $\Vo$, we have 
$$\Upsilon^{g_{|a|+|b|}}(\e_\vv)=\Vo^b \circ \Ho^a(\e_\vv)=\e_\vv+b(\e_{\vv-1}+\e_{\vv+1}).$$
Then there is a non-zero $c$ of maximal absolute value such that $h^c v^b h^a = g_{|a|+|b|+|c|}$.
We have
$$\Upsilon^{g_{|a|+|b|+|c|}}(\e_vv)=\e_\vv+b(\e_{\vv-1}+\e_{\vv+1})+bc(\e_{\vv-2}+2 \e_\vv+\e_{\vv+2}).$$
Therefore $|\Upsilon^{g_{|a|+|b|+|c|}}(\e_vv)(\v)|=|1+2bc|$. This quantity is larger than one unless 
$(b,c) \in \{(1,-1),(-1,1)\}$. This handles the case (1) of the proposition. 

If case (1) does not apply, then $b=\pm 1$, and $c=-b$. There is a maximal integer $d \geq 1$ such that
$(h^{-b} v^b)^d h^a=g_{|a|+2d}$. 
By conjugating by an element of the dihedral group and applying Proposition \ref{prop:no_crit}, we see that for some 
$\alpha, \beta \in \{\pm 1\}$ we have 
$$\Upsilon^{g_{|a|+2d}}(\e_\vv)(\vv)=\Upsilon^{g_{|a|+2d}}(\e_\vv)(\vv \pm 2)=\alpha
\quad \textrm{and} \quad
\Upsilon^{g_{|a|+2d}}(\e_\vv)(\vv \pm 1)=\beta.
$$
The choices of $\alpha$ and $\beta$ are given by the following rules that
$$\alpha=(-1)^d
\quad \textrm{and} \quad
\beta=-b(-1)^d.$$
The element $g_{|a|+2d+1}$ must be given by either $h^{-b} g_{|a|+2d}$ or $v^{\pm 1} g_{|a|+2d}$.
Assume that $g_{a+2d+1}=h^{-b} g_{|a|+2d+1}$. Then,
$$\Upsilon^{g_{|a|+2d+1}}(\e_\vv)(\vv)=
\Upsilon^{g_{|a|+2d}}(\e_\vv)(\vv)-2 \Upsilon^{g_{|a|+2d+1}}(\e_\vv)(\vv \pm 1)=\alpha-2b\beta=3(-1)^d,$$
which is of absolute value larger than $2$. This handles case (2). 

The only remaining possibility is that $g_{|a|+2d+1}=v^{\pm 1} g_{|a|+2d}$.
Then there is a nonzero $e$ of maximal magnitude for which $g_{|a|+2d+|e|}=v^{e} g_{|a|+2d}$. We compute
$$\Upsilon^{g_{|a|+2d+|e|}}(\e_\vv)(\vv)=\alpha 
\quad \textrm{and} \quad 
\Upsilon^{g_{|a|+2d+|e|}}(\e_\vv)(\vv \pm 1)=\beta-2 e \alpha=\alpha(-b+2e).$$
There is then an integer $f \neq 0$ such that $h^f g_{|a|+2d+|e|}=g_{|a|+2d+|e|+|f|}$. We compute that
$$\Upsilon^{g_{|a|+2d+|e|+|f|}}(\e_\vv)(\vv)=\alpha+2 f \Upsilon^{g_{|a|+2d+|e|}}(\e_\vv)(\vv \pm 1)=
\alpha\big(1+2f(-b+2e)\big).$$
This quantity has magnitude one only if $f(-b+2e)=-1$. Therefore, we must have $e=b$, and $f=-b$. 
This handles case (3).
Finally, if $e=b$ and $f=-b$, then 
$$g_{|a|+2d+|e|+|f|}=h^{-b} v^b (h^{-b} v^b)^d h^a,$$
which contradicts the maximality of $d$. 
\end{proof}

\begin{corollary}
\label{cor:growth_list}
Let $\langle g_n \rangle$ be any renormalizing sequence. Then there is an $N$ such that for all $n \geq N$.  
$$|\Upsilon^{g_n}(\e_\vv)(\vv)| \geq 2$$
for all $\vv \in \V$. 
\end{corollary}
\begin{proof}
Applying Proposition \ref{prop:growth_list}, there is an integer $n_1$ for which 
$|\Upsilon^{g_{n_1}}(\e_\va)(\va)| \geq 2$
for all $\va \in \Alpha$. By applying dihedral symmetry and Proposition \ref{prop:growth_list} again,
there is a integer $n_2$ for which $|\Upsilon^{\gamma(g_{n_2})}(\e_\vb)(\vb)| \geq 2$
for all $\vb \in \Beta$. Set $n=\max~\{n_1, n_2\}$. By Lemma \ref{lem:growth},
for all $\vv \in \V$ we have $|\Upsilon^{\gamma(g_{n_1})}(\e_\vv)(\vv)| \geq 2$.
\end{proof}

\begin{proof}[Proof of Lemma \ref{lem:eff_int} (Effective decay for the integers).]
From the discussion below the statement of the lemma, we see it is sufficient to prove the statement in equation \ref{eq:half2}. 
Since $\bm \theta$ is renormalizable, its shrinking sequence $\langle g_n \rangle$ is renormalizable.
The automorphism $\gamma$ sends renormalizable sequences to renormalizable sequences by Theorem \ref{thm:characterization}.
Therefore, the sequence $\langle \gamma(g_n) \rangle$ is also renormalizable. Thus, equation \ref{eq:half2}
follows directly from Corollary \ref{cor:growth_list}.
\end{proof}

\subsection{Effective decay for graphs without vertices of valance one}
\label{ss:effective_decay2}
Essentially, we use covers and subgraphs to deduce 
effective decay for graphs with no vertices of valance one from effective decay for $\G_\Z$. 

\begin{lemma}[Effective decay]
\label{lem:eff}
Let $\G$ be any infinite connected bipartite graph with bounded valance and without vertices of valance one.
Let $\langle g_i \rangle$ be a shrinking sequence for a renormalizable direction $\bm \theta$.
Then, there is a critical time $t>0$ for which 
$$|\Act_\G^{g_{t}}(\f)(\vv)| \leq \frac{1}{2} |\f(\vv)|$$
for any $\bm \theta$-survivor $\f \in \R^{\V_\G}$ and any $\vv \in \V_\G$. 
\end{lemma}
\begin{proof}
The sequence $\langle \gamma(\g_n) \rangle$ is a renormalizable sequence.
By Corollary \ref{cor:growth_list}, there is an $N$ such that for all $n \geq N$, 
$|\Upsilon_\Z^{\gamma(g_n)}(\e_\vv)(\vv)| \geq 2$ for all $\vv \in \Z$, where $\Upsilon_\Z$ denotes the action
associated to the graph $\G_\Z$. We will show that for $n \geq N$ we also have
$|\Upsilon_\G^{\gamma(g_n)}(\e_\vv)(\vv)| \geq 2$ for all $\vv \in \V_\G$. 

Let $\langle s_i \rangle$ denote the sign sequence of $\bm \theta$. Without loss of generality, we may assume $s_0=++$, $g_1 \in \{h^{-1}, v^{-1}\}$,
and $\gamma(g_1) \in \{h, v\}$. In particular, $\rhoa{\gamma(g_1)}(\Q_{++}) \subset \Q_{++}$. 
Then by Corollary \ref{cor:expansion_sign_action}, for any infinite connected graph bipartite $\GH$ 
we have $\Upsilon_\GH(\hat \Q_{++}) \subset \Q_{s'_i}$  where $s_i'=\Sigma^{\gamma(g_n)}(++)$.
Choose any $\vv \in \V$. By Lemma \ref{lem:covers}, there is a $\widetilde \G$ which is either $\G$ or a double cover of $\G$, and an embedding $\widetilde \phi:\G_\Z \to \widetilde \G$ such that $\widetilde \phi(0)$ is a lift of $\vv$. Let $\pi:\widetilde \G \to \G$ denote this covering, and 
$\pi_\ast:\R^{\V_{\widetilde G}} \to \R^{\V_\G}$ be as in equation \ref{eq:pi_ast}. 
Then by Proposition \ref{prop:cover_commute} for $n \geq N$, we have 
$$|\Upsilon_\G^{\gamma(g_n)}(\e_\vv)(\vv)| = \big|\Upsilon_{\widetilde \G}^{\gamma(g_n)}(\e_{\phi(0)})\big(\phi(0)\big)\big|.$$
Then by Theorem \ref{thm:subgraph}, we know that
$$\big|\Upsilon_{\widetilde \G}^{\gamma(g_n)}(\e_{\phi(0)})\big(\phi(0)\big)\big| \geq |\Upsilon_\Z^{\gamma(g_n)}(\e_{0})(0)| \geq 2,$$
so that $|\Upsilon_\G^{\gamma(g_n)}(\e_\vv)(\vv)| \geq 2$ as desired.

We complete the proof by following the logic applied to the case of $\G=\Z$. By Corollary \ref{cor:critical_times_occur}, there is a critical time 
$t \geq N$. Then, for any $\bm \theta$-survivor $\f \in \R^{\V_\G}$ and any $\vv \in \V_\G$,
$$\f(\vv)=\langle \f, \e_\vv\rangle=
\langle \Upsilon^{g_t}(\f), \Upsilon^{\gamma(g_t)}(\e_\vv)\rangle \geq 
\big|\Upsilon^{g_t}(\f)(\vv)\big| \big|\Upsilon^{\gamma(g_t)}(\e_\vv)(\vv)\big|\geq 2 |\Upsilon^{g_t}(\f)(\vv)|.$$
Here the first inequality follows because both $\Upsilon^{g_t}(\f), \Upsilon^{\gamma(g_t)}(\e_\vv) \in \Q_{s_t}$ by 
Proposition \ref{prop:gamma_critical}. See equation \ref{eq:sum} for another example of this.
We deduce that $|\Upsilon^{g_t}(\f)(\vv)| \leq \frac{1}{2} \f(\vv)$ as claimed by the lemma.
\end{proof}

\subsection{A perturbed group action}
\label{ss:adjacency_sign}

Let $\A \x=\y$. 
For any $g \in G$, we can compute $\Upsilon^{g}(\x)-\x$ inductively in terms of $\y$.
In this section we will explain how this is done.

Fixing any $\y \in \R^\V_c$, define the following ``affine'' actions on $\R^\V_c$.
$$\Ho_\y(\zz)=\Ho(\zz)+\pi_\Alpha(\y) 
\quad \textrm{and} \quad
\Vo_\y(\zz)=\Vo(\zz)+\pi_\Beta(\y) .$$
Powers of these functions are given by
$$\Ho^k_\y(\zz)=\Ho^k(\zz)+k \pi_\Alpha(\y) 
\quad \textrm{and} \quad
\Vo^k_\y(\zz)=\Vo^k(\zz)+k \pi_\Beta(\y) \label{not:perturb generators}.$$
From these formulas, we can check that $\Ho_\y$ and $\Vo_\y$ generate a nonlinear group action $\Chi_\y:G \times \R^\V_c \to \R^\V_c \label{not:perturb action}$, defined by
$\Chi^{h^k}_\y=\Ho^k_\y$ and $\Chi^{v^k}_\y=\Vo^k_\y$. 

\begin{proposition}
\label{prop:chi_props}
For all $g \in G$, $c \in \R$ and $\y, \y_1, \y_2, \zz_1, \zz_2 \in \R^\V_c$, we have 
\begin{enumerate}
\item $\displaystyle \Chi_\y^g(\zz_1+\zz_2)=\Chi_\y^g(\zz_1)+\Upsilon^g(\zz_2)$.
\item $\displaystyle \Chi_{\y_1+\y_2}^g(\0)=\Chi_{\y_1}^g(\0)+\Chi_{\y_2}^g(\0)$.
\item $\displaystyle \Chi_{c \y}^g(\0)=c \Chi_{\y}^g(\0)$.
\end{enumerate}
\end{proposition}
\begin{proof}
All statements follow from induction on the word length of $g$. The statements are clearly true when $g$ is the identity.
Assume statement (1) is true for $g_0$. We will show it is true for $h^k g_0$.
$$
\begin{array}{rcl}
\Chi_\y^{h^k g_0} (\zz_1+\zz_2) & = & \Chi_\y^{h^k} \big(\Chi_\y^{g_0} (\zz_1)+\Upsilon^{g_0}(\zz_2)\big)
=\Ho^k\big(\Chi_\y^{g_0} (\zz_1)+\Upsilon^{g_0}(\zz_2)\big)+k \pi_\Alpha(\y) \\
& = & \Ho^k\big(\Chi_\y^{g_0} (\zz_1)\big)+k \pi_\Alpha(\y) +\Upsilon^{h^k g_0}(\zz_2)
=\Chi_\y^{h^k g_0}(\zz_1)+\Upsilon^{h^k g_0}(\zz_2).
\end{array}
$$
A similar equation holds for $v^k g_0$. Now assume (2) is true for $g_0$. Then,
$$
\begin{array}{rcl}
\Chi_{\y_1+\y_2}^{h^k g_0}(\0) & = & \Ho^k\big(\Chi_{\y_1+\y_2}^{g_0} (\0)\big)+k\pi_\Alpha(\y_1+\y_2) \\
& = &
\Ho^k\big(\Chi_{\y_1}^{g_0} (\0)+\Chi_{\y_2}^{g_0}(\0)\big)+k\pi_\Alpha(\y_1)+k\pi_\Alpha(\y_2) 
= \Chi_{\y_1}^{h^k g_0}(\0)+\Chi_{\y_2}^{h^k g_0}(\0).
\end{array}
$$
Again, a similar formula holds for $v^k g_0$. Statement (3) holds for similar reasons.
\end{proof}

The following proposition connects this group action to the adjacency sign property. 

\begin{proposition}
\label{prop:chi2}
Let $\x \in \R^\V_c$ and set $\y=\A\x$. Then for all $g \in G$,
$$\Upsilon^g(\x)-\x=\Chi_\y^g(\0).$$
\end{proposition}
\begin{proof}
We may prove this by induction on the word length of $g$. The statement is clearly true when $g$ is the identity. Now suppose
we know the statement for $g_0 \in G$. We must prove the statement holds for $h^k g_0$ and $v^k g_0$ for $k=\pm 1$. We write
$$\begin{array}{rcl}
\Upsilon^{h^k g_0}(\x)-\x & = & \Ho^k \circ \Upsilon^{g_0}(\x)-\x=\Ho^k \big( \Upsilon^{g_0}(\x)-\x\big)+\Ho^k(\x)-\x \\
& = & \Ho^k \circ \Chi_\y^g(\0)+k \pi_\Alpha \circ \A(\x)=\Ho^k \circ \Chi_\y^g(\0)+k \pi_\Alpha (\y)=\Chi_\y^{h^k g_0}(\0).
\end{array}$$
A similar statement holds for the case of $v^k g_0$.
\end{proof}

We will now connect this operation to the adjacency sign property. Recall that for $\vv \in \V$ the function $\e_\vv \in \R^\V_c$ is the function that assigns one
to $\vv$ and assigns $\0$ to all other vertices.

\begin{definition}
Let $\lambda \geq 2$ and suppose $\bm \theta \in \Rn_\lambda \cap \Q_{++}$ has $\lambda$-shrinking sequence $\langle g_i \rangle$. 
Let $s_i'=\Sigma^{\gamma(g_i)}(++)$. 
For a vertex $\vv \in \V$ we say the graph $\G$ has $\vv$-ASP if for all $i \geq 0$, we have 
$\Chi^{\gamma(g_i)}_{\e_\vv}(\0) \in \hat \Q_{s'_i}$.
\end{definition}

\begin{proposition}
\label{prop:adj_sign_master}
Suppose that $\G$ has $\vv$-ASP for all $\vv \in \V$. Then $\G$ has the adjacency sign property.
\end{proposition}
\begin{proof}
Let $\x \in \R^\V_c$ and assume that $\y=\A(\x) \in \hat \Q_{++}$. Let $t$ be a critical time.
Then Proposition \ref{prop:chi2} implies that $\Upsilon^{\gamma(g_t)}(\x)-\x=\Chi_\y^{\gamma(g_t)}(\0)$.  
By Proposition \ref{prop:chi_props}, we can write
\begin{equation}
\label{eq:chi_sum}
\Chi_\y^{\gamma(g_t)}(\0)=\sum_{\vv \in \V} \Chi_{\y(\vv) \e_{\vv}}^{\gamma(g_t)}(\0)=
\sum_{\vv \in \V} \y(\vv) \Chi_{\e_{\vv}}^{\gamma(g_t)}(\0).
\end{equation}
Note that each $\y(\vv) \geq 0$ by the assumption that $\y \in \hat \Q_{++}$. In addition, each 
$\Chi_{\e_{\vv}}^{\gamma(g_t)}(\0) \in \hat \Q_{s'_t}$ because $\G$ has $\vv$-ASP.
Now let $\langle s_i\rangle$ denote the sign sequence of $\bm \theta$, and let $\f \in \Surv_{\bm \theta}$ be a $\bm \theta$-survivor. By Proposition \ref{prop:equiv_survivors},
we know $\Upsilon^{g_t} (\f) \in \hat \Q_{s_t}$. By Proposition \ref{prop:gamma_critical}, for all critical times $t$ we have $s_t'=s_t$.
Therefore, each $\Chi_{\e_{\vv}}^{\gamma(g_t)}(\0) \in \hat \Q_{s_t}$. Each term in equation \ref{eq:chi_sum} lies in $\hat \Q_{s_t}$.
It follows that $\langle \Upsilon^{g_t} (\f), \Chi_\y^{\gamma(g_t)}(\0)\rangle \geq 0$ as desired.
\end{proof}

\begin{lemma}[Subgraphs and $\vv$-ASP]
\label{lem:subgraphs_and_ASP}
Let $\vv \in \V$, and assume that there is an infinite connected subgraph $\GH \subset \G$ containing the vertex $\vv$ such that $\GH$ has $\vv$-ASP. Then, $\G$ also has $\vv$-ASP. 
\end{lemma}

\begin{proof}
We recall our notation from section \ref{sect:subgraphs_and_covers}. We distinguish the actions 
of $\Upsilon_\G:G \times \R^{\V_\G} \to \R^{\V_\G}$ and $\Upsilon_\GH:G \times \R^{\V_\GH} \to \R^{\V_\GH}$. To distinguish the two $\Chi_\y$ actions, we use
$\Chi_{\G,\y}:G \times \R^{\V_\G} \to \R^{\V_\G}$ and $\Upsilon_{\GH,\y}:G \times \R^{\V_\GH} \to \R^{\V_\GH}$. We will abuse notation by identifying
$\R^{\V_\GH}$ with the subset of $\R^{\V_\G}$ which is supported on $\V_\GH$.

We must show that $\Chi^{\gamma(g_i)}_{\G,\e_\vv}(\0) \in \hat \Q_{s'_i}$ for all $i$. The statement of the proposition guarantees
$\Chi^{\gamma(g_i)}_{\GH,\e_\vv}(\0) \in \hat \Q_{s'_i}$ for all $i$. We will show that for all $i$,
\begin{equation}
\label{eqn:prove_super_asp}
\Chi^{\gamma(g_i)}_{\G,\e_\vv}(\0)-\Chi^{\gamma(g_i)}_{\GH,\e_\vv}(\0) \in \hat \Q_{s'_i}.
\end{equation}
This implies the proposition.

Now we will inductively define elements $\y_i \in \R^{\V_\G}_c$ for each integer $i \geq 1$. 
Assuming $\y_1, \ldots, \y_{i-1}$ are defined, we define $\y_i$ to be the unique element of 
 so that the following equation holds
\begin{equation}
\label{eq:sg_sum}
\Chi^{\gamma(g_i)}_{\G,\e_\vv}(\0)-\Chi^{\gamma(g_i)}_{\GH,\e_\vv}(\0)=
\Upsilon^{\gamma(g_i g_1^{-1})}_{\G}(\y_1)+\Upsilon^{\gamma(g_i g_2^{-1})}_{\G}(\y_2)+\ldots+\y_i.
\end{equation}
We will show that 
\begin{equation}
\label{eqn:sg_to_show} 
\Upsilon^{\gamma(g_i g_j^{-1})}_{\G}(\y_j) \in \hat \Q_{s'_i}.
\quad \textrm{for all $i \geq j$.}
\end{equation}
This implies that equation \ref{eqn:prove_super_asp} holds, because it holds for each term in the sum given in equation \ref{eq:sg_sum}.

We observe by combining the cases $i=j$ and $i=j-1$ of equation \ref{eq:sg_sum} that
$$\y_j=\Chi^{\gamma(g_j)}_{\G,\e_\vv}(\0)-\Chi^{\gamma(g_j)}_{\GH,\e_\vv}(\0)-\Upsilon_\G^{\gamma(g_j g_{j-1}^{-1})} \big(\Chi^{\gamma(g_{j-1})}_{\G,\e_\vv}(\0)-\Chi^{\gamma(g_{j-1})}_{\GH,\e_\vv}(\0)\big).$$
Define the following two quantities. 
$$\a_j=\Chi^{\gamma(g_j)}_{\G,\e_\vv}(\0)-\Upsilon_\G^{\gamma(g_j g_{j-1}^{-1})} \circ \Chi^{\gamma(g_{j-1})}_{\G,\e_\vv}(\0).$$
$$\b_j=-\Chi^{\gamma(g_j)}_{\GH,\e_\vv}(\0)+\Upsilon_\GH^{\gamma(g_j g_{j-1}^{-1})} \circ \Chi^{\gamma(g_{j-1})}_{\GH,\e_\vv}(\0).$$
Observe that
$$\y_j=\a_j+\b_j+
\Upsilon_\G^{\gamma(g_j g_{j-1}^{-1})} \circ \Chi^{\gamma(g_{j-1})}_{\GH,\e_\vv}(\0)
-\Upsilon_\GH^{\gamma(g_j g_{j-1}^{-1})} \circ \Chi^{\gamma(g_{j-1})}_{\GH,\e_\vv}(\0).
$$
By Proposition \ref{prop:chi_props}, we have
$$\a_j =
\Chi^{\gamma(g_jg_{j-1}^{-1})}_{\G,\e_\vv} \big( \Chi^{\gamma(g_{j-1})}_{\G,\e_\vv}(\0)\big)+\Upsilon_\G^{\gamma(g_j g_{j-1}^{-1})} \big( -\Chi^{\gamma(g_{j-1})}_{\G,\e_\vv}(\0)\big) =
\Chi^{\gamma(g_jg_{j-1}^{-1})}_{\G,\e_\vv}(\0).$$
Similarly, 
$$\b_j=-\Chi^{\gamma(g_jg_{j-1}^{-1})}_{\GH,\e_\vv} \big( \Chi^{\gamma(g_{j-1})}_{\GH,\e_\vv}(\0)\big)-\Upsilon_\GH^{\gamma(g_j g_{j-1}^{-1})} \big( -\Chi^{\gamma(g_{j-1})}_{\G,\e_\vv}(\0)\big) =
-\Chi^{\gamma(g_jg_{j-1}^{-1})}_{\GH,\e_\vv}(\0).$$
Thus $\a_j+\b_j=\0$, because $\Chi^{\gamma(g_jg_{j-1}^{-1})}_{\ast,\e_\vv}(\0)$
equals $\pm \pi_\Alpha(\e_v)$ when $g_jg_{j-1}^{-1}=h^{\pm 1}$ and 
equals $\pm \pi_\Beta(\e_v)$ when $g_jg_{j-1}^{-1}=v^{\pm 1}$. Thus,
$$\y_j=\Upsilon_\G^{\gamma(g_j g_{j-1}^{-1})} \circ \Chi^{\gamma(g_{j-1})}_{\GH,\e_\vv}(\0)
-\Upsilon_\GH^{\gamma(g_j g_{j-1}^{-1})} \circ \Chi^{\gamma(g_{j-1})}_{\GH,\e_\vv}(\0).
$$
Now we will use the assumption that $\GH$ has $\vv$-ASP. Therefore, 
$\Chi^{\gamma(g_{j-1})}_{\GH,\e_\vv}(\0) \in \hat \Q_{s'_j}$. 
Then, by applying Corollary \ref{cor:diff_tracking} to the geodesic ray 
$\langle \gamma(g_{i} g_{j-1}^{-1}) \rangle_{i \geq j-1}$ we see
that 
$$\Upsilon_\G^{\gamma(g_i g_j^{-1})}(\y_j) \in \hat \Q_{s(i,j)},$$
where $s(i,j)=\Sigma^{\gamma(g_i g_{j-1}^{-1})}(s'_{j-1})$. We observe $s(i,j)=s'_i$,
and therefore we have proved equation \ref{eqn:sg_to_show} as desired.
\end{proof}

\begin{lemma}[Covers and $\vv$-ASP]
\label{lem:cover_ASP}
Let $\vv \in \V$, and assume that there is a finite cover $\widetilde \G$ of $\G$ and
a lift $\widetilde \vv$ of $\vv$ for so that $\widetilde \G$ has $\widetilde \vv$-ASP. 
Then, $\G$ has $\vv$-ASP. 
\end{lemma}

Recall our notation for dealing with covers given in section \ref{sect:subgraphs_and_covers}.
When discussing a covering graph $\widetilde G$ of $\G$, we will use tildes to denote
functions in $\R^{\widetilde \V}$ and operators on this space.
For instance,
$\widetilde \Chi_{\widetilde \y}$ denotes an action of $G$ on $\R^{\widetilde \V}$.
Also recall that $\pi:\widetilde \G \to \G$ denotes the covering map, and 
$\pi_\ast:\R^{\widetilde \V} \to \R^{\V}$ is the induced map given in equation
\ref{eq:pi_ast}. We also use $\tilde \pi_{\tilde \Alpha}: \R^{\tilde \V} \to \R^{\tilde \Alpha}$ to denote the variant of the projection $\pi_\Alpha: \R^\V \to \R^\Alpha$ for the cover $\tilde \G$.

The key to the lemma is the following observation.

\begin{proposition}
\label{prop:chi_commute}
Let $\widetilde \G$ be a finite cover of $\G$.
Let $\widetilde \x, \widetilde \y \in \R^{\widetilde \V}$. Then
\begin{enumerate}
\item $\pi_\ast \circ \widetilde \pi_{\widetilde \Alpha}(\widetilde \y)=\pi_\Alpha \circ \pi_\ast(\widetilde \y)$, 
\item  $\pi_\ast \circ \widetilde \pi_{\widetilde \Beta}(\widetilde \y)=\pi_\Beta \circ \pi_\ast(\widetilde \y)$, and 
\item for all $g \in G$, 
$\pi_\ast \circ \widetilde \Chi^g_{\widetilde \y}(\widetilde \x)=
\Chi^g_{\pi_\ast(\widetilde \y)} \circ \pi_\ast(\widetilde \x)$.
\end{enumerate}
\end{proposition}
\begin{proof}
The first two statements follow trivially from the definitions. 
It is enough to check statement (3) on powers of the generators. 
We will consider the case of $g=h^k$. The case of $g=v^k$ follows similarly. By definition of $\widetilde \Chi^{h^k}$,
by statement (1) and by Proposition \ref{prop:cover_commute} we have
$$\pi_\ast \circ \widetilde \Chi^{h^k}_{\widetilde \y}(\widetilde \x)=
\pi_\ast \big(\widetilde \Upsilon^{h^k}(\widetilde \x)+
k\widetilde \pi_{\widetilde \Alpha}({\widetilde \y})\big)=
\Upsilon^{h^k} \circ \pi_\ast(\widetilde \x)+k \pi_\Alpha \circ \pi_\ast({\widetilde \y})$$
So, by definition of $\Chi^{h^k}$, we see $\pi_\ast \circ \widetilde \Chi^{h^k}_{\widetilde \y}(\widetilde \x)=\Chi^{h^k}_{\pi_\ast(\widetilde \y)} \circ \pi_\ast(\widetilde \x)$ as desired.
\end{proof}

\begin{proof}[Proof of Lemma \ref{lem:cover_ASP}]
First note that $\pi_\ast(\e_{\widetilde \vv})=\e_{\vv}$. 
Since $\widetilde \G$ has $\widetilde \vv$-ASP, we know 
$\widetilde \Chi^{\gamma(g_i)}_{\e_{\widetilde \vv}}(\widetilde \0) \in \hat \Q_{s'_i}$
for all $i$. By Proposition \ref{prop:chi_commute},
$$\Chi^{\gamma(g_i)}_{\e_{\vv}}(\0)=\pi_\ast \circ \widetilde \Chi^{\gamma(g_i)}_{\e_{\widetilde \vv}}(\widetilde \0),$$
and therefore $\Chi^{\gamma(g_i)}_{\e_{\vv}}(\0)\in \hat \Q_{s'_i}$
for all $i$.
\end{proof}

\subsection{The adjacency sign property for \texorpdfstring{$\Z$}{Z}}
\label{ss:asp_Z}

In this subsection, we complete the proof that graphs with no vertices of valance one have the adjacency sign property.

\begin{theorem}
\label{thm:adj_sign_int}
The graph $\G_\Z$ has $\vv$-ASP for all $\vv \in \Z$.
\end{theorem}

From the previous subsection, we have the following corollary.

\begin{corollary}
\label{cor:adj_sign}
Suppose $\G$ be any infinite graph without vertices of valance one. 
Then $\G$ has $\vv$-ASP for all $\vv \in \V$. Thus, $\G$ has the adjacency sign property.
\end{corollary}
\begin{proof}
Lemma \ref{lem:covers} guarantees that given any $\vv \in \V$, we can find an embedding of $\G_\Z$ into $\G$ such that $\vv$ lies in the image, or we can find a double cover $\widetilde \G$ of $\G$ and a lift $\widetilde \vv$ of $\vv$ and a embedding of $\G_\Z$
into $\widetilde \G$ which passes through $\widetilde \vv$.
In the first case, $\G$ has $\vv$-ASP by Lemma \ref{lem:subgraphs_and_ASP}. 
In the second, $\widetilde \G$ has $\widetilde \vv$-ASP by Lemma \ref{lem:subgraphs_and_ASP} and $\G$ has $\vv$-ASP by Lemma \ref{lem:cover_ASP}.
The adjacency sign property follows from Proposition \ref{prop:adj_sign_master}.
\end{proof}

\begin{proof}[Proof of Theorem \ref{thm:adj_sign_int}]
By possibly applying reflective dihedral symmetry $\bar \cdot$ (which reflects in the line $x=y$), 
we may assume that $\vv \in \Alpha$. By further applying translational symmetries of the graph $\G_\Z$, we may assume that $\vv=0$. Thus, $\Alpha$ consists of the even integers and $\Beta$ consists of the odd integers. To simplify notation let $\y=\e_0$. 

We now recall what we must prove. Let $\lambda \geq 2$ and suppose $\bm \theta \in \Rn_\lambda \cap \Q_{++}$ has shrinking sequence $\langle g_i \rangle$ and
sign sequence $\langle s_i \rangle$. Let $s_i'=\Sigma^{\gamma(g_i)}(++)$. Then, we must show that for all $i \geq 0$, we have
$\Chi_\y^{\gamma(g_i)}(\0) \in \hat \Q_{s'_i}$.

We primarily view this as a combinatorial statement, however we will find it useful to use some geometric tricks. Therefore, we define $\1 \in \R^\Z$ to be the function such that $\1(\vx)=1$ for all $\vx \in \Z$. This is an eigenvector for the adjacency operator. Namely, $\A(\1)=2\cdot \1$. We can therefore build a surface $S_\1=S(\G_\Z, \1)$. We will find it useful
to compute the holonomies of homology classes on this surface. We use
$$\hol_\1:H_1(S_\1, V; \Z) \to \R^2$$
to denote the holonomy map for this surface.

The idea of the proof is to find $\zz_i, \w_i \in \R^\V_c \cap \hat \Q_{s_i'}$ for integers $i \geq 0$ such that 
the following statements hold.
\begin{enumerate}
\item For all $i \geq 0$ we have $\Chi_\y^{\gamma(g_{i})}(\0)=\zz_i+\w_i$. By Proposition \ref{prop:chi_props}, this ensures that for all $k \geq 0$
$$\Chi_\y^{\gamma(g_{i+k})}(\0)=\Chi_\y^{\gamma(g_{i+k} g_i^{-1})}(\zz_i)+\Upsilon^{\gamma(g_{i+k} g_i^{-1})}(\w_i).$$
\item $\zz_i \in \hat \Q_{s'_i}$ for all $i \geq 0$.
\item For all $i$ and all $k \geq 0$ we have $\Upsilon^{\gamma(g_{i+k} g_i^{-1})}(\w_i) \in \hat \Q_{s'_{i+k}}$. In particular, $\w_i \in \hat \Q_{s'_i}$ for all $i$. 
\end{enumerate}
These statements guarantee $\Chi_\y^{\gamma(g_i)}(\0) \in \hat \Q_{s'_i}$ for all $i$ as desired. Therefore, they imply the lemma.

We will now explain our choice of $\zz_i$. Our choice determined inductively. We set $\zz_0=\0$. 
Note that $\zz_0 \in \hat \Q_{s_0'}=\hat \Q_{++}$. Subsequent choices are determined by following the arrows in 
Figure \ref{fig:adj_sign_diagram}. We begin at the node labeled $\0$ in the quadrant $++$ of the diagram. 
Then to find $\zz_1$ we follow the arrow labeled $\gamma(g_1)$. To find $\zz_i$ we follow the arrows labeled
$\gamma(g_1)$, then $\gamma(g_2 g_1^{-1})$, then $\gamma(g_3 g_2^{-1})$, and continue until we follow 
the arrow labeled $\gamma(g_i g_{i-1}^{-1})$. By comparison with the diagram in Definition \ref{def:sign_action} of
the expanding sign action, we see that $\zz_i \in \hat \Q_{s'_i}$ for all $i$. This verifies statement (2) above.

\begin{figure}
\begin{center}
\includegraphics{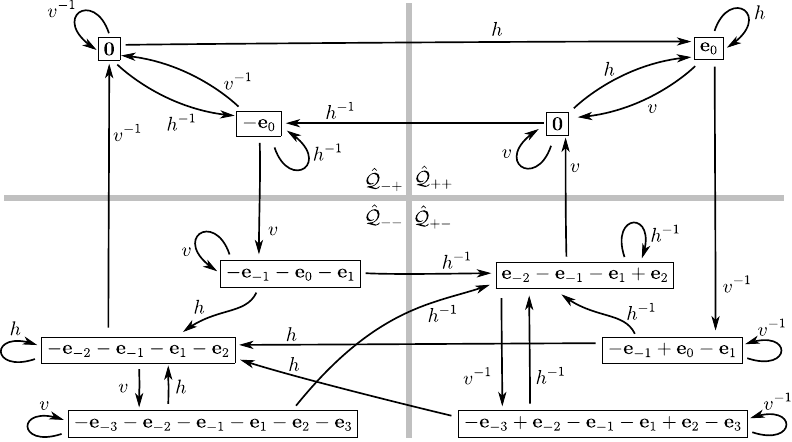}
\caption{Diagram for the proof of Theorem \ref{thm:adj_sign_int}. Arrows labeled $\gamma(g_i g_{i-1}^{-1})$ join $\zz_{i-1}$ to $\zz_i$.}
\label{fig:adj_sign_diagram}
\end{center}
\end{figure}

The formula in statement (1) above inductively determines $\w_i$ from $\zz_i$. We have $\w_0=\0$. For $i \geq 1$ we have
\begin{equation}
\label{eq:w_i}
\begin{array}{rcl}
\w_i & = & \Chi_\y^{\gamma(g_{i})}(\0)- \zz_i=
\Chi_\y^{\gamma(g_{i} g_{i-1}^{-1})} \circ \Chi_\y^{\gamma(g_{i-1})}(\0) -\zz_i \\
& = & 
\Chi_\y^{\gamma(g_{i} g_{i-1}^{-1})} (\zz_{i-1}+\w_{i-1})-\zz_i \\
& = &
\Chi_\y^{\gamma(g_{i} g_{i-1}^{-1})} (\zz_{i-1})-\zz_i+\Upsilon^{\gamma(g_{i} g_{i-1}^{-1})}(\w_{i-1}).\end{array}
\end{equation}
Because of choice of $\w_i$, we have automatically verified statement (1) above.

It remains to verify statement (3). We prove this statement by induction in $i$. Clearly, statement (3) is true for $\w_0=\0$ for all $k$. 
Now assume the statement is true for $i-1$. In particular, we assume that $\Upsilon^{\gamma(g_{i+k} g_{i-1}^{-1})}(\w_{i-1}) \in \hat \Q_{s_{i+k}'}$
for all $k \geq 0$. We wish to show statement (3) holds for $i$. By equation \ref{eq:w_i}, we have
$$\Upsilon^{\gamma(g_{i+k} g_i^{-1})}(\w_i)=\Upsilon^{\gamma(g_{i+k} g_i^{-1})}\big(\Chi_\y^{\gamma(g_{i} g_{i-1}^{-1})} (\zz_{i-1})-\zz_i\big)+\Upsilon^{\gamma(g_{i+k} g_{i-1}^{-1})}(\w_{i-1}).$$
By our inductive hypothesis, it is sufficient to show that for all $k \geq 0$, 
\begin{equation}
\label{eq:to_check}
\Upsilon^{\gamma(g_{i+k} g_i^{-1})}\big(\Chi_\y^{\gamma(g_{i} g_{i-1}^{-1})} (\zz_{i-1})-\zz_i\big) \in \hat \Q_{s'_{i+k}}.
\end{equation}

Despite the fact that we need equation \ref{eq:to_check} for all $k \geq 0$, this is really a finite check. We will use the Quadrant Tracking Proposition \ref{prop:quadrant_tracking} 
to verify this for all $k$ in one step. The number of checks is therefore equal to the number ($30$) of arrows in figure \ref{fig:adj_sign_diagram} above.

In order to cut down on the number of checks, we consider the element of the dihedral group action on the plane $\Delta:\R^2 \to \R^2$ defined by $\Delta(x,y)=(-x,y)$. Note that it satisfies the equations
$$\Delta \circ \rhoa{h}=\rhoa{h^{-1}} \circ \Delta
\quad \textrm{and} \quad
\Delta \circ \rhoa{v}=\rhoa{v^{-1}} \circ \Delta.$$
We define $\delta:G \to G$ to be the automorphism induced by the action on generators
$\delta(h)=h^{-1}$ and $\delta(v)=v^{-1}$. The action satisfies $\Delta \circ \rhoa{g}=\rhoa{\delta(g)} \circ \Delta$ for all $g \in G$. There is an orientation reversing affine automorphism of the surface $S_\1$ which preserves each rectangle which appears as an intersection of a horizontal and vertical cylinder and which has derivative given by $\Delta$. This affine automorphism preserves $\Xi(\R^\Z)$. The action therefore lifts to an action on $\R^\Z$. This action $\Delta_\ast:\R^\V \to \R^\V$ is given by 
$$\Delta_\ast(\f)(\vv)=\begin{cases}
-\f(\vv) & \textrm{if $\vv \in \Alpha$} \\
\f(\vv) & \textrm{if $\vv \in \Beta$.} 
\end{cases}$$
We can use this action to conjugate our group of operators $\Upsilon^G$.
$$\Upsilon^g \circ \Delta_\ast=\Delta_\ast \circ \Upsilon^{\delta(g)}
\quad \textrm{for all $g \in G$.}$$
For our specific choice of $\y=\e_0$ (or more generally for any $\y$ supported on $\Alpha$)
we have
$$\Delta_\ast \circ \Ho_\y=\Ho_\y^{-1} \circ \Delta_\ast
\quad \textrm{and} \quad
\Delta_\ast \circ \Vo_\y=\Vo_\y^{-1} \circ \Delta_\ast.$$
Thus for all $g \in G$, we have $\Chi_\y^g \circ \Delta_\ast=\Delta_\ast \circ \Chi_\y^{\delta(g)}$.
In particular, we can apply $\Delta_\ast$ to both sides of equation \ref{eq:to_check} to obtain
$$\Upsilon^{\delta \circ \gamma(g_{i+k} g_i^{-1})}\big(\Chi_\y^{\delta \circ \gamma(g_{i} g_{i-1}^{-1})} \circ \Delta_\ast(\zz_{i-1})-\Delta_\ast(\zz_i)\big) \in \hat \Q_{\Delta(s'_{i+k})}.$$
Therefore by the invariance of figure \ref{fig:adj_sign_diagram} under $\Delta_\ast$, we only
need to consider the case when $\zz_{i-1} \in \hat \Q_{++} \cup \Q_{+-}$. 
We will show that equation \ref{eq:to_check} holds in each of these cases below.

For all $i$ let $\x_i=\Chi_\y^{\gamma(g_{i} g_{i-1}^{-1})} (\zz_{i-1})-\zz_i$. 
\begin{enumerate}
\item Suppose $\zz_{i-1}=\0 \in \hat \Q_{++}$.
\begin{enumerate}
\item If $\gamma(g_{i} g_{i-1}^{-1})=h$, then $\zz_i=\e_0 \in \hat \Q_{++}$. We compute 
$$\x_i=\Ho_\y (\0)-\e_0=\0.$$
Thus, in this case $\Upsilon^{\gamma(g_{i+k} g_i^{-1})}(\x_i)=\0$ for all $k$. The conclusion follows trivially, because $\0 \in \hat \Q_s$ for all $s \in \SP$. 
\item If $\gamma(g_{i} g_{i-1}^{-1})=v$, then $\zz_i=\0 \in \hat \Q_{++}$. We compute $\x_i=\0$. The conclusion follows as in case (1a).
\item If $\gamma(g_{i} g_{i-1}^{-1})=h^{-1}$, then $\zz_i=-\e_0 \in \hat \Q_{-+}$. We compute $\x_i=\0$. The conclusion follows as in case (1a).
\end{enumerate}

\item Suppose $\zz_{i-1}=\e_0 \in \hat \Q_{++}$.
\begin{enumerate}
\item If $\gamma(g_{i} g_{i-1}^{-1})=h$, then $\zz_i=\e_0 \in \hat \Q_{++}$. We compute
$$\x_i=\Ho_\y (\zz_{i-1})-\zz_i=2 \e_0-\e_0=\e_0.$$
We see that $\x_i=\Ho(\e_0) \in \hat \Q_{++}$. We have $\x_i, \e_0 \in \Q_{++}$, while $\rhoa{h}(\Q_{++}) \subset \Q_{++}$. Thus Proposition \ref{prop:quadrant_tracking}
entails that the quadrant containing $\Upsilon^{\gamma(g_{i+k} g_i^{-1})}(\x_i)$ is governed by the expanding sign action. Namely,
we consider the geodesic ray $\langle \gamma(g_{l+i-1} g_{i-1}^{-1}) \rangle_{l \geq 0}$.
Proposition \ref{prop:quadrant_tracking} implies that $\Upsilon^{\gamma(g_{l+i-1} g_{i-1}^{-1})}(\e_0) \in \hat \Q_{\tilde s_l}$ where $\tilde s_l=\Sigma^{\gamma(g_{l+i-1} g_{i-1}^{-1})}(++)$. Note by induction that $\tilde s_{k+1}=s'_{i+k}$ for all $k \geq 0$, since $s'_{i-1}=\tilde s_0=++$ and $s'_{i+k}=\Sigma^{g_{i+k} g_{i-1}^{-1}}(s'_{i-1})=\tilde s_{k+1}$. Therefore, for all $k \geq 0$,
$$\Upsilon^{\gamma(g_{i+k} g_i^{-1})}(\x_i)=\Upsilon^{\gamma(g_{i+k} g_{i-1}^{-1})}(\e_0)
\in \hat \Q_{s'_{i+k}}.$$
\item If $\gamma(g_{i} g_{i-1}^{-1})=v$, then $\zz_i=\0 \in \hat \Q_{++}$. We compute
$$\x_i=\Vo_\y (\zz_{i-1})-\zz_i=\e_{-1}+\e_0+\e_1.$$
We have $\x_i=\Vo(\e_0)$ with both $\x_i, \e_0 \in \hat \Q_{++}$, while $\rhoa{v}(\Q_{++}) \subset \Q_{++}$. Thus, Proposition \ref{prop:quadrant_tracking}
guarantees that $\Upsilon^{\gamma(g_{i+k} g_i^{-1})}(\x_i) \in \hat \Q_{s'_{i+k}}$ for all $k \geq 0$. See case (2a).
\item If $\gamma(g_{i} g_{i-1}^{-1})=v^{-1}$, then $\zz_i=-\e_{-1}+\e_0-\e_1 \in \hat \Q_{-+}$. We compute $\x_i=\0$. The conclusion follows as in case (1a).
\end{enumerate}

\item Suppose $\zz_{i-1}=-\e_{-1}+\e_0-\e_1 \in \hat \Q_{+-}$.
\begin{enumerate}
\item If $\gamma(g_{i} g_{i-1}^{-1})=h$, then $\zz_i=-\e_{-2}-\e_{-1}-\e_1-\e_{2} \in \hat \Q_{--}$. We compute
$$\x_i=\Ho_\y (\zz_{i-1})-\zz_i=(-\e_{-2}-\e_{-1}-\e_1-\e_{2})-\zz_i=\0.$$
The conclusion follows as in case (1a).
\item If $\gamma(g_{i} g_{i-1}^{-1})=h^{-1}$, then $\zz_i=\e_{-2}-\e_{-1}-\e_1+\e_{2} \in \hat \Q_{+-}$. We compute
$$\x_i=\Ho^{-1}_\y (\zz_{i-1})-\zz_i=(\e_{-2}-\e_{-1}+2\e_0-\e_1+\e_{2})-\zz_i=2\e_0.$$
Therefore $\x_i=\Ho^{-1}(\x_i) \in \hat \Q_{+-}$ while $\rhoa{h^{-1}}(\Q_{+-})\subset \Q_{+-}$. Proposition \ref{prop:quadrant_tracking}
guarantees that $\Upsilon^{\gamma(g_{i+k} g_i^{-1})}(\x_i) \in \hat \Q_{s'_{i+k}}$ for all $k \geq 0$. See case (2a).
\item If $\gamma(g_{i} g_{i-1}^{-1})=v^{-1}$, then $\zz_i=-\e_{-1}+\e_0-\e_1 \in \hat \Q_{+-}$. We compute
$$\x_i=\Vo^{-1}_\y (\zz_{i-1})-\zz_i=(-2\e_{-1}+\e_0-2\e_1)-\zz_i=-\e_{-1}-\e_{1}.$$
Therefore $\x_i=\Vo^{-1}(\x_i) \in \hat \Q_{+-}$. By Proposition \ref{prop:quadrant_tracking},
 $\Upsilon^{\gamma(g_{i+k} g_i^{-1})}(\x_i) \in \hat \Q_{s'_{i+k}}$ for all $k \geq 0$ as in case (2a).
\end{enumerate}

\item Suppose $\zz_{i-1}=\e_{-2}-\e_{-1}-\e_1+\e_{2} \in \hat \Q_{+-}$.
\begin{enumerate}
\item If $\gamma(g_{i} g_{i-1}^{-1})=v$, then $\zz_i=\0 \in \hat \Q_{++}$. We compute
$$\x_i=\Vo_\y (\zz_{i-1})=\Vo (\zz_{i-1})=\e_{-3}+\e_{-2}+\e_{2}+\e_3.$$
In this case $\x_i=\Vo(\zz_{i-1})$. There exist saddle connections $\sigma_1$ and $\sigma_2$ on the surface $S_\1$ such that $\Zem(\hom{\sigma_1})=\e_{-2}-\e_{-1}$ and $\Zem(\hom{\sigma_2})=-\e_1+\e_2$, and so
$\zz_{i-1}=\Zem(\hom{\sigma_1})+\Zem(\hom{\sigma_2})$. These saddle connections are pictured on the left side of Figure \ref{fig:saddles}. For $j \in \{1,2\}$, we have 
$\Vo \circ \Zem(\hom{\sigma_j})=\Zem \circ \Phi^{h^{-1}}(\hom{\sigma_j})$ by Corollary \ref{cor:pullback_action}. We note that the length of the holonomy of these saddle connections is
$$\| \hol_\1 \circ \Phi^{h^{-1}}(\sigma_j) \|=
\|\hol_\1 (\sigma_j)\|=\sqrt{2}.$$
Thus, $\hol~\sigma_{j} \in \cl\big(\Exp_2(h^{-1}) \cap \Q_{--}\big)$. Therefore, 
by considering the geodesic ray $\langle \gamma(g_{i-1+l} g_{i-1}^{-1}) \rangle_{l \geq 0}$,
Corollary \ref{cor:expansion_sign_action} implies that 
$\Upsilon^{\gamma(g_{i-1+l} g_{i-1}^{-1})} \circ \Zem(\hom{\sigma_j})\in \hat \Q_{\tilde s_l}$ where $\tilde s_l=r \circ \Sigma^{g_{i-1+l} g_{i-1}^{-1}}(--)$.
By equation \ref{eq:r_commute}, we have 
$$\tilde s_l=\Sigma^{\gamma(g_{i-1+l} g_{i-1}^{-1})} \circ r(--)=\Sigma^{\gamma(g_{i-1+l} g_{i-1}^{-1})} (+-).$$
Thus, we have $s'_{i+k}=\tilde s_{k+1}$ for all $k \geq 0$, since $s'_{i-1}=\tilde s_{0}=+-$ and $s'_{i+k}=\Sigma^{\gamma(g_{i+k}g_{i-1}^{-1})}(s'_{i-1})$
while $\tilde s_{k+1}=\Sigma^{\gamma(g_{k+i} g_{i-1}^{-1})}(\tilde s_0)$. Therefore, we have 
$$\Upsilon^{\gamma(g_{i+k} g_i^{-1})}(\x_i)=
\Upsilon^{\gamma(g_{i+k} g_{i-1}^{-1})} \circ \Zem(\hom{\sigma_1})+
\Upsilon^{\gamma(g_{i+k} g_{i-1}^{-1})} \circ \Zem(\hom{\sigma_2}) \in 
\hat \Q_{s'_{i+k}}$$
for all $k \geq 0$. 

\begin{figure}
\begin{center}
\includegraphics{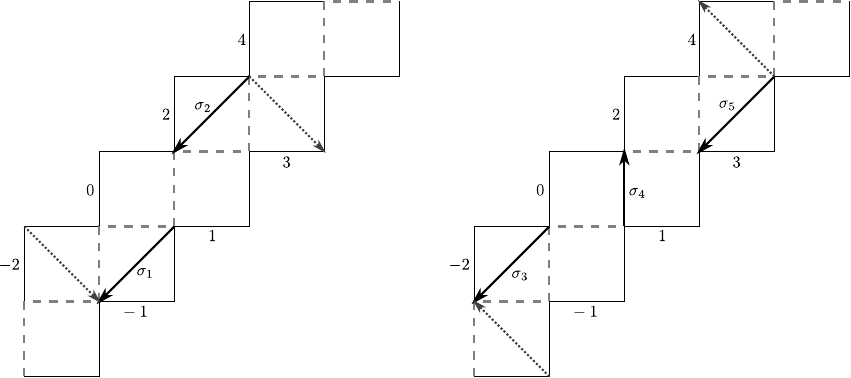}
\caption{Both sides of this figure illustrate the surface $S_\1=S(\G_\Z, \1)$. To obtain the surface, identify opposite horizontal and vertical edges by vertical and horizontal translations (respectively). This figure illustrates the saddle connections $\sigma_1, \ldots, \sigma_5$ used in the proof of Theorem \ref{thm:adj_sign_int}. From left to right,
the dotted arrows represent the saddle connections $\Phi^{h^{-1}}(\sigma_1)$, $\Phi^{h^{-1}}(\sigma_2)$, $\Phi^{v^{-1}}(\sigma_3)$ and $\Phi^{v^{-1}}(\sigma_5)$. Also, $\Phi^{v^{-1}}(\sigma_4)=\sigma_4$. }
\label{fig:saddles}
\end{center}
\end{figure}

\item If $\gamma(g_{i} g_{i-1}^{-1})=h^{-1}$, then $\zz_i=\zz_{i-1}=\e_{-2}-\e_{-1}-\e_1+\e_{2} \in \hat \Q_{+-}$. We compute
$$\x_i=\Ho^{-1}_\y (\zz_{i-1})-\zz_i=(2\e_{-2}-\e_{-1}+\e_0-\e_1+2\e_{2})-\zz_i=\e_{-2}+\e_0+\e_2.$$
Thus $\x_i=\Ho^{-1}(\x_i) \in \hat \Q_{+-}$ while $\rhoa{h^{-1}}(\Q_{+-})\subset \Q_{+-}$. Proposition \ref{prop:quadrant_tracking}
guarantees that $\Upsilon^{\gamma(g_{i+k} g_i^{-1})}(\x_i) \in \hat \Q_{s'_{i+k}}$ for all $k \geq 0$ as in case (2a).
\item If $\gamma(g_{i} g_{i-1}^{-1})=v^{-1}$, then $\zz_i=-\e_{-3}+\e_{-2}-\e_{-1}-\e_1+\e_2-\e_3  \in \hat \Q_{+-}$. We compute
$$\x_i=(-\e_{-3}+\e_{-2}-2\e_{-1}-2\e_1+\e_{2}-\e_3)-\zz_i=-\e_{-1}-\e_1.$$
In particular, we have that 
$\x_i=\Vo^{-1}(\x_i)$. So, Proposition \ref{prop:quadrant_tracking}
guarantees that $\Upsilon^{\gamma(g_{i+k} g_i^{-1})}(\x_i) \in \hat \Q_{s'_{i+k}}$ for all $k \geq 0$ as in case (2a).
\end{enumerate}

\item Suppose $\zz_{i-1}=-\e_{-3}+\e_{-2}-\e_{-1}-\e_1+\e_2-\e_3  \in \hat \Q_{+-}$.
\begin{enumerate}
\item If $\gamma(g_{i} g_{i-1}^{-1})=h$, then $\zz_i=-\e_{-2}-\e_{-1}-\e_1-\e_2 \in \hat \Q_{--}$. We compute
$$\begin{array}{rcl} 
\x_i & = & \Ho_\y (\zz_{i-1})-\zz_i \\
& = &
(-\e_{-4}-\e_{-3}-\e_{-2}-\e_{-1}-\e_0-\e_1-\e_2-\e_3-\e_4)-\zz_i \\
& = & -\e_{-4}-\e_{-3}-\e_0-\e_3-\e_4.\end{array}$$
In this case $\x_i=\Ho(-\e_{-3}+\e_{-2}-\e_0+\e_2-\e_3)$. 
We can find saddle connections $\sigma_3$, $\sigma_4$ and $\sigma_5$ so that
$\Zem(\hom{\sigma_3})=-\e_{-3}+\e_{-2}$, $\Zem(\hom{\sigma_4})=-\e_{0}$
and $\Zem(\hom{\sigma_5})=\e_2-\e_3$. These saddle connections are depicted on the right side of Figure \ref{fig:saddles}. 
For $j \in \{3,4, 5\}$, we have 
$\Ho \circ \Zem(\hom{\sigma_j})=\Zem \circ \Phi^{v^{-1}}(\hom{\sigma_j})$ by Corollary \ref{cor:pullback_action}. For each such $j$ we have
$$\|\hol_\1(\sigma_j)\|=\|\hol_\1 \circ \Phi^{v^{-1}}(\hom{\sigma_j})\|.$$
In particular, $\hol_\1(\sigma_j) \in \cl\big(\Exp_2(v^{-1})\cap (\Q_{-+} \cup \Q_{--})\big)$.
 Corollary \ref{cor:expansion_sign_action} implies that the quadrant containing 
$\Upsilon^{\gamma(g_{i+k} g_{i-1}^{-1})} \circ \Zem(\hom{\sigma_j})$ is given by 
may determined 
by following the expanding sign action. As in case (4a), we consider the geodesic ray $\langle \gamma(g_{i-1+l} g_{i-1}^{-1}) \rangle_{l \geq 0}$, and
$\Upsilon^{\gamma(g_{i+l} g_{i-1}^{-1})} \circ \Zem(\hom{\sigma_j}) \in \hat \Q_{\widetilde s_l}$ where
$$\widetilde s_l=r \circ \Sigma^{g_{i+l} g_{i-1}^{-1}}(s)=\Sigma^{\gamma(g_{i+l} g_{i-1}^{-1})} \circ r(s),$$
where $s=--$ or $s=-+$ depending on $j \in \{3,4,5\}$. In these cases $r(s)=+-$ or $r(s)=--$. We have $\gamma(g_{i}g_{i-1}^{-1})=h$, and
$$\Sigma^h(+-)=\Sigma^h(--)=--.$$
Therefore in either case, we have $\widetilde s_l=\Sigma^{\gamma(g_{i-1+l} g_{i-1}^{-1})}(+-)$ for $l \geq 1$. We also observe 
$s'_{i+k}=\widetilde s_{1+k}$ for all $k \geq 0$.  We conclude that for all $k \geq 0$, 
$$\Upsilon^{\gamma(g_{i+k} g_i^{-1})}(\x_i)=
\sum_{j=3}^5 \Upsilon^{\gamma(g_{i+k} g_{i-1}^{-1})} \circ \Zem(\hom{\sigma_j})
\in \hat \Q_{s'_{i+k}}.$$
\item If $\gamma(g_{i} g_{i-1}^{-1})=h^{-1}$, then $\zz_i=\e_{-2}-\e_{-1}-\e_1+\e_2 \in \hat \Q_{+-}$. We compute
$$\begin{array}{rcl} \x_i & = & \Ho_\y^{-1} (\zz_{i-1})-\zz_i \\
& = & (\e_{-4}-\e_{-3}+3\e_{-2}-\e_{-1}+\e_0-\e_1+3\e_2-\e_3+\e_4)-\zz_i \\
& =& \e_{-4}-\e_{-3}+2\e_{-2}+\e_0+2\e_2-\e_3+\e_4. \end{array}$$
In particular, 
$\x_i=\Ho^{-1}(-\e_{-3}+\e_{-2}+\e_0+\e_2-\e_3)$. Both $\x_i$ and $-\e_{-3}+\e_{-2}+\e_0+\e_2-\e_3$ lie in $\hat \Q_{+-}$,
while $\rhoa{h^{-1}}(\Q_{+-}) \subset \Q_{+-}$.  
Therefore, Proposition \ref{prop:quadrant_tracking}
guarantees that $\Upsilon^{\gamma(g_{i+k} g_i^{-1})}(\x_i) \in \hat \Q_{s'_{i+k}}$ for all $k \geq 0$ as in case (2a). 
\item If $\gamma(g_{i} g_{i-1}^{-1})=v^{-1}$, then $\zz_i=\zz_{i-1} \in \hat \Q_{+-}$. We compute
$$\x_i=\Vo_\y^{-1} (\zz_{i-1})-\zz_i=-\e_{-3}-\e_{-1}-\e_1-\e_3.$$
Therefore, $\x_i=\Vo^{-1}(\x_i)$ while $\rhoa{v^{-1}}(\Q_{+-})\subset \Q_{+-}$.
So, Proposition \ref{prop:quadrant_tracking}
guarantees that $\Upsilon^{\gamma(g_{i+k} g_i^{-1})}(\x_i) \in \hat \Q_{s'_{i+k}}$ for all $k \geq 0$ as in case (2a).
\end{enumerate}
\end{enumerate}
\end{proof}

\appendix
\section{Invariant measures and coding}
\name{sect:coding}

In this section, we rehash some of the arguments used to understand invariant measures of interval exchange maps. See
\cite[\S 14.5]{KH95}, for instance.

\subsection{Coding orbits}
Let $S=\bigsqcup_{i \in \Lambda} P_i/\sim$ be a translation surface written as a union of polygons with edge identifications.
Let $\Edge$ denote the set of all identified pairs of edges in $\partial P_i \subset S$ for $i \in \Lambda$. 
Let $\bm \theta \in \Circ$ be a direction.
Define $\Edge_\ast \subset \Edge$ to be the collection of those edges which are not parallel to $\bm \theta$. 

\begin{remark}
In this paper, our surfaces are built from polygons with horizontal and vertical sides,
and the directions we consider are $\lambda$-renormalizable which disallows
$\bm \theta$ from being horizontal or vertical. 
So, in our setting $\Edge=\Edge_\ast$. 
\end{remark}

We view each edge $e \in E$ as a closed interval (including its endpoints).
The union of edges, $X_\ast=\bigcup_{e \in \Edge_\ast} e$ is a section for the straight line flow in direction ${\bm \theta}$.
That is, given any point in $S$, its forward orbit under the straight line flow $F_\btheta$
hits a point in $X_\ast$.
We use $\ret: X_\ast \to X_\ast$ to denote the return map of the flow $F_\btheta$ to this section. 

Recall that $V \subset S$ denotes the identified vertices of the polygons making up $S$. This is also the union of endpoints of edges in $E_\ast$. 
When equipped with the measure on $X_\ast$ induced by the Lebesgue transverse measure to the foliation in direction $\btheta$, 
the return map $\ret$ conjugate to an interval exchange involving infinitely many intervals.

As with interval exchange transformations, we 
run into a problem concerning orbits which visit the endpoints of an edge $e \in \Edge_\ast$. (These endpoints lie in $V$.)
We resolve this by splitting all orbits which hit a singularity in two. 
Namely, if the forward or backward orbit of a point $p$ hits an endpoint of an edge, we replace $p$ with two new points $p^-$ and $p^+$. The orbit of $p^-$ tracks the points
to the left of $p$, and the orbit of $p^+$ tracks points to the right. (Left and right make sense once we rotate the picture so that $\bm \theta$ is vertical.)
By {\em track} we mean to follow points infinitesimally nearby. We keep track of both the point's location, and the edge it lies on.
For $e \in \Edge_\ast$, let $\hat e$ denote $e$ with all points with singular
orbits split as described above. Let $l$ and $r$ be the left and right endpoints of $e \in \Edge_\ast$, respectively. We also replace $l$ with $l^+$ and $r$ with $r^-$ in $\hat e$. 
We call $\hat e$ a {\em split edge}. 

For each $e \in \Edge_\ast$, there is a natural ``unsplitting'' map $\pi_e: \hat e \to e$. This map is surjective, and one-to-one except at countably many points, where it is two-to-one. 
We will be considering the disjoint union of all split edges
$\hat X_\ast=\bigsqcup_{e \in \Edge_\ast} \hat e$, and 
we define the map 
$$\pi:\hat X_\ast \to X_\ast; \quad
p \in \hat e \mapsto \pi_e(p).$$
This map may be countable-to-one at points in $V$, but is only finite-to-one at the finite-order cone singularities.

The return map $\ret$ has a natural lift to the map $\hat \ret: \hat X_\ast \to \hat X_\ast$. We define $\hat \ret$ to be the unique continuous map
so that whenever $\pi$ is one-to-one
at $p \in \hat X_\ast$, we have
$$\ret(p)=\pi \circ \hat \ret \circ \pi^{-1}(p).$$ 
With this definition, $\hat T_\btheta^n(q)$ is well defined for all $q \in \hat X_\ast$ and all $n \in \Z$.

Let $\textit{edge}:\hat X_\ast \to \Edge_\ast$ denote the map which recovers the edge a point of $\hat X_\ast$ lies on.
Now consider the coding map
\begin{equation}
\label{eq:coding_map}
\textit{code}:\hat X_\ast \to \Edge_\ast^\Z:q \mapsto \langle \textit{edge} \circ \hat T_\btheta^n(q) \rangle_{n \in \Z}.
\end{equation}
The image of this map is a shift space $\Omega \subset \Edge_\ast^\Z$ on the countable alphabet $\Edge_\ast$. The shift space $\Omega$ has the property that each symbol $e \in \Edge_\ast$ is only followed by 
finitely many other symbols.

We give the space $\hat X_\ast=\bigsqcup_{e \in \Edge_\ast} \hat e$ the coarsest topology
which makes both the map $\pi$ and the map $\textit{code}$
continuous. 

\begin{proposition}
\label{prop:coding}
The coding map is injective if the straight line flow in direction $\bm \theta$ is conservative and has no periodic trajectories.
\end{proposition}
\begin{proof}
Suppose the straight line flow is conservative and has no periodic trajectories. 
We will show that distinct trajectories have distinct codes. 

Suppose we have a pair of distinct trajectories with the same code, $\omega$. The two trajectories
cross the same sequence of edges. Consider the intersection of the two trajectories
with one edge $e \in E$. Since the trajectories are distinct, they bound a closed interval $J$ in $e$.
The straight-line flow of points in this interval must also have the same code. Consider the forward and backward straight-line flows of all points in $J$. This defines a continuous isometric immersion of a bi-infinite strip into the surface $S$.
We will draw a contradiction to the existence of this immersion.

Since the straight-line flow is conservative, there is a trajectory in $J$ which returns to $J$. 
The code of a point and another point on the 
orbit of the point (intersected with $\bigcup_{e \in E} e$) differ by a shift. 
Since a point in $J$ returns to $J$, we know that the code $\omega$ is periodic. Now consider the set of points
whose trajectories cross edges according to a periodic $\omega$. This set can be found by developing the bi-infinite periodic sequence
of polygons crossed into the plane. Lines which run through this sequence of polygons correspond to trajectories
with code $\omega$. Such a line must be preserved by the translation symmetry of the developed sequence of polygons.
The quotient of such a line by this symmetry gives a closed straight-line flow trajectory in this direction.
But, this is ruled out by our assumption of no periodic trajectories.
\end{proof}

\begin{proposition}
\name{prop:Cantor set}
If the coding map is injective, then the coding map is a homeomorphism from $\hat X_\ast$ onto its image, $\Omega \subset \Edge_\ast^\Z$. In this case, the coding map is a topological conjugacy from
$\hat \ret$ to the shift map on $\Omega$.
\end{proposition}
\begin{proof}
For this proof, endow $\hat X_\ast$ with the (apriori new) topology which makes
the coding map a homeomorphism. We must show that $\pi$ is continuous. 
Let $U$ be an open subset of an edge $e \in E$.
We will show that $\pi^{-1}(U)$ is open.
Let $q \in \pi^{-1}(U)$ and let $p=\pi(q)$. 
It suffices to find an open subset of $\pi^{-1}(U)$ which contains $q$. 
For each $n \geq 0$,
let $C_n \subset \Omega$ be the cylinder set defined so that for each $\omega \in C_n$, we have $\omega_m=\textit{code}(q)_m$ for $-n \leq m \leq n$. 
Observe that $\pi\circ\textit{code}^{-1}(C_n)$ is a closed interval in $e$. 
By definition of $C_n$, we have $\bigcap_{n} C_n=\{\textit{code}(q)\}$.
By injectivity 
of the coding map, 
$$\bigcap_n \textit{code}^{-1}(C_n)=\{q\} \quad \text{and} \quad 
\bigcap_n \pi\circ\textit{code}^{-1}(C_n)=\{p\}.$$ 
Since the later is a nested intersection of closed intervals, there must be an $N$ so that $\pi\circ\textit{code}^{-1}(C_{N}) \subset U$. So, our needed open set is
given by $\textit{code}^{-1}(C_{N})$.

The fact that this is a conjugacy follows from the fact that
the coding map is a homeomorphism together with the definition of the coding map given in equation \ref{eq:coding_map}.
\end{proof}

\subsection{Laminations and invariant measures}
We mentioned that $X_\ast=\bigcup_{e \in \Edge_\ast} e$ gives
a section of the straight line flow in direction $\btheta$.
Recall that $\F_\btheta$ denoted the foliation by orbits of this flow.
We can split leaves which hit points in $V$ in a similar manner to how we split points of $X_\ast$ to form $\hat X_\ast$. Namely, if a leaf hits a singularity $v \in V$ in forward or backward time, we replace it by two leaves one of which tracks points to the left, and one which tracks nearby points to the right. We separate the two new leaves by a gap on the surface, and leave the point $v$ in the gap. Since there are only countably many singular leaves, this can be done everywhere. We let $\hat \F_\btheta
\label{not:split foliation2}$
denote the new leaf space.

Recall that we call a measure {\em locally finite} if it is finite on compact sets. Similarly, a transverse measure is {\em locally finite} if
it assigns finite mass to compact transversals. Since we have split leaves 
and orbits in the same way, we have the following:
\begin{proposition}
\label{prop:measure_spaces}
The following two spaces of measures are isomorphic:
\begin{itemize}
\item The space of locally finite $\hat \ret$-invariant measures on $\hat X_\ast$.
\item The space $\M_\btheta$ of locally finite transverse measures to $\hat \F_\btheta$.
\end{itemize}
If the coding map is injective, these two spaces are also isomorphic to 
\begin{itemize}
\item The space of locally finite shift-invariant measures on $\Omega$.
\end{itemize}
\end{proposition}

Suppose the straight line flow in direction $\bm \theta$ is conservative and has no periodic trajectories. Proposition \ref{prop:Cantor set} guarantees that each split edge
$\hat e$ is homeomorphic to a Cantor set. It then follows that,
in this case, $\hat \F_\btheta$ is a {\em lamination}, i.e., it is locally homeomorphic to a Cartesian product of a Cantor set and a line.
This lamination can also be produced via the standard construction of a lamination from an interval exchange
obtained by placing a hyperbolic structure on the suspended IET (the surface $S$) and straightening the leaves to geodesics. See \cite[Part I]{Bonahon01}.

\subsection{Interaction with homology}
\label{sect:homology}

In this section, we will consider the edges of the set $\Edge_\ast$ to be oriented, with an arbitrary choice of orientation made for each $e \in \Edge_\ast$. The homology classes of oriented edges  
generate $H_1(S, V, \Z)$. 

We recall some notation from Section \ref{sect:reinterpret}. The cohomology space $\Coh$ is the space of all linear maps $H_1(S, V, \R) \to \R$. 
We let $\M_\btheta$ of locally finite transverse measures to the oriented leaf space $\hat \F_\btheta$. We will now formally define the map 
$\Psi_{\bm \theta}:\M_{\bm \theta} \to \Coh$.
Given a measure $\mu \in \M_\theta$, we define
a linear map $\Psi_{\bm \theta}(\mu):H_1(S, V, \R) \to \R$. This is the unique map for which $\Psi_{\bm \theta}(\mu)(\hom{\gamma})=\mu(\gamma)$ if
the homology class $\hom{\gamma}$ can be realized by a curve $\gamma$
which has the property that whenever it crosses a leaf of 
$\hat \F_\btheta$, it crosses with positive algebraic sign.
We are following the convention that if the leaf of $\hat \F_\btheta$ is vertical,
then $\gamma$ crosses with positive algebraic sign if it moves rightward across the leaf. This determines the values of $\Psi_{\bm \theta}(\mu)$
on a set which generates $H_1(S, V, \R)$, and we extend linearly.

\begin{lemma}
\label{lem:Psi}
Suppose the coding map given in equation \ref{eq:coding_map} is injective.
Then the map $\Psi_{\bm \theta}$ is also injective. 
\end{lemma}
\begin{proof}
We will utilize Proposition \ref{prop:measure_spaces} to translate the question to the shift space $\Omega$. 
By Carath\'eodory's extension theorem,
a Borel measure on a shift space is determined by the measures of the cylinder sets.
Let $A$ be a cylinder set. Then, $\pi \circ \textit{code}^{-1}(A)$ is a strip of trajectories which hit a sequence
of specified edges of $E$. Such a strip of trajectories is precisely the collection of trajectories which cross
some specific saddle connection $\sigma$ on $S$. (Namely, the left and right sides of the strip must hit at least one vertex of an edge. Let $P$ be a vertex on the left side and $Q$ be a vertex on the right. Then $\sigma=\overline{PQ}$ is a saddle connection contained in the strip.) See Figure \ref{fig:cylinder_set}. Thus, if $\nu$ is a shift invariant measure on $\Omega$, and
$\mu$ is the corresponding transverse measure to $\hat \F_\btheta$, then 
$\nu(A)= \mu (\sigma)$.

\begin{figure}[ht]
\begin{center}
\includegraphics[width=3in]{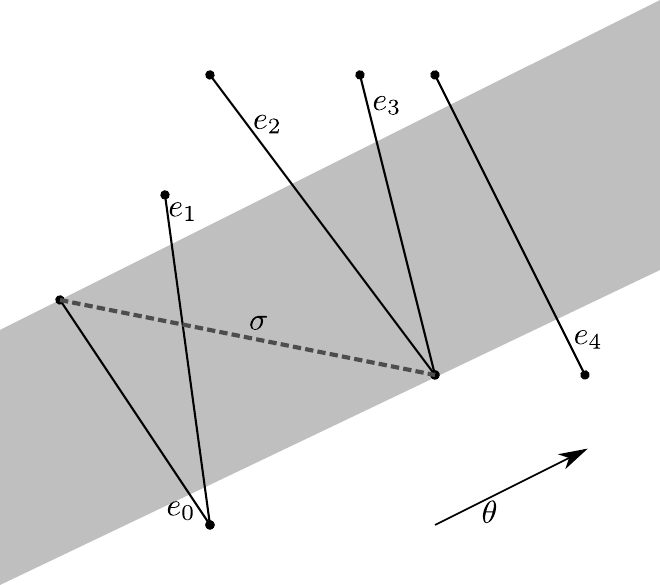}
\caption{This figure indicates how the measure of a cylinder set of $\Omega$ is determined by the measure of a saddle connection. Here the cylinder set is the set of trajectories
which cross the sequence of edges $e_0, \ldots, e_4$. The saddle connection $\sigma$ is chosen so that it spans the width of the strip of trajectories which cross this sequence of edges.}
\label{fig:cylinder_set}
\end{center}
\end{figure}

Observe that for all saddle connections $\sigma$ and all $\mu \in \M_{\bm \theta}$, we have that 
$$\mu(\sigma)=|\Psi_{\bm \theta}(\mu)(\hom \sigma)|.$$
Therefore, we can recover the transverse measures of saddle connections from the image under $\Psi_{\bm \theta}$.
The conclusion follows from Carath\'eodory's extension theorem.
\end{proof}

Recall that Lemma \ref{lem:injective} stated that for a general translation surface, $S$, conservativity and aperiodicity of the straight line flow in direction $\btheta$ implies the injectivity of $\Psi_\btheta$. 

\begin{proof}[Proof of Lemma \ref{lem:injective}]
If the straight line flow has no periodic trajectories and is conservative, then the coding map is injective by Proposition \ref{prop:coding}. 
By Lemma \ref{lem:Psi}, we know $\Psi_{\bm \theta}$ is also injective. 
\end{proof}

Recall that Lemma \ref{lem:sign} stated that for the surface $S=S(\G,\w)$ and for $m \in H^1$, $m \in \Psi_\btheta(\M_\btheta)$ if and only if for every saddle connection $\sigma$ on $S$, 
$$
\label{eq:m temprorary}
m(\hom{\sigma}) \neq 0
\quad \text{implies} \quad
\sgn \big( m(\hom{\sigma})\big)=\sgn \big(\hol(\sigma) \wedge \bm \theta\big).$$

\begin{proof}[Proof of Lemma \ref{lem:sign}]
The ``only if'' part of the lemma follows directly from Definition 
of $\Psi_\btheta$. See page \pagelink{not:Psi}. 
The ``if'' part follows from Carath\'eodory's extension theorem again. Assume $m \in \Coh$ satisfies
equation \ref{eq:m temprorary} for every saddle connection $\sigma$. By Carath\'eodory's extension theorem, we can define a unique measure $\mu$ on $\Omega$ by determining its value on cylinder sets.
Define $f$ on the semi-ring of transversals generated by the saddle connections. We define 
$f(\sigma)=|m(\hom \sigma)|$. This function is finitely additive because of the sign condition.
Thus it extends uniquely to a measure on $\Omega$. The measure is shift invariant by the uniqueness provided by Carath\'eodory's extension theorem. We can use
Proposition \ref{prop:measure_spaces} to pull this measure back to a unique transverse invariant measure $\mu \in \M_\btheta$ so that $\Psi_\btheta(\mu)=m$.
\end{proof}
 
\section{Generalized Farkas' theorem}
\label{sect:farkas}

We will use a generalization of Farkas' theorem given by Craven and Koliha \cite[Theorem 2]{CK77}. We will introduce some of their terminology and then give their result.
Then we will apply this to prove Lemma \ref{lem:farkas}.

If $X$ is a real vector space, its algebraic dual $X^\sharp$ is the collection of linear functionals $X \to \R$. $X^\sharp$ is also a real vector space and we have a bilinear 
pairing $\langle, \rangle:X \times X^\sharp \to \R$ given by $\langle \x, \f \rangle=\f(\x)$. 
A subset $X^+ \subset X^\sharp$ is said to {\em separate points in $X$} if for any distinct $\x_1, \x_2 \in X$ there is a $\f \in X^+$ such that $\langle \x_1, \f \rangle \neq \langle \x_2, \f \rangle$. The pair $\langle X, X^+ \rangle$ is a {\em dual pair} if $X^+$ is a subspace of $X^\sharp$ which separates points in $X$. 
The {\em weak topology} on $X$ with respect to the dual pair $\langle X, X^+\rangle$ is the coarsest topology which makes each functional in $X^+$ continuous.

A {convex cone} in $X$ is a subset $S \subset X$ such that $S+S \subset S$ and $\alpha S \subset S$ for all $\alpha \geq 0$. 
If $\langle X, X^+\rangle$ is a dual pair, then the {\em anticone} of $S$ is 
$$S^+=\{\f \in X^+~:~ \textrm{$\langle\x, \f\rangle \geq 0$ for all $x \in S$}\}.$$

The {\em algebraic adjoint} of a linear map $M:X \to Y$ is the map $M^\sharp:Y^\sharp \to X^\sharp$ defined by
$$\langle \x, M^\sharp (\g) \rangle=\langle M(\x), \g\rangle.$$
If $\langle X, X^+\rangle$ and $\langle Y, Y^+\rangle$ are dual pairs then the linear map $M:X \to Y$ is weakly-continuous if and only if
$M^\sharp(Y^+) \subset X^+$. In this case we define the adjoint $M^+:Y^+ \to X^+$ to be the restriction of $M^\sharp$ to $Y^+$. 

\begin{theorem}[Generalized Farkas' theorem {\cite[Theorem 2]{CK77}}]
Let $\langle X, X^+\rangle$ and $\langle Y, Y^+\rangle$ be dual pairs, let $S \subset X$ be a convex cone, and let $M:X \to Y$ be a weakly-continuous linear map.
If $M(S)$ is weakly-closed then the following are equivalent conditions on $\b \in Y$:
\begin{enumerate}
\item The equation $M \x=\b$ has a solution $\x \in S$.
\item If $\y^+ \in Y^+$ satisfies $M^+(\y^+) \in S^+$ then $\langle \b, \y^+\rangle \geq 0$. 
\end{enumerate}
\end{theorem}

Lemma \ref{lem:farkas} is a special case of this theorem. Consider the dual pair $\langle \R^\V, \R^\V_c\rangle$. Here the weak topologies are simply the topologies of pointwise convergence. We observe that the adjoint of the adjacency operator $\A:\R^\V \to \R^\V$ is just the restriction of $\A$ to $\R^\V_c$. (As in the rest of the paper, we abuse notation by using $\A$ to refer to either of these maps.)
In particular, $\A$ is weakly-continuous. We will prove that the convex cone $\A(\hat \Q_{++})$ is weakly-closed below. The anticone of $\hat \Q_{++} \subset \R^\V$ 
is precisely $\hat \Q_{++} \cap \R^\V_c$. 
By the theorem above, 
given any $\f \in \R^\V$, the following statements are equivalent.
\begin{enumerate}
\item The equation $\A (\g)=\f$ has a solution $\g \in \hat \Q_{++}$.
\item If $\x \in \R^\V_c$ satisfies $\A(\x) \in \hat \Q_{++}$, then $\langle \f, \x \rangle \geq 0$. 
\end{enumerate}
This is precisely the conclusion of Lemma \ref{lem:farkas},
which therefore follows from:

\begin{proposition}
\label{prop:weakly-closed}
The convex cone $\A(\hat \Q_{++}) \subset \R^\V$ is weakly-closed.
\end{proposition}
\begin{proof}
Suppose $\langle \g_{i} \in \A(\hat \Q_{++})\rangle$ is a sequence weakly-converging to $\g_\infty \in \R^\V$.
We must show that $\g_\infty \in \A(\hat \Q_{++})$. We may choose $\f_i \in \hat \Q_{++}$ such that $\A(\f_i)=\g_i$ for all $i$. The idea of the proof is to use a Cantor diagonalization argument to produce a subsequence
of $\langle \f_{i} \rangle$ which converges to some $\f_\infty$, which necessarily lies in $\hat \Q_{++}$. Then, we have
$\A(\f_\infty)=\g_\infty$ by the weak-continuity of $\A$. And therefore $\g_\infty \in \A(\hat \Q_{++})$.

We enumerate $\V=\{\vv_1, \vv_2, \ldots \}$. We will first find a subsequence $\langle \f_{i(1,j)} \rangle_j$ of $\langle \f_{i} \rangle$ such that $\lim_{j \to \infty} \f_{i(1,j)}$ exists. Let $\vw$ be a vertex adjacent to $\vv_1$. 
Then by the formula for $\A(\f_i)$, we know $0 \leq \f_{i}(\vv_1) \leq \A(\f_i)(\vw)=\g_i(\vw)$. Moreover,
$\lim_{i \to \infty} \g_i(\vw) \to \g_\infty(\vw)$. Thus, for all but finitely many $i$ we have 
$$0 \leq \f_{i}(\vv_1) \leq \g_\infty(\vw)+1.$$
Thus by compactness of the interval $[0,g_\infty(\vw)+1]$, we can find a subsequence $\langle \f_{i(1,j)} \rangle_j$ of $\langle \f_{i} \rangle$ such that $\lim_{j \to \infty} \f_{i(1,j)}$ exists. We repeat this argument inductively. For each $n \geq 1$, we can find a subsequence $\langle \f_{i(n+1,j)} \rangle_j$ of $\langle \f_{i(n,j)} \rangle_j$ such that $\lim_{j\to \infty} \f_{i(n+1,j)}(\vv_{n+1})$ exists. 
Then the diagonal sequence $\langle \f_{i(n,n)} \rangle_{n}$ satisfies $\lim_{n \to \infty} \f_{i(n,n)}(\v_j)$ exists for all $j$. We set $\f_\infty=\lim_{n \to \infty} \f_{i(n,n)}$ and proceed as in the previous paragraph.
\end{proof}

\section{The Martin boundaries of a graph}
\label{sect:boundary}
In this section, we will briefly review some relevant facts about the Martin boundaries of the adjacency operator of an infinite connected graph $\G$ with bounded valance. The main goal of this section is to state facts we will use in later appendices.
We will follow the surveys \cite{MW89} and \cite{W00}. The reader is especially encouraged to refer to \cite{MW89}, because it specifically discusses the adjacency operator.
The book \cite{W00} deals exclusively with stochastic matrices. However, our discussion of the adjacency operator can be reduced to the discussion of stochastic matrices; see Remark
\ref{rem:reduction} below.

We may view the adjacency operator as an infinite matrix. For $\vv, \vw \in \V$ and for $n$ a non-negative integer, we define the non-negative number
$$\A^{(n)}_{\vv, \vw}=\langle \A^n(\e_{\vw}), \e_{\vv} \rangle \in \R$$
where $\e_{\vv}$ denote the function $\V \to \R$ which is one at $\vv$ and zero elsewhere.
This means that we can write
$$\A^n(\f)=\sum_{\vv \in  \V} \sum_{\vw \in \V} \A^{(n)}_{\vv, \vw} \f(\vw) \e_{\vv}.$$

Given a $\z>0$, we define the matrix
$\RM_\z=\frac{1}{\z} \sum_{n=0}^{\infty} (\frac{1}{\z} \A)^n$. For $\z$ small this sum will diverge. But there is a real constant $r>0$ for which 
$\RM_\z$ has all finite entries for all $\z>r$, and
$\RM_\z$ has all infinite entries for $\z<r$. 
By definition, this constant $r$ coincides with the spectral radius of the action of $\A$ on $\ell^2(\V)$. 

Note that whenever it exists, the matrix $\RM_\z$ satisfies the equation 
\begin{equation}
\label{eq:rm}
\z\RM_\z=\A \RM_\z+{\mathbf I}.
\end{equation}
From this point of view, the columns $\RM_\z(\e_\vv)$ for $\vv \in \V$ are nearly positive eigenfunctions with eigenvalue $\z$.

We have not explained what happens for $\z=r$. $\A$ is called {\em $r$-transient} if $\RM_r$ has all finite entries, and $\A$ is called {\em $r$-recurrent} if $\RM_r$ has all infinite entries.
No other possibilities can occur. In the $r$-recurrent case, we might try to compute the matrix
$${\mathbf L}=\lim_{n \to \infty} \frac{1}{r^n} \A^n.$$ 
This limit always exists. We say that $\A$ is {\em $r$-null} if all entries of ${\mathbf L}$ are zero, and {\em $r$-positive} if all entries are non-zero. 
Again, no other possibilities can occur. In the $r$-positive case, the columns of ${\mathbf L}$ are positive eigenfunctions. Moreover, since ${\mathbf L}^2={\mathbf L}$,
these eigenfunctions lie in $\ell^2(\V)$. It turns out that all columns are multiples of one another.

\begin{theorem}[{\cite[Theorem 6.2]{MW89}}]
\name{thm:recurrent case}
If $\A$ is $r$-recurrent, then there is a positive solution to the equation $\A(\f)=r \f$. 
This solution is unique up to scaling.
\end{theorem}

The following treats the even more special $r$-positive case.

\begin{theorem}[{\cite[p.~215]{MW89}}]
\name{thm:l2}
$\A$ is $r$-positive if and only if $\A$ has a positive eigenfunction in $\ell^2(\V)$.
If there is such an eigenfunction, its eigenvalue coincides with the spectral radius $r$. 
\end{theorem}

\begin{definition}[Martin kernel]
Suppose that $\z>r$ or that $\z=r$ and $\A$ is $r$-transient. 
Choose a {\em root} vertex $\vo \in \V$. 
The {\em $\z$-Martin kernel} is the matrix satisfying
$$\K_{\z}:\V^2 \to \R; \quad (\vv, \vw) \mapsto \frac{\RM_\z(\e_\vw)(\vv)}{\RM_\z(\e_\vw)(\vo)}.$$
\end{definition}
We may view $\K_{\z}$ as a modification of the matrix $\RM_\z$ where all columns have been rescaled so that the row associated to the root consists of all ones.
Any non-trivial pointwise limit of the columns $\K_\z(\e_\vw)$ of $\K_\z$ produces a positive eigenfunction by equation \ref{eq:rm}.

\begin{definition}[Martin boundary]
\label{def:Martin_boundary}
The {\em $\z$-Martin compactification} $\V_\z \label{not:hat V lambda 2}$ of the vertex set $\V$ is the smallest compactification of  $\V$ to which the function
$$\V \to \R^\V; \quad \vw \mapsto \K_\z(\e_\vw)$$
extends continuously. The {\em $\z$-Martin boundary} is ${\mathcal M}_\z= \V_\z \smallsetminus \V \label{not:M lambda 2}$, and if $\zeta \in {\mathcal M}_\z \label{not:zeta2}$
we use $\k_\zeta \label{not:k zeta2}$ to denote the image of $\zeta$ under this extension.
\end{definition}
It follows from the above discussion that if $\zeta \in {\mathcal M}_\z$, we have $\A \k_\zeta=\z \k_\zeta$.


\begin{definition}[Minimal Martin boundary]
A point $\zeta \in {\mathcal M}_\z$ is called {\em minimal} if whenever $\eta_1, \eta_2 \in {\mathcal M}_\z$ and $0 < t < 1$ satisfy 
$$t \k_{\eta_1}+(1-t)\k_{\eta_2} = \k_\zeta,$$
we have $\zeta=\eta_1=\eta_2$. We use ${\mathcal M}^{\textit{min}}_\z \subset  {\mathcal M}_\z \label{not:M lambda min 2}$ to denote the set of all minimal $\zeta \in {\mathcal M}_\z$.
The subset ${\mathcal M}^{\textit{min}}_\z \subset  {\mathcal M}_\z$ is a Borel subset. 
\end{definition}

From the Poisson-Martin Representation Theorem (Theorem \ref{thm:poisson-martin}), we have the following.

\begin{corollary}
The function ${\mathcal M}_\z \to \R^\V; \zeta \mapsto \k_\zeta$
restricts to a bijection between the minimal Martin boundary and the extremal positive eigenfunctions of $\A$ with eigenvalue $\z$ which take the value $1$ at the root, $\vo$.
\end{corollary}

Since an infinite set with a discrete topology is not compact, we have the following which yields an alternate definition of $r$ in terms of positive eigenfunctions.
\begin{corollary}
There exists a positive function satisfying $\A(\f)=\z \f$ for all $\z > r$
and for $\z=r$ when $\A$ is $r$-transient. 
\end{corollary}

\begin{remark}[Reduction to the stochastic case]
\label{rem:reduction}
Much of the literature on this subject is concerned with {\em stochastic matrices} $\PM$, which are defined to have the property that $\PM (\one)=\one$.
In this case, functions $\w$ satisfying $\PM(\w)=\w$ are called {\em harmonic}.
Suppose we have a positive function $\f \in \R^\V$ satisfying $\A(\f)=\z \f$. Define $\DM$ to be diagonal matrix with $\f(\vv)$ in the diagonal entry associated to $\vv \in \V$. 
Consider the matrix $\PM=\frac{1}{\z} \DM^{-1} \A \DM$. The matrix $\PM$ is easily seen stochastic. Moreover, if $\g$ is another function satisfying $\A(\g)=\z' \g$, then
$\PM(\DM^{-1} \g)=\frac{\z'}{\z} \DM^{-1} \g$. And conversely, if $\h$ satisfies $\frac{\z'}{\z} \PM(\w)=\w$ then $\A(\DM \w)=\lambda' \DM \w$.  Therefore, $\DM$ induces a linear bijection between
the eigenfunctions of $\PM$ with eigenvalue $\frac{\z'}{\z}$ and the eigenfunctions of $\A$ with eigenvalue $\lambda'$. This map also respects the definitions above.
\end{remark}

\section{Cylinder decompositions of translation surfaces}
\name{sect:cylinder decompositions}

Up to the affine group, a surface arises from Thurston's construction if and only if the surface has a pair of cylinder decompositions, each of which supports an affine multi-twist. We use this in later appendices to give a more geometrically natural view of our main results.

Let $S$ be a translation surface as defined in \S \ref{sect:trans}. An {\em (open) cylinder} in $S$ is a subset of $S$ isometric to $\R/k\Z \times (0,h)$. The constants $k$ and $h$ are called the {\em circumference} and {\em height} of the cylinder, respectively. Recall from \S \ref{sect:graphs}, the {\em modulus} of a cylinder is the ratio $\frac{h}{k}$. 
A {\em cylinder decomposition of $S$ in direction $\u \in \Circ \subset \R^2$} is a collection $\sC=\{C_i:u \in {\mathcal I}\}$ of disjoint open cylinders in $S$ whose circumferences are parallel to $\u$ and whose closures cover
$S$. We will call $\sC$ {\em infinite} if $\sC$ is an infinite set.
We call $\sC$ {\em twistable} if there is a positive constant $\kappa$ so that $\kappa m_i \in \Z$ where $m_i$ is the modulus of the cylinder $C_i$ with $i \in {\mathcal I}$. We call the minimal such $\kappa$ the {\em twisting constant}. A surface with a twistable cylinder decomposition supports an affine multi-twist which preserves the partition into cylinders given by the decomposition. See \cite[\S 9]{V}.

\begin{proposition}[Cylinder decompositions and Thurston's construction]
\name{prop:cylinder decomp}
Suppose translation surface $S_0$ has two twistable cylinder directions: a cylinder decomposition $\sC$ in direction $\u \in \Circ$ with twisting constant $\kappa$, and a cylinder decomposition $\sC'$ in direction $\u' \in \Circ$ not parallel to $\u$ with twisting constant $\kappa'$. Then, there is a connected, bipartite ribbon graph $\G$ with bounded valance and a 
positive eigenfunction $\w$ with eigenvalue $\lambda=\sqrt{\kappa \kappa'}~|\u \wedge \u'|$ so
that $S_0$ is affinely equivalent to $S(\G,\w)$. 
Concretely, we have that $S_0$ is translation equivalent
to the image of $S(\G,\w)$ under a linear map $A:\R^2 \to \R^2$ so that
$A(\lambda,0)=\sqrt{\kappa}\u$ and $A(0,\lambda)=\sqrt{\kappa'}\u'$.
\end{proposition}
\begin{proof}[Sketch of proof]
By possibly subdividing the cylinders in the two decompositions, we can assume that all cylinders
in decomposition $\sC$ have modulus $\frac{1}{\kappa}$,
and all cylinders in $\sC'$ have modulus $\frac{1}{\kappa'}$. A calculation reveals that the image of a cylinder in direction $\u$ of modulus $\frac{1}{\kappa}$
under the linear map $A^{-1}$ is a horizontal cylinder
with modulus $\frac{1}{\lambda}$. Similarly, the image
of a cylinder in direction $\u'$ under $A^{-1}$ is a
vertical cylinder with modulus $\frac{1}{\lambda}$.
Then the combinatorics of how the cylinders intersect
on $A^{-1}(S_0)$ determines the ribbon graph $\G$, and the widths of the cylinders determine the eigenfunction $\w$. 
\end{proof}

\begin{definition}
Suppose $S$ has two twistable cylinder decompositions, $\sC$ and $\sC'$, and let $\G$, $\w$, $\lambda$, and $A$ be as above. We will say that a direction $\btheta \in \Circ$ is $(\sC,\sC')$-renormalizable if the vector $A^{-1}(\btheta)$ points in a $\lambda$-renormalizable direction.
\end{definition}

The point of the above is that statements that
hold for $\lambda$-renormalizable directions also
hold for $(\sC,\sC')$-renormalizable directions. The only difference is an affine change of coordinates.
For applying our \hyperref[thm:ergodic2]{Ergodic Measure Characterization Theorem}, we need to consider is when the associated graph has a vertex of valance one. This is dealt with as below:

\begin{proposition}[Valance one condition]
\name{prop:valance one criterion}
Suppose $S$ has two twistable cylinder decompositions, $\sC$ and $\sC'$, with twisting constants
$\kappa$ and $\kappa'$, respectively. Let $\G$ be the graph constructed using Proposition
\ref{prop:cylinder decomp}. Then, $\G$ has a vertex of valance one if and only if one of the following holds:
\begin{itemize}
\item There is a cylinder of $\sC$ which intersects only one cylinder of $\sC'$
counting multiplicity, and the cylinder intersected has modulus $\frac{1}{\kappa'}$.
\item There is a cylinder of $\sC'$ which intersects only one cylinder of $\sC$
counting multiplicity, and the cylinder intersected has modulus $\frac{1}{\kappa}$.
\end{itemize}
\end{proposition}
\begin{proof}[Discussion of proof]
As mentioned in the proof of Proposition \ref{prop:cylinder decomp}, we first subdivide cylinders so that
they have moduli $\frac{1}{\kappa}$ or $\frac{1}{\kappa'}$. The graph is given by the intersection pattern of such cylinders, so the condition that a vertex is $1$-valent is equivalent to the associated cylinder intersecting only one other cylinder counting multiplicity.
\end{proof}

We note that the valance one condition is not stable under subdivision.
A cylinder $\R/k\Z \times (0,h)$ can be cut into two half cylinders, namely $\R/k\Z \times (0,\frac{h}{2})$ and
$\R/k\Z \times (\frac{h}{2},h)$. Given a cylinder decomposition $\sC$, we can get a new cylinder decomposition
in the same direction by cutting each cylinder of $\sC$ in half.
\begin{corollary}
Suppose that $S$ has two twistable cylinder decompositions, $\sC$ and $\sC'$. Suppose the graph obtained
from $\sC$ and $\sC'$ as in Proposition
\ref{prop:cylinder decomp} has a vertex of valance one. Then, the graph associated to the two decompositions
formed by cutting each cylinder in $\sC$ and $\sC'$ in half has no vertices of valance one.
\end{corollary}
\begin{proof}[Sketch of proof]
Observe that the neither criterion of Proposition \ref{prop:cylinder decomp} can hold for the subdivided cylinder decomposition.
\end{proof}

In particular, our \hyperref[thm:ergodic2]{Ergodic Measure Characterization Theorem} can still be used even if a graph $\G$ has a vertices of valance one. We must subdivide cylinders, which has the effect of changing $\G$ and decreasing the set of renormalizable directions. This set still has Hausdorff dimension larger than $\frac{1}{2}$. (See Remark \ref{rem:renormalizable directions}.) 
\section{Infinite Interval Exchange Transformations}
\name{sect:skew rotations}
In this appendix, we describe some natural infinite interval exchange maps to which our results apply.
Given any surface produced from Thurston's construction, the return map of the straight-line flow $F^t_\btheta$ to the horizontal boundaries of rectangles is an infinite interval exchange map. So, our results hold generally for such infinite IETs which arise from flows in renormalizable directions. But, in this appendix we will describe some examples which seem more natural from the point of view of infinite IETs which have been studied by other authors.

\subsection{Skew products}
\name{sect:skew products}
A well studied class of infinite IETs comes from the construction of skew product transformations.
Let $\tau:I \to I$ be a finite IET, let $\Grp$ be a countably infinite discrete group, and let $\psi:I \to \Grp$ be a function which is locally constant away from finitely many discontinuities. Then we define the {\em skew product} of $\tau$ and $\psi$ to be
\begin{equation}
\label{eq:skew_def}
T:I \times \Grp \to I \times \Grp \quad \textrm{defined by} \quad T(x,g)=\big(\tau(x), \psi(x)g\big).
\end{equation}
We consider $T$ an infinite IET, because $T$ is an orientation preserving piecewise isometry of a space piecewise isometric to an interval in $\R$. 

\subsection{Maharam measures}
\name{ss:Maharam}
Let $T$ be a skew product of $\tau$ and $\psi$ as above. Let $h:I \to \R_+$ be a Borel measurable map to the positive real numbers. We call a probability measure $\mu$ an {\em $(h, \tau)$-conformal measure} if 
$\mu \circ \tau$ is absolutely continuous with respect to $\mu$ and the Radon-Nikodym derivative satisfies
$\frac{d \mu \circ \tau}{d \mu}(x)=h(x)$ for $\mu$-a.e. $x \in I$.

Let $\chi:\Grp \to \R$ be a group homomorphism, and suppose that $\mu$
is a $(e^{\chi \circ \psi}, \tau)$-conformal Borel probability measure. 
For each $g \in \Grp$, consider the maps $\pi_g:I \to I \times \Grp$ given by $\pi_g(x)=(x,g)$. 
The {\em $\chi$-Maharam measure} associated to $\mu$ and $\chi$ is the measure $\widetilde \mu_\chi$ on $I \times \Grp$ defined so that
\begin{equation}
\label{eq:maharam intro}
\widetilde \mu_\chi \circ \pi_g = \frac{1}{e^{\chi(g)}} \mu.
\end{equation}
Such a measure is always invariant under the skew product transformation $T$. The scaling factor of $\chi$ from $(e^{\chi \circ \psi}, \tau)$-conformality cancels with the $\frac{1}{e^{\chi}}$ factor that arises from the change in the $\Grp$-coordinate under $T$.
Note that some scalar multiple of Lebesgue measure is $\chi$-Maharam measure where $\chi$ is the trivial group homomorphism.
Note also that Maharam measures are normalized in the sense that $\widetilde \mu(I \times \{e\})=1$, where $e \in \Grp$ denotes the identity element.

\subsection{Skew rotations}
\name{ss:skew_rotations}
A {\em skew rotation} is the skew product of a rotation $\tau:[0,1) \to [0,1)$ given by $\tau(x)=x+\alpha \pmod{1}$ and some $\psi:I \to \Grp$ as described above. Let $n \geq 2$ be an integer and choose generators $\gamma_1, \ldots, \gamma_n \in \Grp$. 
We will consider the special case when 
$\psi$ is defined so that $\psi(x)=\gamma_i$ if and only if $x \in [\frac{i-1}{n}, \frac{i}{n})$. In general, such a skew rotation is not even recurrent. (Consider the case when $\gamma_1, \ldots, \gamma_n$ freely generate.) It is therefore natural to impose a no-drift condition. We choose to highlight
the no-drift condition 
\begin{equation}
\name{eq:relation}
\gamma_n \gamma_{n-1} \ldots \gamma_1=e,
\end{equation}
in this paper, but other choices could be made to produce similar results. 

The best studied system of this form is the case when $\Grp=\Z$, $n=2$, $\gamma_1=1$ and $\gamma_2=-1$.
Ergodicity of this skew rotation was proved for some irrational $\alpha$ by Schmidt 
\cite[Theorem 2.6]{S78}, and later shown to hold for all irrational $\alpha$ by Conze and Keane 
\cite{CK76}. In \cite[Theorem 1.4]{ANSS02}, the following was proved:

\begin{theorem}[Aaronson-Nakada-Sarig-Solomyak]
\label{thm:ANSS}
Let $T_\alpha$ be the skew rotation where $\alpha$ is irrational, $\Grp=\Z$, $n=2$, $\gamma_1=1$ and $\gamma_2=-1$.
Then,
\begin{enumerate}
\item For every group homomorphism $\chi$, there is a unique $\chi$-Maharam measure which is invariant under $T_\alpha$.
\item Each Maharam measure for $T_\alpha$ is ergodic.
\item All locally finite ergodic $T_\alpha$-invariant measures are scalar multiples of Maharam measures.
\end{enumerate}
\end{theorem}

In Appendix \ref{app:nilpotent}, we will prove the following generalization.

\begin{theorem}[Nilpotent case]
\name{thm:nilpotent}
Let $\Grp$ be a nilpotent group generated by $\gamma_1, \ldots, \gamma_n$ and satisfying equation \ref{eq:relation}. 
Then statements (1)-(3) of the above theorem hold for the corresponding skew rotation $T_\alpha$ whenever the unit vector in direction $(\alpha-\frac{1}{n}, \frac{1}{n})$ is $n$-renormalizable.
\end{theorem}

This theorem implies Theorem \ref{thm:ANSS}, because the set of $2$-renormalizable directions
is the directions of irrational slope. For $n>2$, the above theorem
gives all but countably many $\alpha$ in a Cantor set of Hausdorff dimension bigger than $\frac{1}{2}$. 
(See Remark \ref{rem:renormalizable directions} for a description of the sizes of the set of $n$-renormalizable directions.)

\subsection{Skew rotations from translation surfaces}
The goal of this section is to explain that the skew rotations $T_\alpha$ defined in the prior subsection
appear as return maps to a section of the straight line flow on an infinite translation surface. The definition of $T_\alpha$ requires the choice of an infinite discrete group $\Grp$ and a choice of generators $\gamma_1, \ldots, \gamma_n$ so that $\gamma_n \ldots \gamma_1=e$. We then define
$\tau:[0,1) \to [0,1)$ to be a rotation and $\psi:[0,1) \to \Grp$ to be so that 
\begin{equation}
\name{eq:psi}
\psi(x)=\gamma_i \quad \textrm{if $x \in [\frac{i-1}{n}, \frac{i}{n})$.}
\end{equation}
This determines a skew product $T:[0,1) \times \Grp \to [0,1) \times \Grp$ as in equation \ref{eq:skew_def}.

We will use the same data to define a translation surface. 
Let $C$ denote the cylinder $\R/\Z \times [0,1/n]$ with a decomposition into $n$ squares and $n$ top and bottom edges labeled $t_1, \ldots, t_n$ and $b_1, \ldots, b_n$.
See Figure \ref{fig:b3}.
We apply the label $t_i$ to the segment $[\frac{i-1}{n},\frac{i}{n}] \times \{\frac{1}{n}\} \subset C$ and 
$b_i$ to the segment $[\frac{i-1}{n},\frac{i}{n}] \times \{0\} \subset C$. 
We let 
$S$ be the translation surface $\Grp \times C/\sim$ where $\sim$ is a gluing of edges. For each $i\in\{1, \ldots, n\}$ and $g \in \Grp$, we glue edge $t_i$ of cylinder $\{g\} \times C$ to edge $b_{i+1 \pmod{n}}$ of cylinder $\{\gamma_i g\} \times C$ by parallel translation. Our surface has a decomposition into vertical cylinders as well, because a flow in the vertical direction results in passing through edges in the order $t_i, t_{i+1}, \ldots, t_{i+n-1}$ with subscripts written modulo $n$. This results in visiting a list of horizontal cylinders of the form 
\begin{equation}
\label{eq:vertical_crossing}
\{g\} \times C, ~\{\gamma_i g\} \times C, ~\{\gamma_{i+1} \gamma_i g\} \times C, ~\ldots, ~\{\gamma_{i+n-1}\ldots \gamma_i g\} \times C.
\end{equation}
By our assumed group relation, $\gamma_{i+n-1} \ldots \gamma_{i+1} \gamma_i g=g$
and the vertical trajectory returns to its starting point after crossing $n$ horizontal edges. That is, we have a vertical decomposition into cylinders of inverse modulus $n$ as well. In particular, $S$ is of the form 
$S(\G, \w_{\frac{1}{n}})$, where $\G$ is the valance $n$ cylinder intersection graph and $\w_{\frac{1}{n}} \in \R^\V$ is the constant function with value $\frac{1}{n}$. This is an eigenfunction of eigenvalue $\lambda=n$.  
Note that, if we choose the trajectory to start on the edge labeled $t_1$ of ${g} \times C$, then it passes through the cylinders
\begin{equation}
\label{eq:vertical_crossing2}
\{\eta_1 g\} \times C, ~\{\eta_2 g\} \times C, ~\ldots, ~\{\eta_n g\} \times C,
\end{equation}
where the group elements $\eta_i$ are defined to be
\begin{equation}
\name{eq:eta}
\eta_1=e, \quad \eta_2=\gamma_1, \quad \textrm{and} \quad 
\eta_i=\gamma_{i-1} \ldots \gamma_2 \gamma_1\textrm{ for $2 \leq i \leq n$.}
\end{equation}

\begin{figure}[ht]
\begin{center}
\includegraphics[width=2in]{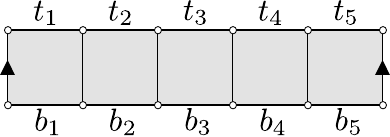}
\caption{The horizontal cylinder $C$ when $n=5$.}
\label{fig:b3}
\end{center}
\end{figure}

We will describe the graph $\G$ obtained by this construction. 
The horizontal cylinders and hence the nodes of the graph in $\Alpha$ are identified with the group $\Grp$. We use $\va_g \in \Alpha$ to denote the node associated to $g \in \G$. If $x$ is a point of $t_1$ the vertical straight line flow crosses the cylinders as described by equation \ref{eq:vertical_crossing2}. Every vertical cylinder passes through some $t_1$. Thus the vertical cylinders can by identified by an ordered $n$-tuple of the form
\begin{equation}
\name{eq:g}
[g]=(\eta_1 g, \eta_2 g, \ldots, \eta_n g) \in \Grp^n.
\end{equation} 
We denote the vertex associated to this cylinder by $\vb_g \in \Beta$. Finally the collection of edges is given by the condition 
\begin{equation}
\label{eq:skew_edge_condition}
\va_g \sim \vb_{g'} \textrm{ if and only if $g=\eta_i g'$ for some $1 \leq i \leq n$.}
\end{equation}
Hence edges correspond to the choice of a $n$-tuple of $[g]=(\eta_1 g, \eta_2 g, \ldots, \eta_n g)$ and an element $\eta_i g$ of the $n$-tuple. This choice refers to the edge
$\overline{\va_{\eta_i g}~\vb_g}$. 
The ribbon graph structure is given by 
$$\North\big(\overline{\va_{\eta_i g}~\vb_g}\big)=\overline{\va_{\eta_{i+1} g}~\vb_g} \quad \textrm{and} \quad
\East\big(\overline{\va_{\eta_i g}~\vb_g} \big)=\overline{\va_{\eta_i g}~\vb_{\eta_{i+1}^{-1} \eta_i g}},$$
where subscript addition is taken modulo $n$. 

\begin{proposition}
\name{prop:graph}
Let $\Grp$ be a discrete group with generators $\gamma_1, \ldots, \gamma_n \in \Grp$ satisfying the relation
given in equation \ref{eq:relation}. Define $\G$ from this data as in the above paragraphs.
Consider embedding $\phi:\R/\Z \times \{g\} \to S(\G, \w_{\frac{1}{n}})$ which sends the interval 
$[\frac{i-1}{n}, \frac{i}{n}] \times \{g\}$ to the top edge $t_i$ of the cylinder $\{g\} \times C$. 
Let $T$ be the skew rotation determined by a rotation $\tau(x)=x+\alpha \pmod{1}$ and the map $\psi:[0,1) \to \Grp$ as in equation \ref{eq:psi}.
Then the embedding $\phi$ conjugates the skew rotation $T$ to the return map of the straight-line flow $\flow^t$ in direction $\bm \theta=(\alpha-\frac{1}{n}, \frac{1}{n})/ \|(\alpha-\frac{1}{n}, \frac{1}{n})\|.$ 
\end{proposition}
\begin{proof}
Consider the point $(x,g) \in [0,1) \times \Grp$. We must show that $\phi \circ T(x,g)$ is the image of $\phi(x,g)$ under the return map of the straight-line flow. Assume $x \in [\frac{i-1}{n}, \frac{i}{n})$. Observe that $\bm \theta$ has positive $y$-coordinate. So flowing $\phi(x,g)$ in the direction of $\bm \theta$ immediately moves into the cylinder $\{\gamma_i g\} \times C$ through the bottom edge $b_{i+1}$, with subscript addition taken modulo $n$. 
The point $\phi(x,g)$ is identified with the point with coordinates $\big(x+\frac{1}{n},0\big)$ of $\{\gamma_i g\} \times C$, where the $x$-coordinate is taken modulo $1$. 
Continuing flowing in direction $\bm \theta$ we reach the point $(x+\alpha,\frac{1}{n})$ on the top of the cylinder $\{\gamma_i g\} \times C$. This point coincides with 
$\phi \circ T(x,g)$ as desired. 
\end{proof}

As a consequence, we see that whenever the direction $(\alpha-\frac{1}{n}, \frac{1}{n})/ \|(\alpha-\frac{1}{n}, \frac{1}{n})\|$ is $n$-renormalizable,
we can say something about the measures of the associated infinite IET. For later use, we give this set of $\alpha$ a name:
\begin{equation}
\name{eq:Omega}
\Omega_n=\{\alpha~:~\textrm{$(\alpha-\frac{1}{n}, \frac{1}{n})/ \|(\alpha-\frac{1}{n}, \frac{1}{n})\|$ is $n$-renormalizable}\}.
\end{equation}
Recall that $\lambda$-renormalizable directions were defined in \S \ref{ss:renormalizable directions}.

To conclude this section, we work out a special case of our skew rotation. We consider the case when our no-drift relation is the only relation in our group. This will lead to a primary example considered in Appendix \ref{sect:hyperbolic}. See Theorem \ref{thm:free}.

\begin{proposition}
\name{prop:tree}
Suppose $\Grp=\langle \gamma_1, \ldots, \gamma_n~|~\gamma_n \ldots \gamma_1=e\rangle$.
Then the associated bipartite graph $\G$ constructed above is the valance $n$ tree. 
\end{proposition}
\begin{proof}
From the above remarks we know $\G$ is $n$-valent. We must show that the graph contains no homotopically non-trivial loops. 
Observe that the elements $\eta_2, \ldots, \eta_{n}$ defined in equation \ref{eq:eta} freely generate the group $\Grp$, while $\eta_1=e$. 
Note that there is a unique way to write each element $g \in \Grp$
as a product of the generators $\eta_2, \dots, \eta_{n}$ and their inverses which minimizes the word length. 
Moreover, if we have a word written as a product of $\eta_2, \ldots, \eta_{n}$ and their inverses, then this product is the minimal one unless there
are is a pair of adjacent terms of the form $\eta_i \eta_i^{-1}$ or  $\eta_i^{-1} \eta_i$ with $i \in \{2, \ldots, n\}$. 

Recall that $\Alpha$ is identified with $\Grp$ and points in $\Beta$ correspond to $n$-tuples in $\Beta$ of the form $[g]=(\eta_1 g, \ldots, \eta_{n-1} g, g)$. See equation \ref{eq:g}. 
The elements of an $n$-tuple in $\Beta$ correspond to adjacent vertices in $\Alpha$.

Consider a non-backtracking path in $\G$ starting at $\va_e$. Denote this sequence
$$\va_0 \sim \vb_0 \sim \va_1 \sim \vb_1 \sim \va_2 \sim \ldots.$$ 
By non-backtracking we mean that $\va_i \neq \va_{i+1}$ and $\vb_i \neq \vb_{i+1}$ for all $i$. 
For each $i \geq 0$, there is a group element $g_i$ so that $\va_i=\va_{g_i}$. Similarly, let $h_i \in \Grp$ be so that $\vb_i=\vb_{h_i}$. 
We claim that the word length (measured with respect to the generators
$\eta_2, \ldots, \eta_n$) of $g_i$ is strictly increasing in $i$. 
If this is true, then a non-backtracking path can not close up. So, there are no homotopically non-trivial loops. 
We prove this by induction. Since $g_0=e$, we see that $h_0=\eta_{i(0)}^{-1}$ for some choice of $i(0) \in \{1, \ldots, n\}$. Then 
we have that $g_1$ lies in the list $[h_0]$. Therefore, we have $g_1=\eta_{j(0)} \eta_{i(0)}^{-1}$ for some $j(0)$. Since $g_1 \neq g_0$, we must have that
$i(0) \neq j(0)$. Since only one of $i(0)$ or $j(0)$ can equal $1$, we see $g_1 \neq e$ so the word length has gone up. More generally, we see that for each $m$, there are distinct $i(m+1)$ and $j(m+1)$ so that 
$$h_m=\eta_{i(m)}^{-1} g_{m-1} \quad \textrm{and} \quad g_{m+1}=\eta_{j(m)} \eta_{i(m)}^{-1} g_{m-1}.$$
Also, $h_{m}=\eta_{i(m)}^{-1} \eta_{j(m-1)} h_{m-1}$ and $h_{m} \neq h_{m-1}$ implies that $i(m) \neq j(m-1)$ for all $m$. 
Therefore, for each $m$ we can write
$$g_{m+1}=\eta_{j(m)} \eta_{i(m)}^{-1} \ldots \eta_{j(0)} \eta_{i(0)}^{-1},$$
and we have $j(k) \neq i(k)$ and $i(k+1) \neq j(k)$ for all $k$. 
We claim that aside from removing terms of the form $\eta_1=e$ and $\eta_1^{-1}=e$ there can be no cancellation to reduce the word length. No adjacent terms can be canceled because $j(m) \neq i(m)$ and $i(m+1) \neq j(m)$ for all $m$. Moreover,
we can have no ``canceling sandwiches'' of the form $\eta_i \eta_1^{\pm 1} \eta_i^{-1}$ or $\eta_i^{-1} \eta_1^{\pm 1} \eta_i$, because the sign of the exponent in the terms in the product is alternating. Therefore, all simplification is simply the removal of terms of the form $\eta_1^{\pm 1}$ as claimed.
Now we see that the word length of $g_{m}$ is at least one larger than $g_{m-1}$, since $g_{m}=\eta_{j(m)} \eta_{i(m)}^{-1} g_{m-1}$
and we have either $j(m) \neq 1$ or $i(m) \neq 1$. This proves the increasing word length claim.
\end{proof}

\section{Translation surfaces and hyperbolic graphs}
\name{sect:hyperbolic}

A graph is called {\em hyperbolic} if it is $\delta$-hyperbolic as a metric space equipped with the edge metric for some $\delta>0$. We refer the reader to \cite[\S 8.4]{BBI} for the definition of and background for $\delta$-hyperbolicity. If $\G$ is hyperbolic, we can compactify $\G$ with the $\delta$-hyperbolic boundary $\partial_{\textit{hyp}} \G$.

\begin{theorem}[{\cite[Theorem IV.27.1]{W00}}]
\label{thm:hyperbolic}
Suppose the graph $\G$ is hyperbolic, and let $\z > r$. Then,
every point in ${\mathcal M}_\z$ is minimal and ${\mathcal M}_\z$ is homeomorphic to $\partial_{\textit{hyp}} \G$.
\end{theorem}

\begin{example}
\name{ex:tripod}
Consider the graph $\G$ given in Figure \ref{fig:triple staircase}. The eigenfunction $\w$ given in the figure lies in $\ell^p(\V)$ for $p \in [1,\infty]$. We conclude that $\A$ is $r$-positive and that the spectral radius is the associated eigenvalue $r=\frac{3 \sqrt{2}}{2}$. By Theorem \ref{thm:recurrent case}, we conclude that the only positive 
functions satisfying $\A(\f)=r \f$ are the multiples of the function $\w$. For $\z>r$,
we may apply Theorem \ref{thm:hyperbolic}. In the case of a tree, the hyperbolic boundary of $\G$ is homeomorphic to the space of ends of $\G$. In this case, we have exactly $3$ ends. We conclude the space of positive functions satisfying $\A(\f)=\z \f$ is linearly isomorphic to the cone on a triangle; we have $3$ such extremal eigenfunctions up to scaling. One such function is given in Figure \ref{fig:eig3}; the others are the same up to the automorphism group of the graph. The space of all extremal positive eigenfunctions which take the value one at the root is homeomorphic to the graph itself. Using our measure characterization, this information determines the ergodic invariant measures
for the straight-line flow in a $r$-renormalizable direction on the surface $S(\G,\w)$ of Figure \ref{fig:triple staircase}.
\end{example}

\begin{figure}[ht]
\begin{center}
\includegraphics[width=6in]{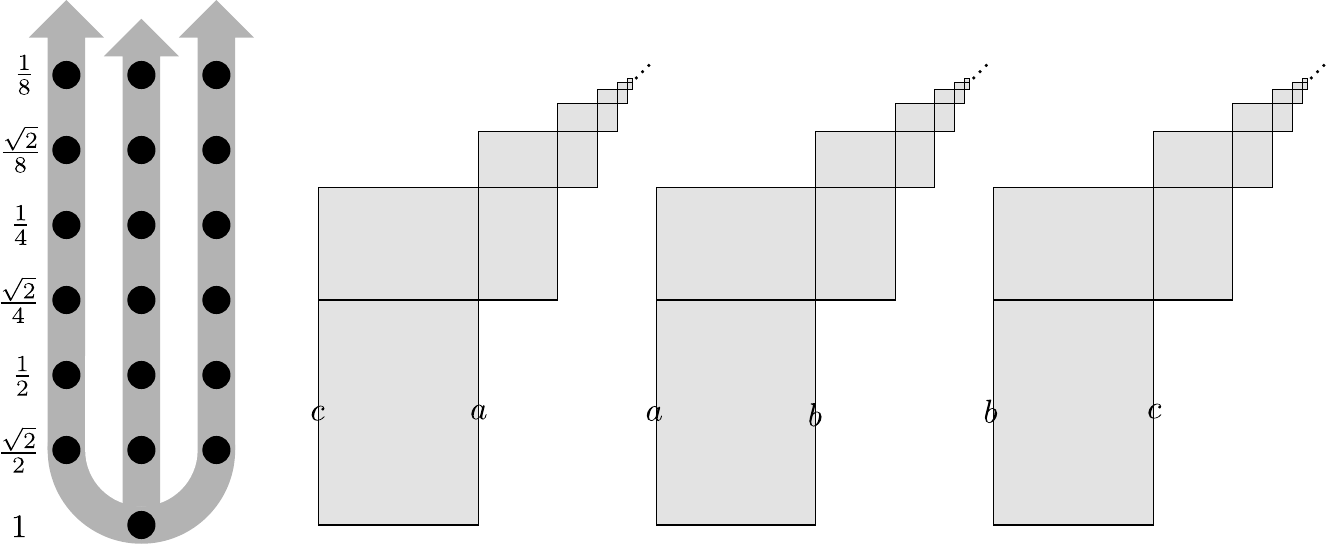}
\caption{A surface $S(\G, \w)$ with finite area. Labeled and opposite unlabeled edges are glued by horizontal or vertical translations.
At left, the graph $\G$ is shown with the positive eigenfunction $\w$
which lies in $\ell^p(\V)$ for all $p \in [1,\infty]$.}
\name{fig:triple staircase}
\end{center}
\end{figure}

\begin{figure}[ht]
\begin{center}
\includegraphics[width=4in]{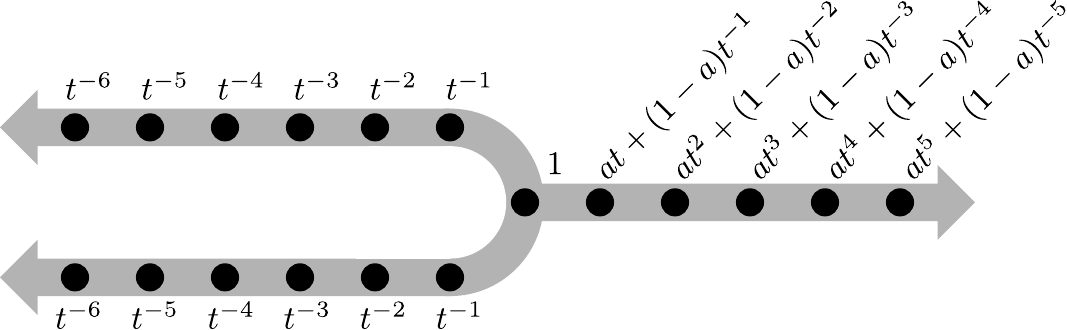}
\caption{An extremal positive eigenfunction is given above with eigenvalue $\lambda=t+t^{-1}$ when we set $a=\frac{t-2t^{-1}}{t-t^{-1}}$.}
\label{fig:eig3}
\end{center}
\end{figure}

The theorem above leaves open the question of what the Martin boundary is in the case that 
$\z=r$ and $\A$ is $r$-transient. In the special case of trees, we have an answer.

\begin{theorem}[{\cite[p.~459]{PW88}}]
\name{thm:tree}
If $\G$ is a tree and $\A$ is $r$-transient with $\z \geq r$ or if $\A$ is $r$-recurrent and $\z > r$, then every point in ${\mathcal M}_\z$ is minimal and ${\mathcal M}_\z$ is homeomorphic to the space of ends of $\G$ (which is also homeomorphic to $\partial_{\textit{hyp}} \G$).
\end{theorem}

\begin{remark} As far as the author knows, there is no general description of the Martin boundary of a hyperbolic graph $\G$ in the case that $\z=r$,
even in the particular case when $\G$ is the Cayley graph of a Gromov hyperbolic group.
\end{remark}

\begin{theorem}[Non-abelian free case]
\name{thm:free}
Let $n \geq 3$ and let $\Grp$ be the non-abelian free group of rank $n-1$
generated by $\gamma_1, \ldots, \gamma_{n-1}$ and define $\gamma_n=(\gamma_{n-1} \gamma_{n-2} \ldots \gamma_1)^{-1}$ 
so that the no-drift condition of equation \ref{eq:relation} is satisfied.
Let $\Omega_n \subset [0,1)$ be as in equation \ref{eq:Omega}.
Then, for every  $\alpha \in \Omega_n$, the following holds for the associated skew product $T_\alpha$:
\begin{enumerate}
\item $T_\alpha$ is conservative but no Maharam measure is ergodic (including Lebesgue measure). 
\item Points in the projectivization of the space of locally finite ergodic $T_\alpha$-invariant measures are 
in bijective correspondence with the Cartesian product of a ray $[0, \infty)$ and the Gromov boundary of $\Grp$. 
\end{enumerate}
\end{theorem}
\begin{proof}[Proof of Theorem \ref{thm:free}]
By Propositions \ref{prop:graph} and \ref{prop:tree}, the skew product $T_\alpha$ with $\alpha \in \Omega_n$ described in the theorem is the return map of the straight-line flow $\flow^t$ in a direction $\bm \theta \in \Rn_\lambda$ on the surface $S(\G, \w_{\frac{1}{n}})$, where $\G$ is the valence $n$-tree. By Theorem \ref{thm:pr},
this flow is conservative. Theorem \ref{thm:ergodic2} characterizes the locally finite ergodic transverse measures to the foliations $\F_{\bm \theta}$.
Such measures are in bijective correspondence with positive eigenfunctions of the adjacency operator. The spectral radius of $\G$ is $r=2 \sqrt{n-1}$ (see \cite[p. 225]{MW89})
and the graph is known to be $r$-transient when $n>2$ (see \cite[p. 10]{W00}). Therefore, Theorem \ref{thm:tree} characterizes the 
positive eigenfunctions (normalized to take the value $1$ at the root node) in terms of the choice of a $\lambda \geq r$ and the choice of a point
in the Gromov boundary of the graph. This bijection then extends to projective equivalence classes of locally finite ergodic invariant transverse measures by Theorem \ref{thm:ergodic2} and to the skew rotation by restricting these measures to the section of horizontal edges of rectangles.
\end{proof}

\section{Surfaces with cocompact nilpotent actions}
\label{app:nilpotent}

\subsection{Eigenfunctions of graphs with nilpotent actions}
For the following is a variant of a theorem of Margulis \cite{M66}. 
\begin{theorem}[Margulis]
\name{thm:margulis2}
Let $\Grp$ be a nilpotent group acting cocompactly by graph automorphisms on the graph $\G$. Let 
$\f \in \R^\V$ be an extremal positive eigenfunction for the adjacency operator. Then for all $g \in \Grp$, the
quantity $\f \circ g(\vv)/\f(\vv)$ is independent of the choice of $\vv \in \V$. Moreover,
the function $h:\Grp \to \R$ given by $h(g)=\log \big(\f \circ g(\vv)/\f(\vv)\big)$ is a group homomorphism to $\R$. 
The map $\f \mapsto h$ is a bijection from extremal positive eigenfunctions which take the value $1$ at an arbitrary chosen root node of $\G$ to 
the collection of homeomorphisms $\Grp \to \R$. 
\end{theorem}

Our formulation of the result does not appear in the literature. However, it follows quickly from the following version:

\begin{theorem}[{\cite[Theorem 25.8]{W00}}]
\name{thm:margulis1}
Let $\G$ be a locally finite, connected graph, and let $P$ be a stochastic matrix determining a random walk on $\G$. Suppose the automorphism group of the pair $(\G,P)$ contains a discrete nilpotent group $\Grp$ which acts co-compactly on $\G$. Then, for any extremal positive eigenfunction $\f:\R^\V \to \R$ and every $\vv \in \V$, the function $h_\vv:\Grp \to \R$ given by $h_\vv(g)=\log \big(\f \circ g(\vv)/\f(\vv)\big)$ is a group homomorphism.
\end{theorem}

\begin{proof}[Discussion of proof of Theorem \ref{thm:margulis2}]
This theorem is sufficiently different from Theorem \ref{thm:margulis1} that it warrants some discussion. First of all, Woess' statement involves
stochastic matrices, but this can be resolved by Remark \ref{rem:reduction}. Second, Woess only states that for each $\vv$, the function 
$$h_\vv(\vv)=\log \big(\f \circ g(\vv)/\f(\vv)\big)$$
is a group homomorphism. To see that this is independent of the choice of $\vv$, take an adjacent vertex $\vw$. We will use the fact that since
$\A \f=\lambda \f$, whenever $\va \sim \vb$ we have $\lambda^{-1} \f(b) \leq \f(a) \leq \lambda \f(b)$. Since automorphisms preserve adjacency, for all $g \in \Grp$ we have
$$\frac{\lambda^{-1} \f \circ g(\vw)}{\lambda \f(\vw)} \leq \frac{\f \circ g(\vv)}{\f(\vv)} \leq
\frac{\lambda \f \circ g(\vw)}{\lambda \f(\vw)}$$
Therefore we have $h_\vw(g)-2\log \lambda \leq h_\vv(g) \leq h_\vw(g)+2\log \lambda.$ This equation is independent of the choice of $g$,
so we may apply it to $g^n$ for all $n$. Then we may use the fact that $h_\vv(g^n)=n h_\vv(g)$ and $h_\vv(g^n)=n h_\vv(g)$ to say that
$$n h_\vw(g)-2\log \lambda \leq n h_\vv(g) \leq n h_\vw(g)+2\log \lambda \quad \textrm{for all $n$.}$$
Therefore $h_\vv(g)=h_\vw(g)$. 

Finally, we need to say something about existence and uniqueness of the positive eigenfunction associated to a group homomorphism. 
Fix a group homomorphism $h:\Grp \to \R$. Then consider the subspace of $\R^\V$ given by 
$$L=\{\f \in \R^\V~:~\textrm{$\f \circ g(\vv)=e^{h(g)} \f(\vv)$ for all $g \in \Grp$}\}.$$
Since $\Grp$ is acting by graph automorphisms, this subspace is $\A$-invariant. Moreover, $L$ has dimension equal to the number of elements 
of $\Grp \backslash \G$. Choose representatives $\vv_1, \ldots, \vv_k$ for the orbit equivalence classes
in $\Grp \backslash \G$. The functions $\f_1, \ldots, \f_k$ defined so that $\f_i(\vv_j)=1$ if $i=j$ and $\f_i(\vv_j)=0$ otherwise
form a basis for $L$. Observe that a function in $L$ is positive if and only if it can be written as a positive linear combination of $\f_1, \ldots, \f_k$. 
Finally, observe that $\A$ acts as a Perron-Frobenius matrix in this basis. Therefore, there is a unique positive eigenvector up to scaling. 
\end{proof}

\subsection{Nilpotent covers of translation surfaces}
\name{sect:covers}
Suppose that $S$ is an infinite translation surface, and let $\Grp$ be a discrete group which acts faithfully on $S$ by homeomorphisms which are translations in local coordinates. Then, $S/\Grp$ is also a translation surface.
If $S/\Grp$ is a closed surface, we call $S$ a {\em $\Grp$-cover of the translation surface $S/\Grp$}. We will let $D$ be a measurable fundamental domain for the $\Grp$-action. 

Let $\chi:\Grp \to \R$ be a group homomorphism. We call an $F_{\bm \theta}^t$-invariant measure, $\nu$, 
{\em $\chi$-Maharam} if $\nu(D)=1$ and for each $g \in \Grp$, we have $\nu \circ g=\frac{1}{e^{\chi(g)}} \nu$.
We call $\nu$ Maharam because the induced invariant measure for the return map to a periodic section is Maharam in the sense of \S \ref{ss:Maharam}.
Note that some scalar multiple Lebesgue measure is $\chi$-Maharam when $\chi$ is the trivial homomorphism.

\begin{theorem}[Nilpotent covers]
\name{thm:nilpotent cover}
Suppose that the connected translation surface $S$ is a $\Grp$-cover of a closed translation surface, where $\Grp$ is a discrete nilpotent group. Further suppose that $\sC$ and $\sD$ have two twistable cylinder decompositions and that the associated intersection graph has no vertices of valance one (see Proposition \ref{prop:valance one criterion}). Then, in any 
$(\sC, \sD)$-renormalizable direction $\bm \theta$, the following statements are satisfied. 
\begin{enumerate}
\item For every homomorphism $\chi:\Grp \to \R$, there is a unique $\chi$-Maharam measure of $F_{\bm \theta}^t$.
\item The collection of locally finite ergodic $F_{\bm \theta}^t$-invariant measures is the collection of scalar multiples of the Maharam measures.
\end{enumerate}
\end{theorem}
\begin{proof}[Discussion of proof]
We begin by showing that all locally finite ergodic $F^t_\btheta$-invariant measures are $\chi$-Maharam,
and there is one for each $\chi$.
Since our surface has two twistable cylinder decompositions, we may assume that (up to an affine change of coordinates) $S=S(\G,\w)$ for some graph $\G$ and positive eigenfunction $\w$. Since $\Grp$ acts on $S$
by translation symmetries, it induces a $\Grp$ action on $\G$. The quotient surface $S/\Grp$ inherits a pair of quotient
cylinder decompositions whose intersection data is given by the quotient graph $\G/\Grp$. Since $S/\Grp$ is a compact translation surface, we know that $\G/\Grp$ is a finite graph. Then since $\Grp$ is nilpotent,
Theorem \ref{thm:margulis2} gives a description of the extremal positive eigenfunctions of the adjacency operator. Namely, they are in bijection with the collection of group homomorphisms $h:\Grp \to \R$. 
Fix such a $\chi$. Then, by Theorem \ref{thm:margulis2}, there is an extremal positive eigenfunction so that 
$h(g)=\log \big(\f \circ g(\vv)/\f(\vv)\big)$ for all $g \in \Grp$ and all vertices $\vv$ of $\G$. 
The \hyperref[thm:ergodic2]{Ergodic Measure Characterization Theorem} guarantees that the locally finite ergodic invariant measures arise from a pullback construction from surfaces built from such extremal positive eigenfunctions. (That is, we pullback the transverse measures. The associated $F_{\bm \theta}^t$-invariant measures is locally
a product of this transverse measure and Lebesgue measure in the orbit direction.)
Observe that the surface $S(\G,\f)$ has a $\Grp$-action which is conjugate to the action
on $S=S(\G,\w)$, but which acts by dilation. Namely, the $g$-action scales area by a dilation with expansion constant $e^{h(g)}$. It follows that the pullback measure is a $\chi$-Maharam measure where $\chi:\Grp \to \R$ is given by $\chi(g)=-h(g)$.

It remains to show that there are no other $\chi$-Maharam measures. Fix $\chi$, and let $\mu$ be a $\chi$-Maharam measure.
We already know the classification of ergodic invariant measures. Namely, for each group homomorphism $h:\Grp \to \R$, there is a unique ergodic invariant $h$-Maharam measure, which we will denote my $\nu_h$. It then follows that there is a measure $m$ on the collection of such group homomorphisms, $\textrm{Hom}(\Grp,\R)$, so that
$$\mu(A)=\int_{\textrm{Hom}(\Grp,\R)} \nu_h(A)~dm \quad \text{for all measurable $A \subset S$}.$$
We claim that $\mu=\nu_\chi$. For this, it suffices to show that the support of the measure $m$ only includes
$\chi$. Suppose the support included some $h_s \neq \chi$. Since $h_s \neq \chi$, there is a $g$ so that
$h_s(g)>\chi(g)$. Define the constant $k=\frac{h_s(g)+\chi(g)}{2}$ and the set 
$$H=\{h \in \textrm{Hom}(\Grp,\R)~:~h(g) \geq k\}.$$
Then, $m(H)=\epsilon$ for some $\epsilon>0$. Let $D$ be the fundamental domain as in the definition of $\chi$-Maharam. Then $\mu(D)=1$ and $\nu_h(D)=1$ for all $h$. For each integer $n>0$, we have
$$
\begin{array}{rcl}
  \mu \circ g^{-n}(D) & = & \int_{\textrm{Hom}(\Grp,\R)} \nu_h \circ g^{-n}(D)~dm 
  	\geq \int_{H} \nu_h \circ g^{-n}(D)~dm\\
  & = &  \int_{H} e^{n h(g)} \nu_h(D)~dm=\int_{H} e^{n h(g)}~dm \geq 
                 \int_{H} e^{n k}~dm = e^{n k} \epsilon. 
\end{array}
$$
As $k>\chi(g)$, by taking $n$ sufficiently large we can guarantee that
$\mu\circ g^{-n}(D)> e^{n \chi(g)}$. But, we must have equality here for $\mu$ to be $\chi$-Maharam.
\end{proof}

\begin{proof}[Discussion of proof of Theorem \ref{thm:nilpotent}]
Let $T_\alpha$ be a skew rotation defined using a nilpotent group 
$\Grp$ with generators $\gamma_1, \ldots, \gamma_n$ satisfying equation \ref{eq:relation}. 
Assume that the unit vector in direction $(\alpha-\frac{1}{n}, \frac{1}{n})$ is $n$-renormalizable.
Then as described by Proposition \ref{prop:graph}, $T_\alpha$ arises from a return map of straight-line flow $F_{\btheta}^t$ to a section on a surface $S=S(\G,\w_{\frac{1}{n}})$ 
 in an $n$-renormalizable direction $\btheta$. The description of the locally finite ergodic invariant measures of $F^t_\btheta$ given in Theorem \ref{thm:nilpotent cover} gives rise to a similar characterization for $T_\alpha$. This characterization is given in the statement of Theorem \ref{thm:nilpotent}.
\end{proof}

We establish one more corollary to cover the case where we have a lot of different twistable cylinder decompositions. We remark that the papers \cite{HS09} and \cite{HW10} give many examples of $\Z$-covers of closed translation surfaces which admit a twistable cylinder decompositions in a dense set of directions.

\begin{corollary}
\name{cor:density}
Suppose that $S$ is a $\Grp$-cover of a closed translation surface, where $\Grp$ is nilpotent. Suppose also that $S$ has twistable cylinder decompositions in a dense set of directions. Then, there is a dense set of directions $\Theta$ 
of Hausdorff dimension larger than half
so that statements (1) and (2) of Theorem \ref{thm:nilpotent cover} are satisfied for all $\bm \theta \in \Theta$. 
\end{corollary}
\begin{proof}[Sketch of proof]
Fixing any pair of decompositions $\sC$ and $\sD$, we obtain a collection of directions for which Theorem \ref{thm:nilpotent cover}. This uses the discussion at the end of Appendix \ref{sect:cylinder decompositions}. It is either the $(\sC,\sD)$-renormalizable directions, or renormalizable directions of the pair of decompositions obtained by subdividing each cylinder. In either case, the set of renormalizable directions accumulates on the directions of the two cylinder decompositions $\sC$ and $\sD$. We take $\Theta$ to be the union of all such directions over all pairs of decompositions. Because of this accumulation, $\Theta$ is dense. The statement about Hausdorff dimension follows from Remark \ref{rem:renormalizable directions}.
\end{proof}

\subsection{Example: The Ehrenfest Wind-tree model}
\name{sect:Ehrenfest}
In \cite{HW80}, Hardy and Weber began the study billiards in the plane with a periodic family of rectangular barriers. We will follow the treatment of this dynamical system given by Hubert,  Leli\`evre and Troubetzkoy \cite{HLT11}. Further works on these systems include \cite{Troubetzkoy10} and \cite{DHL11}.
These systems are parameterized by a pair of real numbers $a$ and $b$ taken from the interval $(0,1)$. For each $m,n \in \Z$, define rectangle $R_{m,n}=(m,m+a) \times (n,n+b)$. We consider billiards in the table $T_{a,b}=\R^2 \sm \bigcup_{m,n \in \Z} R_{m,n}$. 

The billiard flow on $T_{a,b}$ decomposes into invariant sets corresponding to the fact that a single billiard trajectory can travel in only four directions. These invariant sets are given the structure of a translation surface by the Zemljakov-Katok unfolding construction \cite{ZK}. The billiard flow restricted to any of these invariant sets is conjugate to a straight-line flow on a translation surface $X_{a,b}$. The surface $X_{a,b}$ has a co-compact $\Z^2$ action.
See \cite{HLT11} for more details. For the following, see \cite[Theorem 1]{HLT11}:

\begin{theorem}[Periodic directions in the Wind-tree model]
\name{thm:hlt}
Suppose $a$ and $b$ are rational numbers in $(0,1)$ which can written as the ratio of two integers with odd numerator and even denominator. 
Then, there are twistable cylinder decompositions in a dense set directions on $X_{a,b}$.
\end{theorem}

It should be noted that Theorem \ref{thm:hlt} is quite delicate. It is also shown in \cite[Theorem 2]{HLT11} that when $a$ and $b$ can both be written as rationals with even numerator and odd denominator, then $X_{a,b}$ never admits a decomposition into cylinders.

Because of the above theorem, Corollary \ref{cor:density} can then be specialized to the following.

\begin{corollary}[Ergodic directions in the wind-tree model]
\name{cor:Ehrenfest}
Let $a$ and $b$ be as in the above theorem. Then, there is a dense set of directions of Hausdorff dimension larger than half for which the billiard flow on $T_{a,b}$ is ergodic. For each of these directions, statements (1) and (2) of Theorem \ref{thm:nilpotent cover} hold for the surface $X_{a,b}$ as well.
\end{corollary}

The conclusion here should be contrasted with the work of Fr{\c{a}}czek and Ulcigrai, where it is shown
that when $a$ and $b$ are rational numbers, the billiard flow on $T_{a,b}$ is not ergodic in almost every direction \cite[Theorem 1.2]{FU14}. (This result is also shown to hold for almost every $a$ and $b$.)
\section{Unique ergodicity}
\name{sect:unique ergodicity}

Recall that a dynamical system is said to be uniquely ergodic if there is only one invariant probability measure. 
We have been considering two dynamical systems associated to our surfaces $S(\G,\w)$: the infinite IETs arising from the return maps to the horizontal edges of the rectangles making up the surface, and the straight-line flows on these surfaces. In both cases we will show that when an
invariant probability measure is unique whenever it exists.

The surface in Figure \ref{fig:triple staircase} and the surfaces $X_\alpha$ defined for a rational parameter $0<\alpha<1$ defined in \cite[\S 3]{Chamanara04} (see \cite[Proposition 11]{Chamanara04} for a description of the multi-twists of $X_\alpha$) and orientation covers of the surfaces $X$ and $Y$ of \cite{CGL} represent surfaces our unique ergodicity theorems apply to. We remark that Trevi{\~n}o has a criterion for ergodicity of the straight-line flow which likely also applies here \cite[Theorem 2]{Trev14}.

\subsection{Unique ergodicity of infinite IETs}
Let $\G$ be an infinite, connected, bipartite, ribbon graph with bounded valance and no vertices of valance one,
and let $\w$ be a positive eigenfunction with eigenvalue $\lambda$ for the adjacency operator on $\G$. Let $X$ denote the union of the horizontal edges of the rectangles making up the surface $S(\G,\w)$. (See Definition \ref{def:S} of the surface $S(\G,\w)$ for a description of these rectangles.) Choose a $\lambda$-renormalizable direction $\btheta$, and let $T:X \to X$ be the infinite IET given by the return map of the straight-line flow $F^t_\btheta:S(\G,\w) \to S(\G,\w)$ to the section $X$. 
Then by the Measure Characterization Theorem (Theorem \ref{thm:ergodic2}), the locally finite ergodic transverse measures to the foliation $\F_\btheta$ are given by a pullback of a Lebesgue transverse measure of surfaces $S(\G,\f)$ where $\f:\R^\V \to \R$ iterates over the extremal positive eigenfunctions of the adjacency operator. By restricting these measures to $X$, we obtain the locally-finite ergodic invariant measures $\mu_\f$ for $T:X \to X$. 

The total measure $\mu_\f(X)$ is related to the $\ell^1$-norm of $\f$, which we define to be
$$\|\f\|_1=\sum_{\vv \in \V} |\f(\v)|.$$
\begin{proposition}
We have $\mu_\f(X)<\infty$ if and only if $\|\f\|_1<\infty$. 
\end{proposition}
\begin{proof}
Let $e=\overline{\va \vb}$ be an edge of $\G$. Then, $R_e$ is a rectangle of $S(\G,\w)$. (See Definition \ref{def:S}.) Let $L_e$ denote the lower edge of $R_e$. Because $\mu_\f$ is a pullback of Lebesgue transverse measure on $S(\G,\f)$ in some non-horizontal direction, there is a constant $c>0$ depending on $\f$ so that $\mu_\f(L_e)=c \f(\vb)$. Since $X$ is the union of all lower edges of the rectangles $R_e$, we have the inequality 
$$\mu_\f(X)=\sum_{\vb \in \Beta} c \f(\vb) < c \|\f\|_1.$$
Let $n$ be the largest valance of a vertex in $\G$, and let $\lambda$ be the eigenvalue of $\f$. Then, if the vertex $\va \in \Alpha$ is adjacent to vertex $\vb \in \Beta$, we have 
$\f(\va)< \lambda \f(\vb)$. Since every vertex $\vb \in \Beta$ is adjacent to at most $n$ vertices in $\Alpha$, we have 
$$\|\f\|_1=\sum_{\va \in \Alpha} \f(\va)+\sum_{\vb \in \Beta} \f(\vb)\leq (1+n \lambda)\sum_{\vb \in \Beta} \f(\vb)=\frac{1+n \lambda}{c} \mu_\f(X).$$
Together these inequalities imply that $\mu_\f(X)<\infty$ if and only if $\|\f\|_1<\infty$. 
\end{proof}

\begin{corollary}[Unique ergodicity of $T$]
Let $\G$ be an infinite, connected, bipartite, ribbon graph with bounded valance and no vertices of valance one,
and let $\w$ be a positive eigenfunction with eigenvalue $\lambda$ for the adjacency operator on $\G$.
Let $\btheta$ be a $\lambda$-renormalizable direction.
Then, There is at most one invariant probability measure for the first return map $T:X \to X$ of
the the straight-line flow $F^t_\btheta:S(\G,\w) \to S(\G,\w)$ 
to the union $X$ of horizontal edges of rectangles making up $S(\G,\w)$. 
\end{corollary}
\begin{proof}
Suppose $T:X \to X$ has an invariant probability measure. Then it has an ergodic one, $\mu_\f$. By the prior proposition, $\f$ has finite $\ell^1$-norm. Therefore, it has finite $\ell^2$ norm. In this case, Theorem \ref{thm:l2} guarantees that the eigenvalue of $\f$ is the spectral radius, $r$, of the adjacency operator, $\A$. We also get that $\A$ is $r$-positive. Then by Theorem \ref{thm:recurrent case}, we see that the equation $A\f=r \f$ has a unique solution up to scaling. It follows that $\mu_\f$ is the only ergodic invariant probability measure for $T$, i.e., $T$ is uniquely ergodic.
\end{proof}

\subsection{Unique ergodicity for the straight-line flow}
Let $\G$ be a graph as in the prior section, and let $\w$ be a positive eigenfunction of the adjacency operator with eigenvalue $\lambda$. We will consider the straight-line flow $F_\btheta^t$ on $S(\G,\w)$ in a $\lambda$-renormalizable direction $\btheta$. Let $\f$ be an extremal positive eigenfunction, and let $\nu_\f$ be the $F_\btheta^t$-invariant measure on $S(\G,\w)$ obtained as in the Measure Characterization Theorem by pulling back the Lebesgue transverse invariant measure from $S(\G,\f)$ and then integrating over the leaves. 

\begin{lemma}
\name{lem:l2}
The $\nu_\f$ measure of the surface $S(\G,\w)$ is finite if and only if the $\ell^2$-inner product, 
$$\w \cdot \f=\sum_{\vv \in \V} \w(\vv) \f(\vv)$$
is finite.
\end{lemma}
\begin{proof}
We assume that $\A\w=\lambda \w$ and $\A \f=\lambda'\w$. The measure $\nu_\f$ is locally the product of the pullback of Lebesgue measure on the leaves of the foliation $\F_\btheta$ and the pullback of the Lebesgue transverse measure, $\mu$, on $S(\G,\f)$ to the foliation in some direction $\btheta'$.

We need to compute $\nu\big(S(\G,\w)\big)$ using this local product structure. To do this, let $\cyl_\va$ be a horizontal cylinder
in the surface $S(\G,\w)$, and let $\gamma_\va$ be a horizontal circle winding around the cylinder. By considering that the transverse measure
$\mu(\gamma_\va)$ should be the same as the sum of the $\mu$-measures of edges on the bottom of $\cyl_\va$, we see that this can be computed by looking
at the surface $S(\G, \f)$. Let $\theta'$ be the angle made with the horizontal by $\btheta'$. We have 
$$\mu(\gamma_\va)=\sin(\theta') \sum_{\vb \sim \va} \f(\vb)=\A(\f)(\va) \sin (\theta')=\lambda' \f(\va) \sin (\theta').$$ 
The Lebesgue measures of the (connected components of) intersections of a 
leaf of the foliation in direction $\bm \theta$ on $S(\G,\w)$ and the cylinder $\cyl_\va$ are
given by $\w(\va)/\sin(\theta)$, where $\theta$ is the angle made with the horizontal by the vector $\bm \theta$. Since the cylinder $\cyl_\va$ can be written as the product of a circle $\gamma_\va$ with a intersection of a leaf with $\cyl_\va$ we have
$$\nu\big(S(\G,\w)\big)=\sum_{\va \in \Alpha} \nu(\cyl_\va)=
\sum_{\va \in \Alpha} \frac{\lambda' \sin (\theta') \f(\va) \w(\va)}{\sin(\theta)}=\frac{\lambda' \sin (\theta')}{\sin(\theta)} \sum_{\va \in \Alpha} \f(\va)\w(\va).$$
By swapping the roles of horizontal and vertical, we also obtain an expression for $\nu\big(S(\G,\w)\big)$ in terms of $\sum_{\vb \in \Beta} \f(\vb)\w(\vb)$. The conclusion follows. \end{proof}

Now, we will consider when the $\ell^2$-inner product of two eigenfunctions is finite.
\begin{lemma}
Suppose $\A(\f)=\lambda \f$ and $\A(\g)=\lambda' \g$. Then if $\lambda' \not \neq \lambda$, whenever
$\ell^2$-inner product of $\f$ and $\g$ exists it equals zero.
\end{lemma}
\begin{proof}
Suppose $\A(\f)=\lambda \f$ and $\A(\g)=\lambda' \g$. Let $s$ denote the $\ell^2$-inner product $\sum_{\vv \in \V} \f(\vv) \g(\vv)$, and assume the sum converges. Since $\A(\f)=\lambda \f$,
$$\lambda s= \sum_{\vv \in \V} [\A(\f)(\vv)] \g(\vv)=\sum_{\vv \in \V} \sum_{\vw \sim \vv} \f(\vw) \g(\vv)=
\sum_{\vw \in \V} \sum_{\vv \sim \vw} \f(\vw) \g(\vv).$$
Now observe that fixing any $\vw \in \V$, we have $\sum_{\vv \sim \vw} \f(\vw) \g(\vv)=\f(\vw)[\A(\g)(\vw)]=\lambda' \f(\vw) \g(\vw)$. 
Therefore, we have 
$$\lambda s=\lambda' \sum_{\vw \in \V}  \f(\vw) \g(\vw)=\lambda' s.$$
Since $\lambda'\neq \lambda$, we know that $s=0$.
\end{proof}

\begin{theorem}
$\G$ be an infinite, connected, bipartite, ribbon graph with bounded valance and no vertices of valance one,
and let $\w$ be a positive eigenfunction with eigenvalue $\lambda$ for the adjacency operator on $\G$.
Let $\btheta$ be a $\lambda$-renormalizable direction.
Then, if $S(\G,\w)$ has finite area, then $F_\btheta^t$ is uniquely ergodic. 
\end{theorem}
\begin{proof}
Suppose $S(\G,\w)$ has finite area. Then by Lemma \ref{lem:l2}, we have $\w \cdot \w<\infty$. So, $\w$ is a 
positive eigenfunction in $\ell^2$. So by 
Theorem \ref{thm:l2}, $\lambda$ is the spectral radius and $\A$ is $r$-positive. By Theorem \ref{thm:recurrent case}, we see that the equation $A\f=\lambda \f$ has a unique solution up to scaling. So, if $\f$ is a positive eigenfunction of $\f$ which is not a scalar multiple of $\w$, it must have an eigenvalue other than $\lambda$. 
Since $\f$ and $\w$ are both positive, we can not have $\f \cdot \w=0$. We conclude that $\f \cdot \w=+\infty$. By Lemma \ref{lem:l2}, we conclude that the $F_\btheta^t$-invariant measure $\nu_\f$ assigns infinite measure to the surface. Thus, scalar multiples of Lebesgue measure, $\nu_\w$, are the only finite $F_\btheta^t$-invariant measures.
\end{proof}

\section*{Acknowledgments}
The author would like to thank Corinna Ulcigrai and Barak Weiss for pointing out a number of minor errors and typos in the prior version of this paper. This version also benefited from helpful comments from an anonymous referee which have helped to make the paper more readable.

\section*{List of Notations}
\vspace{1em}
\label{sect:notations}

\begin{center}
\begin{longtable}{@{}llr@{}} 
\toprule
Notation & Brief description & Page\\ \midrule
\endfirsthead
\midrule
Notation & Brief description & Page\\ \midrule
\endhead
\endfoot
\bottomrule
\endlastfoot
$S$ & a translation surface defined as a union of polygons & \pagelink{not:S}\\
$\Circ$ & the unit circle in $\R^2$ & \pagelink{not:Circ}\\
$F_{\bm \theta}^t$ & the straight-line flow & \pagelink{not:F} \\
$\bm \theta$ & an element of $\Circ$ (called a {\em direction}) & \pagelink{not:btheta} \\
$D(\phi)$ & the derivative of the affine automorphism $\phi$ & \pagelink{not:D} \\
$\Aff(S)$ & affine automorphism group of $S$ & \pagelink{not:Aff} \\
$\G$ & an infinite, connected, bipartite, ribbon graph with bounded valance & \pagelink{not:G} \\
$\V$ & the set of vertices of $\G$ & \pagelink{not:V} \\
$\E$ & the set of edges of $\G$ & \pagelink{not:E} \\
$\vv, \vw$ & elements of $\V$ (i.e., vertices of $\G$) & \pagelink{not:v}\\
$\Alpha,\Beta$ & subsets of $\V$ which make $\G$ bipartite & \pagelink{not:Alpha}\\
$\va, \vb$ & elements of $\Alpha$ and $\Beta$ & \pagelink{not:a}\\
$\alpha$, $\beta$ & the projections $\alpha:\E \to \Alpha$
and $\beta:\E \to \Beta$ & \pagelink{not:alpha}\\
$p_\vv$ & permutation of edges containing $\vv \in \V$ which make $\G$ a ribbon graph & \pagelink{not:Pv}\\
$\R^\V$ & the set of all functions from $\V$ to $\R$ & \pagelink{not:RV} \\
$\A$ & the adjacency operator, $\A:\R^\V \to \R^\V$ & \pagelink{def:adj} \\
$\f$, $\g$ & elements of $\R^\V$ & \pagelink{not:f} \\
$\w$ & a element of $\R^\V$ which is a positive eigenfunction of $\A$  & \pagelink{not:w} \\
$\lambda$ & eigenvalue of $\w$ & \pagelink{not:lambda}\\
$\East$, $\North$ & east and north permutations of $\E$ & \pagelink{not:East} \\
$S(\G,\w)$ & surface built from rectangles using $\G$ and $\w$ & \pagelink{not:S2}\\
$R_e$ & rectangle of $S(\G,\w)$ associated to $e \in \E$ & \pagelink{not:R}\\
$\cyl_\vv$ & horizontal or vertical cylinder of $S(\G,\w)$
associated to $\vv \in \V$ & \pagelink{eq:cyl}\\
$G$ & non-abelian free group with two generators & \pagelink{not:Gfree} \\
$h,v$ & generators of $G$ & \pagelink{not:Gfree} \\
$\rho_\lambda$ & group representation $\rho_\lambda:G \to \SL(2,\R)$ depending on $\lambda$ & \pagelink{not:rho}\\
$g$ & arbitrary element of the free group $G$ & \pagelink{not:g} \\
$\Phi$ & group endomorphism of $G$ into $\Aff\big(S(\G,\w)\big)$ & \pagelink{not:Phi} \\
$\R \P^1$ & the real projective line, $\R \P^1=\R^2 \smallsetminus \{\0\} /\R$ & \pagelink{not:RP1} \\
$\Rn_\lambda$ & the $\lambda$-renormalizable directions in $\Circ$ & \pagelink{not:Rn} \\
$\langle g_n \rangle$ & a sequence of elements of $G$ which form a geodesic ray & \pagelink{not:ray}\\
$\btheta(\langle g_n \rangle, \lambda)$ & a $\lambda$-renormalizable direction with $\lambda$-shinking sequence
$\langle g_n \rangle$ & \pagelink{not:theta2}\\
$E_\lambda$ & non-negative solutions to $A(\f)=\lambda \f$ & \pagelink{not:E lambda 1}, \pagelink{not:E lambda} \\
$\F_{\bm \theta}$ & foliation by orbits of the straight line flow in direction $\bm \theta$ & \pagelink{not:foliation}\\
$\hat \F_\btheta$ & leaf space formed by splitting singular leaves of $\F_{\bm \theta}$ & \pagelink{not:split foliation}, \pagelink{not:split foliation2}\\
$\wedge$ & standard wedge product in $\R^2$ & \pagelink{eq:wedge} \\
$V$ & collection of singularities of $S$ & \pagelink{not:singularities}\\
$\M_\btheta$ & space of transverse measures to $\hat \F_\btheta$ & \pagelink{not:measures} \\
$H_1(S, V, \R)$ & space of homology classes of curves in $S/V$ & \pagelink{not:homology} \\
$\Coh$ & space of linear maps $H_1(S, V, \R) \to \R$ & \pagelink{not:cohomology} \\
$\Psi_{\bm \theta}$ & a linear map $\M_\btheta \to \Coh$ & \pagelink{not:Psi}\\
$\sigma$ & a saddle connection on $S$ & \pagelink{not:sigma} \\
$\hol(\sigma)$ & holonomy of a saddle connection & \pagelink{not:holonomy}\\
$\sgn$ & the signum function $\R \to \{-1,0,1\}$ & \pagelink{not:signum}\\
$E$ & edges of rectangles making up $S$ & \pagelink{not:Edges} \\
$\phi_\ast$ & action of an affine automorphism on $\M_\btheta$ or $H^1$ & \pagelink{not:phi ast} \\
$\hom{\cdot}$ & homology class in $H_1(S,V,\Z)$ & \pagelink{not:hom} \\
$\Phi^{g}_\ast$ & action of the affine automorphism $\Phi^g$ on $H^1$ & \pagelink{not:Phi ast} \\
$\Ho$, $\Vo$ & operators on $\R^\V$ & \pagelink{eq:Ho}\\
$\Upsilon^G$ & $G$-action on $\R^\V$ generated by $\Ho$ and $\Vo$ & \pagelink{not:Upsilon}\\
$\Xi$ & linear embedding of $\R^\V$ into $H^1$ & \pagelink{not:Xi} \\
$i$ & the intersection pairing $H_1(S,V,\Z) \times H_1(S \smallsetminus V,\Z) \to \Z$ & \pagelink{not:intersection} \\
$\Surv_{\bm \theta}$ & subset of $\btheta$-survivors in $\R^\V$ & \pagelink{not:survivors} \\
$\SP$ & the set of sign pairs: $\SP=\{(1,1),(1,-1),(-1,1),(-1,-1)\}$ & \pagelink{not:sign pairs} \\
$++$,$+-$, \ldots & abbreviation for sign pairs & \pagelink{not:sign pairs 2} \\
$\Q_s$ & the four quadrants in $\R^2$, parameterized by $s \in \SP$ & \pagelink{not:quadrants 1} \\
$\cl(X)$ & the (topological) closure of $X$ & \pagelink{not:closure} \\
$\hat \Q_s$ & the four quadrants in $\R^\V$, parameterized by $s \in \SP$ & \pagelink{not:quadrants 2} \\
$P_\f$ & parameterization of a $2$-plane inside of $\R^\V$  & \pagelink{eq:P} \\
$\bar \cdot$ & the involution of $\R^2$ or $\SP$ given by $(x,y) \mapsto (y,x)$; & \pagelink {not:bar1}\\
 & also, the involution of $G$ which interchanges $v$ with $h$ & \pagelink {not:bar2}\\
$\pi_{U}(\f)$ & projection of $\f \in \R^V$ to functions supported on $U \subset \V$ & \pagelink{eq:pi_U} \\
$\R^\V_c$ & the set of finitely supported functions in $\R^\V$ & \pagelink{not:RVc} \\
$\langle , \rangle$ & a bilinear pairing $\R^\V \times \R^\V_c \to \R$ analogous to the dot product & \pagelink{eq:pairing} \\
$\gamma$ & an automorphism of $G$ satisfying $\langle \Act^g \f, \Act^{\gamma(g)} \x \rangle= \langle \f, \x \rangle$ & \pagelink{eq:gamma} \\
${\mathcal V}_\lambda$ & Martin compactification of the vertex set $\V$ of $\G$ & \pagelink{not:V lambda 1}, \pagelink{not:hat V lambda 2} \\
${\mathcal M}_\lambda$ & The Martin boundary $\V_\lambda \smallsetminus \V$ & \pagelink{not:M lambda 1}, \pagelink{not:M lambda 2} \\
$\zeta$ & a point in the Martin boundary ${\mathcal M}_\lambda$  & \pagelink{not:zeta}, \pagelink{not:zeta2} \\
$\k_\zeta$ & the positive eigenfunction in $\R^\V$ associated to $\zeta \in {\mathcal M}_\lambda$  & \pagelink{not:k zeta}, \pagelink{not:k zeta2} \\
${\mathcal M}^{\textit{min}}_\lambda$ & the minimal Martin boundary in ${\mathcal M}_\lambda$ & \pagelink{not:M lambda min}, \pagelink{not:M lambda min 2} \\
$\nu_\f$ & measure on ${\mathcal M}^{\textit{min}}_\lambda$ associated to the positive eigenfunction $\f$ & \pagelink{not:nu f} \\
$\widehat E_s$ & solutions to $A^2(\f)=\lambda^2 \f$ with $\f \in \hat \Q_s$ where $s \in \SP$ & \pagelink{not:hat E s} \\
$\f_\Alpha$, $\f_\Beta$ & elements of $\R^\V$ derived from $\f|_\Alpha$ and $f|_\Beta$, where $\f \in \R^\V$ & \pagelink{eq:fsub} \\
$M$ & a space of signed Borel measures on ${\mathcal M}_\lambda$ & \pagelink{not: M} \\
$M^+$, $M^-$ & collections of positive and negative measures in $M$, respectively & \pagelink{not: M +}\\
${\mathcal N}$ & linear map $\bigcup_{s \in \SP} \widehat E_s \to M^2$ & \pagelink{not:Nu} \\
$\Shrink_\lambda(g)$ & set of directions shrunk when moving from $e$ to $g \in G$ & \pagelink{not:shrink} \\
$\Exp_\lambda(g)$ & set of directions expanded when moving from $e$ to $g \in G$ & \pagelink{not:expand} \\
$\pi_\Circ$ & projection $\R^2 \smallsetminus \{\0\} \to \Circ$ recovering a vectors direction & \pagelink{not:pi circ}\\
$\H^2$ & the hyperbolic plane, $\SO(2) \setminus \SL(2, \R)$ &  \pagelink{not:hyperolic plane} \\
$\sa^g(s)$ & expanding sign action $G \times \SP \to \SP$ & \pagelink{not:sa} \\
$R$ & rotation of $\R^2$ by angle $\frac{\pi}{2}$ & \pagelink{eq:R} \\
$r$ & action on $\SP$ induced by the action of $R$ on quadrants & \pagelink{eq:r} \\
$\Zem$ & linear embedding of $H_1(S,V,\R)$ into $\R^\V_c$ & \pagelink{not:Zem} \\
$\textit{val}(\vv)$ & the valance of the vertex $\vv \in \V$ & \pagelink{not:val} \\
$\Ho_\y, \Vo_\y$ & affine perturbations of $\H$ and $\V$ operators & \pagelink{not:perturb generators} \\
$\Chi_\y$ & affine perturbation of the action $\Upsilon^G$ & \pagelink{not:perturb action} \\
$\Omega_n$ & renormalizable parameters for certain special skew products  & \pagelink{eq:Omega} \\
\end{longtable}
\end{center}
\label{sect:notations end}


\bibliographystyle{alpha}
\bibliography{/home/pat/active/my_papers/bibliography}
\end{document}